\documentclass[10pt]{amsart}

\usepackage{latexsym,exscale,enumerate,amsfonts,amssymb, ulem, xparse, mathtools}
\usepackage{amsmath,amsthm,amsfonts,amssymb,amscd, stmaryrd,textcomp,mathscinet}
\usepackage{hyperref}
\usepackage{ulem}
\usepackage{bbold}
\usepackage{tocvsec2}

\usepackage{libertine}
\usepackage[usenames,dvipsnames]{xcolor}
\colorlet{green}{black!30!green} 

\usepackage[all]{xy}
\SelectTips{cm}{}

\usepackage{graphicx}

\usepackage{tikz}
\usetikzlibrary{calc}
\usetikzlibrary{decorations.markings}
\usetikzlibrary{decorations.pathreplacing}
\usetikzlibrary{arrows,shapes,positioning}
\tikzstyle directed=[postaction={decorate,decoration={markings,
    mark=at position #1 with {\arrow{>}}}}]
\tikzstyle rdirected=[postaction={decorate,decoration={markings,
    mark=at position #1 with {\arrow{<}}}}]
\tikzset{anchorbase/.style={baseline={([yshift=-0.5ex]current bounding box.center)}}}

\tikzset{
    partial ellipse/.style args={#1:#2:#3}{
        insert path={+ (#1:#3) arc (#1:#2:#3)}
    }
}

\newcommand{\Sym}{\mathrm{Sym}}

\def\cal#1{\mathcal{#1}}%

\def\1{\mathbbm{1}}%

\theoremstyle{plain}
\newtheorem{theorem}{Theorem}
\newtheorem{theorem*}[theorem]{Theorem*}
\newtheorem{theorem**}[theorem]{Theorem*}
\numberwithin{theorem}{section}

\newtheorem{conjecture}[theorem]{Conjecture}
\newtheorem{corollary}[theorem]{Corollary}
\newtheorem{corollary*}[theorem]{Corollary*}
\newtheorem{corollary**}[theorem]{Corollary*}

\newtheorem{lemma}[theorem]{Lemma}
\newtheorem{lemma*}[theorem]{Lemma*}
\newtheorem{lemma**}[theorem]{Lemma*}

\newtheorem{notation*}[theorem]{Notation*}

\newtheorem{proposition}[theorem]{Proposition}
\newtheorem{proposition*}[theorem]{Proposition*}
\newtheorem{proposition**}[theorem]{Proposition*}

\theoremstyle{definition}
\newtheorem{definition}[theorem]{Definition}
\newtheorem{definition*}[theorem]{Definition*}
\newtheorem{definition**}[theorem]{Definition*}
\newtheorem{example}[theorem]{Example}
\newtheorem{question}[theorem]{Question}
\newtheorem{questions}[theorem]{Questions}

\theoremstyle{remark}
\newtheorem{remark}[theorem]{Remark}

\newcommand{\slnn}[1]{\mf{sl}_{#1}}

\newcommand{\glnn}[1]{\mf{gl}_{#1}}

\newcommand{\sh}{w}
\newcommand{\cat}[1]{\mathchoice
  {\ensuremath{\mbox{\bfseries {\upshape {#1}}}}}
  {\ensuremath{\mbox{\bfseries {\upshape {#1}}}}}
  {\scalebox{.7}{\ensuremath{\mbox{\bfseries {\upshape {#1}}}}}}
  {\scalebox{.5}{\ensuremath{\mbox{\bfseries {\upshape {#1}}}}}}%
  }

\newcommand{\Kar}{\operatorname{Kar}}

\newcommand{\Kh}{\cat{Kh}}
\newcommand{\Bl}{\cat{Bl}}
\newcommand{\APS}{\cat{APS}}

\newcommand{\Web}[1][]{#1\cat{Web}}

\newcommand{\Webq}[1][]{#1\cat{Web}_{q}}

\newcommand{\wrap}{D}

\newcommand{\A}{\cat{A}}

\newcommand{\AWeb}[1][]{#1\cat{A}\cat{Web}}
\newcommand{\AWebq}[1][]{#1\cat{A}\cat{Web}_{q}}

\newcommand{\essAWeb}[1][]{#1\AWeb^{\mathrm{ess}}}

\newcommand{\essbAWeb}[1][]{\overline{#1\AWeb}^{\mathrm{ess}}}

\newcommand{\Afoam}[1][]{#1\cat{A}\cat{Foam}}
\newcommand{\AFoam}[1][]{#1\cat{A}\cat{Foam}}

\def\mC{\cat{C}}
\def\basis{\mathrm{B}_T}
\def\basisJW{\mathrm{B}_{S}}
\def\basisstd{\mathrm{B}}

\newcommand{\T}{\cat{T}}
\newcommand{\TWebq}[1][]{#1\T\cat{Web}_{q}}
\newcommand{\TWebA}[1][]{#1\T\cat{Web}_{A}}
\newcommand{\Tfoam}[1][]{#1\T\cat{Foam}}
\newcommand{\essTfoam}[1][]{#1\Tfoam^{ess}}

\newcommand{\Ttanweb}{\T\cat{TanWeb}}

\newcommand{\Stwo}{\cat{S}^\mathbf{2}}
\newcommand{\Su}{\cat{S}}

\newcommand{\SWebq}[1][]{#1\Su\cat{Web}_{q}}
\newcommand{\SWebA}[1][]{#1\Su\cat{Web}_{A}}
\newcommand{\Sfoam}[1][]{#1\Su\cat{Foam}}

\newcommand{\Sfoamred}[1][]{#1\Sfoam^{\mathrm{red}}}
\newcommand{\Sfoamor}[1][]{#1\Sfoam^{\mathrm or}}
\newcommand{\SCob}{\Su\cat{Cob}}
\newcommand{\Slink}{\Su\cat{Link}}
\newcommand{\Slinko}{\Su\cat{Link}^{\circ}}
\newcommand{\Stanweb}{\Su\cat{TanWeb}}
\newcommand{\Stanwebo}{\Su\cat{TanWeb}^{\circ}}

\newcommand{\stwo}{*_{\wedge}}
\newcommand{\w}[1]{\wedge^{(#1)}}

\newcommand{\twoFoam}{\cat{Foam}}
\newcommand{\Vect}{\cat{Vect}}

\newcommand{\HC}{\operatorname{K}}

\newcommand{\Hom}{{\rm Hom}}

\renewcommand{\to}{\rightarrow}

\newcommand{\id}{{\rm id}}

\newcommand{\End}{{\rm End}}

\def\mf{\mathfrak}
\def\shuffle{\,\raise 1pt\hbox{$\scriptscriptstyle\cup{\mskip
               -4mu}\cup$}\,}

\numberwithin{equation}{section}

\def\comm#1{}%
\def\emph#1{{\sl #1\/}}

\let\phi=\varphi
\let\theta=\vartheta
\let\epsilon=\varepsilon

\usepackage{bbm}
\def\C{{\mathbbm C}}
\def\N{{\mathbbm N}}
\def\R{{\mathbbm R}}
\def\Z{{\mathbbm Z}}
\def\Q{{\mathbbm Q}}

\newcommand{\Ss}{S}

\def\cal#1{\mathcal{#1}}%
\def\1{\mathbbm{1}}%
\def\nn{\notag}

\def\la{\langle}
\def\ra{\rangle}

\newcommand{\torus}[2]{
\draw [green, thick, directed=.25]  (0,0) to (#1,0);
\draw [green, thick, directed=.25] (0,#2)to (#1,#2);
\draw [red, thick,rdirected=.3, rdirected=.25] (#1,#2) to (#1,0);
\draw [red, thick, rdirected=.3, rdirected=.25] (0,#2) to (0,0);}

\newcommand{\torusfront}[3]{
\draw [green, thick, directed=.25] (0,0)to(#1,0) ;
\draw [red, thick, rdirected=.3, rdirected=.25] (#1+0.5*#2,#2) to (#1,0);
\draw [green, thick, directed=.25] (0,0+#3) to (#1,0+#3);
\draw [green, thick, directed=.25] (0+0.5*#2,#2+#3) to (#1+0.5*#2,#2+#3);
\draw [red, thick, rdirected=.3, rdirected=.25] (#1+0.5*#2,#2+#3) to (#1,0+#3);
\draw [red, thick, rdirected=.3, rdirected=.25] (0+0.5*#2,#2+#3) to (0,0+#3);
\draw (0,0) to (0,0+#3);
\draw(#1,0) to(#1,0+#3);
}

\newcommand{\torusback}[3]{
\draw [opacity=.5] (0.5*#2,#2)  to  (0.5*#2,#2+#3); 
\draw (#1+0.5*#2,#2) to (#1+0.5*#2,#2+#3);
\draw [green,opacity=.5, thick, directed=.25] (0+0.5*#2,#2)to (#1+0.5*#2,#2) ;
\draw [red,opacity=.5, thick, rdirected=.3, rdirected=.25] (0+0.5*#2,#2) to (0,0);
}

\newcommand{\annfrontd}[3]{
\draw (0,0)to(#1,0) ;
\draw [red, thick, rdirected=.3, rdirected=.25] (#1+0.5*#2,#2) to (#1,0);
\draw [dashed] (0,0+#3) to (#1,0+#3);
\draw [dashed] (0+0.5*#2,#2+#3) to (#1+0.5*#2,#2+#3);
\draw [dashed,red, thick, rdirected=.3, rdirected=.25] (#1+0.5*#2,#2+#3) to (#1,0+#3);
\draw [dashed, red, thick, rdirected=.3, rdirected=.25] (0+0.5*#2,#2+#3) to (0,0+#3);
\draw (0,0) to (0,0+#3);
\draw(#1,0) to(#1,0+#3);
}

\newcommand{\annback}[3]{
\draw [opacity=.5] (0.5*#2,#2)  to  (0.5*#2,#2+#3); 
\draw (#1+0.5*#2,#2) to (#1+0.5*#2,#2+#3);
\draw [opacity=.5] (0+0.5*#2,#2)to (#1+0.5*#2,#2) ;
\draw [red,opacity=.5, thick, rdirected=.3, rdirected=.25] (0+0.5*#2,#2) to (0,0);
}

\begin{document}
\allowdisplaybreaks

\title[Khovanov homology and categorification of skein modules]{Khovanov homology and categorification of skein modules}

\author{Hoel Queffelec}
 \address{IMAG\\ Univ. Montpellier\\ CNRS \\ Montpellier \\ France}
\email{hoel.queffelec@umontpellier.fr}

 \author{Paul Wedrich}
\address{Mathematical Sciences Institute\\ The Australian National University \\ Australia}
\email{p.wedrich@gmail.com}
\urladdr{paul.wedrich.at}

\begin{abstract}
For every oriented surface of finite type, we construct a functorial Khovanov homology for links in a thickening of the surface, which takes values in a categorification of the corresponding $\glnn{2}$ skein module. The latter is a mild refinement of the Kauffman bracket skein algebra, and its categorification is constructed using a category of $\glnn{2}$ foams that admits an interesting non-negative grading. We expect that the natural algebra structure on the $\glnn{2}$ skein module can be categorified by a tensor product that makes the surface link homology functor monoidal. We construct a candidate bifunctor on the target category and conjecture that it extends to a monoidal structure. This would give rise to a canonical basis of the associated $\glnn{2}$ skein algebra and verify an analogue of a positivity conjecture of Fock--Goncharov and Thurston. We provide evidence towards the monoidality conjecture by checking several instances of a categorified Frohman-Gelca formula for the skein algebra of the torus. Finally, we recover a variant of the Asaeda--Przytycki--Sikora surface link homologies and prove that surface embeddings give rise to spectral sequences between them.
\end{abstract}

\maketitle
\maxtocdepth{subsection}
\tableofcontents

\section{Introduction}
After Khovanov's categorification of the Jones polynomial \cite{Kh1}, one of the most intriguing problems in categorification and quantum topology has been to extend Khovanov homology to $3$-manifolds other than $\R^3$. Since there is already more than one way to extend the Jones polynomial to an invariant of $3$-manifolds, we can also see different approaches to extending Khovanov homology, each posing different technical challenges and uncovering interesting higher representation-theoretic structure. For example, an extension along the lines of the Witten--Reshetikhin--Turaev construction of 3-manifold invariants has to make sense of categorifying the quantum parameter $q$ at a root of unity, e.g. via (K)hopfological algebra \cite{Kh4}. Another extension in the form of a $(4+\epsilon)$-dimensional TQFT has been outlined by Morrison--Walker using the technology of disk-like $4$-categories~\cite{MW}.

The alternative approach that we follow in the present article seeks to categorify skein modules of $3$-manifolds as introduced by Conway, Przytycki~\cite{Prz1} and Turaev~\cite{Tur}. Skein modules are constructed as abelian groups spanned by links embedded in the $3$-manifold, modulo local relations determined by the Jones polynomial. They can be seen as quantizations of character varieties, see Bullock~\cite{Bul}, and have deep relations to quantum Teichm\"{u}ller theory and cluster algebras, see Bonahon--Wong~\cite{BW, BW1}, Thurston~\cite{Thu} and references therein. Since skein modules of $3$-manifolds can be disassembled along Heegaard splittings, the focus of this paper will be on categorifying the skein modules of thickened Heegaard surfaces.

\subsection{Link homology on surfaces}
Khovanov homology and related functorial invariants of links in $\R^3$ are usually defined and computed via link diagrams\footnote{A notable exception being Witten's approach \cite{Witten}.} and thus depend on a direction of projection, or alternatively, the identification of the ambient manifold $\R^3$ with a thickening of $\R^2$. Conversely, many features of the definition of Khovanov homology directly carry over to the case of other thickened surfaces, and related link invariants have been defined and studied by Asaeda--Przytycki--Sikora~\cite{APS}, Turaev--Turner ~\cite{TT}, Boerner~\cite{Boe} and others. 

The purpose of this article is to develop these constructions into honest functorial invariants of links in thickened surfaces and smooth cobordism between them, which take values in target categories that categorify skein modules, and which have desirable gluing properties that we expect to be useful in defining categorical 3-manifold invariants. Our first result is the following.

\begin{theorem}\label{thm:linkhom} Let $\Su$ be an oriented surface of finite type and let $\Slink$ denote the category with objects given by links embedded in $\Su \times [0,1]$ and morphisms given by oriented link cobordisms, properly embedded  in $\Su \times [0,1]^2$, up to isotopy relative to the boundary. Then there exists a functor 
\[\Su\Kh\colon \Slink \to \HC(\Sfoam),\]
 where the target is the bounded homotopy category of chain complexes in the additive $H_1(\Su)\times \Z$-graded category $\Sfoam$ of $\glnn{2}$ webs and foams in $\Su$, whose Grothendieck group is isomorphic to the $\glnn{2}$ skein module $\SWebq$ of $\Su$. Moreover, $\Su\Kh$ categorifies the evaluation of links from $\Slink$ in $\SWebq$.
\end{theorem}
The main challenge in constructing the functors $\Su\Kh$ is in setting up and understanding an appropriate target category. We use Blanchet's foams~\cite{Blan} to construct the category $\Sfoam$ in Section~\ref{sec:foams} and the definition of the functors $\Su\Kh$ in Section~\ref{sec:BKh} then follows a well-known recipe involving cubes of resolutions built from local pieces associated to crossings, and maps induced by link cobordisms computed via movie presentations (analogues for higher rank Khovanov--Rozansky homologies can be constructed from results in Ehrig--Tubbenhauer--Wedrich~\cite{ETW}). The use of foams instead of, for example, Khovanov's or Bar-Natan's cobordisms \cite{Kh1, BN2} in this construction enables properly (i.e. not just projectively) functorial link invariants. The target $\HC(\Sfoam)$ categorifies the $\glnn{2}$ skein module $\SWebq$ of $\Su$, which is spanned by $\glnn{2}$ webs drawn on $\Su$ modulo certain relations. This skein module is closely related, but finer than the Kauffman bracket skein module in the sense that it has an $H_1(\Su)$-grading instead of just an $H_1(\Su,\Z/2\Z)$-grading. We will define and investigate these skein modules in Section~\ref{sec:gl2skein}. We also prove the following properties of $\Su\Kh$.

\begin{proposition}\label{prop:prop} The surface link homologies from Theorem~\ref{thm:linkhom} satisfy the following properties.
\begin{itemize}
\item The invariant $\Su\Kh(L)$ of an oriented link $L$, representing the homology class $[L]\in H_1(\Su\times [0,1])\cong H_1(\Su)$, is supported in the $H_1(\Su)$-degree $[L]$ part of $\HC(\Sfoam)$.
\item A link cobordism of Euler characteristic $k$ induces a homogeneous morphism of $\Z$-degree $-k$ in $\HC(\Sfoam)$, which only depends on the isotopy class of the cobordism.
\item $\Su\Kh$ intertwines the natural actions of $\mathrm{Diff^+(\Su)}$ on $\Slink$ and $\HC(\Sfoam)$.
\item The assignment $\Su \mapsto \Su\Kh$ is functorial under orientation preserving embeddings of surfaces.
\end{itemize}
\end{proposition}

While these properties suggest that the functors $\Su\Kh$ are very natural objects, a disadvantage is that their target categories $\HC(\Sfoam)$ are relatively large and unwieldy. In the search for more manageable algebraic link homology functors, we take a closer look at the foam categories $\Sfoam$ in Section~\ref{sec:foams}. An important observation is that their $\Z$-grading naturally splits into two, one of which is non-negative. 

\begin{theorem} \label{thm:nonneggrad} The category $\Sfoam$ is $H_1(\Su)\times \Z_{\geq 0}\times \Z$-graded, provided $\Su$ is not the sphere $\Stwo$.
\end{theorem} 

As a consequence, there exists a truncation functor $\Sfoam\to \Sfoam_0$, which projects on the degree zero component of the non-negative grading.  The category $\Sfoam_0$ is related (when working over $\Z/2\Z$) to a quotient of Bar-Natan's dotted cobordism category \cite{BN2} over the surface $\Su$, which Boerner \cite{Boe} used to describe the Asaeda-Przytycki-Sikora invariants.  Our invariants $\Su\Kh$, in contrast, also make sense in the untruncated setting, and the action of the truncation functor uncovers interesting additional structure. Evidence for this is found in the case of the annulus $\Su=\A$, where an analogous truncation functor gives rise to a spectral sequence between the annular Khovanov homology of \cite{APS} and the ordinary Khovanov homology in the thickened plane $\R^2$, see Roberts~\cite{Roberts}. Moreover, the annular Khovanov homology has recently been shown by Grigsby--Licata--Wehrli~\cite{GLW} to carry an action by the exterior current algebra for $\slnn{2}$, which uses the positive degree parts with respect to the non-negative grading. In another direction, Beliakova--Putyra--Wehrli~\cite{BPW} have constructed a quantized annular Khovanov homology with an action of $U_q(\slnn{2})$. Extensions of both types of actions to the general surface case are intriguing open problems for further research.

In order to define an algebraic version of $\Su\Kh$, we use the following.

\begin{proposition}\label{prop:semisimple} The Karoubi envelope of the truncation $\Sfoam_0$ is semisimple for $\Su\neq\T,\Stwo$.
\end{proposition} 
Moreover, the simple objects are indexed by pairs of an integer lamination and a first homology class on $\Su$, or alternatively by a standard basis element of the $\glnn{2}$ skein module $\SWebq$, which we define in Section~\ref{sec:gl2skein}. 

In Section~\ref{sec:alginv1}, we define an algebraic link homology $\Su \Kh^\prime$ by composing the functor $\Su\Kh$ with functors induced by the degree zero truncation and the embedding in the Karoubi envelope, and then a representable functor. 

\begin{theorem}\label{thm:alglinkhomology} There exists a surface link homology functor $\Su\Kh^\prime$ with values in the category of vector spaces graded by $\{\text{integer laminations on }\Su\}\times H_1(\Su)\times \Z\times \Z$.
\end{theorem}
The functor $\Su\Kh^\prime$ factors through $\Su\Kh$ by construction, and both functors categorify the evaluation of links in the skein module $\SWebq$. In Section \ref{sec:alginv1} we also sketch the construction of an alternative algebraic surface link homology $\Su\APS$, which agrees with the Asaeda--Przytycki--Sikora invariants when defined over $\Z/2\Z$, but which depends (when defined over $\Q$ or $\Z$) on the conjectural functoriality of Khovanov homology under foams.

\begin{theorem}[Assuming Conjecture~\ref{conj:functoriality} or with $\Z/2\Z$-coefficients] \label{thm:ss} There exists a surface link homology functor $\Su\APS$ with values in the category of vector spaces graded by $\Z\{\text{essential simple closed curves on }\Su\}\times H_1(\Su)\times \Z\times \Z$. Moreover, if $\phi \colon \Su \to \Su^\prime$ is an embedding of surfaces and $L$ is a link in $\Su\times[0,1]$, then there exists a spectral sequence:
\[\Su\APS(L)\rightsquigarrow \Su^\prime\APS(\phi(L)).\] 
\end{theorem}
These spectral sequences generalize the well-known spectral sequence between the annular and the ordinary Khovanov homology, see Roberts~\cite{Roberts} and further Grigsby--Licata--Wehrli~\cite{GLW} and Hubbard--Saltz \cite{HubbardSaltz}.

\subsection{Skein algebra categorification}
\label{sec:skeinalgcat}
The key feature that makes skein modules of thickened (Heegaard) surfaces useful for the description of skein modules of $3$-manifolds is that they inherit a natural algebra structure from the operation of gluing two copies of the thickened surface along the thickening direction.\footnote{This algebra structure was used in the pioneering work of Przytycki~\cite{Prz1} and Turaev~\cite{Tur}, and it has since been studied in numerous publications, see e.g. Przytycki--Sikora~\cite{PS2} and references therein.} Further $2$- and $3$-handle attachments to the thickened surface then correspond to taking quotients by certain ideals with respect to this multiplication. This gives presentations of the skein modules of all closed orientable $3$-manifolds. Finding a categorification of the skein algebra multiplication is thus a natural problem, which is open except in simple cases.

\begin{example} Let $\Su$ be either $\R^2$ or $\A$, then the Khovanov functor $\Su\Kh$ is monoidal and the tensor product on the target category decategorifies to the skein algebra multiplication in $\SWebq$.
\end{example}
These cases are very special since the monoidal structure on $\Slink$ can be understood as disjoint union and it is well-known that the (annular) Khovanov homology of a disjoint union of links is isomorphic to the tensor product of the invariants of the component links. Moreover, in these special cases, the Khovanov functors are actually braided and pivotal.

For other surfaces we conjecture that the functors from Theorem~\ref{thm:linkhom} can also be made monoidal, possibly after proceeding to a slightly different target category. 

\begin{conjecture}\label{conj:mon} For every oriented surface of finite type $\Su$, there exists a monoidal functor
\[\Su\cat{MKh}=\Su \cat{M}\circ \Su\Kh \colon \Slink \to \HC(\Sfoam)\to \Su\mC\]
which categorifies the evaluation of links in the $\glnn{2}$ skein module of $\Su$. More precisely, we require that the monoidal target dg category $\Su\mC$ is $H_1(\Su)\times \Z\times \Z$-graded (the second $\Z$-factor corresponds to the homological grading, the first one to powers of $q$) and carries an action of $\mathrm{Diff}^+(\Su)$ that is intertwined by $\Su\cat{MKh}$ with the natural action on $\Slink$. Further, there is an isomorphism $K_0(\Su\mC) \cong \SWebq$ of $H_1(\Su)$-graded unital $\Z[q^{\pm 1}]$-algebras and $\mathrm{Diff}^+(\Su)$-representations that sends the class $[\Su\cat{MKh}(L)]$ of a link $L$ to its skein algebra evaluation. This would give a categorification of the skein algebra $\SWebq$ in the sense of \cite[Definition 2]{QW}. 
\end{conjecture}

Our main reason for expecting that $\HC(\Sfoam)$ is not an ideal target category for a monoidal Khovanov functor is that it seems impossible to directly define the putative tensor product. To illustrate this, note that the Khovanov functor sends crossingless links in $\Su$ essentially to themselves, considered as a chain complex of a single web, concentrated in homological degree zero. The tensor product of such objects of $\HC(\Sfoam)$ should, thus, be computed, by lifting these elements against the Khovanov functor to actual links, where they can be superposed in $\Slink$, and the tensor product is then given by the image of the superposition under the Khovanov functor. Under an additional assumption on the functoriality of the Khovanov functors under foams, this process actually extends to a bifunctor $\Sfoam\times \Sfoam \to \HC(\Sfoam)$, see Proposition~\ref{prop:bifunctor}. However, this bifunctor sends pairs of foams to chain maps, which are well-defined only up to homotopy. It is then unclear how to extend this bifunctor to chain complexes in $\Sfoam$, or to the homotopy category. Indeed, the differentials of the chain complexes in the two arguments would specify components of the differential of the product chain complex only up to homotopy. The product would, at best, be determined only up to isomorphism, which is not sufficient. 

This problem suggests to replace $\HC(\Sfoam)$ by a target category $\Su\mC$ with more rigid morphism spaces, in which homotopic chain maps are actually forced to become equal. A natural candidate for such a replacement is given by the functor $\HC(\Sfoam)\to \HC(\Kar(\Sfoam_0))$ induced by the degree zero truncation and the embedding of a category into its Karoubi envelope. By Proposition~\ref{prop:semisimple} $\Kar(\Sfoam_0)$ is semisimple for $\Su\neq \T,\Stwo$, which implies that any chain complex over this category retracts onto an essentially unique minimal complex with zero differentials, and homotopic chain maps between such complexes are necessarily equal. The category of such minimal complexes can be considered as $\Kar(\Sfoam_0)$ equipped with an additional homological grading, or equivalently, as a dg category $\Kar(\Sfoam_0)_{dg}=\bigoplus_{k\in \Z} t^k\Kar(\Sfoam_0)$ with trivial differential. We define the $\delta$-grading on objects in $\Kar(\Sfoam_0)_{dg}$ as the difference of the $q$-grading and the homological grading. It seems possible that $\Kar(\Sfoam_0)_{dg}$ could serve as a target category for a monoidal link homology functor as in Conjecture~\ref{conj:mon}.

\begin{questions}\label{q:monoidal} Does the bifunctor $*\colon \Sfoam\times \Sfoam \to \HC(\Sfoam)$ induce a monoidal structure on $\Kar(\Sfoam_0)_{dg}$ for $\Su\neq \T,\Stwo$? If so, does the tensor product $*$ respect the parity of the $\delta$-grading in the sense that $\deg_{\delta}(A* B) = \deg_{\delta}(A)+\deg_{\delta}(B)$ in $\Z/2\Z$?
\end{questions}

The relevance of the second question will soon become clear.

\subsection{Positivity in skein algebras}
A central motivation for Conjecture~\ref{conj:mon} and an important hint about the structure of the putative monoidal target category $\Su\mC$ come from positivity phenomena in skein algebras. More specifically, a conjecture of Fock--Goncharov \cite[Conjecture 12.4]{FoG}, reformulated by Thurston \cite[Conjecture 4.20]{Thu}, claims that a particular basis $\Su\basis$ of the $\glnn{2}$ skein algebra $\SWebq$, which we define in Section~\ref{sec:gl2skein}, has positive structure constants.

\begin{conjecture}\label{conj:positivity} For every oriented surface $\Su$ of finite type, the products of elements of $\Su\basis$ in the skein algebra $\SWebq$ are $\N[-q^{\pm 1}]$-linear combinations of elements of $\Su\basis$. 
\end{conjecture} 

A key observation is that the positivity property expressed in Conjecture~\ref{conj:positivity} for a surface $\Su$ would be a natural consequence of successful categorification of the skein algebra $\SWebq$ as in Conjecture~\ref{conj:mon}, provided that the tensor product respects the parity of the $\delta$-grading, i.e. the difference of the $q$-grading and the homological grading on $\Su\mC$. A positive basis could then arise as the decategorification of the set of isomorphism classes of indecomposable objects in $\Su\mC$. Conversely, this observation feeds the expectation that the putative target category $\Su\mC$ should contain precisely such indecomposable objects, which might be useful for constructing $\Su\mC$.  

The elements of the basis $\Su\basis$ of $\SWebq$ can be understood as skein algebra products of Chebyshev polynomials\footnote{More accurately, since we consider $\glnn{2}$ rather than $\slnn{2}$ skein algebras, we should replace Chebyshev polynomials by power-sum symmetric polynomials in two variables. We intentionally blur this distinction in the introduction.} of the first kind $T_n$, evaluated on essential simple closed curves on $\Su$. The $T_n\in\Z[X]$ are defined by the recursion $T_{n+1}=X T_n - T_{n-1}$ for $n\geq 3$ with initial conditions $T_2=X^2-2$,  $T_1=X$ and $T_0=1$. We have argued in \cite{QW,QW2} that $T_n(c)$ for a simple closed curve $c$ should be interpreted as the curve $c$ colored by the extremal weight spaces $\Sym^n_{\mathrm{ext}}(V)$ in the $U_q(\glnn{2})$-representation $\Sym^n(V)$, where $V$ denotes the vector representation. The more familiar colorings by $\Sym^n(V)$, which feature in the definition of the colored Jones polynomials, can be realized by using Chebyshev polynomials of the second kind $S_n$ instead. These satisfy the same recurrence, but with initial conditions $S_2=X^2-1$, $S_1=X$ and $S_0=1$. In general, any basis for the $\Z$-module $\Z[X]$ will give rise to an associated basis of $\SWebq$. In the following table, we collect the three bases that we will use in this article.
\begin{center}
\begin{tabular}{c|c|c|c}
 $\SWebq$ basis & $\Z[X]$ basis & color & categorification strategy  \\
\hline
 $\Su\basisstd$ & $X^n$ & $V^{\otimes n}$ & standard Khovanov homology \\
 \hline
 $\Su\basisJW$ & $S_n$ & $\Sym^n(V)$ & colored Khovanov homology \\
  \hline
 $\Su\basis$ & $T_n$ & $\Sym^n_{\mathrm{ext}}(V)$ & toric colored Khovanov homology \\
\end{tabular}
\end{center}
The basis changes from $\Su\basisstd$ to $\Su\basisJW$ and further onward to $\Su\basis$ are triangular and positive, which follows from the analogous result for polynomial bases. In the table, we have also indicated the strategies that can be used to categorify the skein module basis elements. To warm up, we will describe this process in the case of the basis $\Su\basisJW$, whose elements can be considered as symmetrically colored, crossingless links in $\Su$ whose components are pairwise non-isotopic (actually, in the skein module $\SWebq$, such elements also appear with additional $2$-labeled edges, but we ignore this feature for now). In \cite{Kh7}, Khovanov has proposed several ways of making sense of symmetrically colored links in the framework of Khovanov homology. These proposals extend to the case of links in thickened surfaces, and for crossingless link diagrams and ground ring $\Q$ it is easy to see that the different proposals agree. 

Let $c$ be an oriented simple closed curve $\Su$, which we may also consider as an object in $\Sfoam$. The $\Sym^n(V)$-colored version of $c$ is defined as an object of $\Kar(\Sfoam)$, i.e. an idempotent morphism in $\Sfoam$. To describe this idempotent, consider the multi-curve $c^n$ consisting of $n$ parallel copies of $c$. By virtue of the functoriality of Khovanov homology, there exists a homomorphism from the braid group $B_n$ to $\End_{\Sfoam}(c^n)$, which takes braids to cobordisms that braid the components of $c^n$ around each other. A simple computation using the Reidemeister II chain maps (see Section~\ref{sec:Reidemeister}) shows that this homomorphism induces a homomorphism $S_n \to \End_{\Sfoam}(c^n)$. The idempotent representing $\Sym^n(V)$-colored version of $c$ is now defined to be the image of the symmetric Young symmetrizer in $\C S_n$. Alternatively, a linear combination of foams representing this idempotent can also be described as rotation foam generated by the $\glnn{2}$ version of the $n$-th Jones--Wenzl projector. This uses the fact that rotation foams satisfy the same relations as their web sections, interpreted as morphisms in a web category at $q=1$, see Queffelec--Rose~\cite{QR2} and Queffelec--Rose--Sartori~\cite{QRS}.

\begin{example}\label{exa:JWfoam} We illustrate the idempotent representing a $\Sym^2(V)$-colored curve, first as the cobordism version of the second Young symmetrizer $\frac{1+s}{2}\in \C S_2$, and then as a linear combination of rotation foams generated by the second Jones--Wenzl projector for $\glnn{2}$. 

  \[
\begin{tikzpicture}[anchorbase,scale=.2]
\draw[thick] (0,0) ellipse (4 and 1);
\node at (1,0) {\tiny $\Sym^2$};
\end{tikzpicture}
=\mathrm{im}\left(\frac{1}{2}\,
\begin{tikzpicture}[fill opacity=.2,anchorbase,scale=.2]
\draw[thick] (0,0) ellipse (4 and 1);
\draw[thick] (0,0) ellipse (3 and 0.75);
\draw[thick] (0,3) ellipse (4 and 1);
\draw[thick] (0,3) ellipse (3 and .75);
\draw (-4,0) to (-4,3);
\draw (-3,0) to (-3,3);
\draw (4,0) to (4,3);
\draw (3,0) to (3,3);
\end{tikzpicture} 
+
\frac{1}{2}\,
\begin{tikzpicture}[fill opacity=.2,anchorbase,scale=.2]
\draw (-4,0) to (-3,3);
\draw[white, line width=.1cm] (-3,0) to (-4,3);
\draw (-3,0) to (-4,3);
\draw (4,0) to (3,3);
\draw[white, line width=.1cm] (3,0) to (4,3);
\draw (3,0) to (4,3);
\draw[thick] (0,0) ellipse (4 and 1);
\draw[thick] (0,0) ellipse (3 and 0.75);
\draw[thick] (0,3) ellipse (4 and 1);
\draw[thick] (0,3) ellipse (3 and .75);
\end{tikzpicture} 
\right)
=\mathrm{im}\left(
\begin{tikzpicture}[fill opacity=.2,anchorbase,scale=.2]
\draw[thick] (0,0) ellipse (4 and 1);
\draw[thick] (0,0) ellipse (3 and 0.75);
\draw[very thick, red] (-4,0) to (-4,3);
\draw[very thick, red] (-3,0) to (-3,3);
\draw (4,0) to (4,3);
\draw (3,0) to (3,3);
\draw[thick] (0,3) ellipse (4 and 1);
\draw[thick] (0,3) ellipse (3 and .75);
\end{tikzpicture} 
\textcolor[rgb]{1.00,0.00,0.00}\, {-\frac{1}{2}}\,
\begin{tikzpicture}[fill opacity=.2,anchorbase,scale=.2]
\draw[thick] (0,0) ellipse (4 and 1);
\draw[thick] (0,0) ellipse (3 and 0.75);
\draw[very thick, red] (-4,0) to [out=90,in=225] (-3.5,1) to (-3.5,2) to [out=135,in=270] (-4,3);
\draw[very thick, red] (-3,0) to [out=90,in=315] (-3.5,1) to (-3.5,2) to [out=45,in=270] (-3,3);
\draw (3,0) to [out=90,in=225] (3.5,1) to (3.5,2) to [out=135,in=270] (3,3);
\draw (4,0) to [out=90,in=315] (3.5,1) to (3.5,2) to [out=45,in=270] (4,3);
\draw[thick] (0,3) ellipse (4 and 1);
\draw[thick] (0,3) ellipse (3 and .75);
\end{tikzpicture} 
\right)
\]
\end{example}

Using rotation foams generated by higher order Jones--Wenzl projectors for $\glnn{2}$, we construct categorifications of all elements of the basis $\Su\basisJW$, and it is easy to see that the image of every object of $\Sfoam$ under the embedding in $\Kar(\Sfoam)$ is isomorphic to a direct sum of shifts of such categorified elements of $\Su\basisJW$. Moreover, if $\Su\neq \T$, these objects of $\Kar(\Sfoam)$ have no non-trivial endomorphisms or morphisms to other objects in degree zero, which implies Proposition~\ref{prop:semisimple}, i.e. that the category $\Kar(\Sfoam_0)$ is semisimple, and that it decategorifies to $\SWebq$. This suggests that a tensor product on $\Kar(\Sfoam_0)_{dg}$ that is compatible with a monoidal surface link homology functor $\Su\cat{MKh}$ as in Conjecture~\ref{conj:mon} and that respects the parity of the $\delta$-degree would require an affirmative answer to the following question. 

\begin{question} \label{quest:JWpositivity}
Is the basis $\Su\basisJW$ positive, i.e. does it have structure constants in $\N[-q^{\pm 1}]$, for $\Su \neq \T$?
\end{question}
Thurston conjectures that this is true in the analogous framework of unquantized Kauffman bracket skein algebras whenever $\Su$ has free fundamental group, i.e. at least one puncture or boundary component, \cite[Conjecture 4.19]{Thu}. However, to the best of our knowledge, no counterexample is known among the closed orientable surfaces of genus greater than one. In any case, the basis $\Su\basisJW$ passes the basic test that it admits a positive change of basis to $\Su\basis$, which is a necessary condition for positivity according to L\^e \cite[Theorem 1.2]{Le}. 

It is easy to see that for the torus the basis $\T\basisJW$ is not positive. Frohman and Gelca~\cite{FG} showed that one should instead consider the finer basis $\T\basis$ modeled on the Chebyshev polynomials of the first kind, which are harder to categorify. Fortunately, the morphism spaces in the category $\Tfoam$ are controlled by $\Ss^1$-equivariant foams, i.e. foams obtained by rotating an annular $\glnn{2}$ web along a slope in $\T$. This allows us to use rotation foams generated by the extremal weight projectors, which we have introduced in \cite{QW,QW2}, to categorify the basis elements in $\T\basis$. We describe this process and related peculiarities of the torus case in the next section.

\subsection{Toric link homology and a categorified Frohman--Gelca formula}
Associated with every foam $F$ in a thickened surface $\Su\times [0,1]$ is a properly embedded compact surface $c(F)\subset \Su\times [0,1]$, not necessarily connected or orientable. The $\Z$-degree of $F$ as a morphism in $\Sfoam$ is given by $2 d-\chi(c(F))$ where $d$ is the number of certain decorations, called dots, on $F$. 

As far as surface link homologies $\Su\Kh$ and foam categories $\Sfoam$ are concerned, the cases $\Su=\T$ and $\Su=\Stwo$ are special because their thickenings contain boundary-parallel incompressible closed surfaces of non-negative Euler characteristic. As a consequence, the endomorphism algebra of the empty web $\emptyset$ is of infinite rank in degree zero in $\Tfoam$, and it even contains negative-degree endomorphisms in $\Stwo\cat{Foam}$.

In order to define a more manageable foam category for the torus, we define a quotient $\essTfoam$ of $\Tfoam$, in which foams $F$ with boundary-parallel tori in $c(F)$ are set to zero (while compressible tori still evaluate, as usual, to $\pm 2$). The category $\essTfoam$ has full \textit{slope subcategories}, see Definition~\ref{def:slopesubcat}, in which all morphisms are linear combinations of rotation foams generated by affine $\glnn{2}$ webs, where rotation is performed along a chosen slope on the torus. Each of these slope subcategories is isomorphic to a quotient $\essbAWeb$ of the category of affine $\glnn{2}$ webs at $q=1$. The defining quotient map takes essential circles in the annulus to zero and precisely corresponds to the quotient map from $\Tfoam$ to $\essTfoam$ that kills the boundary-parallel essential torus.

We introduced the affine web category $\essbAWeb$ in \cite{QW2} (denoted $2\essbAWeb$ there, to set it apart from its cousins for $\glnn{N}$) and proved that it gives rise to a presentation of the representation category of the Cartan subalgebra $\mathfrak{h}\subset \glnn{2}$. In particular, $\essbAWeb$ contains idempotent morphisms which correspond to the projections of the extremal weight spaces in the $\glnn{2}$-representation $\Sym^n(V)$. We call them extremal weight projectors and show that they categorify the Chebyshev polynomials of the first kind in the same sense as the Jones--Wenzl projectors categorify the Chebyshev polynomials of the second kind.\footnote{More accurately, the $\glnn{2}$ extremal weight projectors categorify power-sum symmetric polynomials in two variables, while the $\glnn{2}$ Jones--Wenzl projectors categorify complete symmetric polynomials. Chebyshev polynomials appear for $\slnn{2}$ instead of $\glnn{2}$.}

\begin{example} The second Jones--Wenzl projector and the second extremal weight projector for $\glnn{2}$ in comparison:
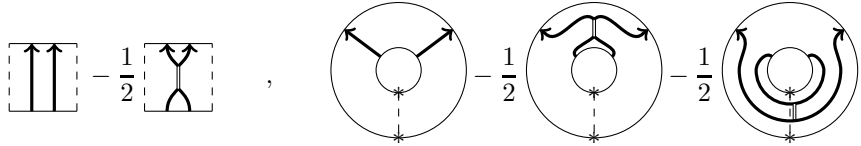
\[
\begin{tikzpicture}[anchorbase, scale=.3]
\draw (0,0) to (3,0);
\draw (0,3) to (3,3);
\draw[dashed] (0,0) to (0,3);
\draw[dashed] (3,0) to (3,3);
\draw[very thick, ->] (1,0) to (1,3);
\draw[very thick, ->] (2,0) to (2,3);
\end{tikzpicture}
\;- \frac{1}{2}\;
\begin{tikzpicture}[anchorbase, scale=.3]
\draw (0,0) to (3,0);
\draw (0,3) to (3,3);
\draw[dashed] (0,0) to (0,3);
\draw[dashed] (3,0) to (3,3);
\draw[very thick] (1,0) to [out=90,in=225](1.5,1);
\draw[very thick] (2,0) to [out=90,in=315] (1.5,1);
\draw[double] (1.5,1) to (1.5,2);
\draw[very thick,->] (1.5,2) to [out=135,in=270](1,2.9) to (1,3);
\draw[very thick,->] (1.5,2) to [out=45,in=270](2,2.9) to (2,3);
\end{tikzpicture}
\qquad, \qquad
 \begin{tikzpicture}[anchorbase, scale=.3]
\draw (0,0) circle (1);
\draw (0,0) circle (3);
\draw [very thick,->] (.8,.6) to (2.4,1.8);
\draw [very thick,->] (-.8,.6) to (-2.4,1.8);
\node at (0,-1) {$*$};
\node at (0,-3) {$*$};
\draw [dashed] (0,-1) to (0,-3);
\end{tikzpicture}
-\frac{1}{2} \;
 \begin{tikzpicture}[anchorbase, scale=.3]
\draw (0,0) circle (1);
\draw (0,0) circle (3);
\draw [very thick] (.8,.6) to [out=45, in=315] (0,1.5);
\draw [very thick] (-.8,.6) to [out=135,in=225] (0,1.5);
\draw [double] (0,1.5) to (0,2.25);
\draw [very thick,->] (0,2.25) to [out=45,in=225](2.16,1.62) to(2.4,1.8);
\draw [very thick,->] (0,2.25) to [out=135,in=315](-2.16,1.62) to(-2.4,1.8);
\node at (0,-1) {$*$};
\node at (0,-3) {$*$};
\draw [dashed] (0,-1) to (0,-3);
\end{tikzpicture}
-\frac{1}{2} \;
 \begin{tikzpicture}[anchorbase, scale=.3]
\draw (0,0) circle (1);
\draw (0,0) circle (3);
\draw [very thick] (.8,.6) to [out=45, in=90] (1.5,0) to [out=270,in=0] (0,-1.5);
\draw [very thick] (-.8,.6) to [out=135,, in=90] (-1.5,0) to [out=270,in=180] (0,-1.5);
\draw[double] (.2,-1.5) to (.2,-2.25);
\draw [very thick,->] (0,-2.25) to [out=0,in=270] (2.25,0) to [out=90,in=225](2.16,1.62) to(2.4,1.8);
\draw [very thick,->] (0,-2.25) to [out=180,in=270] (-2.25,0) to [out=90,in=315](-2.16,1.62) to(-2.4,1.8);
\node at (0,-1) {$*$};
\node at (0,-3) {$*$};
\draw [dashed] (0,-1) to (0,-3);
\end{tikzpicture}\]
\end{example}

In Section~\ref{sec:toric} we use the $\glnn{2}$ extremal weight projectors in $\essbAWeb$ and rotation foams generated by them to prove the following theorem. 

\begin{theorem} The category $\essTfoam$ contains idempotent morphisms, which generate a skeleton of the semisimple Karoubi envelope $\Kar(\essTfoam_0)$, and which then decategorify to the elements of the positive basis $\T\basis$ of $\TWebq$. 
\end{theorem}

We conjecture that a homologically graded version $\Kar(\essTfoam_0)_{\mathrm{dg}}$ of this category is a suitable target category for a monoidal toric Khovanov homology functor. Assuming functoriality of Khovanov homology under foams, we construct a bifunctor $*$ on this category and conjecture that it extends to a monoidal structure that is compatible with the toric link homology. We provide evidence for this conjecture by evaluating the bifunctor on certain categorified basis elements. Assuming the previous conjectures about functoriality of Khovanov homology and the extension to a monoidal structure, these computations also prove that the products of categorified Frohman--Gelca basis elements decompose along the $\glnn{2}$ version of the Frohman--Gelca product-to-sum formula from \cite{FG}:  
\[(m,n)_T * (r,s)_T=  (m+r,n+s)_T + (m-r,n-s)_T*\w{r,s} \]
in $\TWebq$, which we prove in Appendix~\ref{sec:FGfla}.
\begin{conjecture}
\label{conj_catFG} The bifunctor $*$ extends to a monoidal structure on $\Kar(\essTfoam_0)_{\mathrm{dg}}$, giving a categorification of the skein algebra $\TWebq$ as in Conjecture~\ref{conj:mon}. Moreover, the categorified Frohman--Gelca basis elements tensor as follows:
\begin{equation}\label{eqn:catFGintro} (m,n)^F_T * (r,s)^F_T\cong  (m+r,n+s)^F_T \oplus (m-r,n-s)^F_T*\w{r,s} 
\end{equation}
\end{conjecture}
If the first part of Conjecture~\ref{conj_catFG} holds, then the \textit{categorified Frohman--Gelca formula} \eqref{eqn:catFGintro} can be proved by induction from a small number of base cases, see Remark~\ref{rem:catproof}. In Section~\ref{sec:examples} we check these base cases and then provide evidence towards Conjecture~\ref{conj_catFG} by verifying an additional non-trivial case of \eqref{eqn:catFGintro}.

\settocdepth{section}
\subsection*{Acknowledgements}  We would like to thank Anna Beliakova, Francis Bonahon, C\'edric Bonnaf\'e, Ben Cooper, Eugene Gorsky, Mikhail Khovanov, Thang L\^{e}, Tony Licata, Gregor Masbaum, Scott Morrison, Alexandre Nicolas, Jozef Przytycki, Jake Rasmussen, Peter Samuelson, Dylan Thurston, Daniel Tubbenhauer and Emmanuel Wagner for interesting discussions.

\subsection*{Funding}
The work of H.~Q. was partially supported by a PEPS Jeunes Chercheuses et Jeunes Chercheurs, the ANR Quantact and by the CNRS-MSI partnership ``LIA AnGe''. The work of P.~W. was supported by the Leverhulme Trust [Research Grant RP2013-K-017] and the Australian Research Council Discovery Projects ``Braid groups and higher representation theory'' and ``Low dimensional categories'' [DP140103821, DP160103479].

\settocdepth{subsection}
\section{The \texorpdfstring{$\glnn{2}$}{gl(2)} skein algebras} 
\label{sec:gl2skein}

We will use the skein theory of $\glnn{2}$ webs, which are generated in the planar algebra sense by the oriented trivalent merge and split vertices between single ($1$-labeled) and double ($2$-labeled) edges, and which satisfy the following $\Z[q^{\pm1}]$-linear web relations:

\begin{gather}
\label{eqn:circles}
\begin{tikzpicture}[fill opacity=.2,anchorbase,scale=.3]
\draw[very thick, directed=.55] (1,0) to [out=0,in=270] (2,1) to [out=90,in=0] (1,2)to [out=180,in=90] (0,1)to [out=270,in=180] (1,0);
\end{tikzpicture} 
\quad=\quad
(q+ q^{-1}) \emptyset
\quad=\quad 
\begin{tikzpicture}[fill opacity=.2,anchorbase,scale=.3]
\draw[very thick, rdirected=.55] (1,0) to [out=0,in=270] (2,1) to [out=90,in=0] (1,2)to [out=180,in=90] (0,1)to [out=270,in=180] (1,0);
\end{tikzpicture}
\quad,\quad
\begin{tikzpicture}[fill opacity=.2,anchorbase,scale=.3]
\draw[double, directed=.55] (1,0) to [out=0,in=270] (2,1) to [out=90,in=0] (1,2)to [out=180,in=90] (0,1)to [out=270,in=180] (1,0);
\end{tikzpicture} 
\quad=\quad
\emptyset
\quad=\quad \begin{tikzpicture}[fill opacity=.2,anchorbase,scale=.3]
\draw[double, rdirected=.55] (1,0) to [out=0,in=270] (2,1) to [out=90,in=0] (1,2)to [out=180,in=90] (0,1)to [out=270,in=180] (1,0);
\end{tikzpicture}
\\
\label{eqn:digons}
\begin{tikzpicture}[anchorbase, scale=.5]
\draw [double] (.5,0) -- (.5,.3);
\draw [very thick] (.5,.3) .. controls (.4,.35) and (0,.6) .. (0,1) .. controls (0,1.4) and (.4,1.65) .. (.5,1.7);
\draw [very thick] (.5,.3) .. controls (.6,.35) and (1,.6) .. (1,1) .. controls (1,1.4) and (.6,1.65) .. (.5,1.7);
\draw [double, ->] (.5,1.7) -- (.5,2);
\end{tikzpicture}
\quad= \quad
(q+q^{-1})\;
\begin{tikzpicture}[anchorbase, scale=.5]
\draw [double,->] (.5,0) -- (.5,2);
\end{tikzpicture}
\quad,\quad
\begin{tikzpicture}[anchorbase, scale=.5]
\draw [very thick] (.5,0) -- (.5,.3);
\draw [very thick] (.5,.3) .. controls (.4,.35) and (0,.6) .. (0,1) .. controls (0,1.4) and (.4,1.65) .. (.5,1.7);
\draw [double, directed=0.55] (.5,.3) .. controls (.6,.35) and (1,.6) .. (1,1) .. controls (1,1.4) and (.6,1.65) .. (.5,1.7);
\draw [very thick, ->] (.5,1.7) -- (.5,2);
\end{tikzpicture}
\quad= \quad
\begin{tikzpicture}[anchorbase, scale=.5]
\draw [very thick,->] (.5,0) -- (.5,2);
\end{tikzpicture}
\quad= \quad
\begin{tikzpicture}[anchorbase, scale=.5]
\draw [very thick] (.5,0) -- (.5,.3);
\draw [double, directed=0.55] (.5,.3) .. controls (.4,.35) and (0,.6) .. (0,1) .. controls (0,1.4) and (.4,1.65) .. (.5,1.7);
\draw [very thick] (.5,.3) .. controls (.6,.35) and (1,.6) .. (1,1) .. controls (1,1.4) and (.6,1.65) .. (.5,1.7);
\draw [very thick, ->] (.5,1.7) -- (.5,2);
\end{tikzpicture}
\\
\label{eqn:squares}
\begin{tikzpicture}[anchorbase,scale=.5]
\draw [double] (0,0) -- (0,0.5);
\draw [very thick] (1,0) -- (1,.7);
\draw [very thick] (0,0.5) -- (1,.7);
\draw [double] (1,.7) -- (1,1.3);
\draw [very thick] (0,.5) -- (0,1.5);
\draw [very thick] (1,1.3) -- (0,1.5);
\draw [double,->] (0,1.5) -- (0,2);
\draw [very thick, ->] (1,1.3) -- (1,2);
\end{tikzpicture}
\quad = \quad
\begin{tikzpicture}[anchorbase,scale=.5]
\draw [double,->] (0,0) -- (0,2);
\draw [very thick,->] (1,0) -- (1,2);
\end{tikzpicture}
\quad,\quad
\begin{tikzpicture}[anchorbase,scale=.5]
\draw [double] (1,0) -- (1,0.5);
\draw [very thick] (0,0) -- (0,.7);
\draw [very thick] (1,0.5) -- (0,.7);
\draw [double] (0,.7) -- (0,1.3);
\draw [very thick] (1,.5) -- (1,1.5);
\draw [very thick] (0,1.3) -- (1,1.5);
\draw [double,->] (1,1.5) -- (1,2);
\draw [very thick, ->] (0,1.3) -- (0,2);
\end{tikzpicture}
\quad = \quad
\begin{tikzpicture}[anchorbase,scale=.5]
\draw [double,->] (1,0) -- (1,2);
\draw [very thick,->] (0,0) -- (0,2);
\end{tikzpicture}
\quad , \quad
\begin{tikzpicture}[anchorbase,scale=.5]
\draw [double,->] (0,0) to  (0,2);
\draw [double,->] (1,2) to (1,0);
\end{tikzpicture}
\quad =\quad
\begin{tikzpicture}[anchorbase,scale=.5]
\draw [double,->] (0,0) to (0,.5) to [out=90,in=90] (1,.5) to (1,0);
\draw [double,->] (1,2) to (1,1.5) to [out=270,in=270] (0,1.5) to (0,2);
\end{tikzpicture}
\quad , \quad
\begin{tikzpicture}[anchorbase,scale=.5]
\draw [double,<-] (0,0) to  (0,2);
\draw [double,<-] (1,2) to (1,0);
\end{tikzpicture}
\quad =\quad
\begin{tikzpicture}[anchorbase,scale=.5]
\draw [double,<-] (0,0) to (0,.5) to [out=90,in=90] (1,.5) to (1,0);
\draw [double,<-] (1,2) to (1,1.5) to [out=270,in=270] (0,1.5) to (0,2);
\end{tikzpicture}
\end{gather}

This skein theory encodes the pivotal tensor category of representations of $U_q(\glnn{2})$ generated by the vector representation and its exterior square, see e.g. \cite{CKM,QS2,TVW}. It is closely related to the Temperley--Lieb skein theory, which describes a corresponding category of representations of $U_q(\slnn{2})$, in which this exterior square is isomorphic to the trivial representation. For a $\glnn{2}$ web $W$ (in a disk or some other surface ) we denote by $c(W)$ the unoriented multi-curve obtained by erasing all doubled edges and forgetting orientations in $W$. It is easy to see that $c$ induces a map from the $\glnn{2}$ skein theory to the Temperley--Lieb skein theory, since the images of the $\glnn{2}$ web relations hold in the latter.

In fact, $\glnn{2}$ webs satisfy generalizations of the $1$-labeled circle relation in \eqref{eqn:circles}, which we recall from~\cite[Lemma 64]{QW2}.

\begin{lemma}
\label{lem:neckcut}
Let $W$ be a $\glnn{2}$ web (in a disk or some other surface) and suppose that $c(W)$ contains a circle $c$ which bounds a disk $\cat{D}$ in the complement of $c(W)$. Then $W = (q+q^{-1})V$, where $V$ is a web that agrees with $W$ outside a neighborhood of the disk $\cat{D}$ and with underlying curve $c(V)$ obtained by removing the circle in question from $c(W)$.
\end{lemma}
We recall the proof from \cite[Lemma 64]{QW2}.
\begin{proof} We only consider $W$ in a neighborhood of the disk $\cat{D}$ bounded by $c$. We will find a sequence of web relations which reduce the interaction of $2$-labeled edges with $c$ until $c$ can be removed via a circle relation in~\eqref{eqn:circles}.  There are three types of interaction of $c$ with $2$-labeled edges to consider in sequence:
\begin{enumerate}
\item Any $2$-labeled circle contained in $\cat{D}$ can be removed using one of the relations in \eqref{eqn:circles}, starting with an innermost one.
\item Suppose there exists a $2$-labeled edge in the interior of $\cat{D}$ with boundary on $c$. We take an innermost such edge, i.e. one which encloses a region in the disk with no other $2$-labeled edges in the interior. Such an intersection edge can be removed via the digon relations in \eqref{eqn:digons}, provided there are no $2$-labeled edges hitting the boundary of $\cat{D}$ from the outside in the relevant region. Otherwise, jump to (3) to remove external edges first. Note that they always come in pairs for orientation reasons.
\item There is a pair of $2$-labeled edges, hitting $c$ from the outside $\cat{D}$, which are adjacent in the sense that an arc along $c$ connects them without hitting other $2$-labeled edges. Then one application of the saddle relations in \eqref{eqn:squares} creates a $2$-labeled edge connecting two points on $c$ from the outside (see the right side of Figure~\ref{fig:compdiskint}), which can be removed as in (2). 
\end{enumerate}
This algorithm relates $W$ to a web that contains $c$ as an oriented $1$-labeled circle that can be removed via \eqref{eqn:circles}.
\end{proof}

\begin{figure}[ht]
\begin{tikzpicture}[fill opacity=.2,anchorbase]
\draw[very thick] (1,0) to [out=0,in=270](2,1) to [out=90,in=0] (1,2) to [out=180,in=90] (0,1) to [out=270,in=180] (1,0);
\draw[double,directed=.55] (1,.5) to [out=0,in=270](1.5,1) to [out=90,in=0] (1,1.5) to [out=180,in=90] (.5,1) to [out=270,in=180] (1,.5);
\end{tikzpicture}
\quad , \quad
\begin{tikzpicture}[fill opacity=.2,anchorbase]
\draw[very thick,directed=.08,rdirected=.18,directed=.32,directed=.45,rdirected=.57,directed=.70,rdirected=.88] (1,0) to [out=0,in=270](2,1) to [out=90,in=0] (1,2) to [out=180,in=90] (0,1) to [out=270,in=180] (1,0);
\draw[double,directed=.55] (0,1) to [out=0,in=90](1,0);
\draw[double,directed=.55] (1,2) to [out=270,in=180](2,1);
\draw[double,rdirected=.55] (0.29,1.71) to (1.71,0.29);
\draw[double,directed=.55] (1.8,1.6) to (2.16,1.92);
\draw[double,rdirected=.55] (1.6,1.8) to (1.92,2.16);
\end{tikzpicture}
\quad , \quad
\begin{tikzpicture}[fill opacity=.2,anchorbase]
\draw[very thick] (1,0) to [out=0,in=270](2,1) to [out=90,in=0] (1,2) to [out=180,in=90] (0,1) to [out=270,in=180] (1,0);
\draw[double,directed=.55] (1.71,1.71) to [out=45,in=180](2.5,2) ;
\draw[double,rdirected=.55] (1.71,0.29) to [out=315,in=180](2.5,0) ;
\end{tikzpicture}
$\to $\;
\begin{tikzpicture}[fill opacity=.2,anchorbase]
\draw[very thick] (1,0) to [out=0,in=270](2,1) to [out=90,in=0] (1,2) to [out=180,in=90] (0,1) to [out=270,in=180] (1,0);
\draw[double,directed=.55] (1.71,1.71) to [out=45,in=90] (2.25,1) to [out=270,in=315](1.71,0.29);
\draw[double,rdirected=.55] (3,2) to [out=180,in=90] (2.5,1) to [out=270,in=180](3,0);
\end{tikzpicture}
\caption{Types of interaction of $2$-labeled edges with a $1$-labeled circle bounding a disk: internal circles, internal and external edges.
\label{fig:compdiskint}}
\end{figure}
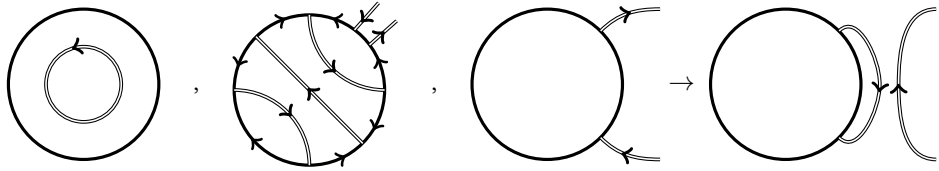

In the following, $\Su$ denotes an oriented surface of finite type, i.e. the result of removing a finite number of points from a compact oriented $2$-manifold. We usually assume that $\Su$ is connected and identify it with the quotient space obtained from gluing the edges of a punctured planar polygon in a suitable way.

\begin{definition}\label{def:skein} We let $\SWebq$ denote the quotient of the free $\Z[q^{\pm 1}]$-module spanned by isotopy classes of $\glnn{2}$ webs embedded in $\Su$ by the ideal generated by the local relations~\ref{eqn:circles}, \ref{eqn:digons} and \ref{eqn:squares}, which we interpret as being supported in disks in $\Su$.

Link diagrams and tangled web diagrams drawn on $\Su$ give elements of $\SWebq$ by resolving crossings:

\begin{gather}
\begin{tikzpicture}[anchorbase, scale=.5]
\draw [very thick, ->] (2,1) to [out=180,in=0] (0,0);
 \draw [white,line width=.15cm] (2,0) to [out=180,in=0] (0,1) ;
\draw [very thick, ->] (2,0) to [out=180,in=0] (0,1);
\end{tikzpicture}
 \;\;:= \;\;
\begin{tikzpicture}[anchorbase, scale=.5]
\draw [very thick, ->] (2,1) to (0,1);
\draw [very thick, ->] (2,0) to (0,0);
\end{tikzpicture}
\;\;-\;\;
 q\;
\begin{tikzpicture}[anchorbase, scale=.5]
\draw [very thick] (2,0) to[out=180,in=315] (1.3,.5);
\draw [very thick] (2,1) to[out=180,in=45] (1.3,.5);
\draw [double] (1.3,.5) -- (.7,.5);
\draw [very thick, ->] (.7,.5) to[out=135,in=0]  (0,1);
\draw [very thick, ->] (.7,.5) to[out=225,in=0] (0,0);
\end{tikzpicture}
\quad,\quad
\begin{tikzpicture}[anchorbase, scale=.5]
\draw [very thick, ->] (2,0) to [out=180,in=0] (0,1);
 \draw [white,line width=.15cm] (2,1) to [out=180,in=0] (0,0) ;
\draw [very thick, ->] (2,1) to [out=180,in=0] (0,0);
\end{tikzpicture}
\;\;:=\;\;
\begin{tikzpicture}[anchorbase, scale=.5]
\draw [very thick, ->] (2,1) to (0,1);
\draw [very thick, ->] (2,0) to (0,0);
\end{tikzpicture}
\;\;-\;\;
q^{-1}   \;
\begin{tikzpicture}[anchorbase, scale=.5]
\draw [very thick] (2,0) to[out=180,in=315] (1.3,.5);
\draw [very thick] (2,1) to[out=180,in=45] (1.3,.5);
\draw [double] (1.3,.5) -- (.7,.5);
\draw [very thick, ->] (.7,.5) to[out=135,in=0]  (0,1);
\draw [very thick, ->] (.7,.5) to[out=225,in=0] (0,0);
\end{tikzpicture}
\nonumber
\\
\label{eq:crossing}
\begin{tikzpicture}[anchorbase, scale=.5]
\draw [very thick, ->] (2,1) to [out=180,in=0] (0,0);
 \draw [white,line width=.15cm] (2,0) to [out=180,in=0] (0,1) ;
\draw [double, ->] (2,0) to [out=180,in=0] (0,1);
\end{tikzpicture}
\;\;:= \;\;-q \;\;\,
\begin{tikzpicture}[anchorbase, scale=.5]
\draw [double] (2,0) -- (1.4,0);
\draw [very thick, ->] (1.4,0) -- (0,0);
\draw [very thick] (2,1) -- (.6,1);
\draw [double, ->] (0.6,1) -- (0,1);
\draw [very thick] (.6,1) -- (1.4,0);
\end{tikzpicture} 
\quad,\quad
\begin{tikzpicture}[anchorbase, scale=.5]
\draw [double, ->] (2,1) to [out=180,in=0] (0,0);
 \draw [white,line width=.15cm] (2,0) to [out=180,in=0] (0,1) ;
\draw [very thick, ->] (2,0) to [out=180,in=0] (0,1);
\end{tikzpicture}
\;\;:= \;\;-q  \;\;\,
\begin{tikzpicture}[anchorbase, scale=.5]
\draw [double] (2,1) -- (1.4,1);
\draw [very thick, ->] (1.4,1) -- (0,1);
\draw [very thick] (2,0) -- (.6,0);
\draw [double, ->] (0.6,0) -- (0,0);
\draw [very thick] (.6,0) -- (1.4,1);
\end{tikzpicture} 
\quad,\quad
\begin{tikzpicture}[anchorbase, scale=.5]
\draw [double, ->] (2,1) to [out=180,in=0] (0,0);
 \draw [white,line width=.15cm] (2,0) to [out=180,in=0] (0,1) ;
\draw [double, ->] (2,0) to [out=180,in=0] (0,1);
\end{tikzpicture}
\;\;:= \;\;q^{2}  \;\;\,
\begin{tikzpicture}[anchorbase, scale=.5]
\draw [double, ->] (2,0) -- (0,0);
\draw [double, ->] (2,1) -- (0,1);
\end{tikzpicture}
\\ \nonumber
\begin{tikzpicture}[anchorbase, scale=.5]
\draw [very thick, ->] (2,0) to [out=180,in=0] (0,1);
 \draw [white,line width=.15cm] (2,1) to [out=180,in=0] (0,0) ;
\draw [double, ->] (2,1) to [out=180,in=0] (0,0);
\end{tikzpicture}
\;\;:= \;\;
-q^{-1}   \;
\begin{tikzpicture}[anchorbase, scale=.5]
\draw [double] (2,1) -- (1.4,1);
\draw [very thick, ->] (1.4,1) -- (0,1);
\draw [very thick] (2,0) -- (.6,0);
\draw [double, ->] (0.6,0) -- (0,0);
\draw [very thick] (.6,0) -- (1.4,1);
\end{tikzpicture} 
\quad,\quad
\begin{tikzpicture}[anchorbase, scale=.5]
\draw [double, ->] (2,0) to [out=180,in=0] (0,1);
 \draw [white,line width=.15cm] (2,1) to [out=180,in=0] (0,0) ;
\draw [very thick, ->] (2,1) to [out=180,in=0] (0,0);
\end{tikzpicture}
\;\;:= \;\;
- q^{-1} \;
\begin{tikzpicture}[anchorbase, scale=.5]
\draw [double] (2,0) -- (1.4,0);
\draw [very thick, ->] (1.4,0) -- (0,0);
\draw [very thick] (2,1) -- (.6,1);
\draw [double, ->] (0.6,1) -- (0,1);
\draw [very thick] (.6,1) -- (1.4,0);
\end{tikzpicture} 
\quad,\quad
\begin{tikzpicture}[anchorbase, scale=.5]
\draw [double, ->] (2,0) to [out=180,in=0] (0,1);
 \draw [white,line width=.15cm] (2,1) to [out=180,in=0] (0,0) ;
\draw [double, ->] (2,1) to [out=180,in=0] (0,0);
\end{tikzpicture}
\;\;:= \;\; q^{-2} \;
\begin{tikzpicture}[anchorbase, scale=.5]
\draw [double, ->] (2,0) -- (0,0);
\draw [double, ->] (2,1) -- (0,1);
\end{tikzpicture} 
\end{gather}
Tangles web diagrams that differ by Reidemeister type II and III or fork slide moves have the same evaluation in $\SWebq$. In particular, $\SWebq$ carries a well-defined $\Z[q^{\pm 1}]$-linear associative multiplication $*$, which is defined on webs $W_1$ and $W_2$ as the element $W_1*W_2$ obtained by superposing $W_2$ with $W_1$ and resolving all crossings. We call $(\SWebq,*)$ the $\glnn{2}$ skein algebra of $\Su$.
\end{definition}

\begin{remark} Alternatively, $\SWebq$ can be defined as the quotient of the free $\Z[q^{\pm 1}]$-module spanned by isotopy classes of framed, tangled $\glnn{2}$ webs embedded in $\Su\times [0,1]$, by the relations \eqref{eqn:circles}, \eqref{eqn:digons}, \eqref{eqn:squares} and \eqref{eq:crossing}. One can then also choose to either impose these relations in every $3$-ball embedded in $\Su\times [0,1]$, where the shown web diagrams arise by projection onto a specified equatorial plane with an appropriate blackboard framing, or just in balls of the form $\cat{D}\times [0,1]$ with $\cat{D}\hookrightarrow \Su$. The results are naturally isomorphic, and we will switch freely between these descriptions.
\end{remark}

\begin{example}
\label{exa:1001} We compute the product of two simple closed curves in the toric skein algebra $\TWebq$.
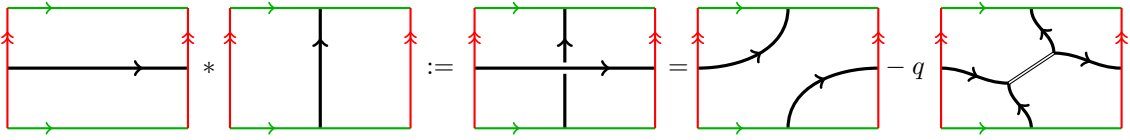
\begin{equation*}
\begin{tikzpicture}[anchorbase]
\draw[very thick, directed=.75] (0,.8) to (2.4,.8);
\torus{2.4}{1.6}
\end{tikzpicture}
\;*\;
\begin{tikzpicture}[anchorbase]
\draw[very thick, directed=.75] (1.2,0) to (1.2,1.6);
\torus{2.4}{1.6}
\end{tikzpicture}
\;:=\;
\begin{tikzpicture}[anchorbase]
\draw[very thick, directed=.75] (1.2,0) to (1.2,1.6);
\draw[white, line width=.15cm] (2.4,.8) to (0,.8);
\draw[very thick, directed=.75] (0,.8) to (2.4,.8);
\torus{2.4}{1.6}
\end{tikzpicture}
=
\begin{tikzpicture}[anchorbase]
\draw[very thick, directed=.55] (1.2,0) to [out=90,in=180] (2.4,.8);
\draw[very thick, directed=.55] (0,.8) to [out=0,in=270] (1.2,1.6);
\torus{2.4}{1.6}
\end{tikzpicture} 
 - q \;\;
\begin{tikzpicture}[anchorbase]
\draw[very thick, directed=.55] (1.2,0) to [out=90,in=270] (0.9,0.6);
\draw[double] (0.9,0.6) to (1.5,1);
\draw[very thick, directed=.55] (1.5,1) to [out=0,in=180] (2.4,.8);
\draw[very thick, directed=.55] (0,.8) to [out=0,in=180] (.9,0.6);
\draw[very thick, directed=.55] (1.5,1) to [out=90,in=270] (1.2,1.6);
\torus{2.4}{1.6}
\end{tikzpicture}
\end{equation*}
\end{example}

In the following, we associate simple topological invariants to webs in $\Su$.

\begin{definition}\label{def:homologyclass} Let $W$ be a web on $\Su$. The class $[W]\in H_1(\Su)$ is defined as the homology class of the multi-curve $\overline{W}$ obtained from $W$ by replacing each $2$-labeled edge by two parallel $1$-labeled edges and smoothing out trivalent vertices in the process.
\end{definition}

Note that the homology class of a web as defined above is invariant under all $\glnn{2}$ web relations.

\begin{lemma} $\SWebq$ is a $H_1(\Su)$-graded algebra.
\end{lemma}
\begin{proof} Let $c\in H_1(\Su)$, then the component of $\SWebq$ of degree $c$ is spanned by those webs with homology class $c$. It is also clear from the crossing resolution rules \eqref{eq:crossing} that the grading is respected by the multiplication, i.e. $[W_1*W_2]=[W_1]+[W_2]$.
\end{proof}

\begin{lemma} \label{lem:2labsub} The subalgebra of $2$-labeled webs in $\SWebq$ is spanned by a collection of $2$-labeled multi-curves $\w{x}$ parametrized by $x\in H_1(\Su)$, with elements satisfying $[\w{x}]=2x$ and $\w{0}=\emptyset$. The skein algebra multiplication is given by
\[\w{x} * \w{y} = q^{2 x\cdot y} \w{x+y} \] where $x\cdot y$ denotes the intersection pairing. In particular, any element of the form $\w{x}$ is invertible in $\SWebq$ with inverse given by $\w{-x}$.
\end{lemma}

Another topological invariant of webs in $\Su$ is their underlying integer lamination.

\begin{definition}
\label{def:intlam}
An integer lamination on $\Su$ is an unordered collection $\{(C_i,n_i)\}$ (possibly empty) where the $n_i$ are positive integers and the $C_i$ are disjointly embedded and pairwise non-isotopic, unoriented simple closed curves in $\Su$. We further require that the $C_i$ are essential in the sense that they do not bound disks in $\Su$. Curves that bound disks are called inessential.

Given a web $W$ on $\Su$ and its underlying unoriented multi-curve $c(W)$ we let $l(W)$ denote the integer lamination obtained from $c(W)$ by erasing all inessential curves and then recording one copy $C_i$ of each isotopy class of remaining essential simple closed curves together with the number $n_i$ of its parallel copies. A web $W$ is called inessential if $l(W)=\emptyset$.
\end{definition}

Note that the integer laminations associated to webs are invariant under all web relations. 

In the following, we denote by $\cal{L}(\Su)$ a set of representatives of equivalence classes of integer laminations on $\Su$ up to isotopy, or equivalently, for the set of crossingless links in $\Su$ without inessential components, including the empty link. For the laminations in the set $\cal{L}(\Su)$ we choose (arbitrarily, but once and for all) an orientation on each $C_i$, which we also put on its parallel copies.

\begin{proposition}\label{prop:stdbasis} $\SWebq$ is a free $\Z[q^{\pm 1}]$-module and a basis is given by 
\[\Su\basisstd:=\{L*\w{x}| L\in \cal{L}(\Su), x\in H_1(\Su) \}\] 
\end{proposition}
We call this the standard basis. 

\begin{proof} The proof of the fact that this set spans $\SWebq$ is given in the next lemma. Provided that all elements of $\Su\basisstd$ are neither $\Z[q^{\pm 1}]$-torsion nor zero, they are linearly independent since each element is uniquely characterized by its homology class and underlying integer lamination. The absence of torsion and non-triviality can be proven directly using a Diamond Lemma argument, or deduced via Lemmas~\ref{lem:gl22sl2-1} and \ref{lem:gl22sl2-2} from the analogous results for the Kauffman bracket skein algebra, see e.g. \cite[Theorem 3.1]{Prz2} or \cite[Proposition 3.8]{Thu}.
\end{proof}

\begin{lemma} \label{lem:equivobjectsdecat} Every web $W$ in $\SWebq$ is equal to a $\N[-q^{\pm 1}]$-multiple of an element of $\Su\basisstd$.
\end{lemma}
\begin{proof} First of all, by Lemma~\ref{lem:neckcut} we may assume that $c(W)$ has no inessential components. Then $c(W)=l(W)$ is equivalent to precisely one lamination $\vec{l}(W)$ in $\cal{\Su}$, which furthermore comes with a prescribed orientation. Next we want to write $W=\vec{l}(W)*W_2$ for some $2$-labeled multi-curve $W_2$. Indeed, up to a power of $-q$, the underlying oriented multi-curve of $W_2$ can be taken to be the concatenation of all $2$-labeled edges in $W$ and all those $1$-labeled edges in $W$ which carry the opposite orientation than the one prescribed by $\vec{l}$. Finally, by relation \eqref{eqn:circles} we may assume that $W_2$ has no inessential components.
\end{proof}

The products of elements in $\Su\basisstd$ are usually quite complicated, with a few exceptions as follows.

\begin{lemma} Let $X$ be an oriented $1$-labeled essential simple closed curve with $[X]=x\in H_1(\Su)$ and denote by $-X$ the same curve with the opposite orientation. Then we have: 
\begin{align}
\label{eq:reorient}(X,n)*\w{-nx}&=(-X,n)\\ 
(X,n)*\w{y} &= q^{2n x\cdot y} \w{y} * (X,n)\\
(X,n)* (X,m) &= (X,m+n)
\end{align}
\end{lemma}

In the following, we describe two alternative bases. Let $\{(C_i,n_i)\}$ be an element of $\cal{L}(\Su)$, which we can interpret as a link on $\Su$ and thus also as a web. More precisely, we identify $\{(C_i,n_i)\}$ with the skein algebra product $\prod_i C_i^{n_i}$. Then we define $\{(C_i,n_i)\}_T$ to be the $\Z$-linear combination of webs obtained as the skein algebra product $\prod_i T_{n_i}(C_i)$ where $T_1(C_i)=C_i$, $T_2(C_i)=C_i^2-2\w{[C_i]}$ and all higher order operations are recursively defined as $T_{m+1}(C_i)=T_m(C_i)*C_i - T_{m-1}(C_i)*\w{[C_i]}$. Similarly, $\{(C_i,n_i)\}_{\mathrm{S}}$ is defined as the skein algebra product $\prod_i S_{n_i}(C_i)$ where $S_1(C_i)=C_i$, $S_2(C_i)=C_i^2-\w{[C_i]}$ and higher order operations are defined by the same recursion as the $T_{i}$.

\begin{definition} \label{def:cheb} We consider the bases
\[\Su\basis:=\{L_T*\w{x}| L\in \cal{L}(\Su), x\in H_1(\Su) \}\] 
and
\[\Su\basisJW:=\{L_{\mathrm{S}}*\w{x}| L\in \cal{L}(\Su), x\in H_1(\Su) \}.\] 
\end{definition}
It is clear that $\Su\basis$ and $\Su\basisJW$ are bases, because they are triangularly equivalent to $\Su\basisstd$. In fact, the basis changes from $\Su\basis$ to $\Su\basisJW$ and further to $\Su\basisstd$ are positive: every element of $\Su\basisstd$ is a $\N$-linear combination of elements of $\Su\basisJW$, which are further $\N$-linear combination of elements of $\Su\basis$. This follows from the fact that these bases are modeled on the bases of the symmetric polynomial ring $\Z[q^{\pm 1}][x_1,x_2]^{\cal{S}_2}\cong \Z[q^{\pm 1}][x_1+x_2, x_1x_2]$ given by $(x_1+x_2)^a(x_1x_2)^b$, $(x_1^a+x_1^{a-1}x_2\cdots x_1x_2^{a-1}+x_2^a)(x_1x_2)^b$, and $(x_1^a+x_2^a)(x_1x_2)^b$ for $a,b\geq 0$.

\subsection{Positivity conjectures}
\label{sec:pos}
The main appeal of $\Su\basis$ is that it seems to have \textit{positive} structure constants in $\N[-q^{\pm 1}]$, see Conjecture~\ref{conj:positivity}. We now argue that this conjecture is equivalent to an analogous conjecture of Fock--Goncharov \cite[Conjecture 12.4]{FoG} and Thurston \cite[Conjecture 4.20]{Thu} for the Kauffman bracket skein algebra $\mathrm{Sk}_A(\Su)$. The latter is defined as the quotient of the free $\Z[A^{\pm 1}]$-module spanned by unoriented framed links in $\Su\times [0,1]$ modulo the following skein relations:
\begin{equation}
\begin{tikzpicture}[anchorbase, scale=.5]
\draw [very thick] (2,1) to [out=180,in=0] (0,0);
 \draw [white,line width=.15cm] (2,0) to [out=180,in=0] (0,1) ;
\draw [very thick] (2,0) to [out=180,in=0] (0,1);
\end{tikzpicture}
\;= \; A\;\;
\begin{tikzpicture}[anchorbase, scale=.5]
\draw [very thick] (2,1) to (0,1);
\draw [very thick] (2,0) to [out=180,in=0] (0,0);
\end{tikzpicture}
\;
+ \; A^{-1}\;
\begin{tikzpicture}[anchorbase, scale=.5]
\draw [very thick] (2,1) to [out=180,in=180] (2,0);
\draw [very thick] (0,0) to [out=0,in=0] (0,1);
\end{tikzpicture}
\quad,\quad
\begin{tikzpicture}[anchorbase, scale=.5]
\draw [very thick] (0,0) to [out=90,in=180] (.5,.5) to [out=0,in=90] (1,0)to [out=270,in=0] (.5,-.5) to [out=180,in=270] (0,0);
\end{tikzpicture} 
\;= \; -A^2 - A^{-2}
\end{equation}
In order to relate this to the $\glnn{2}$ skein algebra $\SWebq$, we have to extend scalars in the latter to $\Z[A^{\pm 1}]$ by letting $q$ act as $-A^{-2}$. We call the result $\SWebA$ and see that the skein relation for the trivial $1$-labeled circle agrees with the one in $\mathrm{Sk}_A(\Su)$, but we also have to adjust the skein relations involving crossings. For example, in $\SWebA$ we have the relation: 
\[
\begin{tikzpicture}[anchorbase, scale=.5]
\draw [very thick, ->] (2,1) to [out=180,in=0] (0,0);
 \draw [white,line width=.15cm] (2,0) to [out=180,in=0] (0,1) ;
\draw [very thick, ->] (2,0) to [out=180,in=0] (0,1);
\end{tikzpicture}
 \;\;= \;\;
\begin{tikzpicture}[anchorbase, scale=.5]
\draw [very thick, ->] (2,1) to (0,1);
\draw [very thick, ->] (2,0) to (0,0);
\end{tikzpicture}
\;\; +\;\;
 A^{-2} \;
\begin{tikzpicture}[anchorbase, scale=.5]
\draw [very thick] (2,0) to[out=180,in=315] (1.3,.5);
\draw [very thick] (2,1) to[out=180,in=45] (1.3,.5);
\draw [double] (1.3,.5) -- (.7,.5);
\draw [very thick, ->] (.7,.5) to[out=135,in=0]  (0,1);
\draw [very thick, ->] (.7,.5) to[out=225,in=0] (0,0);
\end{tikzpicture}\]
Instead, we consider the skein algebra $\SWebA^\prime$ defined just as $\SWebA$, but with positive crossings between $i$- and $j$-labeled strands rescaled by $A^{ij}$ (and negative crossings by the inverse scalar), i.e.:

\begin{gather}
\begin{tikzpicture}[anchorbase, scale=.5]
\draw [very thick, ->] (2,1) to [out=180,in=0] (0,0);
 \draw [white,line width=.15cm] (2,0) to [out=180,in=0] (0,1) ;
\draw [very thick, ->] (2,0) to [out=180,in=0] (0,1);
\end{tikzpicture}
 \;\;= \;\;
 A\; 
\begin{tikzpicture}[anchorbase, scale=.5]
\draw [very thick, ->] (2,1) to (0,1);
\draw [very thick, ->] (2,0) to (0,0);
\end{tikzpicture}
\;\; +\;\;
 A^{-1} \;
\begin{tikzpicture}[anchorbase, scale=.5]
\draw [very thick] (2,0) to[out=180,in=315] (1.3,.5);
\draw [very thick] (2,1) to[out=180,in=45] (1.3,.5);
\draw [double] (1.3,.5) -- (.7,.5);
\draw [very thick, ->] (.7,.5) to[out=135,in=0]  (0,1);
\draw [very thick, ->] (.7,.5) to[out=225,in=0] (0,0);
\end{tikzpicture}
\quad,\quad
\begin{tikzpicture}[anchorbase, scale=.5]
\draw [very thick, ->] (2,0) to [out=180,in=0] (0,1);
 \draw [white,line width=.15cm] (2,1) to [out=180,in=0] (0,0) ;
\draw [very thick, ->] (2,1) to [out=180,in=0] (0,0);
\end{tikzpicture}
\;\;=\;\;
A  \;
\begin{tikzpicture}[anchorbase, scale=.5]
\draw [very thick] (2,0) to[out=180,in=315] (1.3,.5);
\draw [very thick] (2,1) to[out=180,in=45] (1.3,.5);
\draw [double] (1.3,.5) -- (.7,.5);
\draw [very thick, ->] (.7,.5) to[out=135,in=0]  (0,1);
\draw [very thick, ->] (.7,.5) to[out=225,in=0] (0,0);
\end{tikzpicture}
\;\;+\;\;
A^{-1}
\begin{tikzpicture}[anchorbase, scale=.5]
\draw [very thick, ->] (2,1) to (0,1);
\draw [very thick, ->] (2,0) to (0,0);
\end{tikzpicture}\nonumber
\\
\label{eq:thickcrossing}
\begin{tikzpicture}[anchorbase, scale=.5]
\draw [very thick, ->] (2,1) to [out=180,in=0] (0,0);
 \draw [white,line width=.15cm] (2,0) to [out=180,in=0] (0,1) ;
\draw [double, ->] (2,0) to [out=180,in=0] (0,1);
\end{tikzpicture}
\;\;=\;\;
\begin{tikzpicture}[anchorbase, scale=.5]
\draw [double] (2,0) -- (1.4,0);
\draw [very thick, ->] (1.4,0) -- (0,0);
\draw [very thick] (2,1) -- (.6,1);
\draw [double, ->] (0.6,1) -- (0,1);
\draw [very thick] (.6,1) -- (1.4,0);
\end{tikzpicture} 
\quad,\quad
\begin{tikzpicture}[anchorbase, scale=.5]
\draw [double, ->] (2,1) to [out=180,in=0] (0,0);
 \draw [white,line width=.15cm] (2,0) to [out=180,in=0] (0,1) ;
\draw [very thick, ->] (2,0) to [out=180,in=0] (0,1);
\end{tikzpicture}
\;\;=\;\;
\begin{tikzpicture}[anchorbase, scale=.5]
\draw [double] (2,1) -- (1.4,1);
\draw [very thick, ->] (1.4,1) -- (0,1);
\draw [very thick] (2,0) -- (.6,0);
\draw [double, ->] (0.6,0) -- (0,0);
\draw [very thick] (.6,0) -- (1.4,1);
\end{tikzpicture} 
\quad,\quad
\begin{tikzpicture}[anchorbase, scale=.5]
\draw [double, ->] (2,1) to [out=180,in=0] (0,0);
 \draw [white,line width=.15cm] (2,0) to [out=180,in=0] (0,1) ;
\draw [double, ->] (2,0) to [out=180,in=0] (0,1);
\end{tikzpicture}
\;\;=\;\;
\begin{tikzpicture}[anchorbase, scale=.5]
\draw [double, ->] (2,0) -- (0,0);
\draw [double, ->] (2,1) -- (0,1);
\end{tikzpicture}
\\ \nonumber
\begin{tikzpicture}[anchorbase, scale=.5]
\draw [very thick, ->] (2,0) to [out=180,in=0] (0,1);
 \draw [white,line width=.15cm] (2,1) to [out=180,in=0] (0,0) ;
\draw [double, ->] (2,1) to [out=180,in=0] (0,0);
\end{tikzpicture}
\;\;=\;\;
\begin{tikzpicture}[anchorbase, scale=.5]
\draw [double] (2,1) -- (1.4,1);
\draw [very thick, ->] (1.4,1) -- (0,1);
\draw [very thick] (2,0) -- (.6,0);
\draw [double, ->] (0.6,0) -- (0,0);
\draw [very thick] (.6,0) -- (1.4,1);
\end{tikzpicture} 
\quad,\quad
\begin{tikzpicture}[anchorbase, scale=.5]
\draw [double, ->] (2,0) to [out=180,in=0] (0,1);
 \draw [white,line width=.15cm] (2,1) to [out=180,in=0] (0,0) ;
\draw [very thick, ->] (2,1) to [out=180,in=0] (0,0);
\end{tikzpicture}
\;\;=\;\;
\begin{tikzpicture}[anchorbase, scale=.5]
\draw [double] (2,0) -- (1.4,0);
\draw [very thick, ->] (1.4,0) -- (0,0);
\draw [very thick] (2,1) -- (.6,1);
\draw [double, ->] (0.6,1) -- (0,1);
\draw [very thick] (.6,1) -- (1.4,0);
\end{tikzpicture} 
\quad,\quad
\begin{tikzpicture}[anchorbase, scale=.5]
\draw [double, ->] (2,0) to [out=180,in=0] (0,1);
 \draw [white,line width=.15cm] (2,1) to [out=180,in=0] (0,0) ;
\draw [double, ->] (2,1) to [out=180,in=0] (0,0);
\end{tikzpicture}
\;\;=\;\;
\begin{tikzpicture}[anchorbase, scale=.5]
\draw [double, ->] (2,0) -- (0,0);
\draw [double, ->] (2,1) -- (0,1);
\end{tikzpicture} 
\end{gather}

\begin{lemma}\label{lem:gl22sl2-1}
$\SWebA$ is isomorphic to $\SWebA^\prime$ as a $\Z[A^{\pm 1}]$-module via the map that is the identity on the elements of $\Su\basis$ (or equivalently $\Su\basisstd$). Moreover, the pull-back $\bigstar$ of the multiplication on $\SWebA^\prime$ to $\SWebA$ is related to the usual multiplication $*$ as follows:
\[W_1\bigstar W_2 = A^{[W_1]\cdot [W_2]} W_1 * W_2\]
Here $W_1$ and $W_2$ denote webs, and $[W_1]\cdot [W_2]$ is the intersection pairing of their first homology classes. 
\end{lemma}

\begin{lemma}\label{lem:gl22sl2-2}
The map $\mathrm{forget}$ that forgets orientations and $2$-labeled edges in webs induces an $\Z[A^{\pm 1}]$-algebra epimorphism from $\SWebA^\prime$ to $\mathrm{Sk}_A(\Su)$.
\end{lemma}

The following proposition expresses the equivalence of Conjecture~\ref{conj:positivity} and the conjecture of Fock--Goncharov \cite[Conjecture 12.4]{FoG} as presented by Thurston \cite[Conjecture 4.20]{Thu}. Indeed, the sets $\mathrm{forget}(\Su\basis)$, $\mathrm{forget}(\Su\basisJW)$, and $\mathrm{forget}(\Su\basisstd)$ coincide with Thurston's \textit{bracelets}, \textit{bands} and \textit{bangles bases} of $\mathrm{Sk}_A(\Su)$. This follows directly from Definition~\ref{def:cheb} and \cite[Propositions 4.4 and 4.8]{Thu}.

\begin{proposition} The multiplication of elements in $\Su\basis$ is positive in $-q$ in $\SWebq$ if and only if the multiplication of elements of $\mathrm{forget}(\Su\basis)$ is positive in $A$ in $\mathrm{Sk}_A(\Su)$.
\end{proposition}
\begin{proof} First of all, by Lemma~\ref{lem:gl22sl2-1} and the formula $(-q)=A^{-2}$, positivity in $-q$ in $\SWebq$ is equivalent to positivity in $A$ in $\SWebA^\prime$. 

Given two elements $x,y$ of $\mathrm{forget}(\Su\basis)$, we lift them to elements $\overline{x}, \overline{y}$ of $\Su\basis$ in $\SWebA^\prime$, which are automatically $H_1(\Su)$-homogeneous. Assuming that $\overline{x}* \overline{y}$ has an expansion in terms of elements of $\Su\basis$ with coefficients in $\N[A^{\pm 1}]$, we use Lemma~\ref{lem:gl22sl2-2} to deduce that $x*y=\mathrm{forget}(\overline{x}* \overline{y})$, which is thus manifestly positive as well.  

Conversely, suppose that $X$ and $Y$ are elements of $\Su\basis$ with an expansion $X*Y=\sum \alpha_i Z_i$, where $Z_i$ are distinct elements of $\Su\basis$ with $[Z_i]=[X]+[Y]$. Then, again by Lemma~\ref{lem:gl22sl2-2}, $\mathrm{forget}(X)*\mathrm{forget}(Y)=\sum \alpha_i \mathrm{forget}(Z_i)$. Now note that distinct elements of $\Su\basis$ of a fixed first homology class are sent to distinct elements of $\mathrm{forget}(\Su\basis)$. Thus, positivity in $\mathrm{Sk}_A(\Su)$ implies $\alpha_i\in \N[A^{\pm 1}]$, which finishes the proof.
\end{proof}

\subsection{The skein algebra of the torus}
\label{sec:skeinalgtorus}
In this section we study $\TWebq$, the $\glnn{2}$ skein algebra of the torus and its basis $\T\basis$, which satisfies a suitably modified Frohman-Gelca formula that is manifestly positive. We also show that $\TWebq$ is isomorphic to the symmetric subalgebra of a suitable quantum torus. 

We choose, once and for all, two oriented simple closed curves $\lambda$ and $\mu$ that give generators for the first homology of $\T$. The ordered pair $(\lambda$,$\mu)$ determines an orientation of $\T$ and fixes a convention for illustrating $\T$:
\[\T \quad =\quad  \begin{tikzpicture}[anchorbase]
\torus{2.4}{1.6};
\node [green] at (1.2,.2) {$\lambda$};
\node [red] at (0.2,.8) {$\mu$};
\end{tikzpicture}\]
We say that an oriented simple closed curve on $\T$ is an $(a,b)$-curve if it represents the homology class $a[\lambda] + b [\mu]$ for $a,b\in \Z$ and $(a,b)\neq (0,0)$. These curves are essential and satisfy $\gcd(a,b)=1$. More generally, $(k a,k b)$-(multi)curves for $k\geq 1$ are defined to be $k$-fold oriented parallels of $(a,b)$-curves. Here we reserve the notation $(0,0)$ for the empty curve. 

We choose such oriented multi-curves $(m,n)$ as representatives of integer laminations on $\T$:
\[\cal{L}(\T):=\{(m,n)| m,n\in \Z, m>0 \text{ or } n\geq m=0\}\]

The slope of a simple closed curve on $\T$ in the homology class $m [\mu] + n [\lambda]$, is defined to be $m/n \in \overline{\Q}=\Q\cup \{\infty\}$. Inessential curves can be considered to have any element of $\overline{\Q}$ as slope. A multi-curve has slope $m/n$ if each of its components has this slope. In the following, we also assign slopes to webs.

\begin{definition}\label{def:slope} Let $W$ be a web on $\T$. We say $W$ has slope $m/n$ if the underlying $1$-labeled multi-curve $c(W)$ has this slope. If every component of $c(W)$ is inessential we call $W$ inessential and otherwise essential.
\end{definition}

Note that the slope as defined above is invariant under all $\glnn{2}$ web relations. A sufficient (though not necessary) condition for a web $W$ to have slope $m/n$ is that $W$ is supported in a tubular (or rather, annular) neighborhood of an $(m,n)$-curve.

\begin{corollary} \label{cor:2labsub} The subalgebra of $2$-labeled webs in $\TWebq$ is spanned by the multi-curves $\w{m,n}$ which satisfy $[\w{m,n}]=(2m,2n)$ in  $H_1(\T)\cong \Z^2$. The skein algebra multiplication is given by:
\[\w{m,n} * \w{r,s} = q^{2(ms-nr)} \w{m+r,n+s} \]
In particular, the elements $\w{m,n}$ are invertible in $\TWebq$.
\end{corollary}
\begin{proof}
This is a special case of Lemma~\ref{lem:2labsub}.
\end{proof}

We now reformulate the two bases from the previous section in the special case of the torus. The standard basis and the $T$-basis for $\TWebq$ are given by:
\begin{align*} \T\basisstd &:=\{(m,n)*\w{r,s} | m,n,r,s\in \Z, m>0 \text{ or } n\geq m=0   \}\\
\T\basis &:=\{(m,n)_T*\w{r,s}| m,n,r,s\in \Z, m>0 \text{ or } n>m=0\}\cup \{\w{r,s}| r,s\in \Z\}
\end{align*}
where $(m,n)_T$ is defined by induction on $\gcd(m,n)$ as follows: 
\[(m,n)_T:= \begin{cases}
(m,n)&\text{ if } \gcd(m,n)=1
\\
(m,n)-2\w{m/2,n/2}&\text{ if } \gcd(m,n)=2
\\
(m-a,n-b)_T*(a,b)- (m-2a,n-2b)_T * \w{a,b} &\text{ if } \gcd(m,n)=d \geq 3  \\
&\text{ and } (m,n)=(da,db)
\end{cases}\]
Note that the definition of $(m,n)_T$ makes sense for any $m,n\in\Z$, also for those pairs that do not play a role in $\T\basis$. We also set $(0,0)_T:=2$. The following result is a $\glnn{2}$-version of a theorem due to Frohman and Gelca \cite{FG}.

\begin{theorem} \label{thm:FG} The $\Z[q^{\pm 1}]$-algebra $\TWebq$ is isomorphic to the abstract $\Z[q^{\pm 1}]$-algebra with generators $(m,n)_T$ and $\w{m,n}$ with $(m,n)\in \Z^2$, subject to the following relations:
 \begin{align}
 \label{eq:FG}
 (m,n)_T*(r,s)_T &= (m+r,n+s)_T + (m-r,n-s)_T*\w{r,s}
 \\
 \nonumber
 (m,n)_T*\w{r,s} &= q^{2(ms-nr)} \w{r,s} *(m,n)_T
 \\
 \nonumber
 (m,n)_T*\w{-m,-n} &=  (-m,-n)_T
  \\
 \nonumber
 \w{m,n} * \w{r,s} &= q^{2(ms-nr)} \w{m+r,n+s}
 \\
 \nonumber
 (0,0)_T=2, &\quad \w{0,0}=1
 \end{align}
 We will refer to the multiplication rule \eqref{eq:FG} as the Frohman--Gelca formula.
\end{theorem}

\begin{remark} In \cite{MS} Morton--Samuelson give presentations for the $\glnn{N}$ and HOMFLY-PT skein algebras of the torus. In the case of $\glnn{2}$, their generators agree with our basis elements $(m,n)_T$. However, in choosing a different scaling for crossings, they avoid the use of $2$-labeled basis elements $\w{a,b}$, which comes at the cost of passing from multiplication rules as in \eqref{eq:FG} to commutator identities.  
\end{remark}

\begin{example}\label{exa:2101}  We expand $(2,1)*(0,1)$ as follows:
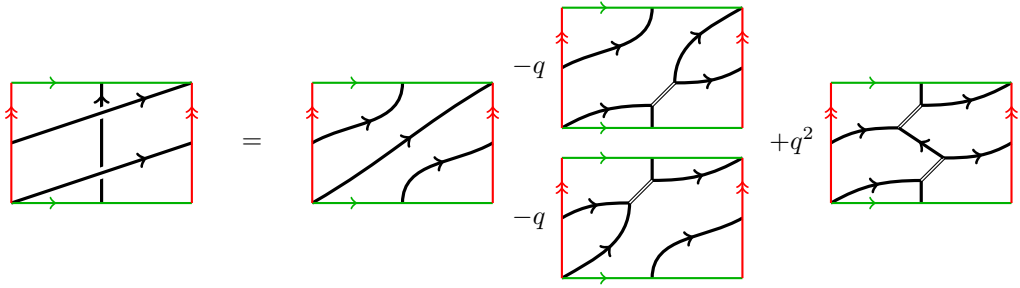
\begin{equation*}
\xy
(-10,0)*{
\begin{tikzpicture}[anchorbase]
\draw[very thick, directed=.9] (1.2,0) to (1.2,1.6);
\draw[white, line width=.15cm] (0,0) to (2.4,.8);
\draw[white, line width=.15cm] (0,.8)  to (2.4,1.6);
\draw[very thick, directed=.75] (0,0) to (2.4,.8);
\draw[very thick, directed=.75] (0,.8)  to (2.4,1.6);
\torus{2.4}{1.6}
\end{tikzpicture}
};
(10,0)*{=
};
(95,0)*{
 + q^2\;\;
\begin{tikzpicture}[anchorbase]
\draw[very thick, directed=.55] (0,0) [out=30,in=180] to (1.2,.3);
\draw[very thick, directed=.55] (0,.8)to [out=30,in=180] (.9,1);
\draw[very thick] (1.2,1.3)  to [out=90,in=270] (1.2,1.6);
\draw[very thick] (1.2,0) [out=90,in=270] to (1.2,.3);
\draw[very thick, directed=.55] (1.5,.6) [out=0,in=210] to (2.4,.8);
\draw[double] (1.2,.3) to (1.5,.6) ;
\draw[double]  (.9,1) to (1.2,1.3) ;
\draw[very thick, directed=.55] (1.5,.6) to (.9,1);
\draw[very thick, directed=.55] (1.2,1.3) [out=0,in=210] to (2.4,1.6);
\torus{2.4}{1.6}
\end{tikzpicture}
};
(60,10)*{
 -q \;\;
\begin{tikzpicture}[anchorbase]
\draw[very thick, directed=.55] (0,0) [out=30,in=180] to (1.2,.3);
\draw[very thick, directed=.55] (0,.8) [out=30,in=270] to (1.2,1.6);
\draw[very thick] (1.2,0) [out=90,in=270] to (1.2,.3);
\draw[very thick, directed=.55] (1.5,.6) [out=0,in=210] to (2.4,.8);
\draw[double] (1.2,.3) to (1.5,.6) ;
\draw[very thick, directed=.55] (1.5,.6) to [out=90,in=210](2.4,1.6);
\torus{2.4}{1.6}
\end{tikzpicture}
};
(60,-10)*{
 -q \;\;
\begin{tikzpicture}[anchorbase]
\draw[very thick, directed=.55] (0,0) [out=30,in=270] to (.9,1);
\draw[very thick, directed=.55] (0,.8)to [out=30,in=180] (.9,1);
\draw[very thick] (1.2,1.3)  to [out=90,in=270] (1.2,1.6);
\draw[very thick, directed=.55] (1.2,0) [out=90,in=210] to (2.4,.8);
\draw[very thick, directed=.55] (1.2,1.3) [out=0,in=210] to (2.4,1.6);
\draw[double]  (.9,1) to (1.2,1.3) ;
\torus{2.4}{1.6}
\end{tikzpicture}
};
(30,0)*{
\begin{tikzpicture}[anchorbase]
\draw[very thick, directed=.55] (0,0) [out=30,in=210] to (2.4,1.6);
\draw[very thick, directed=.55] (0,.8) [out=30,in=270] to (1.2,1.6);
\draw[very thick, directed=.55] (1.2,0) [out=90,in=210] to (2.4,.8);
\torus{2.4}{1.6}
\end{tikzpicture}
};
\endxy
\end{equation*}  After collapsing the digons in the two middle webs, we get the following expression.
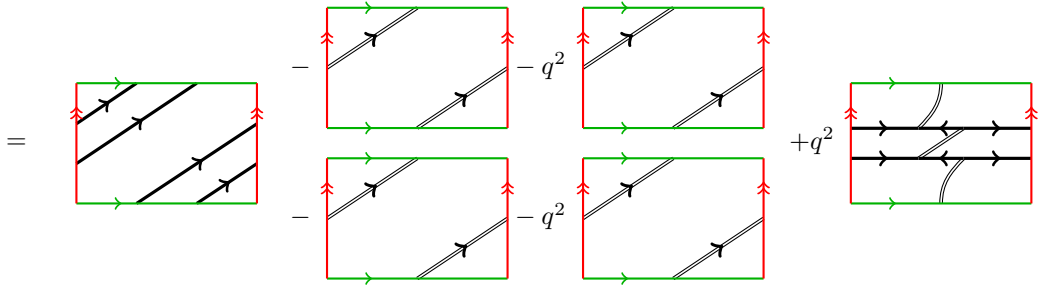
\begin{equation*}
\xy
(10,0)*{=
};
(30,0)*{
\begin{tikzpicture}[anchorbase]
\draw[very thick, directed=.55] (0,0.53) to (1.6,1.6);
\draw[very thick, directed=.55] (0,1.06) to (.8,1.6);
\draw[very thick, directed=.55] (.8,0) to (2.4,1.06);
\draw[very thick, directed=.55] (1.6,0) to (2.4,0.53);
\torus{2.4}{1.6}
\end{tikzpicture}
};
(78,10)*{
- \;\;
\begin{tikzpicture}[anchorbase]
\draw[double, directed=.55] (0,0.8) to (1.2,1.6);
\draw[double, directed=.55] (1.2,0) to (2.4,.8);
\torus{2.4}{1.6}
\end{tikzpicture}
-q^2  \;\;
\begin{tikzpicture}[anchorbase]
\draw[double, directed=.55] (0,0.8) to (1.2,1.6);
\draw[double, directed=.55] (1.2,0) to (2.4,.8);
\torus{2.4}{1.6}
\end{tikzpicture}
};
(78,-10)*{
-\;\;
\begin{tikzpicture}[anchorbase]
\draw[double, directed=.55] (0,0.8) to (1.2,1.6);
\draw[double, directed=.55] (1.2,0) to (2.4,.8);
\torus{2.4}{1.6}
\end{tikzpicture}
-q^2 \;\;
\begin{tikzpicture}[anchorbase]
\draw[double, directed=.55] (0,0.8) to (1.2,1.6);
\draw[double, directed=.55] (1.2,0) to (2.4,.8);
\torus{2.4}{1.6}
\end{tikzpicture}
};
(129,0)*{
 +q^2\;\;
\begin{tikzpicture}[anchorbase]
\draw[very thick, directed=.55] (0,1) to (.9,1);
\draw[very thick, directed=.55] (0,.6) to (.9,.6);
\draw[very thick, directed=.55] (1.5,1) to (2.4,1);
\draw[very thick, directed=.55] (1.5,.6) to (2.4,.6);
\draw[double] (.9,.6) to (1.5,1) ;
\draw[double]  (.9,1) to[out=45,in=270] (1.2,1.6) ;
\draw[double]  (1.2,0) to[out=90,in=225] (1.5,.6) ;
\draw[very thick, directed=.55] (1.5,1) to (.9,1);
\draw[very thick, directed=.55] (1.5,.6) to (.9,.6);
\torus{2.4}{1.6}
\end{tikzpicture}
};
\endxy
\end{equation*}

The left three webs sum to $(2,2)_T$, while the right three webs give $(2,0)_T*\w{0,1}$.
\end{example} 


Frohman--Gelca \cite{FG} showed that the Kauffman bracket skein algebra of the torus is isomorphic to the invariants in the quantum torus $\Z[A^{\pm 1}]\langle X^{\pm 1}, Y^{\pm 1}\rangle / \langle Y X = A^2 X Y\rangle$ under the involution that inverts $X$ and $Y$. The $\glnn{2}$ skein algebra of the torus has an analogous characterization that we describe next.

Consider the algebra $\cal{A}$ over $\Z[q^{\pm 1}]$ generated by invertible generators $X_1$, $X_2$, $Y_1$, $Y_2$, subject to the $q$-commutation relations:
\begin{enumerate}
\item $X_1 X_2=X_2 X_1$, $Y_1 Y_2=Y_2 Y_1$
\item $Y_1 X_1= X_1 Y_1$, $Y_2 X_2= X_2 Y_2$
\item $Y_1 X_2=q^{-2} X_2 Y_1$, $Y_2 X_1=q^{-2} X_1 Y_2$
\end{enumerate}
This algebra is graded by $H_1(\T)\cong \Z^2$ by letting $X_i$ have degree $(1,0)$ and $Y_i$ to have degree $(0,1)$. $S_2$ acts on $\cal{A}$, with the transposition switching $X_1\leftrightarrow X_2$ and $Y_1 \leftrightarrow Y_2$.  

\begin{proposition} There is an isomorphism $\Phi\colon \TWebq \to \cal{A}^{S_2}$ sending
\[
(m,n)_T  \mapsto (X_1^m Y_1^n + X_2^m Y_2^n)\quad ,\quad
\w{r,s}  \mapsto q^{-2 rs} X_1^r X_2^r Y_1^s Y_2^s
\]
\end{proposition}
\begin{proof}
It is easy to check that the images of the relations from \eqref{eq:FG} hold in the quantum torus. 
This means $\Phi$ is well-defined. The quantum torus $\cal{A}$ has a $\Z[q^{\pm 1}]$-basis given by monomials $X_1^a X_2^b Y_1^c Y_2^d$ for $a,b,c,d\in \Z$. The invariant part $\cal{A}^{S_2}$ has a $\Z[q^{\pm 1}]$-basis given by monomials 
\[X_1^a X_2^a Y_1^c Y_2^c =\Phi(q^{2a c} \w{a,c})\] and symmetrizations 
\begin{align*} 
X_1^a X_2^b Y_1^c Y_2^d + X_2^a X_1^b Y_2^c Y_1^d &= q^{-2 b (c-d)} (X_1^{a-b}Y_1^{c-d} + X_2^{a-b} Y_2^{c-d})X_1^b X_2^b Y_1^d Y_2^d \\
&= \Phi(q^{-2b(c-2d)} (a-b,c-d)_T*\w{b,d})\end{align*}
 for $a\neq b$ or $c\neq d$. This shows that $\Phi$ is injective and surjective.
\end{proof} 
Recall from Section~\ref{sec:pos} the skein algebra $\TWebA^\prime$, which results from twisting the multiplication in $\TWebq$ by a half-integer power of $-q$ depending on the intersection form on the first homology. This skein algebra is isomorphic to the $\cal{S}_2$-invariants in the quantum torus $\cal{A}^\prime$ whose $q$-commutation relations are more symmetric:
\begin{enumerate}
\item $X_1 X_2=X_2 X_1$, $Y_1 Y_2=Y_2 Y_1$
\item $Y_1 X_1= q X_1 Y_1$, $Y_2 X_2= q X_2 Y_2$
\item $Y_1 X_2=q^{-1} X_2 Y_1$, $Y_2 X_1=q^{-1} X_1 Y_2$
\end{enumerate} 
From the twisted versions $\SWebA^\prime \cong \cal{A}^{\prime \cal{S}_2}$, the Kauffman bracket skein algebra $\mathrm{Sk}_A(\T)$ can be obtained by erasing $2$-labeled edges, and the quantum torus of Frohman--Gelca is obtained by specializing $X_1\mapsto X $, $X_2\mapsto X^{-1}$, $Y_1\mapsto Y $ and $Y_2\mapsto Y^{-1}$ in $\cal{A}^\prime$.

\section{Foam categories on surfaces}
\label{sec:foams}
After Khovanov's categorification of the Jones polynomial \cite{Kh1}, a categorification of the skein module of $\R^2$ was constructed by Bar-Natan \cite{BN2}. The resulting category has objects corresponding to unoriented curves, but it also has an additional layer of morphisms given by cobordisms between these curves, modulo certain relations. Khovanov homology gives a projective functor from the category of links and link cobordisms up to isotopy to this categorified skein module \cite{Kh2,BN2}. The inherent sign defect in this theory was later on solved by Clark-Morrison-Walker \cite{CMW}, Caprau \cite{Cap} and Blanchet \cite{Blan}, by means of refined categories of cobordisms.

Just as in the case of Bar-Natan cobordisms \cite{APS, Boe}, these categories extend to more general thickened surfaces. We will use Blanchet's version of cobordisms, called foams, to construct properly functorial surface link homologies and categorifications of the $\glnn{2}$ surface skein modules. It is interesting to note that foams were first used by Khovanov in the context of $\slnn{3}$ link homology~\cite{Kh5}, and then extended by Khovanov--Rozansky~\cite{KhR3} and Mackaay--Sto\v{s}i\'{c}--Vaz~\cite{MSV} to $\slnn{N}$ (or rather $\glnn{N}$) link homologies for larger rank. Here, we only consider foams for $\glnn{2}$, but we will comment on $\glnn{N}$ link homologies in Section~\ref{sec:BKh}.

\subsection{Foams in thickened surfaces}
As before, $\Su$ denotes a connected, oriented surface of finite type.
\begin{definition} \label{def:Tfoam}
We define the foam category  $\Sfoam$ to be the graded category with:
\begin{itemize}
\item objects, (direct sums of q-shifted) webs embedded in $\Su$. We highlight the fact that here, webs are not considered up to any relation.
\item morphisms, (matrices of $\Q$-linear combinations of) $\glnn{2}$ foams properly embedded in $\Su\times [0,1]$.
\end{itemize}
\end{definition}

Foams for $\glnn{2}$ are embedded CW-complexes assembled from $1$- and $2$-labeled compact oriented surfaces, called \textit{facets}. These facets are glued along their boundary such that precisely two boundary components of $1$-labeled facets are identified with a single boundary component of a $2$-labeled facet. Around such \textit{seams}, foams have the shape of the letter {\it Y} times an interval or a circle and we require that the orientations of the seam induced by the orientations on the $1$-labeled facets agree with each other and disagree with the orientation induced by the $2$-labeled facet. Facets with label $1$ are furthermore allowed to carry dots. We consider foams modulo isotopy relative to the boundary and Blanchet's local foam relations~\cite{Blan}, which we illustrate in the following with $1$-labeled facets shaded red and $2$-labeled facets shaded yellow:

\begin{equation}\label{sl2closedfoam}
\begin{tikzpicture} [fill opacity=0.2, scale=.65,anchorbase]
	\filldraw [fill=red] (0,0) circle (1);
	\draw (-1,0) .. controls (-1,-.4) and (1,-.4) .. (1,0);
	\draw[dashed] (-1,0) .. controls (-1,.4) and (1,.4) .. (1,0);
\end{tikzpicture}
\quad = \quad 0 \qquad , \qquad
\begin{tikzpicture} [fill opacity=0.2, scale=.65,anchorbase]
	\filldraw [fill=red] (0,0) circle (1);
	\draw (-1,0) .. controls (-1,-.4) and (1,-.4) .. (1,0);
	\draw[dashed] (-1,0) .. controls (-1,.4) and (1,.4) .. (1,0);
	\node [opacity=1]  at (0,0.6) {$\bullet$};
\end{tikzpicture}
\quad = \quad 1
\end{equation}
\begin{equation}\label{sl2neckcutting}
\begin{tikzpicture} [fill opacity=0.2,  decoration={markings, mark=at position 0.6 with {\arrow{>}};  }, scale=.6,anchorbase]
	\draw [fill=red] (0,4) ellipse (1 and 0.5);
	\path [fill=red] (0,0) ellipse (1 and 0.5);
	\draw (1,0) .. controls (1,-.66) and (-1,-.66) .. (-1,0);
	\draw[dashed] (1,0) .. controls (1,.66) and (-1,.66) .. (-1,0);
	\draw (1,4) -- (1,0);
	\draw (-1,4) -- (-1,0);
	\path[fill=red, opacity=.3] (-1,4) .. controls (-1,3.34) and (1,3.34) .. (1,4) --
		(1,0) .. controls (1,.66) and (-1,.66) .. (-1,0) -- cycle;
\end{tikzpicture}
\quad = \quad
\begin{tikzpicture} [fill opacity=0.2,  decoration={markings, mark=at position 0.6 with {\arrow{>}};  }, scale=.6,anchorbase]
	\draw [fill=red] (0,4) ellipse (1 and 0.5);
	\draw (-1,4) .. controls (-1,2) and (1,2) .. (1,4);
	\path [fill=red, opacity=.3] (1,4) .. controls (1,3.34) and (-1,3.34) .. (-1,4) --
		(-1,4) .. controls (-1,2) and (1,2) .. (1,4);
	\path [fill=red] (0,0) ellipse (1 and 0.5);
	\draw (1,0) .. controls (1,-.66) and (-1,-.66) .. (-1,0);
	\draw[dashed] (1,0) .. controls (1,.66) and (-1,.66) .. (-1,0);
	\draw (-1,0) .. controls (-1,2) and (1,2) .. (1,0);
	\path [fill=red, opacity=.3] (1,0) .. controls (1,.66) and (-1,.66) .. (-1,0) --
		(-1,0) .. controls (-1,2) and (1,2) .. (1,0);
		\node[opacity=1] at (0,1) {$\bullet$};
\end{tikzpicture}
\quad + \quad
\begin{tikzpicture} [fill opacity=0.2,  decoration={markings, mark=at position 0.6 with {\arrow{>}};  }, scale=.6,anchorbase]
	\draw [fill=red] (0,4) ellipse (1 and 0.5);
	\draw (-1,4) .. controls (-1,2) and (1,2) .. (1,4);
	\path [fill=red, opacity=.3] (1,4) .. controls (1,3.34) and (-1,3.34) .. (-1,4) --
		(-1,4) .. controls (-1,2) and (1,2) .. (1,4);
	\node[opacity=1] at (0,3) {$\bullet$};
	\path [fill=red] (0,0) ellipse (1 and 0.5);
	\draw (1,0) .. controls (1,-.66) and (-1,-.66) .. (-1,0);
	\draw[dashed] (1,0) .. controls (1,.66) and (-1,.66) .. (-1,0);
	\draw (-1,0) .. controls (-1,2) and (1,2) .. (1,0);
	\path [fill=red, opacity=.3] (1,0) .. controls (1,.66) and (-1,.66) .. (-1,0) --
		(-1,0) .. controls (-1,2) and (1,2) .. (1,0);
\end{tikzpicture}
\quad .
\end{equation}
The neck-cutting relation \eqref{sl2neckcutting} gives the formula:
\begin{equation} \label{sl2handle}
2 \quad
\begin{tikzpicture} [fill opacity=0.2,decoration={markings, mark=at position 0.6 with {\arrow{>}}; }, scale=.65,anchorbase]
	\draw [fill=red] (1,1) -- (-1,2) -- (-1,-1) -- (1,-2) -- cycle;
	\node [opacity=1] at (0,0) {$\bullet$};
\end{tikzpicture}
\quad = \quad
\begin{tikzpicture} [fill opacity=0.2,decoration={markings, mark=at position 0.6 with {\arrow{>}}; }, scale=.65, anchorbase]
	\path [fill=red] (1,1) -- (-1,2) -- (-1,-1) -- (1,-2) -- cycle;
	\draw (-1,.58) -- (-1,2) -- (1,1) -- (1,-2) -- (-1,-1) -- (-1,-.32);
	\draw [fill=red] (0,.75) to [out=225, in=0] (-2,.75) to [out=180, in=90] (-3,0) to [out=270, in=180] (-2,-.5) to [out=0, in=135] (0,-.5);
	\filldraw [fill=white, opacity=1] (-2.5,.15) arc (-120:-60:1) -- (-1.65,.1) arc (70:110:1);
\end{tikzpicture}
 \quad
\end{equation}

\begin{equation} \label{sl2neckcutting_enh_2lab}
\begin{tikzpicture} [scale=0.6,fill opacity=0.2,  decoration={markings, mark=at position 0.6 with {\arrow{>}};  },anchorbase]
	\draw [fill=yellow , fill opacity=0.3] (0,4) ellipse (1 and 0.5);
	\path [fill=yellow , fill opacity=0.3] (0,0) ellipse (1 and 0.5);
	\draw (1,0) .. controls (1,-.66) and (-1,-.66) .. (-1,0);
	\draw[dashed] (1,0) .. controls (1,.66) and (-1,.66) .. (-1,0);
	\draw (1,4) -- (1,0);
	\draw (-1,4) -- (-1,0);
	\path[fill=yellow, opacity=.3] (-1,4) .. controls (-1,3.34) and (1,3.34) .. (1,4) --
		(1,0) .. controls (1,.66) and (-1,.66) .. (-1,0) -- cycle;
\end{tikzpicture}
\quad = \quad - \quad
\begin{tikzpicture} [scale=0.6,fill opacity=0.2,  decoration={markings, mark=at position 0.6 with {\arrow{>}};  },anchorbase]
	\draw [fill=yellow , fill opacity=0.3] (0,4) ellipse (1 and 0.5);
	\draw (-1,4) .. controls (-1,2) and (1,2) .. (1,4);
	\path [fill=yellow, opacity=.3] (1,4) .. controls (1,3.34) and (-1,3.34) .. (-1,4) --
		(-1,4) .. controls (-1,2) and (1,2) .. (1,4);
	\path [fill=yellow , fill opacity=0.3] (0,0) ellipse (1 and 0.5);
	\draw (1,0) .. controls (1,-.66) and (-1,-.66) .. (-1,0);
	\draw[dashed] (1,0) .. controls (1,.66) and (-1,.66) .. (-1,0);
	\draw (-1,0) .. controls (-1,2) and (1,2) .. (1,0);
	\path [fill=yellow, opacity=.3] (1,0) .. controls (1,.66) and (-1,.66) .. (-1,0) --
		(-1,0) .. controls (-1,2) and (1,2) .. (1,0);
\end{tikzpicture}
\end{equation}

\begin{equation}\label{2labeledsphere}
\begin{tikzpicture} [fill opacity=0.3, scale=.65,anchorbase]
	\draw [fill=yellow , fill opacity=0.3] (0,0) circle (1);
	\draw (-1,0) .. controls (-1,-.4) and (1,-.4) .. (1,0);
	\draw[dashed] (-1,0) .. controls (-1,.4) and (1,.4) .. (1,0);
\end{tikzpicture}
\quad =\quad -1
\end{equation}

\begin{equation}\label{sl2thetafoam_enh}
\begin{tikzpicture} [fill opacity=0.2,decoration={markings, mark=at position 0.6 with {\arrow{>}}; }, scale=.65,anchorbase]
	\filldraw [fill=red] (0,0) circle (1);
	\path [fill=yellow , fill opacity=0.3] (0,0) ellipse (1 and 0.3); 	
	\draw [very thick, red, postaction={decorate}] (-1,0) .. controls (-1,-.4) and (1,-.4) .. (1,0);
	\draw[very thick, red, dashed] (-1,0) .. controls (-1,.4) and (1,.4) .. (1,0);
	\node [opacity=1]  at (0,0.7) {$\bullet$};
	\node [opacity=1]  at (0,-0.7) {$\bullet$};
	\node [opacity=1] at (-0.3,0.7) {$\alpha$};
	\node [opacity=1] at (-0.3,-0.7) {$\beta$};
\end{tikzpicture}
\quad = \quad
\left\{
	\begin{array}{rl}
	1 & \text{ if } (\alpha, \beta) = (1,0) \\
	-1 & \text{ if } (\alpha, \beta) = (0,1) \\
	0 & \text{ if } (\alpha, \beta) = (0,0) \text{ or } (1,1)
	\end{array}
\right.
\end{equation}


\begin{equation} \label{sl2Fig5Blanchet_1}
\begin{tikzpicture} [scale=.6,fill opacity=0.2,decoration={markings, mark=at position 0.5 with {\arrow{>}}; },anchorbase]
	\draw[very thick, postaction={decorate}] (.75,0) -- (2,0);
	\draw[very thick, postaction={decorate}] (-2,0) -- (-.75,0);
	\draw[very thick, postaction={decorate}] (.75,0) .. controls (.5,-.5) and (-.5,-.5) .. (-.75,0);
	\draw[dashed, double, postaction={decorate}] (-.75,0) .. controls (-.5,.5) and (.5,.5) .. (.75,0);
	\draw (-2,0) -- (-2,4);
	\draw (2,0) -- (2,4);
	\path [fill=red] (-2,4) -- (-.75,4) -- (-.75,0) -- (-2,0) -- cycle;
	\path [fill=red] (2,4) -- (.75,4) -- (.75,0) -- (2,0) -- cycle;
	\path [fill=yellow , fill opacity=0.3] (-.75,4) .. controls (-.5,4.5) and (.5,4.5) .. (.75,4) --
		(.75,0) .. controls (.5,.5) and (-.5,.5) .. (-.75,0);
	\path [fill=red] (-.75,4) .. controls (-.5,3.5) and (.5,3.5) .. (.75,4) --
		(.75,0) .. controls (.5,-.5) and (-.5,-.5) .. (-.75,0);
	\draw [very thick, red, postaction={decorate}] (-.75,4) -- (-.75,0);
	\draw [very thick, red, postaction={decorate}] (.75, 0) -- (.75,4);
	\draw[very thick, postaction={decorate}] (.75,4) -- (2,4);
	\draw[very thick, postaction={decorate}] (-2,4) -- (-.75,4);
	\draw[very thick, postaction={decorate}] (.75,4) .. controls (.5,3.5) and (-.5,3.5) .. (-.75,4);
	\draw[double, postaction={decorate}] (-.75,4) .. controls (-.5,4.5) and (.5,4.5) .. (.75,4);
\end{tikzpicture}
\quad = \quad - \quad
\begin{tikzpicture} [scale=.6,fill opacity=0.2,decoration={markings, mark=at position 0.5 with {\arrow{>}}; },anchorbase]
	\draw[very thick, postaction={decorate}] (.75,0) -- (2,0);
	\draw[very thick, postaction={decorate}] (-2,0) -- (-.75,0);
	\draw[very thick, postaction={decorate}] (.75,0) .. controls (.5,-.5) and (-.5,-.5) .. (-.75,0);
	\draw[dashed, double, postaction={decorate}] (-.75,0) .. controls (-.5,.5) and (.5,.5) .. (.75,0);
	\draw (-2,0) -- (-2,4);
	\draw (2,0) -- (2,4);
	\path [fill=red] (-2,4) -- (-.75,4) -- (-.75,0) -- (-2,0) -- cycle;
	\path [fill=red] (2,4) -- (.75,4) -- (.75,0) -- (2,0) -- cycle;
	\path [fill=red] (-.75,4) .. controls (-.75,2) and (.75,2) .. (.75,4) --
			(.75, 0) .. controls (.75,2) and (-.75,2) .. (-.75,0);
	\path [fill=yellow , fill opacity=0.3] (-.75,4) .. controls (-.5,4.5) and (.5,4.5) .. (.75,4) --
			(.75,4) .. controls (.75,2) and (-.75,2) .. (-.75,4);
	\path [fill=red] (-.75,4) .. controls (-.5,3.5) and (.5,3.5) .. (.75,4) --
			(.75,4) .. controls (.75,2) and (-.75,2) .. (-.75,4);
	\path [fill=red] (-.75, 0) .. controls (-.75,2) and (.75,2) .. (.75,0) --
			(.75,0) .. controls (.5,-.5) and (-.5,-.5) .. (-.75,0);
	\path [fill=yellow , fill opacity=0.3] (-.75, 0) .. controls (-.75,2) and (.75,2) .. (.75,0) --
			(.75,0) .. controls (.5,.5) and (-.5,.5) .. (-.75,0);
	\draw [very thick, red, postaction={decorate}] (-.75,4) .. controls (-.75,2) and (.75,2) .. (.75,4);
	\draw [very thick, red, postaction={decorate}] (.75, 0) .. controls (.75,2) and (-.75,2) .. (-.75,0);
	\draw[very thick, postaction={decorate}] (.75,4) -- (2,4);
	\draw[very thick, postaction={decorate}] (-2,4) -- (-.75,4);
	\draw[very thick, postaction={decorate}] (.75,4) .. controls (.5,3.5) and (-.5,3.5) .. (-.75,4);
	\draw[double, postaction={decorate}] (-.75,4) .. controls (-.5,4.5) and (.5,4.5) .. (.75,4);
\end{tikzpicture}
\end{equation}

\begin{equation} \label{sl2Fig5Blanchet_2}
\begin{tikzpicture} [scale=.6,fill opacity=0.2,decoration={markings, mark=at position 0.5 with {\arrow{>}}; },anchorbase]
	\draw[very thick, postaction={decorate}] (2,0) -- (.75,0);
	\draw[very thick, postaction={decorate}] (-.75,0) -- (-2,0);
	\draw[very thick, postaction={decorate}] (-.75,0) .. controls (-.5,-.5) and (.5,-.5) .. (.75,0);
	\draw[dashed, double, postaction={decorate}] (.75,0) .. controls (.5,.5) and (-.5,.5) .. (-.75,0);
	\draw (-2,0) -- (-2,4);
	\draw (2,0) -- (2,4);
	\path [fill=red] (-2,4) -- (-.75,4) -- (-.75,0) -- (-2,0) -- cycle;
	\path [fill=red] (2,4) -- (.75,4) -- (.75,0) -- (2,0) -- cycle;
	\path [fill=yellow , fill opacity=0.3] (-.75,4) .. controls (-.5,4.5) and (.5,4.5) .. (.75,4) --
		(.75,0) .. controls (.5,.5) and (-.5,.5) .. (-.75,0);
	\path [fill=red] (-.75,4) .. controls (-.5,3.5) and (.5,3.5) .. (.75,4) --
		(.75,0) .. controls (.5,-.5) and (-.5,-.5) .. (-.75,0);
	\draw [very thick, red, postaction={decorate}] (-.75,0) -- (-.75,4);
	\draw [very thick, red, postaction={decorate}] (.75, 4) -- (.75,0);
	\draw[very thick, postaction={decorate}] (2,4) -- (.75,4);
	\draw[very thick, postaction={decorate}] (-.75,4) -- (-2,4);
	\draw[very thick, postaction={decorate}] (-.75,4) .. controls (-.5,3.5) and (.5,3.5) .. (.75,4);
	\draw[double, postaction={decorate}] (.75,4) .. controls (.5,4.5) and (-.5,4.5) .. (-.75,4);
\end{tikzpicture}
\quad = \quad
\begin{tikzpicture} [scale=.6,fill opacity=0.2,decoration={markings, mark=at position 0.5 with {\arrow{>}}; },anchorbase]
	\draw[very thick, postaction={decorate}] (2,0) -- (.75,0);
	\draw[very thick, postaction={decorate}] (-.75,0) -- (-2,0);
	\draw[very thick, postaction={decorate}] (-.75,0) .. controls (-.5,-.5) and (.5,-.5) .. (.75,0);
	\draw[dashed, double, postaction={decorate}] (.75,0) .. controls (.5,.5) and (-.5,.5) .. (-.75,0);
	\draw (-2,0) -- (-2,4);
	\draw (2,0) -- (2,4);
	\path [fill=red] (-2,4) -- (-.75,4) -- (-.75,0) -- (-2,0) -- cycle;
	\path [fill=red] (2,4) -- (.75,4) -- (.75,0) -- (2,0) -- cycle;
	\path [fill=red] (-.75,4) .. controls (-.75,2) and (.75,2) .. (.75,4) --
			(.75, 0) .. controls (.75,2) and (-.75,2) .. (-.75,0);
	\path [fill=yellow , fill opacity=0.3] (-.75,4) .. controls (-.5,4.5) and (.5,4.5) .. (.75,4) --
			(.75,4) .. controls (.75,2) and (-.75,2) .. (-.75,4);
	\path [fill=red] (-.75,4) .. controls (-.5,3.5) and (.5,3.5) .. (.75,4) --
			(.75,4) .. controls (.75,2) and (-.75,2) .. (-.75,4);
	\path [fill=red] (-.75, 0) .. controls (-.75,2) and (.75,2) .. (.75,0) --
			(.75,0) .. controls (.5,-.5) and (-.5,-.5) .. (-.75,0);
	\path [fill=yellow , fill opacity=0.3] (-.75, 0) .. controls (-.75,2) and (.75,2) .. (.75,0) --
			(.75,0) .. controls (.5,.5) and (-.5,.5) .. (-.75,0);
	\draw [very thick, red, postaction={decorate}] (.75,4) .. controls (.75,2) and (-.75,2) .. (-.75,4);
	\draw [very thick, red, postaction={decorate}] (-.75, 0) .. controls (-.75,2) and (.75,2) .. (.75,0);
	\draw[very thick, postaction={decorate}] (2,4) -- (.75,4);
	\draw[very thick, postaction={decorate}] (-.75,4) -- (-2,4);
	\draw[very thick, postaction={decorate}] (-.75,4) .. controls (-.5,3.5) and (.5,3.5) .. (.75,4);
	\draw[double, postaction={decorate}] (.75,4) .. controls (.5,4.5) and (-.5,4.5) .. (-.75,4);
\end{tikzpicture}
\end{equation}

\begin{equation} \label{sl2NH_enh}
\begin{tikzpicture} [scale=.6,fill opacity=0.2,decoration={markings, mark=at position 0.5 with {\arrow{>}}; },anchorbase]
	\draw[double, postaction={decorate}] (.75,0) -- (2,0);
	\draw[double, postaction={decorate}] (-2,0) -- (-.75,0);
	\draw[very thick, postaction={decorate}] (-.75,0) .. controls (-.5,-.5) and (.5,-.5) .. (.75,0);
	\draw[very thick, dashed, postaction={decorate}] (-.75,0) .. controls (-.5,.5) and (.5,.5) .. (.75,0);
	\draw (-2,0) -- (-2,4);
	\draw (2,0) -- (2,4);
	\path [fill=yellow, opacity=.3] (-2,4) -- (-.75,4) -- (-.75,0) -- (-2,0) -- cycle;
	\path [fill=yellow, opacity=.3] (2,4) -- (.75,4) -- (.75,0) -- (2,0) -- cycle;
	\path [fill=red] (-.75,4) .. controls (-.5,4.5) and (.5,4.5) .. (.75,4) --
		(.75,0) .. controls (.5,.5) and (-.5,.5) .. (-.75,0);
	\path [fill=red] (-.75,4) .. controls (-.5,3.5) and (.5,3.5) .. (.75,4) --
		(.75,0) .. controls (.5,-.5) and (-.5,-.5) .. (-.75,0);
	\draw [very thick, red, postaction={decorate}] (-.75,0) -- (-.75,4);
	\draw [very thick, red, postaction={decorate}] (.75, 4) -- (.75,0);
	\draw[double, postaction={decorate}] (.75,4) -- (2,4);
	\draw[double, postaction={decorate}] (-2,4) -- (-.75,4);
	\draw[very thick, postaction={decorate}] (-.75,4) .. controls (-.5,3.5) and (.5,3.5) .. (.75,4);
	\draw[very thick, postaction={decorate}] (-.75,4) .. controls (-.5,4.5) and (.5,4.5) .. (.75,4);
\end{tikzpicture}
\quad = \quad
\begin{tikzpicture} [scale=.6,fill opacity=0.2,decoration={markings, mark=at position 0.5 with {\arrow{>}}; },anchorbase]
	\draw[double, postaction={decorate}] (.75,0) -- (2,0);
	\draw[double, postaction={decorate}] (-2,0) -- (-.75,0);
	\draw[very thick, postaction={decorate}] (-.75,0) .. controls (-.5,-.5) and (.5,-.5) .. (.75,0);
	\draw[very thick, dashed, postaction={decorate}] (-.75,0) .. controls (-.5,.5) and (.5,.5) .. (.75,0);
	\draw (-2,0) -- (-2,4);
	\draw (2,0) -- (2,4);
	\path [fill=yellow,opacity=.3] (-2,4) -- (-.75,4) -- (-.75,0) -- (-2,0) -- cycle;
	\path [fill=yellow, opacity=.3] (2,4) -- (.75,4) -- (.75,0) -- (2,0) -- cycle;
	\path [fill=yellow, opacity=.3] (-.75,4) .. controls (-.75,2) and (.75,2) .. (.75,4) --
			(.75, 0) .. controls (.75,2) and (-.75,2) .. (-.75,0);
	\path [fill=red] (-.75,4) .. controls (-.5,4.5) and (.5,4.5) .. (.75,4) --
			(.75,4) .. controls (.75,2) and (-.75,2) .. (-.75,4);
	\path [fill=red] (-.75,4) .. controls (-.5,3.5) and (.5,3.5) .. (.75,4) --
			(.75,4) .. controls (.75,2) and (-.75,2) .. (-.75,4);
	\path [fill=red] (-.75, 0) .. controls (-.75,2) and (.75,2) .. (.75,0) --
			(.75,0) .. controls (.5,-.5) and (-.5,-.5) .. (-.75,0);
	\path [fill=red] (-.75, 0) .. controls (-.75,2) and (.75,2) .. (.75,0) --
			(.75,0) .. controls (.5,.5) and (-.5,.5) .. (-.75,0);
	\draw [very thick, red, postaction={decorate}] (.75,4) .. controls (.75,2) and (-.75,2) .. (-.75,4);
	\draw [very thick, red, postaction={decorate}] (-.75, 0) .. controls (-.75,2) and (.75,2) .. (.75,0);
	\draw[double, postaction={decorate}] (.75,4) -- (2,4);
	\draw[double, postaction={decorate}] (-2,4) -- (-.75,4);
	\draw[very thick, postaction={decorate}] (-.75,4) .. controls (-.5,3.5) and (.5,3.5) .. (.75,4);
	\draw[very thick, postaction={decorate}] (-.75,4) .. controls (-.5,4.5) and (.5,4.5) .. (.75,4);
	\node [opacity=1] at (0,4) {$\bullet$};
\end{tikzpicture}
\quad - \quad
\begin{tikzpicture} [scale=.6,fill opacity=0.2,decoration={markings, mark=at position 0.5 with {\arrow{>}}; },anchorbase]
	\draw[double, postaction={decorate}] (.75,0) -- (2,0);
	\draw[double, postaction={decorate}] (-2,0) -- (-.75,0);
	\draw[very thick, postaction={decorate}] (-.75,0) .. controls (-.5,-.5) and (.5,-.5) .. (.75,0);
	\draw[very thick, dashed, postaction={decorate}] (-.75,0) .. controls (-.5,.5) and (.5,.5) .. (.75,0);
	\draw (-2,0) -- (-2,4);
	\draw (2,0) -- (2,4);
	\path [fill=yellow,opacity=.3] (-2,4) -- (-.75,4) -- (-.75,0) -- (-2,0) -- cycle;
	\path [fill=yellow, opacity=.3] (2,4) -- (.75,4) -- (.75,0) -- (2,0) -- cycle;
	\path [fill=yellow, opacity=.3] (-.75,4) .. controls (-.75,2) and (.75,2) .. (.75,4) --
			(.75, 0) .. controls (.75,2) and (-.75,2) .. (-.75,0);
	\path [fill=red] (-.75,4) .. controls (-.5,4.5) and (.5,4.5) .. (.75,4) --
			(.75,4) .. controls (.75,2) and (-.75,2) .. (-.75,4);
	\path [fill=red] (-.75,4) .. controls (-.5,3.5) and (.5,3.5) .. (.75,4) --
			(.75,4) .. controls (.75,2) and (-.75,2) .. (-.75,4);
	\path [fill=red] (-.75, 0) .. controls (-.75,2) and (.75,2) .. (.75,0) --
			(.75,0) .. controls (.5,-.5) and (-.5,-.5) .. (-.75,0);
	\path [fill=red] (-.75, 0) .. controls (-.75,2) and (.75,2) .. (.75,0) --
			(.75,0) .. controls (.5,.5) and (-.5,.5) .. (-.75,0);
	\draw [very thick, red, postaction={decorate}] (.75,4) .. controls (.75,2) and (-.75,2) .. (-.75,4);
	\draw [very thick, red, postaction={decorate}] (-.75, 0) .. controls (-.75,2) and (.75,2) .. (.75,0);
	\draw[double, postaction={decorate}] (.75,4) -- (2,4);
	\draw[double, postaction={decorate}] (-2,4) -- (-.75,4);
	\draw[very thick, postaction={decorate}] (-.75,4) .. controls (-.5,3.5) and (.5,3.5) .. (.75,4);
	\draw[very thick, postaction={decorate}] (-.75,4) .. controls (-.5,4.5) and (.5,4.5) .. (.75,4);
	\node [opacity=1] at (0,0) {$\bullet$};
\end{tikzpicture}
\end{equation}

We furthermore impose that no facet can carry more than one dot:
\begin{equation}\label{twodotsheet}
\begin{tikzpicture}[fill opacity=.2, scale=.8, anchorbase]
\filldraw [fill=red] (-1,-1) rectangle (1,1);
\node [opacity=1] at (0,-.15) {$\bullet$};
\node [opacity=1] at (0,.15) {$\bullet$};
\end{tikzpicture}
\quad = 0
\end{equation}

This in particular implies that non-degree-zero bubbles and blisters are sent to zero, and the following dot-sliding relation:

\begin{equation} \label{sl2DotSliding}
\xy
(0,0)*{
\begin{tikzpicture} [scale=.8,fill opacity=0.2,  decoration={markings,
                        mark=at position 0.5 with {\arrow{>}};    }, scale=.75]               
	\draw [very thick, red] (0,-2) -- (0,0);
	\draw [double] (-1.5,-2) -- (0,-2);
	\draw [very thick] (0,-2) -- (1.8,-1);
	\draw [very thick] (0,-2) -- (1.5,-3);
	\filldraw [fill=red] (0,0) -- (1.8,1) -- (1.8,-1) -- (0,-2) -- cycle;
	\path[fill=yellow, fill opacity=0.3] (-1.5,0) -- (0,0) -- (0,-2) -- (-1.5,-2) -- cycle;
	\draw (-1.5,0) -- (-1.5,-2);
	\filldraw [fill=red] (0,0) -- (1.5,-1) -- (1.5,-3) -- (0,-2) -- cycle;
	\draw [double] (-1.5,0) -- (0,0);
	\draw [very thick] (0,0) -- (1.8,1);
	\draw [very thick] (0,0) -- (1.5,-1);
	\node [opacity=1] at (1,-0.3) {$\bullet$};
\end{tikzpicture}};
\endxy \quad = \quad - \quad
\xy
(0,0)*{
\begin{tikzpicture} [scale=.8,fill opacity=0.2,  decoration={markings,
                        mark=at position 0.5 with {\arrow{>}};    }, scale=.75]                                      	
	\draw [very thick, red] (0,-2) -- (0,0);
	\draw [double] (-1.5,-2) -- (0,-2);
	\draw [very thick] (0,-2) -- (1.8,-1);
	\draw [very thick] (0,-2) -- (1.5,-3);
	\filldraw [fill=red] (0,0) -- (1.8,1) -- (1.8,-1) -- (0,-2) -- cycle;
	\path[fill=yellow, fill opacity=0.3] (-1.5,0) -- (0,0) -- (0,-2) -- (-1.5,-2) -- cycle;
	\draw (-1.5,0) -- (-1.5,-2);
	\filldraw [fill=red] (0,0) -- (1.5,-1) -- (1.5,-3) -- (0,-2) -- cycle;
	\draw [double] (-1.5,0) -- (0,0);
	\draw [very thick] (0,0) -- (1.8,1);
	\draw [very thick] (0,0) -- (1.5,-1);
	\node [opacity=1] at (1.1,-1.8) {$\bullet$};
\end{tikzpicture}};
\endxy \quad .
\end{equation}

We have imported the figures in \eqref{sl2closedfoam}--\eqref{sl2NH_enh} from Lauda--Queffelec--Rose~\cite{LQR}, who prove that Blanchet's foam relations also arise in certain Schur quotient of categorified quantum groups of type A. In fact, the Schur quotient for categorified quantum $\glnn{\infty}$ has the structure of a (weak) $3$-category, and $\Sfoam$ can be considered as the \textit{integral} of this $3$-category over the $2$-manifold $\Su$. The relations shown here are also compatible with the more recent construction of foam categories by Robert--Wagner\cite{RWa}.

\begin{remark}
Relation~\eqref{twodotsheet} can be deformed to yield a foam-based construction of Lee's deformed Khovanov homology \cite{Lee} or an equivariant link homology~\cite{Kh8}. Deformations of this type have been studied in greater generality in \cite{RW}.
\end{remark}

\begin{remark} $\Sfoam$ carries a natural action of the orientation preserving diffeomorphism group of $\Su$.
\end{remark}

The foam categories $\Sfoam$ are furthermore graded. Let $F$ be a foam, $d$ its number of dots, and $c(F)$ the underlying $1$-labeled surface, which is obtained by deleting all $2$-labeled facets. Then the \emph{$q$-degree} of $F$ is given by $\deg(F):=2d-\chi(c(F))$. If $F\colon q^k W_1 \to q^l W_2$, then we require that $\deg(F)=l-k$.

\begin{remark}
In \cite{Blan}, Blanchet introduces a \emph{trivalent category}, a category of abstract closed webs and foams. These \emph{Blanchet foams} are not considered as embedded in any particular manifold, but they are required to satisfy stricter gluing constraints than our foams. In Blanchet foams, the two boundary components glued along each seam are required to come from two distinct facets and part of the data of a foam is an ordering of the two $1$-labeled facets at each seam.
Here, we give an example of a foam $\Tfoam$, which does not qualify as a Blanchet foam:
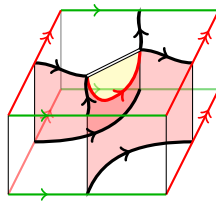
\[\begin{tikzpicture}[fill opacity=.2,anchorbase,xscale=.7, yscale=0.7]
\torusback{3}{2}{1.5}
\fill [fill=yellow] (1.5,2.25) to [out=270,in=200] (1.9,1.8) to [out=20,in=270] (2.5,2.75) to (1.5,2.25) ; 
\draw [fill=red] (2.5,2.75) to [out=0,in=180] (3.5,2.5) to (3.5,1)to [out=180,in=70](1.5,0) to (1.5,1.5) to [out=70,in=250] (1.5,2.25) to [out=270,in=200] (1.9,1.8) to [out=20,in=270] (2.5,2.75);
\draw [fill=red] (2.5,2.75) to [out=70,in=250](2.5,3.5) to (2.5,2) to [out=250,in=0] (.5,1) to  (.5,2.5) to [out=0,in=180] (1.5,2.25) to [out=270,in=200] (1.9,1.8) to [out=20,in=270] (2.5,2.75) ; 
\draw [very thick, red, directed=.55] (1.5,2.25) to [out=270,in=200] (1.9,1.8) to [out=20,in=270] (2.5,2.75);
\draw[very thick, directed=.5] (.5,1) to [out=0,in=250](2.5,2);
\draw[very thick, directed=.5] (1.5,0) to [out=70,in=180](3.5,1);
\draw[very thick, directed=.5] (.5,2.5) to [out=0,in=180] (1.5,2.25);
\draw[very thick, directed=.5] (1.5,1.5) to [out=70,in=250] (1.5,2.25);
\draw[double] (1.5,2.25) to (2.5,2.75);
\draw[very thick, directed=.55] (2.5,2.75) to [out=0,in=180](3.5,2.5);
\draw[very thick, directed=.55] (2.5,2.75) to [out=70,in=250](2.5,3.5);
\torusfront{3}{2}{1.5}
\end{tikzpicture}
\]

Note that for a foam in $\Sfoam$, one can specify a local cyclic ordering of the facets along a seam by using the right-hand rule and the orientation of the seam. 
This precisely fails to produce an abstract foam in Blanchet's sense if the two pieces of facets that locally meet on the seam actually belong to the same facet. We will prove in Proposition \ref{prop:Blanchet} that this does not happen in a certain subcategory of orientable foams in $\Sfoam$.

Blanchet then goes on to define a TQFT from his trivalent category to the category of graded abelian groups. An alternative way of encoding this TQFT is to linearize the morphism spaces in the Blanchet foam category and take the quotient by the relations in the kernel of the TQFT. We call the resulting abstract linear foam category $\twoFoam$. Blanchet's TQFT can then be recovered as the representable functor $\Hom_{\twoFoam}(\emptyset, -)$. The embedded foam relations in Definition~\ref{def:Tfoam} arise in the morphism spaces of $\twoFoam$ by embedding the foams in a 3-manifold and cutting out the relevant 3-balls. 
\end{remark}

The following local relations hold in the category $\Sfoam$ and will be used throughout the paper to simplify webs and foams.

\begin{lemma}\label{lem:webisos}
There are isomorphisms between webs in $\Sfoam$ which differ only in a disk as shown:
\begin{gather}
\label{eq:circles}
\begin{tikzpicture}[fill opacity=.2,anchorbase,scale=.3]
\draw[very thick, directed=.55] (1,0) to [out=0,in=270] (2,1) to [out=90,in=0] (1,2)to [out=180,in=90] (0,1)to [out=270,in=180] (1,0);
\end{tikzpicture} 
\quad\cong\quad
q\;\emptyset \oplus q^{-1} \emptyset
\quad\cong\quad 
\begin{tikzpicture}[fill opacity=.2,anchorbase,scale=.3]
\draw[very thick, rdirected=.55] (1,0) to [out=0,in=270] (2,1) to [out=90,in=0] (1,2)to [out=180,in=90] (0,1)to [out=270,in=180] (1,0);
\end{tikzpicture}
\quad,\quad
\begin{tikzpicture}[fill opacity=.2,anchorbase,scale=.3]
\draw[double, directed=.55] (1,0) to [out=0,in=270] (2,1) to [out=90,in=0] (1,2)to [out=180,in=90] (0,1)to [out=270,in=180] (1,0);
\end{tikzpicture} 
\quad\cong\quad
\emptyset 
\quad\cong\quad \begin{tikzpicture}[fill opacity=.2,anchorbase,scale=.3]
\draw[double, rdirected=.55] (1,0) to [out=0,in=270] (2,1) to [out=90,in=0] (1,2)to [out=180,in=90] (0,1)to [out=270,in=180] (1,0);
\end{tikzpicture}
\\
\label{eg:digons}
\begin{tikzpicture}[anchorbase, scale=.5]
\draw [double] (.5,0) -- (.5,.3);
\draw [very thick] (.5,.3) .. controls (.4,.35) and (0,.6) .. (0,1) .. controls (0,1.4) and (.4,1.65) .. (.5,1.7);
\draw [very thick] (.5,.3) .. controls (.6,.35) and (1,.6) .. (1,1) .. controls (1,1.4) and (.6,1.65) .. (.5,1.7);
\draw [double, ->] (.5,1.7) -- (.5,2);
\end{tikzpicture}
\quad\cong \quad
[2]\;
\begin{tikzpicture}[anchorbase, scale=.5]
\draw [double,->] (.5,0) -- (.5,2);
\end{tikzpicture}
\quad,\quad
\begin{tikzpicture}[anchorbase, scale=.5]
\draw [very thick] (.5,0) -- (.5,.3);
\draw [very thick] (.5,.3) .. controls (.4,.35) and (0,.6) .. (0,1) .. controls (0,1.4) and (.4,1.65) .. (.5,1.7);
\draw [double, directed=0.55] (.5,.3) .. controls (.6,.35) and (1,.6) .. (1,1) .. controls (1,1.4) and (.6,1.65) .. (.5,1.7);
\draw [very thick, ->] (.5,1.7) -- (.5,2);
\end{tikzpicture}
\quad\cong \quad
\begin{tikzpicture}[anchorbase, scale=.5]
\draw [very thick,->] (.5,0) -- (.5,2);
\end{tikzpicture}
\quad\cong \quad
\begin{tikzpicture}[anchorbase, scale=.5]
\draw [very thick] (.5,0) -- (.5,.3);
\draw [double, directed=0.55] (.5,.3) .. controls (.4,.35) and (0,.6) .. (0,1) .. controls (0,1.4) and (.4,1.65) .. (.5,1.7);
\draw [very thick] (.5,.3) .. controls (.6,.35) and (1,.6) .. (1,1) .. controls (1,1.4) and (.6,1.65) .. (.5,1.7);
\draw [very thick, ->] (.5,1.7) -- (.5,2);
\end{tikzpicture}
\\
\label{eq:squares}
\begin{tikzpicture}[anchorbase,scale=.5]
\draw [double] (0,0) -- (0,0.5);
\draw [very thick] (1,0) -- (1,.7);
\draw [very thick] (0,0.5) -- (1,.7);
\draw [double] (1,.7) -- (1,1.3);
\draw [very thick] (0,.5) -- (0,1.5);
\draw [very thick] (1,1.3) -- (0,1.5);
\draw [double,->] (0,1.5) -- (0,2);
\draw [very thick, ->] (1,1.3) -- (1,2);
\end{tikzpicture}
\quad \cong \quad
\begin{tikzpicture}[anchorbase,scale=.5]
\draw [double,->] (0,0) -- (0,2);
\draw [very thick,->] (1,0) -- (1,2);
\end{tikzpicture}
\quad,\quad
\begin{tikzpicture}[anchorbase,scale=.5]
\draw [double] (1,0) -- (1,0.5);
\draw [very thick] (0,0) -- (0,.7);
\draw [very thick] (1,0.5) -- (0,.7);
\draw [double] (0,.7) -- (0,1.3);
\draw [very thick] (1,.5) -- (1,1.5);
\draw [very thick] (0,1.3) -- (1,1.5);
\draw [double,->] (1,1.5) -- (1,2);
\draw [very thick, ->] (0,1.3) -- (0,2);
\end{tikzpicture}
\quad \cong \quad
\begin{tikzpicture}[anchorbase,scale=.5]
\draw [double,->] (1,0) -- (1,2);
\draw [very thick,->] (0,0) -- (0,2);
\end{tikzpicture}
\quad , \quad
\begin{tikzpicture}[anchorbase,scale=.5]
\draw [double,->] (0,0) to  (0,2);
\draw [double,->] (1,2) to (1,0);
\end{tikzpicture}
\quad \cong \quad
\begin{tikzpicture}[anchorbase,scale=.5]
\draw [double,->] (0,0) to (0,.5) to [out=90,in=90] (1,.5) to (1,0);
\draw [double,->] (1,2) to (1,1.5) to [out=270,in=270] (0,1.5) to (0,2);
\end{tikzpicture}
\end{gather}
\end{lemma}
\begin{proof}
Left to the reader.
\end{proof}

We will need a more general neck-cutting relation for identity foams on $1$-labeled circles that have arbitrary interaction with $2$-labeled edges.

\begin{lemma}
\label{lem:neckcut2}
Let $F$ be a foam in $\Sfoam$ whose underlying $1$-labeled surface $c(F)$ has a compression disk $D$ inside $\Su\times [0,1]$, surgery along which produces a surface $S$. Then $F$ can be expressed as a linear combination of foams with underlying surface $S$, each of which carries one more dot than $F$. 
\end{lemma}
\begin{proof} The compression disk $D$ may be assumed to be transverse to $F$. The proof then follows the same strategy as the one of Lemma \ref{lem:neckcut} to reduce the interaction of the foam with $D$ by locally modifying the foam via the web isomorphisms from Lemma~\ref{lem:webisos}. Once $D$ intersects the foams in the resulting linear combination only in a $1$-labeled circle, the neck-cutting relation \eqref{sl2neckcutting} can be applied to produce foams with underlying surface $S$ and the desired number of dots.
\end{proof}

\begin{corollary}\label{cor:neckcut}
Let $W$ be a web in $\Sfoam$, whose underlying curve $c(W)$ contains a circle which bounds a disk in $\Su\setminus c(W)$. Then $W \cong q V \oplus q^{-1} V$, where $V$ is a web that agrees with $W$ outside a neighborhood of the disk and with underlying curve obtained by removing the circle in question from $c(W)$.
\end{corollary}

Recall the standard basis $\Su\basisstd$ for $\SWebq$, which consists of (signed) webs that are given by the skein algebra multiplication of an integer lamination $L\in \cal{L}(\Su)$ and a $2$-labeled multi-curve $\w{x}$ for some $x\in H_1(\Su)$. From now on we interpret $\Su\basisstd$ as a collection of objects in $\Sfoam$ (ignoring minus signs). Lemma~\ref{lem:equivobjectsdecat} then has the following categorified version which admits an analogous proof.

\begin{lemma} \label{lem:equivobjectsS} Every web $W$ in $\Sfoam$ is isomorphic to a direct sum of grading shifts of copies of one element of $\Su\basisstd$.
\end{lemma}

\begin{lemma}\label{lem:closedeval}
Every closed foam that is contained in an embedded 3-ball in $\Su\times[0,1]$ evaluates to a scalar in $\Sfoam$
\end{lemma}
\begin{proof}Let $F$ be such a foam and $c(F)$ its underlying $1$-labeled surface. By Lemma~\ref{lem:neckcut2}, we may apply generalized neck-cutting relations and then assume that $c(F)$ is a union of spheres, possibly dotted. Starting with an innermost sphere, we choose a point on it (disjoint from the interaction with $2$-labeled facets) to obtain a complementary disk. Using the same strategy as in the proof of Lemma \ref{lem:neckcut}, but without actually delooping, we can free the disk, and thus the sphere, from all $2$-labeled interaction. The sphere then evaluates to a scalar. Now inductively proceed to evaluate all spheres in $c(F)$ until a purely $2$-labeled foam remains. This evaluates to a scalar by virtue of $2$-labeled neck-cutting and $2$-labeled sphere evaluation.
\end{proof}

The following lemma implies that the categories $\Sfoam$ (and all its versions considered in this paper) decompose into blocks indexed by first homology classes. 
\begin{lemma}
If $W_1$ and $W_2$ are webs in $\Sfoam$ with $[W_1]\neq [W_2]$, then $\Sfoam(W_1,W_2)=0$.
\end{lemma}
\begin{proof} Suppose there is a foam $F$ between $W_1$ and $W_2$. Interpreting $2$-labeled edges and facets as doubled up, we can perform a small push-off to obtain an oriented cobordism $\overline{F}$ between the multi-curves $\overline{W_1}$ and $\overline{W_2}$, which are thus homologous. 
\end{proof}

We say a foam $F$ is (un)orientable if its underlying $1$-labeled surface $c(F)$ has this property. The following proposition is similar to \cite[Prop 2.2.1]{Queff_PhD} and ensures that the presence of unorientable foams does not lead to unexpected relations between orientable foams.

\begin{proposition} \label{prop:unor}
  The set of unorientable foams in $\Sfoam$ is closed under the local relations from Definition~\ref{def:Tfoam}.
\end{proposition}

\begin{proof}
 Let $F$ be a connected unorientable dotted foam with underlying $1$-labeled surface $c(F)$, possibly with boundary. We want to see whether there are orientable representatives in the equivalence class of $c(F)$ under the local relations from Definition~\ref{def:Tfoam}.
 
The only relations that could change orientability are \eqref{sl2neckcutting} and \eqref{sl2NH_enh}, that, up to the $2$-labeled facet and a sign, look exactly alike. Using \eqref{sl2Fig5Blanchet_1} and \eqref{sl2DotSliding}, one can even deduce the latter one from the former one, so we will focus on this one.

Assume that there exists a compression disk where we can perform neck-cutting, so that the underlying surface $S$ of the resulting foam is orientable. Since $c(F)$ is unorientable and $S$ is orientable, there is a path $\gamma$ in $c(F)$ starting on the annulus $A$ where we perform the relation, and coming back after having crossed an odd number of seams.

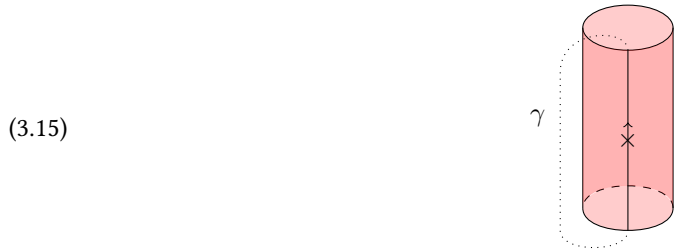
\begin{equation}
\xy
 (0,0)*{
\begin{tikzpicture} [scale=0.6,fill opacity=0.2,  decoration={markings, mark=at position 0.6 with {\arrow{>}};  }]
	\draw [fill=red] (0,4) ellipse (1 and 0.5);
	\path [fill=red] (0,0) ellipse (1 and 0.5);
	\draw (1,0) .. controls (1,-.66) and (-1,-.66) .. (-1,0);
	\draw[dashed] (1,0) .. controls (1,.66) and (-1,.66) .. (-1,0);
	\draw (1,4) -- (1,0);
	\draw (-1,4) -- (-1,0); 
	\path[fill=red, opacity=.3] (-1,4) .. controls (-1,3.34) and (1,3.34) .. (1,4) -- 
		(1,0) .. controls (1,.66) and (-1,.66) .. (-1,0) -- cycle;
	\draw [postaction={decorate}] (0,-.5) -- (0,3.5);
	\node [opacity=1] at (-2,2) {$\gamma$};
	\node [opacity=1] at (0,1.5) {$\times$};
	\draw [dotted] (0,3.5) .. controls (0,4) and (-1.5,4) .. (-1.5,3) -- (-1.5,-.5) .. controls (-1.5,-1) and (0,-1) .. (0,-.5); 
\end{tikzpicture}};
\endxy
\end{equation}

 We can assume that this path $\gamma$ goes only once through $A$. Indeed, if this is not the case, we locally have :
 
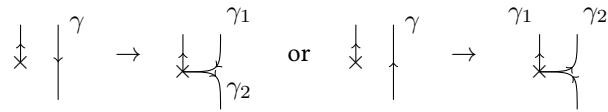
\[
  \xy
   (0,0)*{
 \begin{tikzpicture} [scale=.5,fill opacity=0.2,decoration={markings, mark=at position 0.5 with {\arrow{>}}; }]
  \node[opacity=1] at (0,0) {$\times$};
  \draw [postaction={decorate}] (0,0) -- (0,1);
  \draw [postaction={decorate}](1,1) -- (1,-1);
  \node[opacity=1] at (1.5,1) {$\gamma$};
 \end{tikzpicture}};
  \endxy
  \;\;\rightarrow\;\;
  \xy
   (0,0)*{
 \begin{tikzpicture} [scale=.5,fill opacity=0.2,decoration={markings, mark=at position 0.5 with {\arrow{>}}; }]
  \node[opacity=1] at (0,0) {$\times$};
  \draw[postaction={decorate}] (0,0)  -- (0,1);
  \draw[postaction={decorate}] (1,1) .. controls (1,0) .. (0,0);
  \draw[postaction={decorate}] (0,0) .. controls (1,0) .. (1,-1);
  \node[opacity=1] at (1.5,1.5) {$\gamma_1$};
  \node[opacity=1] at (1.5,-.5) {$\gamma_2$};
  \end{tikzpicture}};
  \endxy
\quad
 \textrm{or}
\quad
  \xy
  (0,0)*{
\begin{tikzpicture} [scale=.5,fill opacity=0.2,decoration={markings, mark=at position 0.5 with {\arrow{>}}; }]
 \node[opacity=1] at (0,0) {$\times$};
 \draw [postaction={decorate}] (0,0) -- (0,1);
 \draw [postaction={decorate}](1,-1) -- (1,1);
 \node[opacity=1] at (1.5,1) {$\gamma$};
\end{tikzpicture}};
  \endxy
  \;\;\rightarrow\;\;
  \xy
  (0,0)*{
\begin{tikzpicture} [scale=.5,fill opacity=0.2,decoration={markings, mark=at position 0.5 with {\arrow{>}}; }]
 \node[opacity=1] at (0,0) {$\times$};
 \draw [postaction={decorate}] (0,0)  -- (0,1);
 \draw [postaction={decorate}] (1,-1) .. controls (1,0) .. (0,0);
 \draw [postaction={decorate}] (0,0) .. controls (1,0) .. (1,1);
 \node[opacity=1] at (-.5,1.5) {$\gamma_1$};
 \node[opacity=1] at (1.5,1.5) {$\gamma_2$};
 \end{tikzpicture}
  };
  \endxy
 \]

 Then, either $\gamma_1$ is a disorientation path, and we focus on it, or $\gamma_2$ is. Applying this reduction process as many times as necessary, we end up with a path going one or zero time through $A$. Zero corresponds to the case where $S$ is unorientable, which contradicts our assumption.

 So, applying the neck-cutting relation, we get a sum :

\begin{equation}
\xy
 (0,0)*{
\begin{tikzpicture} [scale=0.6,fill opacity=0.2,  decoration={markings, mark=at position 0.6 with {\arrow{>}};  }]
	\draw [fill=red] (0,4) ellipse (1 and 0.5);
	\path [fill=red] (0,0) ellipse (1 and 0.5);
	\draw (1,0) .. controls (1,-.66) and (-1,-.66) .. (-1,0);
	\draw[dashed] (1,0) .. controls (1,.66) and (-1,.66) .. (-1,0);
	\draw (1,4) -- (1,0);
	\draw (-1,4) -- (-1,0); 
	\path[fill=red, opacity=.3] (-1,4) .. controls (-1,3.34) and (1,3.34) .. (1,4) -- 
		(1,0) .. controls (1,.66) and (-1,.66) .. (-1,0) -- cycle;
\end{tikzpicture}};
\endxy
\;\;=\;\;
\xy
(0,-1)*{
\begin{tikzpicture} [scale=0.6,fill opacity=0.2,  decoration={markings, mark=at position 0.6 with {\arrow{>}};  }]
	\draw [fill=red] (0,4) ellipse (1 and 0.5);
	\draw (-1,4) .. controls (-1,2) and (1,2) .. (1,4);
	\path [fill=red, opacity=.3] (1,4) .. controls (1,3.34) and (-1,3.34) .. (-1,4) --
		(-1,4) .. controls (-1,2) and (1,2) .. (1,4);
	\node [opacity=1] at (0,3) {$\bullet$};
	\draw [postaction={decorate}] (0,3) -- (0,3.5);
	\node [opacity=1] at (-2,2) {$\gamma$};
	\path [fill=red] (0,0) ellipse (1 and 0.5);
	\draw (1,0) .. controls (1,-.66) and (-1,-.66) .. (-1,0);
	\draw[dashed] (1,0) .. controls (1,.66) and (-1,.66) .. (-1,0);
	\draw (-1,0) .. controls (-1,2) and (1,2) .. (1,0);
	\path [fill=red, opacity=.3] (1,0) .. controls (1,.66) and (-1,.66) .. (-1,0) -- 
		(-1,0) .. controls (-1,2) and (1,2) .. (1,0);
	\draw [postaction={decorate}] (0,-.5) -- (0,0);
	\draw [dotted] (0,3.5) .. controls (0,4) and (-1.5,4) .. (-1.5,3) -- (-1.5,-.5) .. controls (-1.5,-1) and (0,-1) .. (0,-.5); 
\end{tikzpicture}};
\endxy
\;\; + \;\;
\xy
(0,0)*{\begin{tikzpicture} [scale=0.6,fill opacity=0.2,  decoration={markings, mark=at position 0.6 with {\arrow{>}};  }]
	\draw [fill=red] (0,4) ellipse (1 and 0.5);
	\draw (-1,4) .. controls (-1,2) and (1,2) .. (1,4);
	\path [fill=red, opacity=.3] (1,4) .. controls (1,3.34) and (-1,3.34) .. (-1,4) --
		(-1,4) .. controls (-1,2) and (1,2) .. (1,4);
	\path [fill=red] (0,0) ellipse (1 and 0.5);
	\draw (1,0) .. controls (1,-.66) and (-1,-.66) .. (-1,0);
	\draw[dashed] (1,0) .. controls (1,.66) and (-1,.66) .. (-1,0);
	\draw (-1,0) .. controls (-1,2) and (1,2) .. (1,0);
	\path [fill=red, opacity=.3] (1,0) .. controls (1,.66) and (-1,.66) .. (-1,0) -- 
		(-1,0) .. controls (-1,2) and (1,2) .. (1,0);
	\node [opacity=1] at (0,0) {$\bullet$};
\end{tikzpicture}};
\endxy \nn
 \end{equation}
 
The image of $\gamma$ after neck-cutting transports the dot of the left term to the place where the one of the right term lies. This is done by crossing an odd number of seams, and therefore it produces a $-1$ coefficient by \eqref{sl2DotSliding}, and the sum is actually zero. 
\end{proof}

\begin{remark} \label{rem:unor}
  The previous proposition implies in particular that the subcategory of orientable foams $\Sfoamor$ is isomorphic to the quotient of the foam category by the unorientable foams.
\end{remark}

\begin{remark}
A similar argument can be used to show that an unorientable foam $F$ that carries a dot on an unorientable component of $c(F)$ must be zero.  
\end{remark}

\begin{remark}
We denote by $\SCob$ the category of Bar-Natan cobordisms in $\Su\times [0,1]$. Its objects are unoriented multi-curves in $\Su$ and its morphisms are $\Q$-linear combinations of dotted cobordisms, properly embedded in $\Su\times [0,1]$, modulo isotopy relative to the boundary and the relations \eqref{sl2closedfoam}, \eqref{sl2neckcutting} and \eqref{twodotsheet}, see \cite[Section 11.6]{BN2} and \cite{AF}. It is clear from the definition, that $\SCob$ and $\Sfoam$ are closely related when defined over $\Z/2\Z$.
One of the disadvantages of $\SCob$ is that the analogues of Proposition~\ref{prop:unor} and Remark~\ref{rem:unor} are false. Another one is that Khovanov homology has a sign-ambiguity, when defined via $\SCob$.
\end{remark}

\begin{proposition} \label{prop:Blanchet} Every orientable foam in $\Sfoam$ is a Blanchet foam.
\end{proposition}
\begin{proof} Suppose a foam $F$ in $\Sfoam$ is not a Blanchet foam, i.e. there exists a seam to which one connected oriented $1$-labeled facet is glued along two of its boundary components. Then we choose a path in the facet from one boundary to the other, with start and end points that are identified by the gluing. The orientation of the facet is preserved along the path, but it switches across the seam. After erasing the $2$-labeled sheet, the path thus becomes an orientation reversing loop, and so $F$ is not orientable.  
\end{proof}

\begin{definition} We let $\Sfoamred$ denote the full subcategory of $\Sfoam$ with objects given by direct sums of grading shifts of webs without inessential $1$-labeled components.
\end{definition}
Note that by Corollary~\ref{cor:neckcut} the subcategory $\Sfoamred$ is equivalent to $\Sfoam$. All objects in the latter can be decomposed into webs without inessential $1$-labeled components, and so the inclusion is an equivalence.

The simple but key fact in our analysis is that $\Sfoamred$ is non-negatively graded, which follows from an elementary topological lemma.

\begin{lemma}\label{lem:Eulerchar}
Suppose that $\Su\neq \Stwo$. Then a connected surface, properly embedded in $\Su\times [0,1]$, is either a disk, a sphere bounding a ball, or it has non-positive Euler characteristic. In particular, non-orientable surfaces are of negative Euler characteristics.
\end{lemma}
\begin{proof}
  Recall that the Euler characteristics of an orientable surface of genus $g$ with $l$ boundary components is given by $2-2g-l$. This will be non-positive, unless $g=0$ and $l\leq 1$. The $g=0$ and $l=1$ case corresponds to a disk. The $l=0$, $g=0$ case corresponds to a sphere, but all spheres in thickened surfaces (except the thickened sphere) bound balls. If $\Su$ is a disk, this is due to Alexander. For closed surfaces of positive genus, the same result can be deduced from the case of the disk via a covering argument. Finally, it is easy to see that adding punctures cannot create new essential spheres.
  
  Let us now turn towards the unorientable case, when the Euler characteristics is $2-g-l$ where $g\geq 1$ is the number of crosscaps of the surface and $l$ the number of boundary components. $l=0$ is impossible since no unorientable surface can be embedded in $\R^3$ and thus in a thickened surface (which could itself be embedded in $\R^3$). Finally, for the $l=1$ case, note that such a surface bounding a 1-component curve could be made into a closed unorientable surface by taking the union with a $[0,1]\rightarrow [0,-1]$-reflection of it. An embedding of the former in $\Su\times [0,1]$ would give rise to an embedding of a closed unorientable surface in $\Su\times [-1,1]$, which we have already ruled out. Thus $l\geq 2$ and $2-g-l<0$.
  \end{proof}
 
\begin{corollary} \label{cor:nonneggrading} Suppose that $\Su\neq \Stwo$. Then the morphism spaces in $\Sfoamred$ are non-negatively graded. That is $\Sfoamred(W_1,q^k W_2)= 0$ if $k<0$, for any webs $W_1$ and $W_2$ in $\Sfoamred$.
\end{corollary}
\begin{proof} Let $F$ be a foam with $d$ dots between webs in $\Sfoamred$ and $c(F)$ its underlying $1$-labeled surface. Then $\deg(F)=2d-\chi(c(F))$ and by Lemma~\ref{lem:Eulerchar} the only negative contribution to this degree can come from disks or undotted spheres in $c(F)$. However, disks are explicitly ruled out in $\Sfoamred$ and we claim that undotted spheres in $c(F)$ evaluate to zero under the foam relations. This follows directly from the sphere-freeing argument in the proof of Lemma~\ref{lem:closedeval} and the fact that undotted $1$-labeled spheres evaluate to zero.
\end{proof}

\begin{definition} We denote by $\Sfoamred_0$ the degree-zero subcategory of $\Sfoamred$, i.e. the category with the same objects as $\Sfoamred$, but with \[\Sfoamred_0(q^k W_1,q^l W_2 )=\begin{cases} 
 \Sfoamred(q^k W_1,q^k W_2 ) & k=l\\
 0 & k\neq l\end{cases}.\]
\end{definition}
The non-negative grading of the morphism spaces of $\Sfoamred$ implies that $\Sfoamred_0$ can alternatively be seen as a subcategory or as a quotient category of $\Sfoamred$. 

\begin{definition}\label{defi:truncation} Let $\Sfoam_0$ denote the quotient category of $\Sfoam$ induced from the quotient $\Sfoamred_0$ under the equivalence of $\Sfoam$ with $\Sfoamred$. We will call the induced non-negative grading of $\Sfoam$ the \emph{essential $q$-grading}.
\end{definition}

\begin{corollary} \label{cor:oriented}Unorientable foams in $\Sfoam$ are sent to zero in $\Sfoam_0$.
\end{corollary}
\begin{proof} This follows from Lemma~\ref{lem:Eulerchar}.
\end{proof}

The following proposition implies that these categories are non-degenerate.		
											
\begin{proposition}\label{prop:linindep} Consider a collection of foams in a morphism space of $\Sfoamred_0$, which have incompressible underlying $1$-labeled surfaces that are pairwise non-isotopic. Then the elements of this collection are linearly independent and, in particular, non-zero.
\end{proposition}
An analogous result for Bar-Natan cobordism categories $\SCob$ is well-known, see e.g. \cite{AF} and \cite{Kai}.
\begin{proof} By the proof of Lemma~\ref{lem:Eulerchar}, the underlying $1$-labeled surface of a foam $F$ from the chosen collection consist of incompressible annuli and tori.

 We first argue that $F$ is non-zero in $\Sfoam$. Since $F$ is of degree zero, it is orientable. By forgetting the embedding information, $F$ can also be regarded as an abstract Blanchet foam with boundary, see Proposition~\ref{prop:Blanchet}. It can then be completed by another suitable abstract foam---e.g. another copy of $F$ with orientations reversed---to give an abstract foam without boundary. We now use Blanchet's TQFT to evaluate this foam to a scalar in $\Z$ and we claim it is non-zero.
 
 In fact, the envisioned doubling of $F$ has a non-zero evaluation $\pm 2^{t}$ where $t$ is the number of $1$-labeled tori in the underlying surface of the doubled abstract foam. To see this, one can first abstractly neck-cut all $2$-labeled facets and remove $2$-labeled spheres, so that only $2$-labeled disks with boundaries on $1$-labeled tori remain. This only changes the abstract foam evaluation by at most a sign. Now there are two possibilities for the remaining disks. Either they bound a disk on the torus, in which case they can be removed, or their boundary is an essential curve on the torus. In the latter case, these disks come in parallel pairs for orientation reasons, and after neck-cutting and one application of relation~\eqref{sl2Fig5Blanchet_1}, they can be removed as well. Finally, each remaining $1$-labeled torus without $2$-labeled interaction evaluates to $2$. 
 
 Now, since Blanchet's foam evaluation is constant under all foam relations, in particular those performed in the embedded sense in $\Su\times [0,1]$, the foam $F$ cannot be zero in $\Sfoam$. We conclude that a non-trivial linear relation between the foams from the chosen collection would need to involve several distinct foams, whose underlying $1$-labeled surfaces are pairwise non-isotopic by assumption. Such a linear relation would, thus, involve foam relations which change the topology of the underlying $1$-labeled surfaces, i.e. generalized neck-cutting relations as in Lemma~\ref{lem:neckcut2}. However, the foams in question, having incompressible underlying surfaces, could only appear on the neck-cut (i.e. dotted) side of neck-cutting relations. Since they also do not carry dots, this implies that there cannot be a non-trivial linear relation between them. 
\end{proof}

\subsection{Jones--Wenzl foams}
\label{sec:JWFoam}
In this section we describe a set of idempotent foams in $\Sfoam$ that are categorified analogues of the elements of the basis $\Su\basisJW$ of $\SWebq$. In the next section we will see that these idempotents generate all objects in the idempotent completion $\Kar(\Sfoam)$ if $\Su\neq \T$. The foams we will consider are modeled on $\glnn{2}$-versions of the famous Jones--Wenzl projectors, which we now recall.
 
\begin{definition} Let $\Webq$ denote the category of $\glnn{2}$ webs in a horizontal strip, with
\begin{itemize} 
\item objects, finite sequences of elements of the set $\{1,2,1^*,2^*\}$, including the empty sequence,
\item morphisms, $\Q(q)$-linear combinations of $\glnn{2}$ webs properly embedded in the horizontal strip $\R\times [0,1]$, viewed as mapping from the sequence on the bottom boundary to the one on the top boundary. These webs are considered up to isotopy relative to the boundary and modulo the $\glnn{2}$ web relations \eqref{eqn:circles}--\eqref{eqn:squares}.
\end{itemize}
Here $1$ and $2$ encode upward pointing boundary points of associated label, and $1^*$ and $2^*$ encode downward pointing boundary points. Composition is given by stacking strips with matching boundary data. We will also consider the version $\Web$, which is defined over $\Q$ instead of $\Q(q)$, with all relations \eqref{eqn:circles}--\eqref{eqn:squares} specialized to $q=1$.
\end{definition}

\begin{definition}The $U_q(\glnn{2})$ Jones--Wenzl projectors $P_m\in \Webq\otimes \Q(q)$ are defined by $P_1=\id_1$ and then:
 \[\begin{tikzpicture}[anchorbase, scale=.3]
\fill[black,opacity=.2] (0,1) rectangle (3,3);
\draw[thick] (0,1) rectangle (3,3);
\draw [very thick] (.5,0) to (.5,1);
\draw [thick, dotted] (.7,0.5) to (1.3,.5);
\draw [very thick] (1.5,0) to (1.5,1);
\draw [very thick] (2.5,0) to (2.5,1);
\draw [very thick,->] (.5,3) to (.5,4);
\draw [thick, dotted] (.7,3.25) to (1.3,3.25);
\draw [very thick,->] (1.5,3) to (1.5,4);
\draw [very thick,->] (2.5,3) to (2.5,4);
\node at (1.5,1.9) {$P_{m+1}$};
\end{tikzpicture}
   \;\; := \;\;
   \begin{tikzpicture}[anchorbase, scale=.3]
\fill[black,opacity=.2] (0,1) rectangle (2,3);
\draw[thick] (0,1) rectangle (2,3);
\draw [very thick] (.5,0) to (.5,1);
\draw [thick, dotted] (.7,0.5) to (1.3,.5);
\draw [very thick] (1.5,0) to (1.5,1);
\draw [very thick,->] (2.5,0) to (2.5,4);
\draw [very thick,->] (.5,3) to (.5,4);
\draw [thick, dotted] (.7,3.25) to (1.3,3.25);
\draw [very thick,->] (1.5,3) to (1.5,4);
\node at (1,1.9) {$P_{m}$};
\end{tikzpicture}
    \;-\;\frac{[m]}{[m+1]}\;
       \begin{tikzpicture}[anchorbase, scale=.3]
\fill[black,opacity=.2] (0,.5) rectangle (2,1.5);
\draw[thick] (0,.5) rectangle (2,1.5);
\fill[black,opacity=.2] (0,2.5) rectangle (2,3.5);
\draw[thick] (0,2.5) rectangle (2,3.5);
\draw [very thick] (.5,0) to (.5,.5);
\draw [very thick] (1.5,0) to (1.5,.5);
\draw [double] (2,1.75) to (2,2.25);
\draw [very thick] (2.5,0) to (2.5,1.5)to [out=90,in=90] (1.5,1.5);
\draw [very thick,->] (1.5,2.5) to [out=270,in=270] (2.5,2.5) to (2.5,4); 
\draw [very thick] (.5,1.5) to (.5,2.5);
\draw [thick, dotted] (.7,2) to (1.3,2);
\draw [very thick,->] (.5,3.5) to (.5,4);
\draw [very thick,->] (1.5,3.5) to (1.5,4);
\node at (1,.95) {\tiny$P_{m}$};
\node at (1,2.95) {\tiny$P_{m}$};
\end{tikzpicture}   
\]
C.f. Wenzl~\cite{Wen} in the case of $U_q(\slnn{2})$.
\end{definition}
 We give two examples:
\[\begin{tikzpicture}[anchorbase, scale=.3]
\fill[black,opacity=.2] (0,1) rectangle (2,3);
\draw[thick] (0,1) rectangle (2,3);
\draw [very thick] (.5,0) to (.5,1);
\draw [very thick] (1.5,0) to (1.5,1);
\draw [very thick,->] (.5,3) to (.5,4);
\draw [very thick,->] (1.5,3) to (1.5,4);
\node at (1,1.9) {$P_{2}$};
\end{tikzpicture} 
\quad = \quad
\begin{tikzpicture}[anchorbase, scale=.3]
\draw [very thick,->] (.5,0) to (.5,4);
\draw [very thick,->] (1.5,0) to (1.5,4);
\end{tikzpicture}
-
\frac{1}{[2]}
\begin{tikzpicture}[anchorbase, scale=.3]
\draw [very thick] (.5,0) to (.5,.5) to [out=90,in=225] (1,1.5);
\draw [very thick] (1.5,0) to (1.5,.5) to [out=90,in=315] (1,1.5);
\draw[double] (1,1.5) to (1,2.5);
\draw [very thick,->] (1,2.5) to [out=135,in=270] (.5,3.5) to (.5,4);
\draw [very thick,->] (1,2.5) to [out=45,in=270] (1.5,3.5) to (1.5,4);
\end{tikzpicture}
\quad, \quad
\begin{tikzpicture}[anchorbase, scale=.3]
\fill[black,opacity=.2] (0,1) rectangle (3,3);
\draw[thick] (0,1) rectangle (3,3);
\draw [very thick] (.5,0) to (.5,1);
\draw [very thick] (1.5,0) to (1.5,1);
\draw [very thick] (2.5,0) to (2.5,1);
\draw [very thick,->] (.5,3) to (.5,4);
\draw [very thick,->] (1.5,3) to (1.5,4);
\draw [very thick,->] (2.5,3) to (2.5,4);
\node at (1.5,1.9) {$P_{3}$};
\end{tikzpicture} 
\quad = \quad
\begin{tikzpicture}[anchorbase, scale=.3]
\draw [very thick,->] (.5,0) to (.5,4);
\draw [very thick,->] (1.5,0) to (1.5,4);
\draw [very thick,->] (2.5,0) to (2.5,4);
\end{tikzpicture}
-
\frac{[2]}{[3]}
\left(
\begin{tikzpicture}[anchorbase, scale=.3]
\draw [very thick] (.5,0) to (.5,.5) to [out=90,in=225] (1,1.5);
\draw [very thick] (1.5,0) to (1.5,.5) to [out=90,in=315] (1,1.5);
\draw [very thick,->] (2.5,0) to (2.5,4);
\draw[double] (1,1.5) to (1,2.5);
\draw [very thick,->] (1,2.5) to [out=135,in=270] (.5,3.5) to (.5,4);
\draw [very thick,->] (1,2.5) to [out=45,in=270] (1.5,3.5) to (1.5,4);
\draw [very thick,->] (2.5,3) to (2.5,4);
\end{tikzpicture}
+
\begin{tikzpicture}[anchorbase, scale=.3]
\draw [very thick] (.5,0) to (.5,.5) to [out=90,in=225] (1,1.5);
\draw [very thick] (1.5,0) to (1.5,.5) to [out=90,in=315] (1,1.5);
\draw [very thick,->] (-.5,0) to (-.5,4);
\draw[double] (1,1.5) to (1,2.5);
\draw [very thick,->] (1,2.5) to [out=135,in=270] (.5,3.5) to (.5,4);
\draw [very thick,->] (1,2.5) to [out=45,in=270] (1.5,3.5) to (1.5,4);
\draw [very thick,->] (-.5,0) to (-.5,4);
\end{tikzpicture}
\right)
+
\frac{1}{[3]}
\left(
\begin{tikzpicture}[anchorbase, scale=.3]
\draw [very thick] (.5,0) to [out=90,in=225] (1,1);
\draw [very thick] (1.5,0) to [out=90,in=315] (1,1);
\draw [very thick] (2.5,0) to (2.5,1) to [out=90,in=315] (2,2.25);
\draw[double] (1,1) to (1,1.75);
\draw [very thick] (1,1.75) to (2,2.25);
\draw[double] (2,2.25) to (2,3);
\draw [very thick,->] (1,1.75) to [out=135,in=270] (.5,3) to (.5,4);
\draw [very thick,->] (2,3) to [out=135,in=270] (1.5,4);
\draw [very thick,->] (2,3) to [out=45,in=270] (2.5,4);
\end{tikzpicture}
+
\begin{tikzpicture}[anchorbase, scale=.3]
\draw [very thick] (.5,0) to [out=90,in=225] (1,2.25);
\draw [very thick] (1.5,0) to [out=90,in=225] (2,1);
\draw [very thick] (2.5,0) to [out=90,in=315] (2,1);
\draw[double] (2,1) to (2,1.75);
\draw [very thick] (2,1.75) to (1,2.25);
\draw[double] (1,2.25) to (1,3);
\draw [very thick,->] (2,1.75) to [out=45,in=270] (2.5,3) to (2.5,4);
\draw [very thick,->] (1,3) to [out=135,in=270] (0.5,4);
\draw [very thick,->] (1,3) to [out=45,in=270] (1.5,4);
\end{tikzpicture}
\right)
\]

For $\glnn{2}$, we get analogous projectors in $\Web$ by replacing the quantum integers $[m]$ and $[m+1]$ in the recursion by the corresponding integers. It is easy to check that the morphisms $P_m$ are idempotent and that they annihilate \textit{turnbacks} if $m\geq 2$:

\[\begin{tikzpicture}[anchorbase, scale=.3]
\fill[black,opacity=.2] (0,1) rectangle (4,3);
\draw[thick] (0,1) rectangle (4,3);
\draw [very thick] (.5,0) to (.5,1);
\draw [thick, dotted] (.7,0.5) to (1.3,.5);
\draw [thick, dotted] (2.7,0.5) to (3.3,.5);
\draw [very thick] (1.5,0) to (1.5,1);
\draw [very thick] (2.5,0) to (2.5,1);
\draw [very thick] (3.5,0) to (3.5,1);
\draw [very thick,->] (.5,3) to (.5,5);
\draw [thick, dotted] (.7,4.25) to (1.7,4.25);
\draw [thick, dotted] (2.3,4.25) to (3.3,4.25);
\draw [very thick] (1.5,3) to [out=90,in=225] (2,4);
\draw [very thick] (2.5,3) to [out=90,in=315] (2,4);
\draw[double, ->] (2,4) to (2,5);
\draw [very thick,->] (3.5,3) to (3.5,5);
\node at (2,1.9) {$P_{m}$};
\end{tikzpicture}
\quad = \quad 
0
\quad = \quad
\begin{tikzpicture}[anchorbase, scale=.3]
\fill[black,opacity=.2] (0,1) rectangle (4,3);
\draw[thick] (0,1) rectangle (4,3);
\draw [very thick] (.5,-1) to (.5,1);
\draw [thick, dotted] (.7,-0.5) to (1.7,-.5);
\draw [thick, dotted] (2.3,-0.5) to (3.3,-.5);
\draw[double] (2,-1) to (2,0);
\draw [very thick] (2,0) to [out=135,in=270] (1.5,1);
\draw [very thick] (2,0) to [out=45,in=270] (2.5,1);
\draw [very thick] (3.5,-1) to (3.5,1);
\draw [very thick,->] (.5,3) to (.5,4);
\draw [thick, dotted] (.7,3.25) to (1.3,3.25);
\draw [thick, dotted] (2.7,3.25) to (3.3,3.25);
\draw [very thick,->] (1.5,3) to (1.5,4);
\draw [very thick,->] (2.5,3) to(2.5,4);
\draw [very thick,->] (3.5,3) to (3.5,4);
\node at (2,1.9) {$P_{m}$};
\end{tikzpicture}
\]

\begin{remark}
\label{rem:JWprop}
Checking idempotency and turnback annihilation property of the Jones--Wenzl projectors only requires the use of the following web relations:
\begin{equation}
\label{eqn:relevantwebrel}
\begin{tikzpicture}[anchorbase, scale=.5]
\draw [double] (.5,0) -- (.5,.3);
\draw [very thick] (.5,.3) .. controls (.4,.35) and (0,.6) .. (0,1) .. controls (0,1.4) and (.4,1.65) .. (.5,1.7);
\draw [very thick] (.5,.3) .. controls (.6,.35) and (1,.6) .. (1,1) .. controls (1,1.4) and (.6,1.65) .. (.5,1.7);
\draw [double, ->] (.5,1.7) -- (.5,2);
\end{tikzpicture}
\quad= \quad
2\;
\begin{tikzpicture}[anchorbase, scale=.5]
\draw [double,->] (.5,0) -- (.5,2);
\end{tikzpicture}
\quad, \quad
\begin{tikzpicture}[anchorbase,scale=.5]
\draw [double] (0,0) -- (0,0.5);
\draw [very thick] (1,0) -- (1,.7);
\draw [very thick] (0,0.5) -- (1,.7);
\draw [double] (1,.7) -- (1,1.3);
\draw [very thick] (0,.5) -- (0,1.5);
\draw [very thick] (1,1.3) -- (0,1.5);
\draw [double,->] (0,1.5) -- (0,2);
\draw [very thick, ->] (1,1.3) -- (1,2);
\end{tikzpicture}
\quad = \quad
\begin{tikzpicture}[anchorbase,scale=.5]
\draw [double,->] (0,0) -- (0,2);
\draw [very thick,->] (1,0) -- (1,2);
\end{tikzpicture}
\quad,\quad
\begin{tikzpicture}[anchorbase,scale=.5]
\draw [double] (1,0) -- (1,0.5);
\draw [very thick] (0,0) -- (0,.7);
\draw [very thick] (1,0.5) -- (0,.7);
\draw [double] (0,.7) -- (0,1.3);
\draw [very thick] (1,.5) -- (1,1.5);
\draw [very thick] (0,1.3) -- (1,1.5);
\draw [double,->] (1,1.5) -- (1,2);
\draw [very thick, ->] (0,1.3) -- (0,2);
\end{tikzpicture}
\quad = \quad
\begin{tikzpicture}[anchorbase,scale=.5]
\draw [double,->] (1,0) -- (1,2);
\draw [very thick,->] (0,0) -- (0,2);
\end{tikzpicture}
\end{equation}
\end{remark}

The key to obtaining idempotent foams from Jones--Wenzl projectors is the following well-known  fact about foams in the thickened annulus $\A\times[0,1]$. 

\begin{lemma}
\label{lem:WebtoAfoam}
There exists a well-defined functor $\Web\to \Afoam$, which sends webs $W$ to the annular rotation foams $W\times \Ss^1$. 
\end{lemma}
\begin{proof} A direct computation verifies that the foam relations from Definition~\ref{def:Tfoam} imply that rotation foams satisfy the $\glnn{2}$ web relations at $q=1$. 
\end{proof}

Any choice of annular neighborhood of an oriented simple closed curve on $\Su$ determines an embedding $\Afoam\to \Sfoam$ and thus a functor $\Web \to \Sfoam$. The images of Jones--Wenzl projectors under such a functor are idempotent foams, and we call them \textit{Jones--Wenzl foams}. 

\begin{example} The second Jones--Wenzl projector $P_2$ spun around the longitude of the torus:
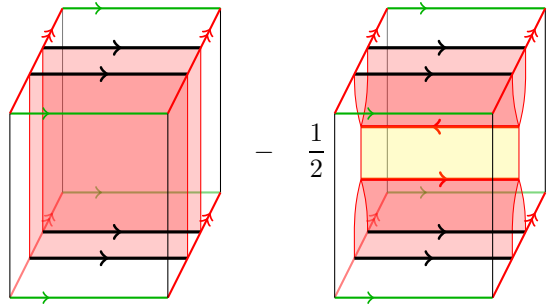
\[
\begin{tikzpicture}[fill opacity=.2,anchorbase,xscale=.7, yscale=0.7]
\torusback{3}{2}{3.5}
\fill [fill=red] (0.625,4.75) to (3.625,4.75) to (3.625,1.25) to (0.625,1.25) to (0.625,4.75);
\fill [fill=red] (0.375,4.25) to (3.375,4.25) to (3.375,0.75) to (.375,0.75) to (0.375,4.25) ;
\draw [red] (0.375,4.25) to (0.375,0.75);
\draw [red] (0.625,4.75) to (0.625,1.25);
\draw [red] (3.375,4.25) to (3.375,0.75);
\draw [red] (3.625,4.75) to (3.625,1.25);
\coordinate (a) at (0.125,-2.75);
\draw[very thick, directed=.5] (.25,3.5)+(a) to ($(3.25,3.5)+(a)$);
\draw[very thick, directed=.5] (.5,4)+(a) to ($((3.5,4)+(a)$);
\coordinate (a) at (0.125,0.75);
\draw[very thick, directed=.5] (.25,3.5)+(a) to ($(3.25,3.5)+(a)$);
\draw[very thick, directed=.5] (.5,4)+(a) to ($(3.5,4)+(a)$);
\torusfront{3}{2}{3.5}
\end{tikzpicture}
\quad - \quad
\frac{1}{2}\;
\begin{tikzpicture}[fill opacity=.2,anchorbase,xscale=.7, yscale=0.7]
\torusback{3}{2}{3.5}
\draw [very thick, red, directed=.55] (.5,2.25) to (3.5,2.25);
\draw [very thick, red, directed=.55] (3.5,3.25) to  (.5,3.25);
\fill [fill=yellow] (3.5,3.25) to  (.5,3.25) to (.5,2.25) to (3.5,2.25) to (3.5,3.25) ;
\fill [fill=red] (3.5,3.25) to (.5,3.25) to [out=75,in=270] (0.625,4.75) to (3.625,4.75) to [out=270,in=75] (3.5,3.25);
\fill [fill=red] (3.5,3.25) to (.5,3.25) to [out=105,in=270] (0.375,4.25) to (3.375,4.25) to [out=270,in=105] (3.5,3.25);
\fill [fill=red] (3.5,2.25) to (.5,2.25) to [out=255,in=90]  (.375,0.75) to (3.375,0.75) to [out=90,in=255] (3.5,2.25);
\fill [fill=red] (3.5,2.25) to (.5,2.25) to [out=285,in=90]  (.625,1.25) to (3.625,1.25) to [out=90,in=285] (3.5,2.25);
\draw [red] (0.375,4.25) to [out=270,in=105] (.5,3.25);
\draw [red] (0.625,4.75) to [out=270,in=75] (.5,3.25);
\draw [red] (3.625,4.75) to [out=270,in=75] (3.5,3.25);
\draw [red] (3.375,4.25) to [out=270,in=105] (3.5,3.25);
\draw [red] (3.5,3.25) to (3.5,2.25);
\draw [red] (0.5,3.25) to (0.5,2.25);
\draw [red] (.625,1.25) to [out=90,in=285] (.5,2.25);
\draw [red] (.375,0.75) to [out=90,in=255] (.5,2.25);
\draw [red] (3.375,0.75) to [out=90,in=255] (3.5,2.25);
\draw [red] (3.625,1.25) to [out=90,in=285] (3.5,2.25);
\coordinate (a) at (0.125,-2.75);
\draw[very thick, directed=.5] (.25,3.5)+(a) to ($(3.25,3.5)+(a)$);
\draw[very thick, directed=.5] (.5,4)+(a) to ($((3.5,4)+(a)$);
\coordinate (a) at (0.125,0.75);
\draw[very thick, directed=.5] (.25,3.5)+(a) to ($(3.25,3.5)+(a)$);
\draw[very thick, directed=.5] (.5,4)+(a) to ($(3.5,4)+(a)$);
\torusfront{3}{2}{3.5}
\end{tikzpicture}
\]
\end{example} 
Given several disjoint oriented simple closed curves on $\Su$, we can simultaneously embed one spun Jones--Wenzl foam per curve. 
\begin{definition}
Let $L=\{(C_i,n_i)\}\in \cal{L}(\Su)$ be an oriented integer lamination on $\Su$ (see Definition~\ref{def:intlam}), then we denote by $L^F_S$ the idempotent foam in $\Sfoam$ obtained by simultaneously embedding the spun Jones--Wenzl projectors $P_{n_i}\times S^1$ in $\A\times [0,1]$ along the oriented simple closed curves $C_i$ into $\Su\times [0,1]$. 
\end{definition} 
The foams $L^F_S$ are designed to decategorify to the basis elements $L_S\in \Su\basisJW$ (we will see that they do so in Lemma~\ref{lem:JWdecat}). However, a generic basis element in $\Su\basisJW$ is of the form $L_S*\w{x}$ where $x$ is a multi-curve on $\Su$ representing a certain first homology class of $\Su$. We also need idempotent foams $L^F_S*\w{x}$ corresponding to such basis elements $L_S*\w{x}$. In Section~\ref{sec:stwo} we will describe a general framework for how superposition with $2$-labeled multi-curves give auto-equivalences $-*\w{x}$ of the foam categories $\Sfoam$, which allow $L^F_S*\w{x}$ to be obtained from $L^F_S$. Since this framework is best explained after introducing the Khovanov functors in Section~\ref{sec:BKh} and since it depends on an additional Functoriality Conjecture~\ref{conj:functoriality}, we will here give a self-contained description of the idempotent foams $L^F_S*\w{x}$.

Recall that the Jones--Wenzl projectors $P_m$ can be written as linear combinations of webs built from the following elementary webs:
\[\id:=\begin{tikzpicture}[anchorbase, scale=.3]
\draw [very thick,->] (.5,0) to (.5,2);
\draw [very thick,->] (1.5,0) to (1.5,2);
\end{tikzpicture}
\quad,\quad
\id_2:=\begin{tikzpicture}[anchorbase, scale=.3]
\draw [double,->] (.5,0) to (.5,2);
\end{tikzpicture}
\quad,\quad
M:=
\begin{tikzpicture}[anchorbase, scale=.3]
\draw [very thick] (.5,0) to [out=90,in=225] (1,1);
\draw [very thick] (1.5,0) to [out=90,in=315] (1,1);
\draw[double,->] (1,1) to (1,2);
\end{tikzpicture}
\quad,\quad
S:=
\begin{tikzpicture}[anchorbase, scale=.3]
\draw[double] (1,2) to (1,3);
\draw [very thick,->] (1,3) to [out=135,in=270] (.5,4);
\draw [very thick,->] (1,3) to [out=45,in=270] (1.5,4);
\end{tikzpicture}
\]
Here $M$ stands for merge and $S$ for split. We define superposed rotation foams locally via:
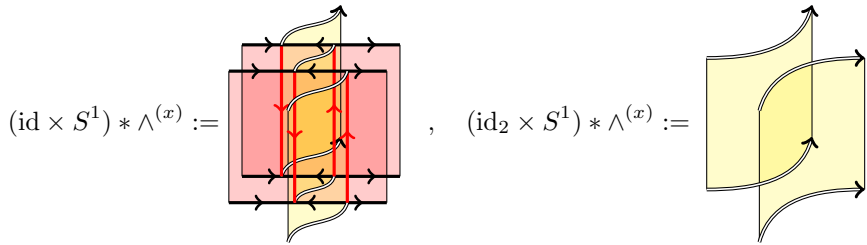
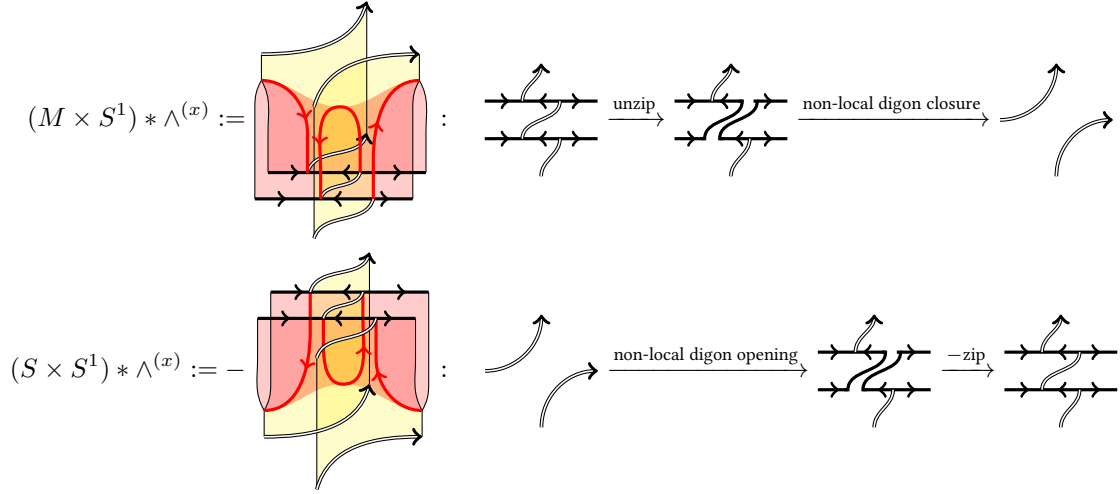
\begin{align*}
(\id\times S^1)*\w{x}:=&\;
  \begin{tikzpicture}[anchorbase,scale=.7, fill opacity=.2]
\fill[red] (.375,.75)  to (.375,3.25)  to  (3.375,3.25) to (3.375,.75) to (.375,.75);
\fill[red] (.625,1.25)  to (.625,3.75)  to (3.625,3.75) to (3.625,1.25) to (.625,1.25);
\fill[yellow] (1.5,0) to [out=75,in=255] (2.625,.75) to (2.625,3.25) to [out=255,in=75] (1.5,2.5) to (1.5,0); 
\fill[yellow] (1.625,.75) to [out=75,in=255] (2.375,1.25) to (2.375,3.75)  to [out=255,in=75] (1.625,3.25) to (1.625,.75); 
\fill[yellow] (1.375,1.25) to [out=75,in=255] (2.5,2) to (2.5,4.5)   to [out=255,in=75] (1.375,3.75) to (1.375,1.25); 
\draw (.375,.75)  to (.375,3.25); 
\draw (3.375,.75)  to (3.375,3.25); 
\draw (.625,1.25)  to (.625,3.75); 
\draw (3.625,1.25)  to (3.625,3.75); 
\draw (2.5,2) to (2.5,4.5) ;
\draw (1.5,2.5) to (1.5,0);
\draw[very thick, directed=.22, rdirected=.55,directed=.85] (.375,.75) to (3.375,.75);
\draw[very thick, directed=.22, rdirected=.52,directed=.85] (.625,1.25) to (3.625,1.25);
\draw[double] (1.5,0) to [out=75,in=255] (2.625,.75);
\draw[double] (1.625,.75) to [out=75,in=255] (2.375,1.25);
\draw[double,->] (1.375,1.25) to [out=75,in=255] (2.5,2);
\draw[red, very thick, directed=.55] (2.625,.75) to (2.625,3.25); 
\draw[red, very thick, rdirected=.55] (1.625,0.75) to (1.625,3.25);
\draw[red, very thick, directed=.55] (2.375,1.25) to (2.375,3.75);  
\draw[red, very thick, rdirected=.55] (1.375,1.25) to (1.375,3.75);  
\draw[very thick, directed=.2, rdirected=.55,directed=.85] (.375,3.25) to (3.375,3.25);
\draw[very thick, directed=.2, rdirected=.52,directed=.85] (.625,3.75) to (3.625,3.75);
\draw[double] (1.5,2.5) to [out=75,in=255] (2.625,3.25);
\draw[double] (1.625,3.25) to [out=75,in=255] (2.375,3.75);
\draw[double,->] (1.375,3.75) to [out=75,in=255] (2.5,4.5);
  \end{tikzpicture}
\quad,\quad
(\id_2\times S^1)*\w{x}:=\; 
  \begin{tikzpicture}[anchorbase,scale=.7, fill opacity=.2]
\fill[yellow] (1.5,0) to [out=75,in=180] (3.5,1) to (3.5,3.5) to [out=180,in=75] (1.5,2.5) to (1.5,0) ; 
\fill[yellow] (.5,1) to [out=0,in=255] (2.5,2) to (2.5,4.5) to [out=255,in=0] (.5,3.5) to (.5,1); 
\draw (3.5,1) to (3.5,3.5); 
\draw (.5,1)  to (.5,3.5); 
\draw (2.5,2) to (2.5,4.5) ;
\draw (1.5,2.5) to (1.5,0);
\draw[double,->] (1.5,0) to [out=75,in=180] (3.5,1);
\draw[double,->] (.5,1) to [out=0,in=255] (2.5,2); 
\draw[double,->] (1.5,2.5) to [out=75,in=180] (3.5,3.5);
\draw[double,->] (.5,3.5) to [out=0,in=255] (2.5,4.5);
  \end{tikzpicture}
\\
(M\times S^1)*\w{x}:=& \;
  \begin{tikzpicture}[anchorbase,scale=.7, fill opacity=.2]
\fill[red] (.375,.75)  to (.375,2.5)  to [out=90,in=240](.5,3) to [out=0,in=180] (2,2.5) to [out=0,in=180] (3.5,3) to [out=240,in=90] (3.375,2.5) to (3.375,.75)  to (.375,.75) ;
\fill[red] (.625,1.25)to (.625,2)  to [out=90,in=300](.5,3) to [out=0,in=180] (2,2.5) to [out=0,in=180] (3.5,3) to [out=300,in=90] (3.625,2) to (3.625,1.25) to (.625,1.25) ;
\fill[yellow] (1.5,0) to [out=75,in=255] (2.625,.75) to (2.625,1.5) to [out=90,in=180] (3.5,3) to (3.5,3.5) to [out=180,in=75] (1.5,2.5) to (1.5,0); 
\fill[yellow] (2.375,1.25) to [out=255,in=75]  (1.625,0.75) to (1.625,1.25) to [out=90,in=180] (2,2.5) to [out=0,in=90] (2.375,1.75) to (2.375,1.25); 
\fill[yellow] (1.375,1.25) to [out=75,in=255] (2.5,2) to (2.5,4.5)  to [out=255,in=0] (.5,3.5) to (.5,3) to [out=0,in=90] (1.375,2) to (1.375,1.25); 
\draw (.375,.75) to (.375,2.5)  to [out=90,in=240](.5,3) to (.5,3.5); 
\draw (3.375,.75) to (3.375,2.5)  to [out=90,in=240](3.5,3) to (3.5,3.5); 
\draw (.625,1.25) to (.625,2)  to [out=90,in=300](.5,3); 
\draw (3.625,1.25) to (3.625,2) to [out=90,in=300](3.5,3); 
\draw (2.5,2) to (2.5,4.5) ;
\draw (1.5,2.5) to (1.5,0);
\draw[very thick, directed=.2, rdirected=.55,directed=.85] (.375,.75) to (3.375,.75);
\draw[very thick, directed=.2, rdirected=.52,directed=.85] (.625,1.25) to (3.625,1.25);
\draw[double] (1.5,0) to [out=75,in=255] (2.625,.75);
\draw[double] (1.625,.75) to [out=75,in=255] (2.375,1.25);
\draw[double,->] (1.375,1.25) to [out=75,in=255] (2.5,2);
\draw[red, very thick, directed=.55] (2.625,.75) to (2.625,1.5) to [out=90,in=180] (3.5,3); 
\draw[red, very thick, rdirected=.35] (1.625,0.75) to (1.625,1.25) to [out=90,in=180] (2,2.5) to [out=0,in=90] (2.375,1.75) to (2.375,1.25);
\draw[red, very thick, directed=.55] (.5,3) to [out=0,in=90] (1.375,2) to (1.375,1.25);  
\draw[double,->] (1.5,2.5) to [out=75,in=180] (3.5,3.5);
\draw[double,->] (.5,3.5) to [out=0,in=255] (2.5,4.5);
  \end{tikzpicture}
  \;
  \colon
  \quad 
  \begin{tikzpicture}[anchorbase,scale=.5]
  \draw[very thick, directed=.22, rdirected=.55,directed=.85] (0,1) to (3,1);
  \draw[very thick, directed=.22, rdirected=.55,directed=.85] (0,2) to (3,2);
  \draw[double] (1.5,0) to [out=90,in=270] (2,1);
\draw[double] (1,1) to [out=90,in=270] (2,2);
\draw[double,->] (1,2) to [out=90,in=270] (1.5,3);
  \end{tikzpicture}
  \xrightarrow{\textrm{unzip}}
  \begin{tikzpicture}[anchorbase,scale=.5]
  \draw[very thick, directed=.22, rdirected=.55,directed=.85] (0,1) to (0.8,1) (1.2,1) to (3,1);
  \draw[very thick, directed=.22, rdirected=.55,directed=.85] (0,2) to (1.8,2) (2.2,2) to(3,2);
  \draw[double] (1.5,0) to [out=90,in=270] (2,1);
\draw[very thick] (.7,1) to (.8,1) to [out=90,in=270] (1.8,2) to (1.7,2);
\draw[very thick] (1.3,1) to(1.2,1) to [out=90,in=270] (2.2,2) to (2.1,2);
\draw[double,->] (1,2) to [out=90,in=270] (1.5,3);
  \end{tikzpicture}
  \xrightarrow{\textrm{non-local digon closure}}
  \begin{tikzpicture}[anchorbase,scale=.5]
  \draw[double,->] (1.5,0) to [out=90,in=180] (3,1.5);
\draw[double,->] (0,1.5) to [out=00,in=270] (1.5,3);
  \end{tikzpicture}
  \\
(S\times S^1)*\w{x}:=-&\; 
  \begin{tikzpicture}[anchorbase,xscale=-.7,yscale=-.7, fill opacity=.2]
\fill[red] (.375,.75)  to (.375,2.5)  to [out=90,in=240](.5,3) to [out=0,in=180] (2,2.5) to [out=0,in=180] (3.5,3) to [out=240,in=90] (3.375,2.5) to (3.375,.75)  to (.375,.75) ;
\fill[red] (.625,1.25)to (.625,2)  to [out=90,in=300](.5,3) to [out=0,in=180] (2,2.5) to [out=0,in=180] (3.5,3) to [out=300,in=90] (3.625,2) to (3.625,1.25) to (.625,1.25) ;
\fill[yellow] (1.5,0) to [out=75,in=255] (2.625,.75) to (2.625,1.5) to [out=90,in=180] (3.5,3) to (3.5,3.5) to [out=180,in=75] (1.5,2.5) to (1.5,0); 
\fill[yellow] (2.375,1.25) to [out=255,in=75]  (1.625,0.75) to (1.625,1.25) to [out=90,in=180] (2,2.5) to [out=0,in=90] (2.375,1.75) to (2.375,1.25); 
\fill[yellow] (1.375,1.25) to [out=75,in=255] (2.5,2) to (2.5,4.5)  to [out=255,in=0] (.5,3.5) to (.5,3) to [out=0,in=90] (1.375,2) to (1.375,1.25); 
\draw (.375,.75) to (.375,2.5)  to [out=90,in=240](.5,3) to (.5,3.5); 
\draw (3.375,.75) to (3.375,2.5)  to [out=90,in=240](3.5,3) to (3.5,3.5); 
\draw (.625,1.25) to (.625,2)  to [out=90,in=300](.5,3); 
\draw (3.625,1.25) to (3.625,2) to [out=90,in=300](3.5,3); 
\draw (2.5,2) to (2.5,4.5) ;
\draw (1.5,2.5) to (1.5,0);
\draw[double,<-] (1.5,2.5) to [out=75,in=180] (3.5,3.5);
\draw[double,<-] (.5,3.5) to [out=0,in=255] (2.5,4.5);
\draw[red, very thick, directed=.55] (2.625,.75) to (2.625,1.5) to [out=90,in=180] (3.5,3); 
\draw[red, very thick, rdirected=.35] (1.625,0.75) to (1.625,1.25) to [out=90,in=180] (2,2.5) to [out=0,in=90] (2.375,1.75) to (2.375,1.25);
\draw[red, very thick, directed=.55] (.5,3) to [out=0,in=90] (1.375,2) to (1.375,1.25); 
\draw[very thick, rdirected=.2, directed=.55,rdirected=.85] (.375,.75) to (3.375,.75);
\draw[very thick, rdirected=.2, directed=.52,rdirected=.85] (.625,1.25) to (3.625,1.25);
\draw[double,<-] (1.5,0) to [out=75,in=255] (2.625,.75);
\draw[double] (1.625,.75) to [out=75,in=255] (2.375,1.25);
\draw[double] (1.375,1.25) to [out=75,in=255] (2.5,2); 
  \end{tikzpicture}
  \;
  \colon
  \quad 
  \begin{tikzpicture}[anchorbase,scale=.5]
  \draw[double,->] (1.5,0) to [out=90,in=180] (3,1.5);
\draw[double,->] (0,1.5) to [out=00,in=270] (1.5,3);
  \end{tikzpicture}
  \xrightarrow{\textrm{non-local digon opening}}
  \begin{tikzpicture}[anchorbase,scale=.5]
  \draw[very thick, directed=.22, rdirected=.55,directed=.85] (0,1) to (0.8,1) (1.2,1) to (3,1);
  \draw[very thick, directed=.22, rdirected=.55,directed=.85] (0,2) to (1.8,2) (2.2,2) to(3,2);
  \draw[double] (1.5,0) to [out=90,in=270] (2,1);
\draw[very thick] (.7,1) to (.8,1) to [out=90,in=270] (1.8,2) to (1.7,2);
\draw[very thick] (1.3,1) to(1.2,1) to [out=90,in=270] (2.2,2) to (2.1,2);
\draw[double,->] (1,2) to [out=90,in=270] (1.5,3);
  \end{tikzpicture}
   \xrightarrow{-\textrm{zip}}
  \begin{tikzpicture}[anchorbase,scale=.5]
  \draw[very thick, directed=.22, rdirected=.55,directed=.85] (0,1) to (3,1);
  \draw[very thick, directed=.22, rdirected=.55,directed=.85] (0,2) to (3,2);
  \draw[double] (1.5,0) to [out=90,in=270] (2,1);
\draw[double] (1,1) to [out=90,in=270] (2,2);
\draw[double,->] (1,2) to [out=90,in=270] (1.5,3);
  \end{tikzpicture}
 \end{align*}
These local models (and their back-to-front reflections) are used wherever the $2$-labeled curve $\w{x}$ has an intersection with the annulus $\A\subset \Su$, along which the spun webs $\id\times S^1$, $M\times S^1$ and $S\times S^1$ are embedded. Away from these intersections, we do not change the spun foam, but we add vertical $2$-labeled sheets over $\w{x}$. 

There are also obvious choices for the local replacements in the case of identity webs with more than two strands and for split and merge vertices with several parallel strands added. Any web can be written as the composition $W=W_m\circ \cdots \circ W_2 \circ W_1 $ of such pieces $W_i$ and we now define $(W\times S^1)*\w{x}$ to be the foam obtained from the rotation foam $W\times S^1$ by the appropriate composition of local replacements  

\begin{remark} The minus sign in the local replacement for $(S\times S^1)*\w{x}$ may appear arbitrary to the reader. This is because it is indeed arbitrary, although not unmotivated. In Section~\ref{sec:stwo} we will explain how superposition $-*\w{x}$ induces an auto-equivalence of $\Sfoam$, which is defined via the Khovanov functor. In this framework, the local replacements for $(M\times S^1)*\w{x}$ and $(S\times S^1)*\w{x}$ are induced by mutually inverse fork-slide moves (see e.g. the first isomorphism in Lemma~\ref{lem:forksliding2}). In order for the fork-slide foams to be inverse to each other, one has to carry a sign. Which one carries a sign depends on an ordering of the resolutions of crossings between $1$- and $2$-labeled edges. In the following computations concerning superposed rotation foams, this arbitrary choice is irrelevant, as long as it is chosen consistently on the equal domain and target webs.
\end{remark}

\begin{definition}\label{def:JWfoam}
Let $L=\{(C_i,n_i)\}\in \cal{L}(\Su)$ be an oriented integer lamination on $\Su$ and $\w{x}$ a $2$-labeled multi-curve on $\Su$ transverse to all $C_i$, which minimizes the intersection number with all $C_i$ in its homology class. Then we define $L^F_S*\w{x}$ to be the foam in $\Sfoam$ obtained from $L^F_S$ by local replacements as described above for every intersection between $C_i$ and $\w{x}$. We call them \emph{Jones--Wenzl basis foams}.
\end{definition} 

\begin{definition} A foam $F\colon W_1 \to W_2$ in $\Sfoamred$ is said to \emph{have a turnback at $W_i$} if $c(F)$ contains a boundary-parallel annulus with both boundary circles on $c(W_i)$. It is said have a \emph{standard turnback} if $F=F^\prime \circ F_S$ (or $F=F_M\circ F^\prime$), where $F_S$ is locally modeled on a superposed split foam $(S\times S^1)*\w{x}$ near the above-mentioned boundary circles, and on the identity foam elsewhere (similarly, $F_M$ is modeled on a merge foam). We say a morphism $G\colon W_1 \to W_2$ in $\Sfoamred$ \emph{annihilates turnbacks} if $F_2\circ G = 0 =G \circ F_1$ whenever $F_i$ has a turnback at $W_i$.   
\end{definition}

\begin{lemma}\label{lem:turnbacks} Let $W$ be the source web of one of the Jones--Wenzl basis foams $L^F_S*\w{x}$ from Definition~\ref{def:JWfoam}. Then any foam $F$ which has a turnback at $W$ can be expressed as a linear combination of foams having standard turnbacks at $W$.
\end{lemma}

\begin{proof}
Consider the two circles $c_1,c_2$ in $c(W)$ which are connected by the turnback annulus in $c(F)$. We choose two points $p_1\in c_1$ and $p_2\in c_2$, such that corresponding $1$-labeled edges in $W$ have the orientation prescribed by the oriented integer lamination $L$. By assumption there exist paths $\pi_i$ on the $1$-labeled facets of $F$, which connect the points $p_i$ to a point on a seam. Due to orientation reasons, the composite path $\pi=\pi_2^{-1}\circ \pi_1$ passes an odd number of seams. By applying foam relations along a strategy analogous to the one employed in the proof of Lemma~\ref{lem:neckcut}, we may assume this number of seams to be one. The cycle $c$ resulting from a small push-off of $c_2^{-1} \circ\pi_2^{-1}\circ \pi_1\circ c_1$ bounds a compression disk and we perform neck-cutting along it using Lemma~\ref{lem:neckcut2}. We illustrate a simple case, in which we focus only on the neighborhood of $c_1$ and $c_2$ in the bottom web and suppress dots arising from neck-cutting:
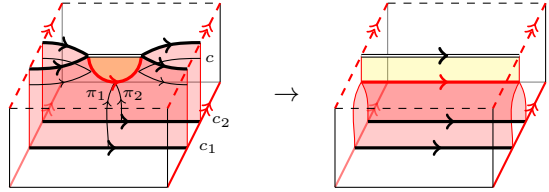
\[
  \begin{tikzpicture}[anchorbase,scale=.7, fill opacity=.2]
\annback{3}{2}{1.5}
\fill[red] (.375,.75)  to (.375,2.25)  to [out=0,in=210] (1.5,2.5) to (2.5,2.5)  to [out=330,in=180] (3.375,2.25) to (3.375,.75) to (.375,.75);
\fill[red] (.625,1.25)  to (.625,2.75)  to [out=0,in=150] (1.5,2.5) to (2.5,2.5)  to [out=30,in=180] (3.625,2.75) to (3.625,1.25) to (.625,1.25);
\fill[yellow](1.5,2.5) to [out=270,in=180] (2,2) to [out=0,in=270](2.5,2.5) to (1.5,2.5); 
\draw[red] (.375,.75)  to (.375,2.25); 
\draw[red] (3.375,.75)  to (3.375,2.25); 
\draw[red] (.625,1.25)  to (.625,2.75); 
\draw[red] (3.625,1.25)  to (3.625,2.75); 
\draw[directed=.70] (1.875,.75) to [out=90,in=240] (2,2);
\draw[directed=.55] (2.125,1.25) to [out=90,in=300] (2,2);
\draw (.625,2.45) to [out=0,in=150] (1.55,2.2); 
\draw[directed=.55] (.375,1.95) to [out=0,in=210] (1.55,2.2);
\draw (2.45,2.2) to [out=30,in=180] (3.625,2.45);
\draw[directed=.55] (2.45,2.2) to [out=330,in=180] (3.375,1.95);
\draw[very thick, directed=.55] (.375,.75) to (3.375,.75);
\draw[very thick, directed=.55] (.625,1.25) to (3.625,1.25);
\draw[red, very thick, directed=.55] (1.5,2.5) to [out=270,in=180] (2,2) to [out=0,in=270] (2.5,2.5); 
\draw[very thick, directed=.55] (.625,2.75) to [out=0,in=150] (1.5,2.5); 
\draw[very thick, directed=.55] (.375,2.25) to [out=0,in=210] (1.5,2.5);
\draw[double] (1.5,2.5) to (2.5,2.5);
\draw[very thick, directed=.55] (2.5,2.5) to [out=30,in=180] (3.625,2.75);
\draw[very thick, directed=.55] (2.5,2.5) to [out=330,in=180] (3.375,2.25);
\node[opacity=1] at (3.775,.75) {\tiny $c_1$};
\node[opacity=1] at (4.025,1.25) {\tiny $c_2$};
\node[opacity=1] at (3.8,2.45) {\tiny $c$};
\node[opacity=1] at (1.65,1.75) {\tiny $\pi_1$};
\node[opacity=1] at (2.35,1.75) {\tiny $\pi_2$};
\annfrontd{3}{2}{1.5}
  \end{tikzpicture}
 \quad  \to \quad
  \begin{tikzpicture}[anchorbase,scale=.7, fill opacity=.2]
\annback{3}{2}{1.5}
\fill[red] (.375,.75)  to [out=90,in=240](.5,2) to (3.5,2) to [out=240,in=90] (3.375,.75) to (.375,.75) ;
\fill[red] (.625,1.25)  to [out=90,in=300](.5,2) to (3.5,2) to [out=300,in=90] (3.625,1.25) to (.625,1.25) ;
\fill[yellow](.5,2.5) to (3.5,2.5)to (3.5,2)to (.5,2)to (.5,2.5); 
\draw[red] (.375,.75)  to [out=90,in=240](.5,2) to (.5,2.5); 
\draw[red] (3.375,.75)  to [out=90,in=240](3.5,2) to (3.5,2.5); 
\draw[red] (.625,1.25)  to [out=90,in=300](.5,2); 
\draw[red] (3.625,1.25)  to [out=90,in=300](3.5,2); 
\draw[very thick, directed=.55] (.375,.75) to (3.375,.75);
\draw[very thick, directed=.55] (.625,1.25) to (3.625,1.25);
\draw[red, very thick, directed=.55] (.5,2) to (3.5,2);
\draw[double,directed=.55] (.5,2.5) to (3.5,2.5);
\annfrontd{3}{2}{1.5}
  \end{tikzpicture}
  \]

The result is a linear combination of two foams, both of which have a standard turnback (up to isotopy) and a dot each. In fact, one of the summands has a dot on its turnback annulus, but this summand is equal to zero, because its underlying surface contains an undotted sphere. The remaining summand has a dotted sphere in its underlying surface, which can be removed at the expense of a sign. 
\end{proof}

For the next proposition, we need two foam relations, which follow from the ones in Definition~\ref{def:Tfoam}:
\begin{align} \label{eqn:pinchdisk}
\begin{tikzpicture} [anchorbase,scale=.5,fill opacity=0.2]
\path[fill=red]  (2,5) to [out=90,in=270] (2,1) to [out=180,in=0] (0.25,1) to  (0.25,1.75) to [out=0, in=270] (1,2.5) to [out=90, in=0] (0.25,3.25) to [out=180, in=90] (-0.5,2.5) to [out=270,in=180] (0.25,1.75) to (0.25,1)  to (-2,1) to (-2,5) to (2,5);
\path[fill=red]  (2.5,4) to [out=90,in=270] (2.5,0) to [out=180,in=0] (0.25,0) to  (0.25,1.75) to [out=0, in=270] (1,2.5) to [out=90, in=0] (0.25,3.25) to [out=180, in=90] (-0.5,2.5) to [out=270,in=180] (0.25,1.75) to (0.25,0)  to (-1.5,0) to (-1.5,4) to (2.5,4);
\path[fill=yellow] (0.25,1.75) to [out=0, in=270] (1,2.5) to [out=90, in=0] (0.25,3.25) to [out=180, in=90] (-0.5,2.5) to [out=270,in=180] (0.25,1.75);
	\draw[very thick, directed=.65] (2,1) to [out=180,in=0] (-2,1);
	\draw[very thick, directed=.55] (2.5,0) to [out=180,in=0] (-1.5,0);
	\draw (2,1) to (2,5);
	\draw (2.5,0) to (2.5,4);
	\draw (-1.5,0) to (-1.5,4);
	\draw (-2,1) to (-2,5);	
	\draw[dashed] (2,3) to [out=180,in=45] (1,2.5);
	\draw[dashed] (2.5,2) to [out=180,in=315] (1,2.5);
	\draw[dashed] (1,2.5) to (-.5,2.5);
	\draw[dashed] (-.5,2.5) to [out=225,in=0] (-1.5,2);
	\draw[dashed] (-.5,2.5) to [out=135,in=0] (-2,3);
	\draw[very thick, red, directed=.65] (1,2.5) to [out=270,in=0]  (0.25,1.75) to [out=180, in = 270] (-0.5,2.5) to [out=90, in = 180] (0.25,3.25) to [out=0, in = 90] (1,2.5);
	\draw[very thick, directed=.65] (2,5) to [out=180,in=0] (-2,5);
	\draw[very thick, directed=.55] (2.5,4) to [out=180,in=0] (-1.5,4);
\end{tikzpicture}
\quad&=\quad
\begin{tikzpicture} [anchorbase,scale=.5,fill opacity=0.2]
\path[fill=red]  (2,5) to [out=90,in=270] (2,1) to (-2,1) to (-2,5) to (2,5);
\path[fill=red]  (2.5,4) to [out=90,in=270] (2.5,0)  to (-1.5,0) to (-1.5,4) to (2.5,4);
	\draw[very thick, directed=.65] (2,1) to [out=180,in=0] (-2,1);
	\draw[very thick, directed=.55] (2.5,0) to [out=180,in=0] (-1.5,0);
	\draw (2,1) to (2,5);
	\draw (2.5,0) to (2.5,4);
	\draw (-1.5,0) to (-1.5,4);
	\draw (-2,1) to (-2,5);	
	\draw[dashed] (2,3) to (-2,3);
	\draw[dashed] (2.5,2) to (-1.5,2);
	\draw[very thick, directed=.65] (2,5) to [out=180,in=0] (-2,5);
	\draw[very thick, directed=.55] (2.5,4) to [out=180,in=0] (-1.5,4);
	\node[opacity=1] at (1.75,.5) {$\bullet$};
\end{tikzpicture}
\quad-\quad
\begin{tikzpicture} [anchorbase,scale=.5,fill opacity=0.2]
\path[fill=red]  (2,5) to [out=90,in=270] (2,1) to (-2,1) to (-2,5) to (2,5);
\path[fill=red]  (2.5,4) to [out=90,in=270] (2.5,0)  to (-1.5,0) to (-1.5,4) to (2.5,4);
	\draw[very thick, directed=.65] (2,1) to [out=180,in=0] (-2,1);
	\draw[very thick, directed=.55] (2.5,0) to [out=180,in=0] (-1.5,0);
	\draw (2,1) to (2,5);
	\draw (2.5,0) to (2.5,4);
	\draw (-1.5,0) to (-1.5,4);
	\draw (-2,1) to (-2,5);	
	\draw[dashed] (2,3) to (-2,3);
	\draw[dashed] (2.5,2) to (-1.5,2);
	\draw[very thick, directed=.65] (2,5) to [out=180,in=0] (-2,5);
	\draw[very thick, directed=.55] (2.5,4) to [out=180,in=0] (-1.5,4);
	\node[opacity=1] at (-1,4.5) {$\bullet$};
\end{tikzpicture}
\\
\label{eqn:blister}
\begin{tikzpicture} [anchorbase,scale=.5,fill opacity=0.2]
\node[opacity=1] at (0.1,3.45) {\tiny $b$};
	\node[opacity=1] at (-0.2,3.3) {$\bullet$};
\path[fill=red]  (1,2.5) to [out=90,in=0] (0,4) to [out=180, in=90] (-1,2.5) to [out=270,in=180] (0,1) to [out=0,in=270] (1,2.5) ;
\path[fill=red]  (1,2.5) to [out=90,in=0] (0,4) to [out=180, in=90] (-1,2.5) to [out=270,in=180] (0,1) to [out=0,in=270] (1,2.5) ;
\path[fill=yellow] (2,4.5) to (2,2.5) to (1,2.5) to [out=90,in=0] (0,4) to [out=180, in=90] (-1,2.5) to [out=270,in=180] (0,1) to [out=0,in=270] (1,2.5) to (2, 2.5) to (2,0.5) to (-2 , 0.5) to (-2, 4.5) to (2,4.5);
	\draw[double, directed=.55] (2,0.5) to (-2,0.5);
	\draw[very thick, red, directed=.65] (1,2.5) to [out=90,in=0] (0,4) to [out=180,in=90] (-1,2.5);
	\draw[very thick, red, directed=.65] (-1,2.5) to [out=270,in=180] (0,1) to [out=0,in=270] (1,2.5);
	\draw (2,0.5) to (2,4.5);
	\draw (-2,0.5) to (-2,4.5);
\draw[dashed] (2,2.5) to (1,2.5);
	\draw[dashed] (-1,2.5) to  (-2,2.5);
	\draw[dashed] (1,2.5) to [out=135,in=0] (0,3) to [out=180, in=45] (-1,2.5);
	\draw[dashed] (1,2.5) to [out=225,in=0] (0,2) to [out=180,in=315] (-1,2.5);
	\draw[double, directed=.55] (2,4.5) to (-2,4.5);
	\node[opacity=1] at (0.5,1.65) {\tiny $a$};
	\node[opacity=1] at (0.2,1.5) {$\bullet$};	
\end{tikzpicture}
\quad&= \quad (\delta_{a,1}\delta_{b,0}-\delta_{a,0}\delta_{b,1})\;
\begin{tikzpicture} [anchorbase,scale=.5,fill opacity=0.2]
\path[fill=yellow] (2,4.5) to (2,0.5) to (-2 , 0.5) to (-2, 4.5) to (2,4.5);
	\draw[double, directed=.55] (2,0.5) to (-2,0.5);
	\draw (2,0.5) to (2,4.5);
	\draw (-2,0.5) to (-2,4.5);
        \draw[dashed] (2,2.5) to (1,2.5);
	\draw[dashed] (2,2.5) to  (-2,2.5);
	\draw[double, directed=.55] (2,4.5) to (-2,4.5);
\end{tikzpicture}
\end{align} 

\begin{proposition}\label{prop:idempotent} The morphisms $L^F_S*\w{x}$ in $\Sfoamred$ are idempotent and annihilate turnbacks.
\end{proposition}
\begin{proof} It suffices to prove that the relations \eqref{eqn:relevantwebrel} on spun foams continue to hold after the local replacements necessitated by the superposition with $\w{x}$. Then the two claimed properties follow from the corresponding properties of the $\glnn{2}$ Jones--Wenzl projectors via Lemma~\ref{lem:WebtoAfoam}, Remark~\ref{rem:JWprop} and the fact that Jones--Wenzl basis foams annihilate turnbacks if they annihilate standard turnbacks, see Lemma~\ref{lem:turnbacks}.

  We inspect the first relation in \eqref{eqn:relevantwebrel}, whose left-hand side is a digon. A spun digon superposed with $\w{x}$ can be written as $((M\circ S)\times S^1)*\w{x}$. Locally, at each intersection we see a foam of the following type:
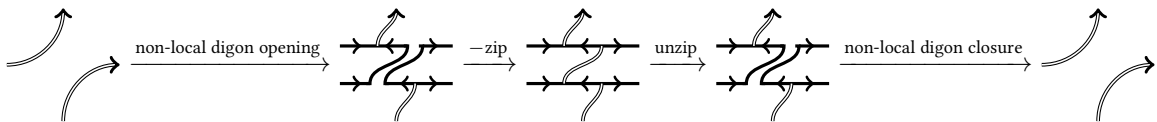
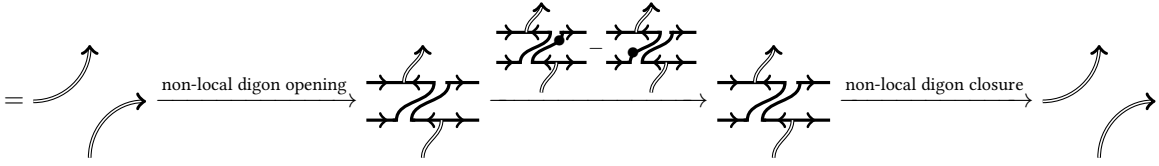
\begin{align*}
  \begin{tikzpicture}[anchorbase,scale=.5]
  \draw[double,->] (1.5,0) to [out=90,in=180] (3,1.5);
\draw[double,->] (0,1.5) to [out=00,in=270] (1.5,3);
  \end{tikzpicture}
  \xrightarrow{\textrm{non-local digon opening}}
  \begin{tikzpicture}[anchorbase,scale=.5]
  \draw[very thick, directed=.22, rdirected=.55,directed=.85] (0,1) to (0.8,1) (1.2,1) to (3,1);
  \draw[very thick, directed=.22, rdirected=.55,directed=.85] (0,2) to (1.8,2) (2.2,2) to(3,2);
  \draw[double] (1.5,0) to [out=90,in=270] (2,1);
\draw[very thick] (.7,1) to (.8,1) to [out=90,in=270] (1.8,2) to (1.7,2);
\draw[very thick] (1.3,1) to(1.2,1) to [out=90,in=270] (2.2,2) to (2.1,2);
\draw[double,->] (1,2) to [out=90,in=270] (1.5,3);
  \end{tikzpicture}
   \xrightarrow{- \textrm{zip}}
  \begin{tikzpicture}[anchorbase,scale=.5]
  \draw[very thick, directed=.22, rdirected=.55,directed=.85] (0,1) to (3,1);
  \draw[very thick, directed=.22, rdirected=.55,directed=.85] (0,2) to (3,2);
  \draw[double] (1.5,0) to [out=90,in=270] (2,1);
\draw[double] (1,1) to [out=90,in=270] (2,2);
\draw[double,->] (1,2) to [out=90,in=270] (1.5,3);
  \end{tikzpicture}
   \xrightarrow{\textrm{unzip}}
  \begin{tikzpicture}[anchorbase,scale=.5]
  \draw[very thick, directed=.22, rdirected=.55,directed=.85] (0,1) to (0.8,1) (1.2,1) to (3,1);
  \draw[very thick, directed=.22, rdirected=.55,directed=.85] (0,2) to (1.8,2) (2.2,2) to(3,2);
  \draw[double] (1.5,0) to [out=90,in=270] (2,1);
\draw[very thick] (.7,1) to (.8,1) to [out=90,in=270] (1.8,2) to (1.7,2);
\draw[very thick] (1.3,1) to(1.2,1) to [out=90,in=270] (2.2,2) to (2.1,2);
\draw[double,->] (1,2) to [out=90,in=270] (1.5,3);
  \end{tikzpicture}
  \xrightarrow{\textrm{non-local digon closure}}
  \begin{tikzpicture}[anchorbase,scale=.5]
  \draw[double,->] (1.5,0) to [out=90,in=180] (3,1.5);
\draw[double,->] (0,1.5) to [out=00,in=270] (1.5,3);
  \end{tikzpicture}
  \\
   = \begin{tikzpicture}[anchorbase,scale=.5]
  \draw[double,->] (1.5,0) to [out=90,in=180] (3,1.5);
\draw[double,->] (0,1.5) to [out=00,in=270] (1.5,3);
  \end{tikzpicture}
  \xrightarrow{\textrm{non-local digon opening}}
  \begin{tikzpicture}[anchorbase,scale=.5]
  \draw[very thick, directed=.22, rdirected=.55,directed=.85] (0,1) to (0.8,1) (1.2,1) to (3,1);
  \draw[very thick, directed=.22, rdirected=.55,directed=.85] (0,2) to (1.8,2) (2.2,2) to(3,2);
  \draw[double] (1.5,0) to [out=90,in=270] (2,1);
\draw[very thick] (.7,1) to (.8,1) to [out=90,in=270] (1.8,2) to (1.7,2);
\draw[very thick] (1.3,1) to(1.2,1) to [out=90,in=270] (2.2,2) to (2.1,2);
\draw[double,->] (1,2) to [out=90,in=270] (1.5,3);
  \end{tikzpicture}
   \xrightarrow{
 \begin{tikzpicture}[anchorbase,scale=.4]
  \draw[very thick, directed=.22, rdirected=.55,directed=.85] (0,1) to (0.8,1) (1.2,1) to (3,1);
  \draw[very thick, directed=.22, rdirected=.55,directed=.85] (0,2) to (1.8,2) (2.2,2) to(3,2);
  \draw[double] (1.5,0) to [out=90,in=270] (2,1);
\draw[very thick] (.7,1) to (.8,1) to [out=90,in=270] (1.8,2) to (1.7,2);
\draw[very thick] (1.3,1) to(1.2,1) to [out=90,in=270] (2.2,2) to (2.1,2);
\draw[double,->] (1,2) to [out=90,in=270] (1.5,3);
\node at (2.1,1.7) {$\bullet$};
  \end{tikzpicture}   
  -
  \begin{tikzpicture}[anchorbase,scale=.4]
  \draw[very thick, directed=.22, rdirected=.55,directed=.85] (0,1) to (0.8,1) (1.2,1) to (3,1);
  \draw[very thick, directed=.22, rdirected=.55,directed=.85] (0,2) to (1.8,2) (2.2,2) to(3,2);
  \draw[double] (1.5,0) to [out=90,in=270] (2,1);
\draw[very thick] (.7,1) to (.8,1) to [out=90,in=270] (1.8,2) to (1.7,2);
\draw[very thick] (1.3,1) to(1.2,1) to [out=90,in=270] (2.2,2) to (2.1,2);
\draw[double,->] (1,2) to [out=90,in=270] (1.5,3);
\node at (.9,1.3) {$\bullet$};
  \end{tikzpicture}
   }
  \begin{tikzpicture}[anchorbase,scale=.5]
  \draw[very thick, directed=.22, rdirected=.55,directed=.85] (0,1) to (0.8,1) (1.2,1) to (3,1);
  \draw[very thick, directed=.22, rdirected=.55,directed=.85] (0,2) to (1.8,2) (2.2,2) to(3,2);
  \draw[double] (1.5,0) to [out=90,in=270] (2,1);
\draw[very thick] (.7,1) to (.8,1) to [out=90,in=270] (1.8,2) to (1.7,2);
\draw[very thick] (1.3,1) to(1.2,1) to [out=90,in=270] (2.2,2) to (2.1,2);
\draw[double,->] (1,2) to [out=90,in=270] (1.5,3);
  \end{tikzpicture}
  \xrightarrow{\textrm{non-local digon closure}}
  \begin{tikzpicture}[anchorbase,scale=.5]
  \draw[double,->] (1.5,0) to [out=90,in=180] (3,1.5);
\draw[double,->] (0,1.5) to [out=00,in=270] (1.5,3);
  \end{tikzpicture}
\end{align*}
Here we have used relation \eqref{eqn:pinchdisk}. If the relevant annulus has $k$ transverse intersections with $\w{x}$, then all will be either of the type shown, or its reflection (this follows from the assumed minimality of the intersection number). Then the entire morphism will be a sum of $2^k$ signed foams, each of which is the composition of opening $k$ digons, placing $k$ dots and closing the $k$ digons. Among these foams with $k$ dotted blisters, only two are non-zero, and they arise when each blister is decorated by precisely one dot. In these cases, all blisters carry their dot on the same side. The term with all dots on the left side evaluates to the identity according to \eqref{eqn:blister}. The term with all dots on the right side carries a sign $(-1)^k$, which is precisely canceled by the sign resulting from the blister evaluation. In sum, we get twice the identity, or: 
\[((M\circ S)\times S^1)*\w{x} = 2 (\id \times S^1)*\w{x}\]
This is the spun and superposed version of the first relation in \eqref{eqn:relevantwebrel}.

The left-hand side of the second relation in \eqref{eqn:relevantwebrel} can be written as $(M\otimes \id_1)\circ(\id_1\otimes S)\circ(\id_1\otimes M)\circ(S\otimes \id_1)$ and the corresponding rotation foam, superposed with $\w{x}$, has the following local description at the intersection points:
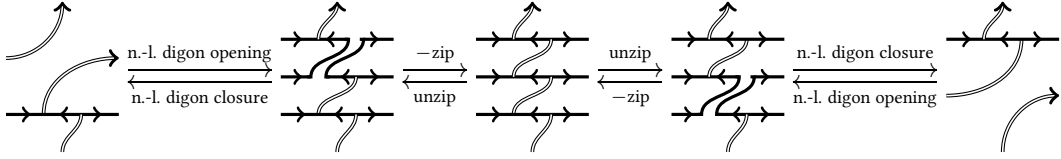
\begin{align*}
\xymatrix{
  \begin{tikzpicture}[anchorbase,scale=.5]
  \draw[double,->] (1,0) to [out=90,in=180] (3,1.5);
\draw[double,->] (0,1.5) to [out=00,in=270] (1.5,3);
\draw[very thick, directed=.22, rdirected=.55,directed=.85] (0,0) to (3,0);
\draw[double] (1.5,-1) to [out=90,in=270] (2,0);
  \end{tikzpicture}
\ar@<.5ex>[rr]^{\textrm{n.-l. digon opening}}
& &
\ar@<.5ex>[ll]^{\textrm{n.-l. digon closure}}
  \begin{tikzpicture}[anchorbase,scale=.5]
  \draw[very thick, directed=.22, rdirected=.55,directed=.85] (0,1) to (0.8,1) (1.2,1) to (3,1);
  \draw[very thick, directed=.22, rdirected=.55,directed=.85] (0,2) to (1.8,2) (2.2,2) to(3,2);
  \draw[double] (1,0) to [out=90,in=270] (2,1);
\draw[very thick] (.7,1) to (.8,1) to [out=90,in=270] (1.8,2) to (1.7,2);
\draw[very thick] (1.3,1) to(1.2,1) to [out=90,in=270] (2.2,2) to (2.1,2);
\draw[double,->] (1,2) to [out=90,in=270] (1.5,3);
\draw[very thick, directed=.22, rdirected=.55,directed=.85] (0,0) to (3,0);
\draw[double] (1.5,-1) to [out=90,in=270] (2,0);
  \end{tikzpicture}
   \ar@<.5ex>[r]^{- \textrm{zip}}
   &
   \ar@<.5ex>[l]^{\textrm{unzip}}
  \begin{tikzpicture}[anchorbase,scale=.5]
  \draw[very thick, directed=.22, rdirected=.55,directed=.85] (0,1) to (3,1);
  \draw[very thick, directed=.22, rdirected=.55,directed=.85] (0,2) to (3,2);
  \draw[double] (1,0) to [out=90,in=270] (2,1);
\draw[double] (1,1) to [out=90,in=270] (2,2);
\draw[double,->] (1,2) to [out=90,in=270] (1.5,3);
\draw[very thick, directed=.22, rdirected=.55,directed=.85] (0,0) to (3,0);
\draw[double] (1.5,-1) to [out=90,in=270] (2,0);
  \end{tikzpicture}
   \ar@<.5ex>[r]^{\textrm{unzip}}
   &
   \ar@<.5ex>[l]^{- \textrm{zip}}
  \begin{tikzpicture}[anchorbase,scale=.5]
  \draw[very thick, directed=.22, rdirected=.55,directed=.85] (0,1) to (0.8,1) (1.2,1) to (3,1);
  \draw[very thick, directed=.22, rdirected=.55,directed=.85] (0,2) to (1.8,2) (2.2,2) to(3,2);
  \draw[double] (1.5,0) to [out=90,in=270] (2,1);
\draw[very thick] (.7,1) to (.8,1) to [out=90,in=270] (1.8,2) to (1.7,2);
\draw[very thick] (1.3,1) to(1.2,1) to [out=90,in=270] (2.2,2) to (2.1,2);
\draw[double] (1,2) to [out=90,in=270] (2,3);
\draw[very thick, directed=.22, rdirected=.55,directed=.85] (0,3) to (3,3);
\draw[double,->] (1,3) to [out=90,in=270] (1.5,4);
  \end{tikzpicture}
\ar@<.5ex>[rr]^{\textrm{n.-l. digon closure}}
  & &
  \ar@<.5ex>[ll]^{\textrm{n.-l. digon opening}}
  \begin{tikzpicture}[anchorbase,scale=.5]
  \draw[double,->] (1.5,0) to [out=90,in=180] (3,1.5);
\draw[double] (0,1.5) to [out=00,in=270] (2,3);
\draw[very thick, directed=.22, rdirected=.55,directed=.85] (0,3) to (3,3);
\draw[double,->] (1,3) to [out=90,in=270] (1.5,4);
  \end{tikzpicture}
 }
\end{align*}
In fact, the local model at each intersection point is the composition starting at the left and proceeding all the way to the right and back again to the left. The two steps in the middle far-commute, thus an equivalent description is the following:

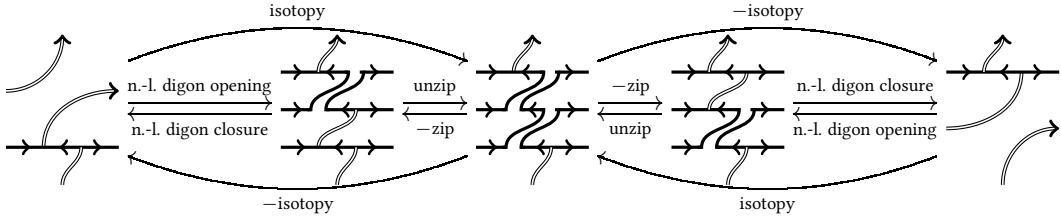
\begin{align*}
\xymatrix{
  \begin{tikzpicture}[anchorbase,scale=.5]
  \draw[double,->] (1,0) to [out=90,in=180] (3,1.5);
\draw[double,->] (0,1.5) to [out=00,in=270] (1.5,3);
\draw[very thick, directed=.22, rdirected=.55,directed=.85] (0,0) to (3,0);
\draw[double] (1.5,-1) to [out=90,in=270] (2,0);
  \end{tikzpicture}
\ar@<.5ex>[rr]^{\textrm{n.-l. digon opening}}
\ar@<1.5ex>@/^2pc/[rrr]^{\textrm{isotopy}}
& &
\ar@<.5ex>[ll]^{\textrm{n.-l. digon closure}}
  \begin{tikzpicture}[anchorbase,scale=.5]
  \draw[very thick, directed=.22, rdirected=.55,directed=.85] (0,1) to (0.8,1) (1.2,1) to (3,1);
  \draw[very thick, directed=.22, rdirected=.55,directed=.85] (0,2) to (1.8,2) (2.2,2) to(3,2);
  \draw[double] (1,0) to [out=90,in=270] (2,1);
\draw[very thick] (.7,1) to (.8,1) to [out=90,in=270] (1.8,2) to (1.7,2);
\draw[very thick] (1.3,1) to(1.2,1) to [out=90,in=270] (2.2,2) to (2.1,2);
\draw[double,->] (1,2) to [out=90,in=270] (1.5,3);
\draw[very thick, directed=.22, rdirected=.55,directed=.85] (0,0) to (3,0);
\draw[double] (1.5,-1) to [out=90,in=270] (2,0);
  \end{tikzpicture}
      \ar@<.5ex>[r]^{\textrm{unzip}}
   &
   \ar@<.5ex>[l]^{- \textrm{zip}}
  \begin{tikzpicture}[anchorbase,scale=.5]
  \draw[very thick, directed=.22, rdirected=.55,directed=.85] (0,1) to (0.8,1) (1.2,1) to (1.8,1) (2.2,1) to (3,1);
  \draw[very thick, directed=.22, rdirected=.55,directed=.85] (0,0) to (0.8,0) (1.2,0)  to(3,0);
\draw[very thick] (.7,1) to (.8,1) to [out=90,in=270] (1.8,2) to (1.7,2);
\draw[very thick] (1.3,1) to(1.2,1) to [out=90,in=270] (2.2,2) to (2.1,2);
\draw[very thick] (.7,0) to (.8,0) to [out=90,in=270] (1.8,1) to (1.7,1);
\draw[very thick] (1.3,0) to(1.2,0) to [out=90,in=270] (2.2,1) to (2.1,1);
\draw[double,->] (1,2) to [out=90,in=270] (1.5,3);
\draw[very thick, directed=.22, rdirected=.55,directed=.85] (0,2) to (1.8,2) (2.2,2)to (3,2);
\draw[double] (1.5,-1) to [out=90,in=270] (2,0);
  \end{tikzpicture}
  \ar@<1.5ex>@/^2pc/[rrr]^{-\textrm{isotopy}}
\ar@<1.5ex>@/^2pc/[lll]^{-\textrm{isotopy}}
   \ar@<.5ex>[r]^{- \textrm{zip}}
   &
   \ar@<.5ex>[l]^{\textrm{unzip}}
  \begin{tikzpicture}[anchorbase,scale=.5]
  \draw[very thick, directed=.22, rdirected=.55,directed=.85] (0,1) to (0.8,1) (1.2,1) to (3,1);
  \draw[very thick, directed=.22, rdirected=.55,directed=.85] (0,2) to (1.8,2) (2.2,2) to(3,2);
  \draw[double] (1.5,0) to [out=90,in=270] (2,1);
\draw[very thick] (.7,1) to (.8,1) to [out=90,in=270] (1.8,2) to (1.7,2);
\draw[very thick] (1.3,1) to(1.2,1) to [out=90,in=270] (2.2,2) to (2.1,2);
\draw[double] (1,2) to [out=90,in=270] (2,3);
\draw[very thick, directed=.22, rdirected=.55,directed=.85] (0,3) to (3,3);
\draw[double,->] (1,3) to [out=90,in=270] (1.5,4);
  \end{tikzpicture}
\ar@<.5ex>[rr]^{\textrm{n.-l. digon closure}}
  & &
  \ar@<.5ex>[ll]^{\textrm{n.-l. digon opening}}
  \ar@<1.5ex>@/^2pc/[lll]^{\textrm{isotopy}}
  \begin{tikzpicture}[anchorbase,scale=.5]
  \draw[double,->] (1.5,0) to [out=90,in=180] (3,1.5);
\draw[double] (0,1.5) to [out=00,in=270] (2,3);
\draw[very thick, directed=.22, rdirected=.55,directed=.85] (0,3) to (3,3);
\draw[double,->] (1,3) to [out=90,in=270] (1.5,4);
  \end{tikzpicture}
 }
\end{align*}
Note that the composites of the outer maps are just (signed) isotopies. The total composition from the left to the right and back is thus the identity. So we have verified:
\[((M\otimes \id_1)\circ(\id_1\otimes S)\circ(\id_1\otimes M)\circ(S\otimes \id_1))*\w{x}= (\id_2\otimes \id_1)*\w{x}\]
This is the spun and superposed version of the second relation in \eqref{eqn:relevantwebrel}. The third relation is analogous.
\end{proof}

Recall from Lemma~\ref{lem:webisos} that the defining web relations and isotopies in $\SWebq$ are lifted to isomorphisms in $\Sfoam$. This implies that we have a well-defined map $\gamma\colon \SWebq \to K_0(\Sfoam)$. We can further compose this with the natural map induced by the embedding of the foam category into its idempotent completion to get a map $\gamma^\prime\colon \SWebq \to K_0(\Kar(\Sfoam))$. 

\begin{lemma}\label{lem:JWdecat} For every basis element $L_S*\w{x}\in \Su\basisJW$ we have $\gamma^\prime(L_S*\w{x})=[L^F_S*\w{x}]\in K_0(\Kar(\Sfoam))$.
\end{lemma}
\begin{proof}
Via the proof of Proposition~\ref{prop:idempotent}, the statement can be deduced from $\gamma^\prime(L_S)=[L^F_S]\in K_0(\Kar(\Sfoam))$, which can be checked one simple closed curve in $L$ at a time. This corresponds to the case $\Su=\A$, where the positive triangular basis change from $\A\basisstd$ to $\A\basisJW$ is easily seen to agree with the decomposition of identity foams in $\gamma(\A\basisstd)$ into Jones--Wenzl basis foams.
\end{proof}
In fact, we have just proved more:

\begin{corollary} \label{cor:JWdecomp} Let $L*\w{x}$ denote the source web of $L_S*\w{x}$, which we interpret as a standard basis element of $\SWebq$. Then the identity foam on $L*\w{x}$ can be written as a sum of primitive idempotents in $\Sfoam$, which are isomorphic to Jones--Wenzl basis foams in $\Kar(\Sfoam)$. Moreover, the decomposition multiplicities agree with those in the basis change from $\Su\basisstd$ to $\Su\basisJW$.
\end{corollary}

\subsection{Decategorification}

\begin{theorem}\label{thm:decat} The map $\gamma$ that sends webs in $\SWebq$ to the class of the corresponding object in $K_0(\Sfoam)$ is an isomorphism of $H_1(\Su)$-graded $\Z[q^{\pm 1}]$-modules. 
\end{theorem}
\begin{proof}
 We have already seen that $\gamma$ is well-defined. Since every object of $\Sfoam$ is isomorphic to a direct sum of shifts of standard basis webs, it is clear that $\gamma$ is surjective. For injectivity, we shall argue that the image $\cal{C}:=\gamma(\Su\basisstd)$ of the standard basis of $\SWebq$ is a linearly independent set in $K_0(\Sfoam)$. In fact, the standard basis webs generate an equivalent full subcategory of $\Sfoam$, so any relation between their classes in $K_0$ stems from an isomorphism of the form:
\[\bigoplus_{i\in I} q^{k_i} A_i \cong \bigoplus_{j\in I} q^{l_j} B_j\]
where the $A_i$ and $B_i$ are standard basis webs. We claim that if such an isomorphism exists, then there exists a bijection $\sigma\colon I \to J$ such that $A_i=B_{\sigma(i)}$ and $k_i=l_{\sigma(i)}$. This, in turn, would imply that such an isomorphism does not impose any non-trivial relation in $K_0(\Sfoam)$.

In order to prove the claim, we recall that the standard basis webs lie in the subcategory $\Sfoamred$ and the morphism spaces between them are non-negatively graded. This implies that we may assume that all components of the isomorphism are of degree zero (a priori, it might have higher degree components, but after truncation to degree zero, one will still have an isomorphism between the same objects). Without loss of generality, we may now assume that $k_i=0=l_j$ for all $i\in I$ and $j\in J$, i.e.:

\begin{equation}
\label{eqn:K0iso}\bigoplus_{i\in I} A_i \cong \bigoplus_{j\in I} B_j
\end{equation}

Next, we embed $\Sfoam$ into its Karoubi envelope $\Kar(\Sfoam)$, where \eqref{eqn:K0iso} still gives an isomorphism. Furthermore, thanks to Corollary~\ref{cor:JWdecomp} we can now decompose the $A_i$ (and $B_j$) further into the direct sum of objects of the Karoubi envelope, which are isomorphic to Jones--Wenzl basis foams. By Lemma~\ref{lem:JWdecat} their classes are contained in the image of the natural map $K_0(\Sfoam) \to K_0(\Kar(\Sfoam))$ and coincide with the images of the Jones--Wenzl basis elements $\Su\basisJW$ of the skein algebra $\SWebq$ under $\gamma$.

Next, we compute the $q$-degree zero morphism spaces between Jones--Wenzl basis foams. Note that if $\Su\neq \T$, then all turnback annuli that can appear in the $1$-labeled part of foams between such elements are supported in neighborhoods of simple closed curves in the corresponding lamination. As such, they are killed by Jones--Wenzl foams in the same way as turnbacks are killed by the usual Jones--Wenzl projectors, see Proposition~\ref{prop:idempotent}. This implies that only those degree zero foams which consist exclusively of vertical annuli induce non-zero morphisms between Jones--Wenzl basis foams. In particular, there are no morphisms between basis foams of different lamination type. 

The existence of the isomorphism \eqref{eqn:K0iso} implies that the same decomposed objects appear on both sides, with equal multiplicities. Since the decomposition multiplicities of standard basis elements into Jones--Wenzl basis elements are triangular, we infer that the two sides of \eqref{eqn:K0iso} already contained isomorphic objects with equal multiplicities.

In the case of the torus $\Su=\T$, the Jones--Wenzl basis foams do not kill all turnbacks, and we cannot use them in the above argument. However, in Section~\ref{sec:toric}, we introduce the rotation foams generated by extremal weight projectors, which are a suitable replacement that allow the completion of the proof in the torus case.
\end{proof}

We record the main step in this proof in a separate corollary. 
 
\begin{corollary} If $\Su\neq \T$, then $\Kar(\Sfoam_0)$ is semisimple, the simple objects are isomorphic to Jones--Wenzl basis foams, and the split Grothendieck group is isomorphic to $\SWebq$.
\end{corollary}

\begin{definition} \label{def:cheb2} Let $\Su\neq \T$. We will write $\Su\basisJW^F$ for the set of objects of $\Kar(\Sfoam)$ given by Jones--Wenzl projector foams on standard basis webs. 
\end{definition}

\section{The Khovanov functors}
\label{sec:BKh}
\begin{definition}
The category $\Slinko$ is the category with:
\begin{itemize}
\item objects, oriented link embeddings in $\Su\times [0,1]$ with generic projection to $\Su$ and with link components labeled by colors from the set $\{1,2\}$ and with an ordering of the crossings,
\item morphisms, oriented, color-preserving link cobordisms embedded in $\Su\times[0,1]\times [0,1]$, modulo isotopy relative to the boundary, as well as trivial cobordisms that only reorder crossings.
\end{itemize}
More generally, we denote by $\Slink$ the category of link embeddings and cobordisms without the genericity assumptions.
\end{definition}
The projection of a link embedding is generic if it produces a link diagram. Since every link embedding can be isotoped into generic position, $\Slinko$ is an equivalent full subcategory of $\Slink$. Cobordisms in $\Slinko$ can also be isotoped into generic position and then presented as movies of link diagrams. Indeed, the time slices of a cobordism in generic position are link diagrams, except in finitely many points across which the link diagrams differ by handle attachments or Reidemeister moves. Isotopies of cobordisms in generic position can be assumed to be composed out of Carter--Rieger--Saito movie moves, \cite{CSp,CRS}.

The Khovanov functor as constructed by Blanchet~\cite{Blan} assigns to a crossing a chain complex built out of webs and foams between them, whose boundary data agree with those of the crossing: 
\begin{gather}
\Kh\left( \;
\begin{tikzpicture}[anchorbase, scale=.5]
\draw [very thick, ->] (2,1) to [out=180,in=0] (0,0);
 \draw [white,line width=.15cm] (2,0) to [out=180,in=0] (0,1) ;
\draw [very thick, ->] (2,0) to [out=180,in=0] (0,1);
\end{tikzpicture}
\;\right)
 \;\;= \;\;
\begin{tikzpicture}[anchorbase, scale=.5]
\draw [very thick, ->] (2,1) to (0,1);
\draw [very thick, ->] (2,0) to (0,0);
\end{tikzpicture}
\;\;\to\;\;
 q t^{-1} \;
\begin{tikzpicture}[anchorbase, scale=.5]
\draw [very thick] (2,0) to[out=180,in=315] (1.3,.5);
\draw [very thick] (2,1) to[out=180,in=45] (1.3,.5);
\draw [double] (1.3,.5) -- (.7,.5);
\draw [very thick, ->] (.7,.5) to[out=135,in=0]  (0,1);
\draw [very thick, ->] (.7,.5) to[out=225,in=0] (0,0);
\end{tikzpicture}
\quad,\quad
\Kh\left( \;
\begin{tikzpicture}[anchorbase, scale=.5]
\draw [very thick, ->] (2,0) to [out=180,in=0] (0,1);
 \draw [white,line width=.15cm] (2,1) to [out=180,in=0] (0,0) ;
\draw [very thick, ->] (2,1) to [out=180,in=0] (0,0);
\end{tikzpicture}
\;\right)
\;\;=\;\;
q^{-1} t  \;
\begin{tikzpicture}[anchorbase, scale=.5]
\draw [very thick] (2,0) to[out=180,in=315] (1.3,.5);
\draw [very thick] (2,1) to[out=180,in=45] (1.3,.5);
\draw [double] (1.3,.5) -- (.7,.5);
\draw [very thick, ->] (.7,.5) to[out=135,in=0]  (0,1);
\draw [very thick, ->] (.7,.5) to[out=225,in=0] (0,0);
\end{tikzpicture}
\;\;\to\;\;
\begin{tikzpicture}[anchorbase, scale=.5]
\draw [very thick, ->] (2,1) to (0,1);
\draw [very thick, ->] (2,0) to (0,0);
\end{tikzpicture}\nonumber
\\
\label{eq:thickcrossing-2}
\Kh\left(\begin{tikzpicture}[anchorbase, scale=.5]
\draw [very thick, ->] (2,1) to [out=180,in=0] (0,0);
 \draw [white,line width=.15cm] (2,0) to [out=180,in=0] (0,1) ;
\draw [double, ->] (2,0) to [out=180,in=0] (0,1);
\end{tikzpicture}
\right)
 = 
q t^{-1} \;
\begin{tikzpicture}[anchorbase, scale=.5]
\draw [double] (2,0) -- (1.4,0);
\draw [very thick, ->] (1.4,0) -- (0,0);
\draw [very thick] (2,1) -- (.6,1);
\draw [double, ->] (0.6,1) -- (0,1);
\draw [very thick] (.6,1) -- (1.4,0);
\end{tikzpicture} 
\;\;,\;\;
\Kh\left(\begin{tikzpicture}[anchorbase, scale=.5]
\draw [double, ->] (2,1) to [out=180,in=0] (0,0);
 \draw [white,line width=.15cm] (2,0) to [out=180,in=0] (0,1) ;
\draw [very thick, ->] (2,0) to [out=180,in=0] (0,1);
\end{tikzpicture}
\right)
 = 
q t^{-1} \;
\begin{tikzpicture}[anchorbase, scale=.5]
\draw [double] (2,1) -- (1.4,1);
\draw [very thick, ->] (1.4,1) -- (0,1);
\draw [very thick] (2,0) -- (.6,0);
\draw [double, ->] (0.6,0) -- (0,0);
\draw [very thick] (.6,0) -- (1.4,1);
\end{tikzpicture} 
\;\;,\;\;
\Kh\left(\begin{tikzpicture}[anchorbase, scale=.5]
\draw [double, ->] (2,1) to [out=180,in=0] (0,0);
 \draw [white,line width=.15cm] (2,0) to [out=180,in=0] (0,1) ;
\draw [double, ->] (2,0) to [out=180,in=0] (0,1);
\end{tikzpicture}
\right)
 = q^{2} t^{-2} \;
\begin{tikzpicture}[anchorbase, scale=.5]
\draw [double, ->] (2,0) -- (0,0);
\draw [double, ->] (2,1) -- (0,1);
\end{tikzpicture}
\\ \nonumber
\;\;\;\;\;\;\;\Kh\left( 
\begin{tikzpicture}[anchorbase, scale=.5]
\draw [very thick, ->] (2,0) to [out=180,in=0] (0,1);
 \draw [white,line width=.15cm] (2,1) to [out=180,in=0] (0,0) ;
\draw [double, ->] (2,1) to [out=180,in=0] (0,0);
\end{tikzpicture}
\right)
=
q^{-1} t  \;
\begin{tikzpicture}[anchorbase, scale=.5]
\draw [double] (2,1) -- (1.4,1);
\draw [very thick, ->] (1.4,1) -- (0,1);
\draw [very thick] (2,0) -- (.6,0);
\draw [double, ->] (0.6,0) -- (0,0);
\draw [very thick] (.6,0) -- (1.4,1);
\end{tikzpicture} 
\;\;,\;\;
\Kh\left(
\begin{tikzpicture}[anchorbase, scale=.5]
\draw [double, ->] (2,0) to [out=180,in=0] (0,1);
 \draw [white,line width=.15cm] (2,1) to [out=180,in=0] (0,0) ;
\draw [very thick, ->] (2,1) to [out=180,in=0] (0,0);
\end{tikzpicture}
\right)
 = 
q^{-1} t \;
\begin{tikzpicture}[anchorbase, scale=.5]
\draw [double] (2,0) -- (1.4,0);
\draw [very thick, ->] (1.4,0) -- (0,0);
\draw [very thick] (2,1) -- (.6,1);
\draw [double, ->] (0.6,1) -- (0,1);
\draw [very thick] (.6,1) -- (1.4,0);
\end{tikzpicture} 
\;\;,\;\;
\Kh\left( 
\begin{tikzpicture}[anchorbase, scale=.5]
\draw [double, ->] (2,0) to [out=180,in=0] (0,1);
 \draw [white,line width=.15cm] (2,1) to [out=180,in=0] (0,0) ;
\draw [double, ->] (2,1) to [out=180,in=0] (0,0);
\end{tikzpicture}
\right)
 = q^{-2} t^{2} \;
\begin{tikzpicture}[anchorbase, scale=.5]
\draw [double, ->] (2,0) -- (0,0);
\draw [double, ->] (2,1) -- (0,1);
\end{tikzpicture} 
\end{gather}
Above and throughout the rest of the paper, the $t$ variable will account for homological degree.
The non-zero differentials in the complexes associated to crossings of $1$-labeled strands are given by a single zip and unzip foam respectively. The other chain complexes all consist of a single object and only trivial differentials.

In order to define the functor on a link diagram $L$ in $\Su$, one first cuts out a little disk around each crossing to get a crossingless link diagram $L^\prime$ in the surface with disks removed. The functor then associates to $L$ a chain complex in $\Sfoam$, which is constructed as a formal tensor product of the chain complexes associated to all crossings. The tensor product is taken in the order specified by the ordering of crossings and it acts on webs by gluing them into $L^\prime$ and on foams by glueing them into $L^\prime\times [0,1]$. 

Recall that morphisms in $\Slinko$, i.e. link cobordisms in $\Su\times [0,1]^2$  between links embedded in $\Su\times [0,1]$ can generically be described by movies of link diagrams. 

\begin{lemma}\label{lem:mm} The movies associated to isotopic link cobordisms in $\Su\times [0,1]^2$ are related by a finite sequence of the Carter--Rieger--Saito movie moves of \cite[Section~7]{CSp} supported over disks in $\Su$.
\end{lemma}
\begin{proof} The proof in \cite{CSp} immediately extends to the case of $\Su\times [0,1]$. 
\end{proof}

All invariants that we consider take values in homotopy categories, i.e. they associate chain complexes to links and chain maps modulo homotopy to link cobordism. Isotopies between such link cobordisms imply equality of the corresponding morphisms and so we do not care about the isotopy classes of isotopies. As a consequence, the functoriality of the Blanchet--Khovanov construction for $\Su$ follows from the case $\Su=\R^2$, which was proved in \cite{Blan}.

\begin{theorem}\label{thm:functoriality} The Blanchet--Khovanov construction produces a well-defined functor 
\[\Su\Kh \colon \Slinko\to \HC(\Sfoam).\]
\end{theorem}
Via the equivalence between $\Slinko$ and $\Slink$, the latter can be taken as source category for $\Su\Kh$. The properties of $\Su\Kh$ listed in Proposition~\ref{prop:prop} follow directly from the construction and Theorem~\ref{thm:functoriality}. The claim that $\Su\Kh$ categorifies the evaluation of links in the skein module $\SWebq$ is a consequence of Theorem~\ref{thm:decat} and the fact that the Grothendieck group of the homotopy category of an additive category is naturally isomorphic to the split Grothendieck group of the additive category, see \cite{RoseNote}. This completes the proof of Theorem~\ref{thm:linkhom}.

More generally, we have:
\begin{theorem} The bigraded colored Khovanov--Rozansky link homologies \cite{KhR} extend to functorial invariants of links in thickened surfaces $\Su$, with target the homotopy category of the category of $\mathfrak{gl}_N$ foams in $\Su\times [0,1]$ as constructed in \cite{ETW}. 
\end{theorem}

The skein modules as well as the foam categories for $\glnn{N}$ with $N\geq 3$ are significantly more complex than their $\glnn{2}$ counterparts. Open questions include: 
\begin{itemize}
\item Does the category of $\glnn{N}$ foams in $\Su\times[0,1]$ admit a non-negative grading as in Corollary~\ref{cor:nonneggrading}?
\item Does it categorify the $\glnn{N}$ skein module of $\Su$ as in Theorem~\ref{thm:decat}?
\item Does its homotopy category admit an additional structure that categorifies the Morton--Samuelson commutator identities in the case of the torus \cite{MS}?
\end{itemize}

Khovanov homology and its higher rank cousins should also give categorifications of \textit{relative skein modules} of $3$-manifolds with boundary, in which one allows properly embedded tangles or webs with boundary. A first interesting case is the $3$-ball $B^3$ with a specified number of boundary points. Such relative skein modules can be categorified via Khovanov(--Rozansky) functors after fixing a projection of the ball onto a disk. However, the relevant target foam categories and the corresponding actions of diffeomorphisms of $B^3$ are only understood in very special cases, e.g. when there are at most four boundary points \cite{Wed1,Wed2} in which case the target categories are related to the motivic Donaldson--Thomas theory of Kontsevich--Soibelman~\cite{KoS}, see Sto\v{s}i\'{c}--Wedrich~\cite{SW}.

\subsection{Algebraic surface link homologies} \label{sec:alginv2}
In this section, we define an algebraic version of $\Su\Kh$, which has a direct decategorification relationship with the skein algebra $\SWebq$. Here we assume that $\Su\neq \T$ and we postpone the discussion of the case $\Su=\T$ to Section~\ref{sec:toric}.

\begin{definition} We define $\Su \Kh^\prime $ to be the composition of 
\begin{itemize}
\item the Khovanov functor $\Su\Kh \colon \Slinko \to \HC(\Sfoam)$,
\item  the projection to $\HC(\Sfoam_0)$,
\item  the natural functor $\HC(\Sfoam_0) \to \HC(\Kar(\Sfoam_0))$,
\item the functor induced by the representable functor $\bigoplus_{F \in {\Su}\basisJW^F}\bigoplus_{k\in \Z} \Hom_{\Kar(\Sfoam_0)}(F,q^k -)$, and
\item the functor of taking the homology of a $\Su\basisJW\times \Z$-graded chain complex.
\end{itemize}
\end{definition}
It is clear from the definition that $\Su\Kh^\prime$ is a link homology functor. From Theorem~\ref{thm:decat} it follows that for any $W\in \Su\basisJW$, the graded Euler characteristic $\chi_q(\Su\Kh^\prime_{W,*,*}(L))$ agrees with the coefficient of $W$ in the evaluation of the link $L$ in $\SWebq$. Furthermore, recall that the elements of $\Su\basisJW$ are determined by their underlying integer lamination and first homology class. We can, thus, interpret the target as the category of vector spaces graded by $\{\textrm{integer laminations on } \Su\}\times H_1(\Su)\times \Z\times \Z$ as stated in Theorem~\ref{thm:alglinkhomology}.

\subsection{Functoriality under foams} A beautiful aspect of the original construction of Khovanov homology is that it starts with a $1+1$-dimensional TQFT,  i.e. a functorial invariant for un-embedded or planar links, and then uses homological algebra to encode the embedding or crossing information described by the link diagram. The construction via Blanchet foams, that we use in this article, has the advantage that it produces an invariant that is properly functorial under link cobordisms, but it has the aesthetic disadvantage that it utilizes an intermediate category of webs and foams, whose set of objects is strictly larger than the set of planar links. It is then a natural questions whether Khovanov homology can be extended from the category of links and link cobordisms to a larger category of tangled webs and foams between them. An analogous question was raised by Khovanov--Rozansky at the end of \cite{KhR}. A first approach to resolve it in the case of Khovanov homology due to Clark--Morrison--Walker \cite{CMW} led to the introduction of the well-named concept of confusions, but no conclusion. In the construction using Blanchet foams, however, extending the Khovanov functor appears to be more natural. We will now explain the precise framework for this extension and conjecture that it is well-defined. In Section~\ref{sec:superposition}, we will assume that this conjecture holds and explore its consequences for categorified skein modules.

\begin{definition}
The category $\Stanwebo$ of tangled webs is the category with:
\begin{itemize}
\item objects, embeddings of framed webs in $\Su\times [0,1]$, with a fixed cyclic order of the edges around each vertex, with generic projection onto $\Su$ and an ordering of crossings,
\item morphisms, framed foams embedded in $\Su\times [0,1]\times [0,1]$, with boundary contained in $\Su\times [0,1]\times \{0,1\}$, with orientations on seams and a fixed cyclic order of the facets around each seam, modulo isotopy relative to the boundary. Additionally, there are identity foams that reorder crossings.
\end{itemize}
We will also consider the category $\Stanweb$, in which the genericity assumption is dropped. The framing conditions are explained in the following.
\end{definition}

Any web drawn on a surface admits a thickening to a ribbon graph, in which edges are replaced by bands, which join disks around the vertices. A framing on a web embedded in $\Su\times [0,1]$ is determined by an embedding of an orientable thickening of the web, up to isotopy. 

Similarly, any foam embedded in a 3-manifold can be thickened and a framing of a foam embedded in $\Su\times [0,1]\times [0,1]$ is determined by an embedding of such a thickening, up to isotopy. The boundary of a framed foam is a framed web, with cyclic order of edges around each vertex determined by the corresponding cyclic order of edges around a seam in the foam. 

When illustrating foams $F$ by projection onto $\Su\times \{1/2\} \times [0,1]$, we use the convention that the cyclic ordering of facets around each seam is determined by the right-hand rule from the orientation of the seam. Additionally, we require that the framing of $F$ induces the blackboard framing (parallel to $\Su$) on each of the illustrated time slices $F\cap \Su \times \{1/2\} \times \{t\}$.

The objects in the category $\Stanwebo$ can be visualized as tangled web diagrams by the generic projection onto $\Su\times \{1/2\}$. The morphisms then admit a description as movies of such diagrams, whose frames differ only by handle attachments, Reidemeister moves (featuring a framed version of the usual Reidemeister I move) or the new fork slide moves. As usual, this description is not faithful, and there exist additional movie moves, which relate movies that describe isotopic foams, see \cite{CS,Carter_foams}.

The Blanchet--Khovanov construction immediately extends to diagrams of tangled webs, and the invariants of diagrams that are related by fork slide moves are homotopy equivalent.

\begin{lemma} \label{lem:forksliding2} The following fork slide isomorphism holds in $\Sfoam$ and $\HC(\Sfoam)$ respectively:
\begin{gather*}
\Kh\left(\;
\begin{tikzpicture}[anchorbase, scale=.3]
\draw [very thick] (1,0) to [out=90,in=225] (.5,3);
\draw [very thick] (2,0) to [out=90,in=315] (.5,3);
\draw [double, ->] (.5,3) -- (.5,4);
\draw [white,line width=.15cm] (0,0) to [out=90,in=270] (2,4);
\draw [double, ->] (0,0) to [out=90,in=270] (2,4);
\end{tikzpicture}
\;\right)
\quad \cong \quad
\Kh\left(\;
\begin{tikzpicture}[anchorbase, scale=.3]
\draw [double,->]  (1.5,1) to [out=90,in=270] (.5,4);
\draw [very thick] (2,0) to[out=90,in=315] (1.5,1);
\draw [very thick] (1,0) to[out=90,in=225] (1.5,1);
\draw [white,line width=.15cm] (0,0) to [out=90,in=270] (2,4);
\draw [double, ->] (0,0) to [out=90,in=270] (2,4);
\end{tikzpicture}
\;\right)
\quad,\quad 
\Kh\left( \;
\begin{tikzpicture}[anchorbase, scale=.3]
\draw [very thick] (1,0) to [out=90,in=225] (.5,3);
\draw [very thick] (2,0) to [out=90,in=315] (.5,3);
\draw [double, ->] (.5,3) -- (.5,4);
\draw [white,line width=.15cm] (0,0) to [out=90,in=270] (2,4);
\draw [very thick, ->] (0,0) to [out=90,in=270] (2,4);
\end{tikzpicture}
\; \right)
\quad \cong \quad 
\Kh\left( \;
\begin{tikzpicture}[anchorbase, scale=.3]
\draw [double,->]  (1.5,1) to [out=90,in=270] (.5,4);
\draw [very thick] (2,0) to[out=90,in=315] (1.5,1);
\draw [very thick] (1,0) to[out=90,in=225] (1.5,1);
\draw [white,line width=.15cm] (0,0) to [out=90,in=270] (2,4);
\draw [very thick, ->] (0,0) to [out=90,in=270] (2,4);
\end{tikzpicture}
\; \right)
\end{gather*}
together with their variants obtained from changing positive into negative crossing or changing merge into split vertices. 
\end{lemma}

The reason for the importance of a framing structure on foams comes from the functoriality question: whether it is possible to make the Khovanov construction into a functor from $\Stanwebo$ to the homotopy category of $\Sfoam$. If we disregard framings, then the following two movies of tangled web diagrams represent isotopic foams. This is Carter's twist zipper move:

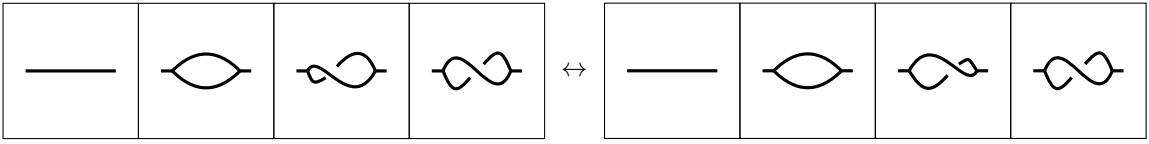
\[
\begin{tikzpicture}[anchorbase]
  \node at (3,-2) {
\begin{tikzpicture}[anchorbase,scale=.3]
\draw (0,0) rectangle (6,6);
\draw [very thick] (1,3) -- (1.5,3);
\draw [very thick] (1.5,3) .. controls (2.3,5) and (3.7,1) .. (4.5,3);
\draw [very thick] (1.5,3) .. controls (1.9,2) and (2.1,2) ..  (2.7,2.7);
\draw [very thick] (3.3,3.3) .. controls (3.9,4) and (4.1,4) .. (4.5,3);
\draw [very thick] (4.5,3) -- (5,3);
\end{tikzpicture}
};
\node at (1.2,-2) {
\begin{tikzpicture}[anchorbase,scale=.3]
\draw (0,0) rectangle (6,6);
\draw [very thick] (1,3) -- (1.5,3);
\draw [very thick] (1.5,3) .. controls (2,4) and (3.4,1) .. (4.5,3);
\draw [very thick] (1.5,3) .. controls (1.7,2.4) and (1.8,2.4) ..  (2.3,2.7);
\draw [very thick] (2.8,3.2) .. controls (3.5,4) and (4.1,4) .. (4.5,3);
\draw [very thick] (4.5,3) -- (5,3);
\end{tikzpicture}
};
\node at (-.6,-2) {
\begin{tikzpicture}[anchorbase,scale=.3]
\draw (0,0) rectangle (6,6);
\draw [very thick] (1,3) -- (1.5,3);
\draw [very thick] (1.5,3) .. controls (2.5,4) and (3.5,4) .. (4.5,3);
\draw [very thick] (1.5,3) .. controls (2.5,2) and (3.5,2) ..  (4.5,3);
\draw [very thick] (4.5,3) -- (5,3);
\end{tikzpicture}
};
\node at (-2.4,-2) {
\begin{tikzpicture}[anchorbase,scale=.3]
\draw (0,0) rectangle (6,6);
\draw [very thick] (1,3) -- (5,3);
\end{tikzpicture}
};
\end{tikzpicture}
\leftrightarrow
\begin{tikzpicture}[anchorbase]
\node at (3,0) {
\begin{tikzpicture}[anchorbase,scale=.3]
\draw (0,0) rectangle (6,6);
\draw [very thick] (1,3) -- (1.5,3);
\draw [very thick] (1.5,3) .. controls (2.3,5) and (3.7,1) .. (4.5,3);
\draw [very thick] (1.5,3) .. controls (1.9,2) and (2.1,2) ..  (2.7,2.7);
\draw [very thick] (3.3,3.3) .. controls (3.9,4) and (4.1,4) .. (4.5,3);
\draw [very thick] (4.5,3) -- (5,3);
\end{tikzpicture}
};
\node at (1.2,0) {
\begin{tikzpicture}[anchorbase, scale=.3]
\draw (0,0) rectangle (6,6);
\draw [very thick] (1,3) -- (1.5,3);
\draw [very thick] (1.5,3) .. controls (2.6,5) and (4,2) .. (4.5,3);
\draw [very thick] (1.5,3) .. controls (2.1,2) and (2.5,2) .. (3.2,2.8);
\draw [very thick] (3.7,3.3) .. controls (4.2,3.6) and (4.1,3.6) .. (4.5,3);
\draw [very thick] (4.5,3) -- (5,3);
\end{tikzpicture}
};
\node at (-.6,0) {
\begin{tikzpicture}[anchorbase, scale=.3]
\draw (0,0) rectangle (6,6);
\draw [very thick] (1,3) -- (1.5,3);
\draw [very thick] (1.5,3) .. controls (2.5,4) and (3.5,4) .. (4.5,3);
\draw [very thick] (1.5,3) .. controls (2.5,2) and (3.5,2) ..  (4.5,3);
\draw [very thick] (4.5,3) -- (5,3);
\end{tikzpicture}
};
\node at (-2.4,0) {
\begin{tikzpicture}[anchorbase, scale=.3]
\draw (0,0) rectangle (6,6);
\draw [very thick] (1,3) -- (5,3);
\end{tikzpicture}
};
\end{tikzpicture}
\]
However, it seems impossible to define vertex twist maps between the Khovanov invariants of these tangled webs in a coherent way, such that the construction assigns homotopic chain maps to both sides of this move. For a more detailed discussion about this issue, we refer to \cite{Queff_PhD}. In the framed case, however, this move is disallowed, as are all other movies that involve the problematic vertex twist. 

In order to show that the conjectural Khovanov functor is indeed well-defined in the framed setup, one needs to verify that the functor respects a generating set for the movie moves that relate movie descriptions of foams that are isotopic in $\Stanwebo$. While good progress has been made in this direction (see in particular \cite{ETW}), a full proof would still require a significant amount of additional work. For the following section we will thus assume that the following conjecture holds.

\begin{conjecture}\label{conj:functoriality}The Blanchet--Khovanov construction produces a functor from $\Stanwebo$ to the homotopy category of $\Sfoam$.
\end{conjecture}

\begin{notation*} We will mark definitions and results that depend on Conjecture~\ref{conj:functoriality} by an asterisk as done here.
\end{notation*}

\section{The superposition product} \label{sec:superposition}
The purpose of this section is to explore potential consequences of Conjecture~\ref{conj:functoriality}. The main application we have in mind is the construction of a bifunctor $\Sfoam\times \Sfoam \to \HC(\Sfoam)$, which is a first step towards a categorification of the skein algebra multiplication on $\SWebq$. Another consequence, which only requires a weaker version of Conjecture~\ref{conj:functoriality}, concerns auto-equivalences of the foam category $\Sfoam$ induced by superposing with $2$-labeled multi-curves. This also allows for another algebraic version of the surface link homology functor $\Su\Kh$, which is closer in spirit to the Asaeda--Przytycki--Sikora invariants and admits spectral sequences associated to surface embeddings, see Sections~\ref{sec:alginv1} and \ref{sec:ss}. 

We start by considering the superposition operation on $\Stanweb$ (and thus also its subcategory $\Slink$) that assigns to two tangled webs in $\Su\times [0,1]$ and $\Su\times [1,2]$ their union through the map induced by division by $2$ on $[0,2]\rightarrow [0,1]$. More generally, this produces a bifunctor $\star\colon\Stanweb\times \Stanweb\to \Stanweb$.

\begin{lemma}\label{lem:assoc} The superposition operation $\star$ extends to a monoidal structure on $\Stanweb$. 
\end{lemma} 
\begin{proof} For tangled webs $L$, $M$, $N$, the products $L\star (M\star N)$ and $(L\star M)\star N$ are isomorphic via a \textit{vertical} isotopy foam, which plays the role of an associator $a_{L,M,N}$.  Moreover, if $C_L\colon L\to L^\prime$, $C_M\colon M\to M^\prime$ and $C_N\colon N\to N^\prime$ are foams, then we have:
\[ C_L\star (C_M\star C_N)= a_{L^\prime,M^\prime,N^\prime}^{-1} \circ (C_L\star C_M)\star C_N\circ a_{L,M,N}.  \] This shows that the associators are natural in the three arguments. The monoidal unit is given by the empty web and the left and right unitors are again given by suitable vertical isotopy foams, which are also natural. The coherence conditions between the associators and unitors are easily checked.
\end{proof}

After choosing an essential inverse to the inclusion $\Stanwebo\hookrightarrow \Stanweb$, the tensor product $\star$ restricts to $\Stanwebo$. Note that such a choice is necessary since the superposition of generic web embeddings no longer needs to be generic (though generically it is and we will sometimes implicitly make this assumption). 

When composed with the Khovanov functor, the superposition product gives a bi-functor $\Stanwebo\times \Stanwebo\to \HC(\Sfoam)$.  Note that when the associators and unitors are vertical isotopies, e.g. in the generic case, their images under the Khovanov functor are identity morphisms. In the following we study to which extent the superposition operation $\star$ can be made to intertwine with the Khovanov functor. A key ingredient is the following lemma. 

\begin{lemma*}
\label{lem:factor4d} There exists a functor $\iota$ that makes the following diagram commutative:
\begin{equation}
\label{eq:sqdiag}
\begin{tikzpicture}
\draw (0,0) node {$\Sfoam$};
\draw (0,1) node {$\Sfoam_{free}$};
\draw (4,0) node {$\HC(\Sfoam)$};
\draw (4,1) node {$\Q\Stanweb$};
\draw [->] (1,0)--(3,0);
\draw [->] (0,.75)--(0,.25);
\draw [->] (1,1)--(2.9,1);
\draw [->] (4,.75)--(4,.25);
\draw (4.5,.5) node {$\Su\Kh$};
\draw (2,1.2) node {$\iota$};
\end{tikzpicture}
\end{equation}
The bottom horizontal arrow denotes the natural map from $\Sfoam$ into its homotopy category, $\Sfoam_{free}$ denotes a free version of $\Sfoam$ described in the proof, the left vertical arrow is the corresponding quotient functor, $\Q\Stanweb$ is the $\Q$-enrichment of $\Stanweb$ and $\Su\Kh$ denotes the $\Q$-linearized Khovanov functor (Conjecture~\ref{conj:functoriality}).
\end{lemma*}
\begin{proof} The category $\Sfoam_{free}$ is a version of the foam category in which no local relations are imposed, isotopies of foams are disallowed and dots are translated into small handle attachments multiplied by the scalar $1/2$. We define $\iota$ as the functor from $\Sfoam_{free}$ to the $\Q$-enriched tangled web category $\Q\Stanweb$ which embeds webs in $\Su\times \{1/2\}$ and toric foams in $\Su \times \{1/2\}\times [0,1]$. Note that $\iota$ takes the \textit{thickening direction} of $\Sfoam_{free}$ to the \textit{time direction} in $\Q\Stanweb$. From the construction of $\Su\Kh$, it follows that $\Su\Kh\circ \iota$ sends webs and foams in $\Sfoam_{free}$ to themselves in $\HC(\Sfoam)$ when considered as complexes concentrated in homological degree zero and chain maps thereof. It is then clear that $\Su\Kh \circ \iota$ actually factors through the quotient $\Sfoam$ of $\Sfoam_{free}$ since all isotopy and local relations of $\Sfoam$ also hold in the target category.
\end{proof}

\begin{proposition*}\label{prop:bifunctor} The superposition product $\star$ on $\Stanweb$ induces a well-defined biadditive, bilinear bifunctor $*$ from $\Sfoam\times \Sfoam$ to $\HC(\Sfoam)$.
\end{proposition*}

\begin{proof} 
In an adaption of the proof of Lemma~\ref{lem:factor4d}, we consider the following diagram:
\begin{center}
\begin{tikzpicture}
\draw (0,0) node {$\Sfoam^{\times 2}$};
\draw (0,1) node {$\Sfoam_{free}^{\times 2}$}; 
\draw (4,1) node {$\Q\Stanweb^{\times 2}$};
\draw (8,1) node {$\Q\Stanweb$};
\draw (8,0) node {$\HC(\Sfoam)$};

\draw [->, dashed] (1,0)--(7,0);
\draw [->] (1,.75)--(7,.25);
\draw [->] (0,.75)--(0,.25);
\draw [->] (1,1)--(2.7,1);
\draw [->] (5.3,1)--(6.9,1);
\draw [->] (8,.75)--(8,.25);
\draw (8.3,.5) node {$\Su\Kh$};
\draw (2,1.2) node {$\iota^{\times 2}$};
\draw (6,1.2) node {$\star$};
\draw (6,0.6) node {$*$};
\end{tikzpicture}
\end{center}
We consider the composite bifunctor  $-*-:=\Su\Kh (\iota(-)\star\iota(-))\colon \Sfoam_{free}^{\times 2} \to \HC(\Sfoam)$. We would like to show that this functor factors through $\Sfoam^{\times 2}$, i.e. that it respects isotopy of foams and local foam relations in each argument. By symmetry, we only consider the first argument and fix a web $W$ and its identity foam $\id_W$ in the second argument. Let $F$ and $F^\prime$ denote two foams in $\Sfoam_{free}$ which are identified via an isotopy in $\Sfoam$, then $\iota(F)\star\iota(\id_W)$ and $\iota(F^\prime)\star\iota(\id_W)$ are also identified in $\Q\Stanweb$ and thus have equal image under $\Su\Kh$ by Conjecture~\ref{conj:functoriality}. Next, let $F$ and $F^\prime$ denote two (linear combinations of) foams in $\Sfoam_{free}$ which are identified via a local foam relation in $\Sfoam$. By invariance under isotopy, we may assume that the local foam relation is applied in a region whose projection onto $\Su$ is disjoint from the web $W$. Since the same foam relation holds on morphisms in $\HC(\Sfoam)$, the two foams produce equal images under the composite bifunctor $*$. We thus conclude that $*$ factors through $\Sfoam^{\times 2}$.
\end{proof}

\begin{remark}
\label{rem:spec}
As discussed in Section~\ref{sec:skeinalgcat}, it is doubtful whether the superposition bifunctor $*$ directly extends to a tensor product on $\HC(\Sfoam)$ that categorifies the skein algebra multiplication. A better candidate category for this would be $\HC(\Kar(\Sfoam_0))$ or the equivalent dg category $\Kar(\Sfoam_0)_{\mathrm{dg}}$ of minimal complexes, which actually have zero differentials for $\Su\neq \T$. These are known to be suitable choices for $\Su=\R^2,\A$ and in Section~\ref{sec:toric} we give supporting evidence in the case $\Su=\T$ for the suitability of a slightly modified target category.
\end{remark}

We now revisit Example~\ref{exa:1001} in the categorified setup to illustrate the superposition bifunctor.

\begin{example}
\label{exa:1001-2} The superposition of the toric $(0,1)$ and $(1,0)$ curves evaluates under the Khovanov functor to the following chain complex in $\Tfoam$.
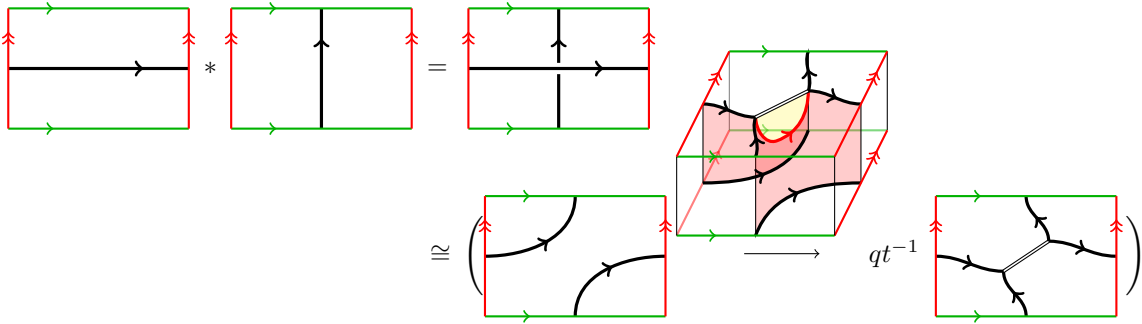
\begin{equation*}
\xy
(0,25)*{
\begin{tikzpicture}[anchorbase]
\draw[very thick, directed=.75] (0,.8) to (2.4,.8);
\torus{2.4}{1.6}
\end{tikzpicture}
\;*\;
\begin{tikzpicture}[anchorbase]
\draw[very thick, directed=.75] (1.2,0) to (1.2,1.6);
\torus{2.4}{1.6}
\end{tikzpicture}
\;=\;
\begin{tikzpicture}[anchorbase]
\draw[very thick, directed=.75] (1.2,0) to (1.2,1.6);
\draw[white, line width=.15cm] (2.4,.8) to (0,.8);
\draw[very thick, directed=.75] (0,.8) to (2.4,.8);
\torus{2.4}{1.6}
\end{tikzpicture}
};
(90,0)*{
 q t^{-1} \;\;
\begin{tikzpicture}[anchorbase]
\draw[very thick, directed=.55] (1.2,0) to [out=90,in=270] (0.9,0.6);
\draw[double] (0.9,0.6) to (1.5,1);
\draw[very thick, directed=.55] (1.5,1) to [out=0,in=180] (2.4,.8);
\draw[very thick, directed=.55] (0,.8) to [out=0,in=180] (.9,0.6);
\draw[very thick, directed=.55] (1.5,1) to [out=90,in=270] (1.2,1.6);
\torus{2.4}{1.6}
\end{tikzpicture}
\;\Bigg)
};
(60,0)*{
\begin{tikzpicture}[anchorbase]
\draw[->] (0,0) to (1,0);
\end{tikzpicture}
};
(29,0)*{\cong\,
\Bigg(
\begin{tikzpicture}[anchorbase]
\draw[very thick, directed=.55] (1.2,0) to [out=90,in=180] (2.4,.8);
\draw[very thick, directed=.55] (0,.8) to [out=0,in=270] (1.2,1.6);
\torus{2.4}{1.6}
\end{tikzpicture}
};
(60,15)*{
\begin{tikzpicture}[fill opacity=.2,anchorbase,xscale=.7, yscale=0.7]
\torusback{3}{2}{1.5}
\fill [fill=yellow] (1.5,2.25) to [out=270,in=200] (1.9,1.8) to [out=20,in=270] (2.5,2.75) to (1.5,2.25) ; 
\draw [fill=red] (2.5,2.75) to [out=0,in=180] (3.5,2.5) to (3.5,1)to [out=180,in=70](1.5,0) to (1.5,1.5) to [out=70,in=250] (1.5,2.25) to [out=270,in=200] (1.9,1.8) to [out=20,in=270] (2.5,2.75);
\draw [fill=red] (2.5,2.75) to [out=70,in=250](2.5,3.5) to (2.5,2) to [out=250,in=0] (.5,1) to  (.5,2.5) to [out=0,in=180] (1.5,2.25) to [out=270,in=200] (1.9,1.8) to [out=20,in=270] (2.5,2.75) ; 
\draw [very thick, red, directed=.55] (1.5,2.25) to [out=270,in=200] (1.9,1.8) to [out=20,in=270] (2.5,2.75);
\draw[very thick, directed=.5] (.5,1) to [out=0,in=250](2.5,2);
\draw[very thick, directed=.5] (1.5,0) to [out=70,in=180](3.5,1);
\draw[very thick, directed=.5] (.5,2.5) to [out=0,in=180] (1.5,2.25);
\draw[very thick, directed=.5] (1.5,1.5) to [out=70,in=250] (1.5,2.25);
\draw[double] (1.5,2.25) to (2.5,2.75);
\draw[very thick, directed=.55] (2.5,2.75) to [out=0,in=180](3.5,2.5);
\draw[very thick, directed=.55] (2.5,2.75) to [out=70,in=250](2.5,3.5);
\torusfront{3}{2}{1.5}
\end{tikzpicture}
};
\endxy
\end{equation*}
The second term in this complex is a web whose underlying curve is the $(1,-1)$ curve. More precisely, note that after an isotopy we see:
\begin{equation}
\label{eq:2mult} q t^{-1} \;\;
\begin{tikzpicture}[anchorbase]
\draw[very thick, directed=.55] (1.2,0) to [out=90,in=270] (0.9,0.6);
\draw[double] (0.9,0.6) to (1.5,1);
\draw[very thick, directed=.55] (1.5,1) to [out=0,in=180] (2.4,.8);
\draw[very thick, directed=.55] (0,.8) to [out=0,in=180] (.9,0.6);
\draw[very thick, directed=.55] (1.5,1) to [out=90,in=270] (1.2,1.6);
\torus{2.4}{1.6}
\end{tikzpicture}
\;\;
\cong 
\;q t^{-1} \;\;
\begin{tikzpicture}[anchorbase]
\draw[very thick, directed=.55] (1,0.6) to [out=90,in=270] (.4,1);
\draw[double] (.4,1) to [out=60,in=270] (.7,1.6);
\draw[double] (.7,0) to [out=90,in=240] (1,.6);
\draw[very thick, directed=.55] (1.2,1.6) to [out=270,in=180] (2.4,.8);
\draw[very thick, directed=.55] (0,.8) to [out=0,in=180] (.4,1);
\draw[very thick, directed=.55] (1,.6) to [out=0,in=90] (1.2,0);
\torus{2.4}{1.6}
\end{tikzpicture}
\;\;=\;\;
 \begin{tikzpicture}[anchorbase]
\draw[very thick, directed=.55] (1.2,1.6) to [out=270,in=180] (2.4,.8);
\draw[very thick, directed=.55] (0,.8) to [out=0,in=90] (1.2,0);
\torus{2.4}{1.6}
\end{tikzpicture}
\; *\;
\begin{tikzpicture}[anchorbase]
\draw[double, directed=.55] (1,0) to (1,1.6);
\torus{2.4}{1.6}
\end{tikzpicture}
\end{equation}
So the Khovanov functor sends the product $(1,0)\star (0,1)$ to a chain complex built out of $(1,1)$ and $(1,-1)*\w{0,1}$. Moreover, in the degree zero truncation $\Tfoam_0$, the differential is set to zero and we obtain a decomposition $(1,0)\star (0,1)\cong (1,1) \oplus (1,-1)*\w{0,1}$.
\end{example} 

We call superposition with $2$-labeled curves $\stwo$-operations. In the next section, we will see that $\stwo$-operations give auto-equivalences of the foam category $\Sfoam$ that intertwine with the Khovanov functor, assuming Conjecture~\ref{conj:functoriality}.

\subsection{Superposition with 2-labeled webs}
\label{sec:stwo}
Recall that we denote by $\w{x}$ for $x\in H_1(\Su)$ a choice of $2$-labeled multicurve with $[\w{x}]=2x \in H_1(\Su)$. For the following, we consider a homologically graded version of the foam category, which we formally define as the dg category $\Sfoam_{\mathrm{dg}}:=\bigoplus_{k\in \Z} t^k \Sfoam$ with trivial differential. Our goal is to prove the following result, assuming Conjecture~\ref{conj:functoriality}.

\begin{theorem*}\label{thm:autoequ}  The $\stwo$-operation $-*\w{x}$ is a well-defined auto-equivalence on $\Sfoam_{\mathrm{dg}}$ of $H_1(\Su)$-degree $2x$, which is of $q$-degree $c\cdot x$ and homological degree $- c \cdot x$ on the block of $\Sfoam_{\mathrm{dg}}$ indexed by $c\in H_1(\Su)$.
\end{theorem*}

Here $*$ denotes the bifunctor from $\Sfoam^{\times 2}$ to $\HC(\Sfoam)$ from Proposition~\ref{prop:bifunctor}. When contracted with the $2$-labeled multicurve $\w{x}$ this at first sight gives a functor from $\Sfoam$ to $\HC(\Sfoam)$. However, if $W$ is a web with $[W]=c\in H_1(\Su)$, then $W*\w{x}$ is a chain complex concentrated in homological degree $-c \cdot x$, which we may again consider as web, although homologically shifted. This means $-*\w{x}$ naturally maps to the target $\Sfoam_{\mathrm{dg}}$ and it then trivially extends to an endofunctor of $\Sfoam_{\mathrm{dg}}$.

Before giving a proof of Theorem~\ref{thm:autoequ} modulo Conjecture~\ref{conj:functoriality}, we illustrate how $\stwo$-operations behave on morphisms by computing another example of superposition on the torus.

\begin{example}\label{exa:2101-2}  The superposition $(2,1)*(0,1)$ has the following expansion as a chain complex in $\Tfoam$.
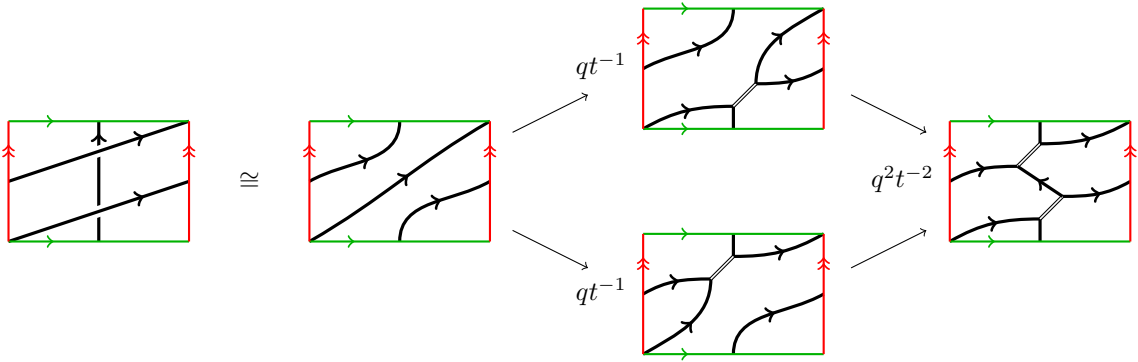
\begin{equation*}
\xy
(-10,0)*{
\begin{tikzpicture}[anchorbase]
\draw[very thick, directed=.9] (1.2,0) to (1.2,1.6);
\draw[white, line width=.15cm] (0,0) to (2.4,.8);
\draw[white, line width=.15cm] (0,.8)  to (2.4,1.6);
\draw[very thick, directed=.75] (0,0) to (2.4,.8);
\draw[very thick, directed=.75] (0,.8)  to (2.4,1.6);
\torus{2.4}{1.6}
\end{tikzpicture}
};
(10,0)*{\cong
};
(110,0)*{
 q^2 t^{-2} \;\;
\begin{tikzpicture}[anchorbase]
\draw[very thick, directed=.55] (0,0) [out=30,in=180] to (1.2,.3);
\draw[very thick, directed=.55] (0,.8)to [out=30,in=180] (.9,1);
\draw[very thick] (1.2,1.3)  to [out=90,in=270] (1.2,1.6);
\draw[very thick] (1.2,0) [out=90,in=270] to (1.2,.3);
\draw[very thick, directed=.55] (1.5,.6) [out=0,in=210] to (2.4,.8);
\draw[double] (1.2,.3) to (1.5,.6) ;
\draw[double]  (.9,1) to (1.2,1.3) ;
\draw[very thick, directed=.55] (1.5,.6) to (.9,1);
\draw[very thick, directed=.55] (1.2,1.3) [out=0,in=210] to (2.4,1.6);
\torus{2.4}{1.6}
\end{tikzpicture}
};
(70,15)*{
 q t^{-1} \;\;
\begin{tikzpicture}[anchorbase]
\draw[very thick, directed=.55] (0,0) [out=30,in=180] to (1.2,.3);
\draw[very thick, directed=.55] (0,.8) [out=30,in=270] to (1.2,1.6);
\draw[very thick] (1.2,0) [out=90,in=270] to (1.2,.3);
\draw[very thick, directed=.55] (1.5,.6) [out=0,in=210] to (2.4,.8);
\draw[double] (1.2,.3) to (1.5,.6) ;
\draw[very thick, directed=.55] (1.5,.6) to [out=90,in=210](2.4,1.6);
\torus{2.4}{1.6}
\end{tikzpicture}
};
(70,-15)*{
 q t^{-1} \;\;
\begin{tikzpicture}[anchorbase]
\draw[very thick, directed=.55] (0,0) [out=30,in=270] to (.9,1);
\draw[very thick, directed=.55] (0,.8)to [out=30,in=180] (.9,1);
\draw[very thick] (1.2,1.3)  to [out=90,in=270] (1.2,1.6);
\draw[very thick, directed=.55] (1.2,0) [out=90,in=210] to (2.4,.8);
\draw[very thick, directed=.55] (1.2,1.3) [out=0,in=210] to (2.4,1.6);
\draw[double]  (.9,1) to (1.2,1.3) ;
\torus{2.4}{1.6}
\end{tikzpicture}
};
(30,0)*{
\begin{tikzpicture}[anchorbase]
\draw[very thick, directed=.55] (0,0) [out=30,in=210] to (2.4,1.6);
\draw[very thick, directed=.55] (0,.8) [out=30,in=270] to (1.2,1.6);
\draw[very thick, directed=.55] (1.2,0) [out=90,in=210] to (2.4,.8);
\torus{2.4}{1.6}
\end{tikzpicture}
};
(50,9)*{
\begin{tikzpicture}[anchorbase]
\draw[->] (0,0) to (1,.5);
\end{tikzpicture}
};
(50,-9)*{
\begin{tikzpicture}[anchorbase]
\draw[->] (0,0) to (1,-.5);
\end{tikzpicture}
};
(95,-9)*{
\begin{tikzpicture}[anchorbase]
\draw[->] (0,0) to (1,.5);
\end{tikzpicture}
};
(95,9)*{
\begin{tikzpicture}[anchorbase]
\draw[->] (0,0) to (1,-.5);
\end{tikzpicture}
};
\endxy
\end{equation*} Here the differentials are given by the usual signed zip foams. After collapsing the digons in the two objects in homological degree $-1$, we get the following isomorphic complex.
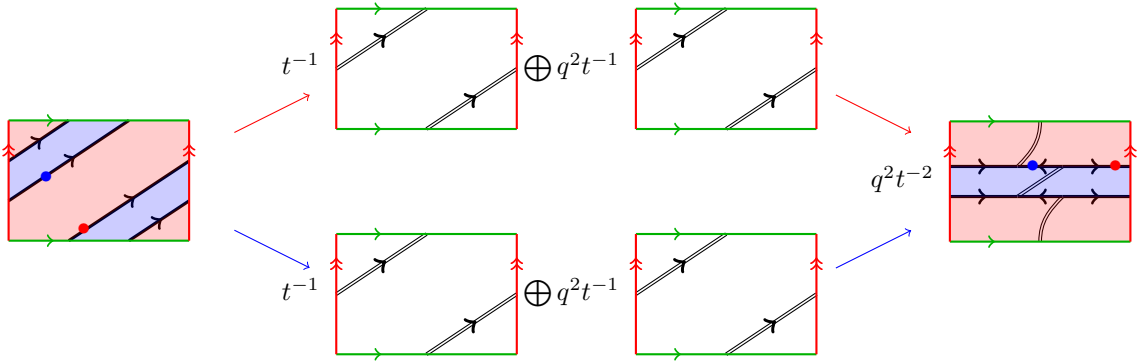
\begin{equation*}
\xy
(30,0)*{
\begin{tikzpicture}[anchorbase]
\draw[very thick, directed=.55] (0,0.53) to (1.6,1.6);
\draw[very thick, directed=.55] (0,1.06) to (.8,1.6);
\draw[very thick, directed=.55] (.8,0) to (2.4,1.06);
\draw[very thick, directed=.55] (1.6,0) to (2.4,0.53);
 \fill[red,opacity=.2] (0,0.53) -- (1.6,1.6) -- (2.4,1.6) -- (2.4,1.06) -- (.8,0) -- (0,0);
    \fill[red,opacity=.2] (0,1.06) -- (.8,1.6) -- (0,1.6);
    \fill[red,opacity=.2] (1.6,0) -- (2.4,0.53) -- (2.4,0);
      \fill[blue,opacity=.2] (0,0.53) -- (1.6,1.6) -- (.8,1.6) -- (0,1.06);
    \fill[blue,opacity=.2] (.8,0) -- (2.4,1.06) -- (2.4,0.53) -- (1.6,0);
    \node[red] at (1,.15) {$\bullet$};
    \node[blue] at (.5,.85) {$\bullet$};
\torus{2.4}{1.6}
\end{tikzpicture}
};
(90,15)*{
 t^{-1} \;\;
\begin{tikzpicture}[anchorbase]
\draw[double, directed=.55] (0,0.8) to (1.2,1.6);
\draw[double, directed=.55] (1.2,0) to (2.4,.8);
\torus{2.4}{1.6}
\end{tikzpicture}
\bigoplus
q^2 t^{-1} \;\;
\begin{tikzpicture}[anchorbase]
\draw[double, directed=.55] (0,0.8) to (1.2,1.6);
\draw[double, directed=.55] (1.2,0) to (2.4,.8);
\torus{2.4}{1.6}
\end{tikzpicture}
};
(90,-15)*{
 t^{-1} \;\;
\begin{tikzpicture}[anchorbase]
\draw[double, directed=.55] (0,0.8) to (1.2,1.6);
\draw[double, directed=.55] (1.2,0) to (2.4,.8);
\torus{2.4}{1.6}
\end{tikzpicture}
\bigoplus
q^2 t^{-1} \;\;
\begin{tikzpicture}[anchorbase]
\draw[double, directed=.55] (0,0.8) to (1.2,1.6);
\draw[double, directed=.55] (1.2,0) to (2.4,.8);
\torus{2.4}{1.6}
\end{tikzpicture}
};
(150,0)*{
 q^2 t^{-2} \;\;
\begin{tikzpicture}[anchorbase]
\draw[very thick, directed=.55] (0,1) to (.9,1);
\draw[very thick, directed=.55] (0,.6) to (.9,.6);
\draw[very thick, directed=.55] (1.5,1) to (2.4,1);
\draw[very thick, directed=.55] (1.5,.6) to (2.4,.6);
\draw[double] (.9,.6) to (1.5,1) ;
\draw[double]  (.9,1) to[out=45,in=270] (1.2,1.6) ;
\draw[double]  (1.2,0) to[out=90,in=225] (1.5,.6) ;
\draw[very thick, directed=.55] (1.5,1) to (.9,1);
\draw[very thick, directed=.55] (1.5,.6) to (.9,.6);
\fill[red,opacity=.2] (0,1) -- (2.4,1) -- (2.4,1.6) -- (0,1.6);
    \fill[red,opacity=.2] (0,.6) -- (2.4,.6) -- (2.4,0) -- (0,0);
    \fill[blue,opacity=.2] (0,.6) -- (2.4,.6) -- (2.4,1) -- (0,1);
    \node[blue] at (1.1,1) {$\bullet$};
    \node[red] at (2.2,1) {$\bullet$};
\torus{2.4}{1.6}
\end{tikzpicture}
};
(53,9)*{
\begin{tikzpicture}[anchorbase]
  \draw[red, ->] (0,0) to (1,.5);
\end{tikzpicture}
};
(53,-9)*{
\begin{tikzpicture}[anchorbase]
\draw[blue,->] (0,0) to (1,-.5);
  \end{tikzpicture}
  };
(133,-9)*{
\begin{tikzpicture}[anchorbase]
\draw[blue,->] (0,0) to (1,.5);
\end{tikzpicture}
};
(133,9)*{
\begin{tikzpicture}[anchorbase]
  \draw[red,->] (0,0) to (1,-.5);
\end{tikzpicture}
};
\endxy
\end{equation*}

The differentials on the left-hand side merge the parallel $1$-labeled strands, either directly (blue), or going around the torus (red). The degree shift by $2$ is accomplished by placing dots in the indicated locations. The differentials on the right-hand side are given by analogous (dotted) splitter foams, acted upon by $\w{0,1}$. In the degree zero truncation, the entire chain complex splits into the visible $q$-degree $0$ and $2$ parts and we will see in Section~\ref{sec:examples} that these parts are resolutions of the categorifications of the basis elements $(2,2)_T$ and $(2,0)_T*\w{0,1}$ of $\TWebq$.
\end{example}

Theorem~\ref{thm:autoequ} is implied by the following lemma. For this we let $\w{-x}$ denote a parallel copy of $\w{x}$ with the opposite orientation.

\begin{lemma*} The endofunctors $-*\w{x}$ and $-*\w{-x}$ are inverse auto-equivalences of $\Sfoam_{\mathrm{dg}}$. 
\end{lemma*}

More precisely, the identity functor on $\Sfoam$ is isomorphic to the composition of the endofunctors $-*\w{x}$ and $-*\w{-x}$ via the natural transformation $\mu$ which associates to a web $W$ an invertible foam $\mu_W$ from $W$ to $W*\w{x}*\w{-x}$, which is modeled in a neighborhood of the intersection of $\id_W$ with $\id_{\w{x}}$ by the following foams
                  
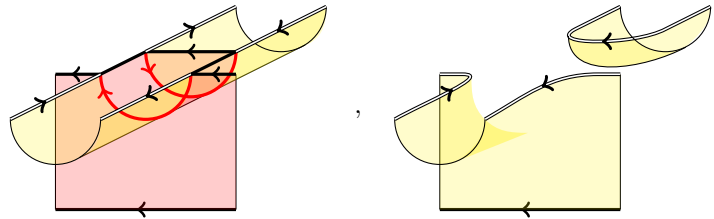
\begin{equation*} 
\begin{tikzpicture} [scale=.6,fill opacity=0.2,anchorbase]
	\draw[very thick, directed=.55] (2,1) to (-2,1);
	\draw (-2,1) -- (-2,4);
	\draw (2,1) -- (2,4);
		\draw (3.53,4.65) to (-1.47,2.15);
	\path [fill=red] (-2,4) to (-1,4) to [out=270,in=180](0,3)to [out=0,in=270] (1,4) to (2,4) to (2,1) to (-2,1) to (-2,4);
	\path [fill=red] (2, 4.5) to  (0,4.5) to (-1,4) to [out=270,in=180] (0,3) to [out=0,in=210] (.65,3.2) to (1.65,3.7) to[out=30,in=270] (2,4.5) ;
	\path [fill=yellow] (-1,4) to (-3,3) to  [out=270,in=180] (-2,2) to [out=0,in=210] (-1.35,2.2) to (.65,3.2) to[out=30,in=270] (1,4) to (-.6,3.2) to[out=135,in=270] (-1,4);
	\path [fill=yellow]  (-1,3) to [out=270,in=30] (-1.35,2.2) to (.65,3.2) to[out=30,in=270] (1,4) ;
\path [fill=yellow] (2,5.5) to (0,4.5) to  [out=270,in=180] (1,3.5) to [out=0,in=210] (1.65,3.7) to (3.65,4.7) to[out=30,in=270] (4,5.5) to (2.4,4.7) to[out=135,in=270] (2,5.5);
	\path [fill=yellow]  (2,4.5) to [out=270,in=45] (1.65,3.7) to (3.65,4.7) to[out=45,in=270] (4,5.5) ;	
	\draw [very thick, red, directed=.9] (1,4) to [out=270,in=0] (0,3) to [out=180,in=270] (-1,4);
	\draw [very thick, red, rdirected=.9] (2,4.5) to [out=270,in=0] (1,3.5) to [out=180,in=270] (0,4.5);
		\draw[very thick, directed=.55] (2,4) to (1,4);
		\draw[very thick,rdirected=.55] (0,4.5) to (2,4.5);
\draw[very thick] (0,4.5) to (-1,4);
\draw[very thick] (2,4.5) to (1,4);
	\draw[very thick, directed=.65] (-1,4) to (-2,4);
		\draw[double, directed=.5] (1,4) to (-1,3);
		\draw[double, rdirected=.70] (-1,4) to (-3,3);
		\draw (-1,3) to [out=270,in=0] (-2,2) to [out=180,in=270] (-3,3);
		\draw (4,5.5) to [out=270,in=0] (3,4.5) to [out=180,in=270] (2,5.5);
		\draw[double, directed=.55] (0,4.5) to (2,5.5);
		\draw[double, rdirected=.55] (2,4.5) to (4,5.5);
\end{tikzpicture}
\quad,\quad
\begin{tikzpicture} [scale=.6,fill opacity=0.2,anchorbase]
	\draw[very thick, directed=.55] (2,1) to (-2,1);
	\draw (-2,1) -- (-2,4);
	\draw (2,1) -- (2,4);
	\path [fill=yellow] (.85,4.85) to [out=270,in=208](3.45,4.6) to[out=45,in=270] (4,5.5) to (3,5) to [out=208,in=208] (1,5) ;
	\path [fill=yellow] (.85,4.85) to [out=270,in=208](3.45,4.6) to[out=208,in=0] (3,4.5) to[out=180,in=270] (2,5.5) to (1,5) to  [out=208,in=180] (2.5,4.8) ;
	\path [fill=yellow]  (-1,3) to [out=270,in=0] (-2,2) to (-2,1) to (2,1) to (2,4) to (1.5,4) to [out=180,in=30] (-1,3);
\path [fill=yellow]  (-2,4) to (-1.5,4) to [out=0,in=30] (-3,3) to [out=270,in=180] (-2,2) to [out=0,in=270] (-1,3) to[out=135,in=270] (-1.3,3.95) to (-2,3.5);
\path [fill=yellow]  (-1,3) to[out=135,in=270]  (-1.3,3.95) to (-2,3.5) to (-2,2) to [out=0,in=208] (-1.55,2.1) 
 to(-.05,2.7)  to [out=180, in=315](-1,3);
 \path [fill=yellow] (-1.55,2.1) 
 to(-.05,2.7)  to [out=180, in=315](-1,3) to[out=270,in=45](-1.55,2.1) ;
		\draw[double, directed=.55] (2,4) to (1.5,4) to [out=180,in=30] (-1,3);
		\draw[double, directed=.55] (-3,3) to [out=30,in=0] (-1.5,4) to (-2,4);
		\draw[double, directed=.55] (4,5.5) to (3,5) to  [out=208,in=208] (1,5) to (2,5.5);
		\draw (-1,3) to [out=270,in=0] (-2,2) to [out=180,in=270] (-3,3);
		\draw (.85,4.85) to [out=270,in=208](3.45,4.6);
		\draw (4,5.5) to [out=270,in=0] (3,4.5) to [out=180,in=270] (2,5.5);
\end{tikzpicture}
\end{equation*} 
as well as their reflections if the sign of intersection is opposite. Away from $\id_{\w{x}}$, the foam $\mu_W$ is simply given by the disjoint union of $\id_W$ and the continuation of the yellow $\mathrm{half circle}\times [0,1]$ facets. Motivated by their appearance we call the $\mu_W$ \emph{roof gutter foams}. They are invertible with inverses given by (signed) reflection in a horizontal plane. 
   
   \begin{proof}
To see that $\mu$ is a natural transformation, we need to check that for any foam $F\colon W_1 \to W_2$ we have
\begin{equation}
\label{eq:nattrans}\mu_{W_2} \circ F = (F*\w{x}*\w{-x})\circ \mu_{W_1},\end{equation} 
in other words, that roof gutters can be pushed through any foam. This again follows from Conjecture~\ref{conj:functoriality} since both sides of \eqref{eq:nattrans} can be written as $\Su\Kh(\iota(F)\star \mu_\emptyset)$ and related by isotoping $\mu_\emptyset$.
\end{proof}

As a consequence we get that $\mu$ induces isomorphisms $\Sfoam(W_1,W_2)\cong \Sfoam(W_1*\w{x},W_2*\w{x})$. 

\begin{remark} 
Note that $\w{x}*\w{y}\cong q^{2x\cdot y}t^{-2x\cdot y}\w{x+y}$. Similarly, if $F$ is a foam in $\Sfoam$, then we have that $(F*\w{x})*\w{y}$ is equivalent to $F*\w{x+y}$ in the sense that they coincide up to grading shifts on their source and target objects and conjugation by a natural isomorphism.
\end{remark}

\subsection{Recovering the Asaeda--Prztycki--Sikora link homologies}
\label{sec:alginv1}
In this section we sketch the construction of another type of algebraic categorical link invariant from the invariants obtained in Theorem~\ref{thm:functoriality}. These invariants are analogous to the Asaeda--Przytycki--Sikora link homologies and agree with them when defined over $\Z/2\Z$. Their construction uses Blanchet's trivalent TQFT and Conjecture~\ref{conj:functoriality}. Recall that Blanchet's TQFT is a functor that associates a $\Z$-graded $\Z$-module to any abstract Blanchet web and a homogeneous homomorphism to any abstract Blanchet foam between webs. We will extend this to field coefficients ($\Q$ or $\Z/2\Z$) and again write $\Hom_{\twoFoam}(\emptyset ,-)$ for the resulting functor.

The basic idea is to first apply Blanchet's TQFT to the Khovanov chain complexes on $\Su$, but regarded as chain complexes of abstract webs and foams between them, and then to refine the resulting invariant by a surface-specific grading. 

There are two main problems with this basic idea. First, not all foams in $\Sfoam$ qualify as Blanchet foams, and second, there can be webs in $\Sfoam$ that bound no abstract Blanchet foams at all. 

In order to remedy the first problem, we can project to the orientable part $\Sfoam^{\textrm{or}}$, whose morphisms qualify as Blanchet foams by Proposition~\ref{prop:Blanchet}. 
 
As regards the second problem, note that for any web $W$ on $\Su$, we can find a $2$-labeled web $\wedge^W$ such that the superposed web $W\ast \wedge^W$ admits an invertible foam $\Gamma_{W}$ from a purely $1$-labeled web. $\Gamma_W$ can be chosen such that its underlying $1$-labeled surface is an identity cobordism, so it is a Blanchet foam. The source web of $\Gamma_W$ is abstractly just a union of circles, thus it abstractly bounds, and by composition so does the web $W\ast \wedge^{W}$. In other words, any web can be made into a Blanchet web up to $\stwo$-operations.

Now, let us decompose the category $\Sfoam^{\textrm{or}}$ into blocks. The homology class $[W]$ induced by a web $W$ as in Definition~\ref{def:homologyclass} is invariant in each block, however, in general the decomposition into blocks will be finer than the decomposition along $H_1(\Su)$. For example, the two webs in Example~\ref{exa:1001} have the same first homology class, but lie in different blocks: indeed, only unorientable foams could map between them. We now choose one web $W_b$ per block, and make a choice of a $2$-labeled web $\wedge^{W_b}$, such that $W\ast\wedge^{W_b}$ is a Blanchet web.

Then it follows that the $\stwo$-operation $-\ast\wedge^{W_b}$ turns all webs in the block $b$ into Blanchet webs. Indeed, for $V\in b$ there exists an orientable foam $F$ from $W_b$ to $V$ (given as the composition of non-zero morphisms in $\Sfoam^{\textrm{or}}$) and $F\ast\wedge^{W_b}$ pre-composed with $\Gamma_{W_b}$ produces a Blanchet foam that maps between a disjoint union of $1$-labeled circles and $V\ast \wedge^{W_b}$. This implies that $V\ast \wedge^{W_b}$ bounds a Blanchet foam and thus is a Blanchet web. 

Let $\cal{L}_{\Su}$ denote the free $\Z$-module spanned by unoriented, essential simple closed curves on $\Su$ up to isotopy.

\begin{definition**} 
The \emph{twisted Blanchet TQFT} for $\Su$ is the functor $\Su\Bl\colon \Sfoam^{\textrm{or}} \to \Vect^{\Z}$ from the foam category $\Sfoam^{\textrm{or}}$ to the category of $\Z$-graded vector spaces and grading preserving linear maps, which is defined as follows. For the block $b$ of $\Sfoam^{\textrm{or}}$, the functor is defined as $\Su\Bl(-):=\Hom_{\twoFoam}(\emptyset ,-\ast \wedge^{W_b})$. Evaluated on a web $V\in b$, this produces a vector space $\Su\Bl(V)$ that is spanned by abstract foams $G$ with boundary $V\ast \wedge^{W_b}$, whose $1$-labeled parts can be assumed to consist of dotted and undotted disks, as a result of neck-cutting relations. The $\Z$-grading of such a basis element is given by twice the number of dots, minus the number of connected components of $c(V)=c(V\ast \wedge^{W_b})$. 

The vector spaces $\Su\Bl(V)$ furthermore admit an additional $\cal{L}_{\Su}$-grading, which is defined on a spanning foam $G$ as follows. If a disk in $c(G)$ bounds a curve that is inessential in $\Su$, then it does not contribute. Otherwise, let $c$ denote the essential boundary component. Then the disk contributes $c$ to the $\cal{L}_{\Su}$-grading if the disk is undotted and $-c$ if the disk is dotted.
\end{definition**}

\begin{proposition**}\label{prop:gradpres} The images of morphisms in $\Sfoam_0$ under the functor $\Su\Bl$ preserve the $\cal{L}_{\Su}$-grading and the complementary $q$-grading. Thus, we get an induced functor $\Su\Bl\colon \Sfoam_0 \to \Vect^{\cal{L}_{\Su}\times \Z}$ to the category of $\cal{L}_{\Su}\times \Z$-graded vector spaces and grading-preserving linear maps between them.
\end{proposition**}
\begin{proof} Let $F$ be a foam in $\Sfoam_0$. Since the refinement of the gradings depends only on the topology of the underlying surface and the presence of dots, we focus on $c(F)$. We have seen in the proof of Proposition \ref{prop:linindep} that we may assume that this dotted surface consists of undotted incompressible annuli and tori as well as other components involving only null-homologous circles. Under the twisted Blanchet TQFT, these other components make no contributions to the $\cal{L}_{\Su}$-grading and they preserve the complementary $q$-grading because they do so in $\Sfoam_0$. Similarly, undotted tori and undotted vertical annuli, i.e. those with one boundary component at the top and one at the bottom of $\Su\times [0,1]$, trivially preserve both gradings. Thus, it remains to deal with annuli that either have both boundary components on the bottom or on the top of $\Su\times [0,1]$. In the first case, the twisted Blanchet TQFT produces a map which can be non-zero only on basis elements that consist of two disks, exactly one of which carries a dot (since only spheres with precisely one dot have a non-zero evaluation). Such basis elements are of degree zero, as is the empty basis element. In other words, the morphism assigned to the annulus with bottom boundary preserves all gradings. The morphism assigned to the annulus with top boundary is computed by abstract neck-cutting: it sends zero disks to the sum of two configurations of two disks, which differ in the location of a single dot. Again, both the domain and the target object carry the same degrees.   
\end{proof}

\begin{definition**} \label{def:SKh}
  We define the link homology functor $\Su\APS\colon \Slinko \to \Vect^{\cal{L}_{\Su}\times \Z\times \Z}$ as the composition of the Khovanov functor $\Slinko\to \HC(\Sfoam)$, the projection to $\HC(\Sfoam_0)$, the twisted Blanchet TQFT $\Su\Bl$ and the functor of taking the homology of an $\cal{L}_{\Su}\times\Z$-graded chain complex. 
\end{definition**}

The functor $\Su\APS$ is closely related to the one originally defined by Asaeda--Przytycki--Sikora \cite{APS}.
Because of the lack of functoriality that the APS construction inherits from Khovanov's construction, we can only compare our process with theirs over $\Z/2\Z$. The main ingredient for constructing their link homologies is a functor from Bar-Natan's cobordism category $\SCob$ to the category of $\cal{L}_{\Su}\times \Z$-graded vector spaces. Over $\Z/2\Z$ this can be pre-composed with the forgetful functor $\Sfoam\to \SCob$ defined on webs and foams by erasing all $2$-labeled edges and facets and by forgetting orientations. The following lemma follows by directly comparing this composition with the twisted Blanchet TQFT on elementary cobordisms.

\begin{lemma} \label{lem:APSComp}
  The composition of degree zero projection and the twisted Blanchet TQFT $\Su\Bl$, as used in Definition~\ref{def:SKh}, agrees over $\Z/2\Z$ with the APS TQFT after forgetting all $2$-labeled information and orientations in webs and foams. 
\end{lemma}

Since both approaches follow essentially the same the cube-of-resolutions strategy to resolve link diagrams, we conclude with the following comparison result.

\begin{corollary} For each link $L$ in $\Su\times [0,1]$, the invariant $\Su\APS(L)$ agrees with the APS invariant of $L$ when defined over $\Z/2\Z$, up to overall grading shifts.
\end{corollary}

\subsection{Spectral sequences from surface embeddings}
\label{sec:ss}
In this section we prove that each embedding of surfaces $\phi\colon \Su \to \Su^\prime$ induces spectral sequences $\Su\APS(L)\rightsquigarrow \Su^\prime\APS(\phi(L))$ for links $L$ in $\Su\times [0,1]$, i.e. Theorem~\ref{thm:ss}. This generalizes the well-known spectral sequences between annular Khovanov homology and the usual Khovanov homology. In this section we work with $\Z/2\Z$-coefficients, but all results hold over $\Q$ and indeed $\Z$ modulo appropriate versions of Conjecture~\ref{conj:functoriality}.

\begin{definition} Let $W$ be a web in $\Sfoam$ and $G$ an $\cal{L}_{\Su}\times \Z$-homogeneous element of $\Su\Bl(W)$ with $\cal{L}_{\Su}$-degree $\sum_{i} m_i c_i$ for $m_i\in\Z$ and distinct isotopy classes of essential simple closed curves $c_i$. Then we define the \emph{$\Su$-weight} of $G$ as $\sum_i m_i\in \Z$.
\end{definition}

If $\Su$ is an annulus, the $\Su$-weight agrees with the $\slnn{2}$-weight in annular Khovanov homology.

\begin{lemma}\label{lem:weightfilt} Let $F$ be a foam between webs $W_1$ and $W_2$ in $\Sfoam^{\textrm{or}}$. Then $F$ induces a linear map between $\Su\Bl(W_1)$ and $\Su\Bl(W_2)$ that does not increase the $\Su$-weight. Moreover, if $F$ is homogeneous of essential $q$-degree $d$, then the linear map lowers the $\Su$-weight by $d$. 
\end{lemma}
\begin{proof} The $\cal{L}_{\Su}$-grading and weight depend only on the $1$-labeled surface $S=c(F)$ and we may assume that $S$ is a disjoint union of incompressible connected components $S_i$ and, without loss of generality, none of the $S_i$ is a (dotted) disk or a (dotted) closed surface. (Neither contribute to the weight, and the latter are sent to zero by the representable functor if they have positive essential $q$-degree.)

Suppose that $F$, and more concretely the component $S_1$, is of positive essential $q$-degree. This means $S_1=S_1^\prime \circ X$ where $X$ is either:
\begin{enumerate}
\item an identity cobordism with a dot on the cylinder over an essential simple closed curve $c$
\item a saddle cobordism merging two distinct essential simple closed curves $c_1,c_2$ into another essential simple closed curve $c$
\item a saddle splitting an essential simple closed curve $c$ into two essential simple closed curves $c_1,c_2$.
\end{enumerate}
In these cases, $X$ acts as the identity on everything except on cup generators of the following degrees:
\begin{enumerate}
\item $c\mapsto -c$
\item $c_1+c_2\mapsto c$, $c_1-c_2 \mapsto -c$ and $c_2-c_1 \mapsto -c$
\item $-c\mapsto -c_1-c_2$  $c\mapsto c_1-c_2, c_2-c_1$
\end{enumerate}
In each case, the $\Su$-weight decreases by the essential $q$-degree of $X$. We have already seen in Proposition~\ref{prop:gradpres} that morphisms of essential $q$-degree zero preserve the $\cal{L}_{\Su}$-grading and thus the $\Su$-weight. Together, these statements imply the lemma.
\end{proof}

\begin{lemma}\label{lem:comsq} For a suitable choice of correcting multi-curves $\gamma$, we have a commutative diagram
\[
\xymatrix{
\Sfoam^{\textrm{or}} \ar@{->>}[r] \ar[d]_{\Su\Bl} & \Sfoam_0 \ar[d]^{\Su\Bl} \\
  \cal{F}_{\Su}\Vect^{\Z} \ar@{->>}[r]^{\textrm{ass. gr.}}& \Vect^{\cal{L}_{\Su}\times \Z}
}
\]
where $\cal{F}_{\Su}\Vect^{\Z}$ is the category of $\Z$-graded vector spaces and grading-preserving linear maps, whose objects are additionally $\cal{L}_{\Su}$-graded, and the linear maps are filtered with respect to the $\Su$-weight as in Lemma~\ref{lem:weightfilt}. The bottom horizontal arrow is the functor of taking the associated graded with respect to the filtration.
\end{lemma}
\begin{proof}
The existence of the representable functor on the left vertical arrow follows from Lemma~\ref{lem:weightfilt}. Moreover, we have also seen that taking the associated graded with respect to the filtration kills the images of all morphisms of positive essential $q$-degree and acts as the identity on all morphisms of essential $q$-degree zero. 
\end{proof}

For the following, let $\phi\colon \Su \to \Su^\prime$ denote an embedding of surfaces.

\begin{lemma}\label{lem:essemb} If a foam $F$ in $\Sfoam$ is homogeneous of essential $q$-degree zero, the same is true for $\phi(F)$ in $\Su^\prime\twoFoam$. 
\end{lemma}

\begin{proof} By the essential $q$-degree zero assumption, we may assume that the underlying surface of $F$ consists of undotted incompressible annuli and tori as well as potentially dotted disks. The embedding $\phi$ will generate no new non-disk incompressible surfaces with dots or negative Euler characteristic, so $\phi(F)$ is of essential $q$-degree zero in $\Su^\prime\twoFoam$. 
\end{proof}

\begin{theorem} Given an embedding $\phi\colon \Su \to \Su^\prime$ and a link $L$ in $\Slinko$, there exists a spectral sequence\[\Su\APS(L)\rightsquigarrow \Su^\prime\APS(\phi(L)).\] 
\end{theorem}
\begin{proof}
We will show that $\Su^\prime\APS(\phi(L))$ is the homology of a filtered chain complex, whose associated graded has homology $\Su\APS(L)$, with the $\cal{L}_{\Su}$-grading collapsed to a $\cal{L}_{\Su^\prime}$-grading via the embedding $\phi$. This implies the existence of the spectral sequence.

Recall that the definitions of $\Su\Bl$ and thus $\APS$ depend on the choice of correcting $2$-labeled multi-curves, which are used to turn all webs on the relevant surface into Blanchet webs. Here, we first choose a collection of correcting multi-curves $\gamma$ on $\Su$ and then complete it to a collection $\gamma^\prime$ for $\Su^\prime$. Now we consider the diagram in Figure~\ref{fig:ss}.

\begin{figure}[ht]
\[
\xymatrix@C=1.3em{
& & \Slinko \ar[d]^{\Su\Kh} \ar@/_2pc/[dddll]_{\Su\APS} \ar[r]^\phi& \Su^\prime\cat{Link}^\circ \ar[d]^{\Su^\prime\Kh} \ar@/^2pc/[dddrr]^{\Su^\prime\APS} & &  
\\
& & \HC(\Sfoam) \ar@{->>}[d]\ar@{->>}[dl]\ar[r]^\phi& \HC(\Su^\prime\twoFoam) \ar@{->>}[dr] \ar@{->>}[d]& &
\\
& \HC(\Sfoam_0) \ar[d]^{\Su\Bl}  &\HC(\Sfoam^{\mathrm{or}})\ar@{->>}[l]\ar[d]^{\Su\Bl}\ar[r]^\phi & \HC(\Su^\prime\twoFoam^{\mathrm{or}})\ar[d]^{\Su^\prime\Bl}\ar@{->>}[r] &  \HC(\Su^\prime\twoFoam_0) \ar[d]^{\Su^\prime\Bl}&
\\
\Vect^{\cal{L}_{\Su}\times \Z^2} \ar[d]^{\mathrm{forget}} & \HC(\Vect^{\cal{L}_{\Su}\times \Z}) \ar[l]_{H_*} \ar[d]^{\mathrm{forget}}  &\HC(\cal{F}_{\Su}\Vect^{\Z}) \ar[l]_{\textrm{ass. gr.}} \ar[d]^{\textrm{ass. gr.}} \ar[r]^{\phi} &\HC(\cal{F}_{\Su^\prime}\Vect^{\Z})\ar[r]^{\textrm{ass. gr.}} &\HC(\Vect^{\cal{L}_{\Su^\prime}\times \Z}) \ar[r]^{H_*}& \Vect^{\cal{L}_{\Su^\prime}\times \Z^2}
\\
\Vect^{\cal{L}_{\Su^\prime}\times \Z^2} \ar@{~>}@/_3pc/[rrrrru]_{\textrm{spectral sequence}} &\HC(\Vect^{\cal{L}_{\Su^\prime}\times \Z})  \ar[l]_{H_*} & \HC(\cal{F}_{\Su}\Vect^{\cal{L}_{\Su^\prime}\times \Z}) \ar[ul]_{\textrm{ass. gr.}} \ar[rru]_{\mathrm{forget}}& &  
}
\]
\caption{The origin of embedding spectral sequences.}
\label{fig:ss}
\end{figure}
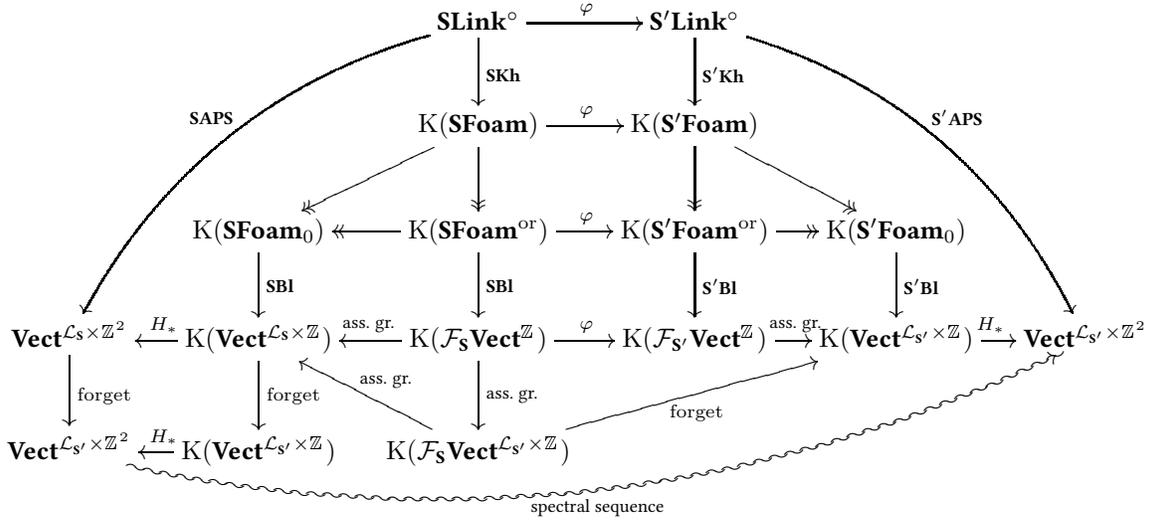
The outside pentagons commute by definition. The upper two central squares commute by the naturality of $\Su\Kh$ in $\Su$ and the fact that orientable foams stay orientable under surface embeddings. The two triangles in the second row commute since they express an iterated quotient. The first and the third square in the third row commute by Lemma~\ref{lem:comsq}. Commutativity of the middle square follows from our choice of correcting curves $\gamma$ and $\gamma^\prime$ and the fact that the $\cal{L}_{\Su}$-degree of a morphism determines the $\cal{L}_{\Su^\prime}$-degree of its image under the surface embedding. In the last row we first see a square of homology and forgetful functors which trivially commutes. The next triangle is more interesting. To make sense of it, we use the following observation.

\textbf{Claim:} Any morphism in $\textrm{im}(\Su\Bl) \subset\cal{F}_{\Su}\Vect^{\Z}$ that preserves the $\Su$-weight also preserves the $\Su^\prime$-weight. 

We fix such a morphism and by linearity we may and do assume that it appears in the image of single foam $F$ in $\Sfoam$ which is homogeneous with respect to essential $q$-degree. By Lemma~\ref{lem:weightfilt}, the essential $q$-degree of $F$ is zero. Then $\phi(F)$ is also of essential $q$-degree zero in $\Su^\prime\twoFoam$ by Lemma~\ref{lem:essemb}, and its image under $\Su^\prime\Bl$ preserves the $\Su^\prime$-weight by Lemma~\ref{lem:weightfilt}.

 Now returning to the interpretation of the triangle, the horizontal arrow kills all maps which do not preserve the $\Su$-weight, which in particular includes all maps that lower the $\Su^\prime$-weight. Alternatively, the vertical arrow first kills only the maps that lower the $\Su^\prime$-weight. The surviving maps in the target $\HC(\cal{F}_{\Su}\Vect^{\cal{L}_{\Su^\prime}\times \Z})$ preserve the $\Su^\prime$-grading on objects---thus the superscript $\cal{L}_{\Su^\prime}\times \Z$---but there can be maps left which decrease the $\Su$-weight grading---indicated by the filtration symbol $\cal{F}_{\Su}$. The diagonal arrow then takes the associated graded with respect to this filtration, which makes the triangle commutative. Alternatively, the $\Su$-weight filtration on $\HC(\cal{F}_{\Su}\Vect^{\cal{L}_{\Su^\prime}\times \Z})$ can be entirely forgotten, which is what the right-and-upward arrow does. The triangle-shaped square above it trivially commutes. This establishes the commutativity of the entire diagram. 

Now, given a link $L$ in $\Slinko$, following the vertical functors downward produces a filtered chain complex in $\HC(\cal{F}_{\Su}\Vect^{\cal{L}_{\Su^\prime}\times \Z})$, whose total homology is isomorphic to $\Su^\prime\APS(L)$ and whose associated graded has homology $\Su\APS(L)$, with the $\cal{L}_{\Su}$-grading collapsed to a $\cal{L}_{\Su^\prime}$-grading via the embedding $\phi$. 
\end{proof}

\section{Toric link homology}
\label{sec:toric}
In this section, we finish the discussion of the foam categories $\Sfoam$ by dealing with the remaining case of the torus $\Su=\T$. In particular, we introduce replacements for the Jones--Wenzl basis foams, which allow the completion of the proof of the isomorphism $\TWebq\cong K_0(\Tfoam)$ and the definition of an algebraic toric Khovanov homology $\T\Kh^\prime$ as in Section~\ref{sec:alginv2}.

\subsection{Affine web categories and extremal weight projectors} \label{sec:projectors}
The purpose of this section is to recall definitions and results from \cite{QW,QW2} on affine web categories and extremal weight projectors for $\glnn{2}$. In the following, we will use these concepts to study a quotient of the toric foam category $\Tfoam$.

\begin{definition} The affine $\glnn{2}$ web category $\AWebq$ is the category with 
\begin{itemize} 
\item objects, finite sequences of elements of the set $\{1,2,1^*,2^*\}$, including the empty sequence,
\item morphisms, $\Z[q^{\pm 1}]$-linear combinations of $\glnn{2}$ webs properly embedded in the annulus $\A$, viewed as mapping from the sequence on the inner boundary circle to the sequence of the outer boundary circle. These webs are considered up to isotopy relative to the boundary and modulo the $\glnn{2}$ web relations \eqref{eqn:circles}--\eqref{eqn:squares}.
\end{itemize}
Here $1$ and $2$ encode radially outward pointing boundary points of associated label, and $1^*$ and $2^*$ encode inward pointing boundary points. Composition is given by stacking annuli with matching boundary data.
\end{definition}
The category $\AWebq$ is monoidal with the tensor product acting by concatenation on objects and by the skein algebra product on morphisms, see Figure~\ref{fig:affinewebs}. The morphisms furthermore admit a $\Z$-grading by winding number, which is computed by the algebraic intersection number of web edges, weighted by label, with the dashed segment shown in Figure~\ref{fig:affinewebs}. In the following we will mostly consider the $q=1$ specialization $\AWeb$, which is actually symmetric monoidal. 

\begin{figure}[ht]
\[ 
\begin{tikzpicture}[anchorbase, scale=.3]
\draw (0,0) circle (1);
\draw (0,0) circle (5);
\draw [very thick,directed=.45] (-.6,.8) to (-.72,.96) to [out=120,in=90] (-2.5,-.5) to [out=270,in=180] (0,-2.5) to [out=0,in=270] (3,0.5) to [out=90, in=0] (1.75,2.5) to [out=180,in=45]  (1,2);
\draw [double,<-] (.6,.8) to [out=60,in=270] (1,1.75) to (1,2);
\draw [double] (-3,4) to (-2.4,3.2);
\draw [very thick,directed=.55] (-2.4,3.2) to [out=0,in=240] (0,4);
\draw [very thick,directed=.55] (-2.4,3.2) to (-2.4,3) to [out=270,in=120] (1,2);
\draw [double,->](0,4)to (0,5);
\draw [very thick,directed=.55] (3,4) to (2.7,3.6) to [out=240,in=300] (0,4);
\node at (.1,.3) {\tiny$1\; 2^*$};
\node at (0,5.5) {\tiny$2$};
\node at (-3.3,4.4) {\tiny$2^*$};
\node at (3.33,4.44) {\tiny$1^*$};
\draw [dashed] (0,-1) to (0,-5);
\end{tikzpicture}
\quad,\quad
W_1\otimes W_2:=
 \begin{tikzpicture}[anchorbase, scale=.3]
\draw[thick] (0,0) circle (2.5);
\fill[black,opacity=.2] (0,0) circle (2.5);
\draw[thick,fill=white] (0,0) circle (1.5);
\draw[thick] (0,0) circle (4.5);
\fill[black,opacity=.2] (4.5,0) arc (0:360:4.5) -- (3.5,0) arc (360:0:3.5);
\draw [thick] (0,0) circle (3.5);
\draw (0,0) circle (1);
\draw (0,0) circle (5);
\draw[dotted] (-2.29,2.29) to [out=225,in=90] (-3.25,0) to [out=270,in=135] (-2.29,-2.29);
\draw[dotted] (-3.4,3.4) to [out=225,in=90] (-4.75,0) to [out=270,in=135] (-3.4,-3.4);
\draw[dotted] (-1.93,1.93) to [out=225,in=90] (-2.75,0) to [out=270,in=135] (-1.93,-1.93);
\draw[dotted] (-0.88,0.88) to [out=225,in=90] (-1.25,0) to [out=270,in=135] (-0.88,-0.88);
\draw [white,line width=.15cm] (-.6,.8) to (-1.8,2.4);
\draw [white,line width=.15cm] (-.6,-.8) to (-1.8,-2.4);
\draw [very thick] (-.6,.8) to (-2.1,2.8);
\draw [very thick] (-.6,-.8) to (-2.1,-2.8);
\draw [very thick] (-2.7,3.6) to (-3,4);
\draw [very thick] (-2.7,-3.6) to (-3,-4);
\draw[dotted] (2.29,2.29) to [out=315,in=90] (3.25,0) to [out=270,in=45] (2.29,-2.29);
\draw[dotted] (3.4,3.4) to [out=315,in=90] (4.75,0) to [out=270,in=45] (3.4,-3.4);
\draw[dotted] (1.93,1.93) to [out=315,in=90] (2.75,0) to [out=270,in=45] (1.93,-1.93);
\draw[dotted] (0.88,0.88) to [out=315,in=90] (1.25,0) to [out=270,in=45] (0.88,-0.88);
\draw [very thick] (1.5,2) to (1.98,2.64);
\draw [very thick] (2.79,3.72) to (3,4);
\draw [very thick] (1.5,-2) to (1.98,-2.64);
\draw [very thick] (2.79,-3.72) to (3,-4);
\draw [very thick] (.6,-.8) to (.9,-1.2);
\draw [very thick] (.6,.8) to (.9,1.2);
\node at (0,-1) {$*$};
\node at (0,-5) {$*$};
\draw [dashed] (0,-1) to (0,-5);
\node at (0,1.95) {\tiny$W_2$};
\node at (0,3.95) {\tiny$W_1$};
\end{tikzpicture}
\]
\caption{
An example of an affine web and the tensor product on affine webs.
}
\label{fig:affinewebs}
\end{figure}
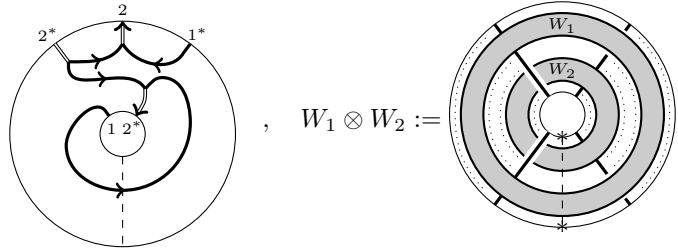
Particularly useful are the auto-equivalence $\lambda(-):=-\otimes \id_2$ given by superposing with a $2$-labeled strand and its quasi-inverse $\lambda^*(-):=-\otimes \id_{2^*}$.

\begin{definition} The category $\essbAWeb$ is defined as the quotient of $\AWeb$ by the tensor ideal generated by the endomorphisms of $\emptyset$ given by essential $1$-labeled circles.
\end{definition}

\begin{definition}\label{def:extrwp} The extremal weight projectors $T_m$ are idempotent endomorphisms of the object $1^{\otimes m}$ in $\essbAWeb$, which are defined recursively starting with $T_1=\id_1$ and
\[T_2= \begin{tikzpicture}[anchorbase, scale=.3]
\draw (0,0) circle (1);
\draw (0,0) circle (3);
\draw [very thick,->] (.8,.6) to (2.4,1.8);
\draw [very thick,->] (-.8,.6) to (-2.4,1.8);
\node at (0,-1) {$*$};
\node at (0,-3) {$*$};
\draw [dashed] (0,-1) to (0,-3);
\end{tikzpicture}
-\frac{1}{2} \;
 \begin{tikzpicture}[anchorbase, scale=.3]
\draw (0,0) circle (1);
\draw (0,0) circle (3);
\draw [very thick] (.8,.6) to [out=45, in=315] (0,1.5);
\draw [very thick] (-.8,.6) to [out=135,in=225] (0,1.5);
\draw [double] (0,1.5) to (0,2.25);
\draw [very thick,->] (0,2.25) to [out=45,in=225](2.16,1.62) to(2.4,1.8);
\draw [very thick,->] (0,2.25) to [out=135,in=315](-2.16,1.62) to(-2.4,1.8);
\node at (0,-1) {$*$};
\node at (0,-3) {$*$};
\draw [dashed] (0,-1) to (0,-3);
\end{tikzpicture}
-\frac{1}{2} \;
 \begin{tikzpicture}[anchorbase, scale=.3]
\draw (0,0) circle (1);
\draw (0,0) circle (3);
\draw [very thick] (.8,.6) to [out=45, in=90] (1.5,0) to [out=270,in=0] (0,-1.5);
\draw [very thick] (-.8,.6) to [out=135,, in=90] (-1.5,0) to [out=270,in=180] (0,-1.5);
\draw[double] (.2,-1.5) to (.2,-2.25);
\draw [very thick,->] (0,-2.25) to [out=0,in=270] (2.25,0) to [out=90,in=225](2.16,1.62) to(2.4,1.8);
\draw [very thick,->] (0,-2.25) to [out=180,in=270] (-2.25,0) to [out=90,in=315](-2.16,1.62) to(-2.4,1.8);
\node at (0,-1) {$*$};
\node at (0,-3) {$*$};
\draw [dashed] (0,-1) to (0,-3);
\end{tikzpicture},\]
and $T_{m+1}:= (\id_{1^{\otimes m-1}} \otimes T_2)(T_m \otimes \id_1)$ for $m\geq 2$.
\end{definition}

\begin{remark} Additionally setting the $2$-labeled essential circles in $\essbAWeb$ equal to $-\id_\emptyset$ produces the affine web category $\essAWeb$, which gives a diagrammatic presentation of the representation category of the Cartan subalgebra of $\glnn{2}$, see \cite[Theorem 1]{QW2}. Under this presentation, the idempotent $T_m$ encodes the endomorphism of $V^{\otimes m}$ given by projection onto the extremal weight spaces in $\Sym^m(V)$.
\end{remark}

When considering the Karoubi envelope of $\essbAWeb$, whose morphism spaces have an addition $\Z$-grading by winding number, we introduce additional winding grading shifts of objects and then consider only morphisms whose winding number is given by the difference of the winding grading of the target and the source object. For more details, see the discussion after \cite[Definition 42]{QW2}. It turns out that for most objects $W$ in $\Kar(\essbAWeb)$, all winding grading shifts $\sh^k W$ are isomorphic to each other, but this is not the case for $\lambda^k(\emptyset)$.

\begin{proposition}\label{prop:webdecomp} The category $\Kar(\essbAWeb)$ is semisimple with a skeleton generated by the objects $\lambda^k(T_{m})$ for $m>1$, $\lambda^k(\emptyset)$ and $\sh\lambda^k(\emptyset)$, where $k \in \Z$. In particular, these objects have $1$-dimensional endomorphism algebras and there are no other non-trivial morphisms between distinct objects from this collection.
\end{proposition}
\begin{proof}
See \cite[Section 4.2]{QW2}.
\end{proof}

\begin{proposition} Tensor products of extremal weight projectors with $m,n\geq 1$ decompose as follows in $\Kar(\essbAWeb)$:
\[\lambda^a(T_m)\otimes \lambda^b(T_n) \cong \lambda^{a+b}(T_{m+n}) \oplus \lambda^{a+b+\min(m,n)}(T_{|m-n|})\]
where $T_0:=\emptyset \oplus \sh \emptyset$ by definition.
\end{proposition}
\begin{proof}
See \cite[Section 4.1]{QW2}.
\end{proof}

Finally, we will need the following relation from \cite[Lemma 23]{QW2}:

\begin{equation}
\label{eqn:doublewrap}
 \begin{tikzpicture}[anchorbase, scale=.4]
\draw [white,line width=.15cm] (.9,.4) to [out=30,in=0](0,1.5) to [out=180,in=90] (-1.75,0) to [out=270,in=180](0,-1.75) to [out=0,in=270] (2,0) to [out=90,in=0] (0, 2.25) to [out=180,in=90]  (-2.5,0) to [out=270,in=180]  (0,-2.5) to [out=0,in=270] (2.5,0) to [out=90,in=210](2.8,1.2);
\draw [very thick,->] (.9,.4) to [out=30,in=0](0,1.5) to [out=180,in=90] (-1.75,0) to [out=270,in=180](0,-1.75) to [out=0,in=270] (2,0) to [out=90,in=0] (0, 2.25) to [out=180,in=90]  (-2.5,0) to [out=270,in=180]  (0,-2.5) to [out=0,in=270] (2.5,0) to [out=90,in=210](2.8,1.2);
\node at (0,-1) {$*$};
\node at (0,-3) {$*$};
\draw [dashed] (0,-1) to (0,-3);
\draw (0,0) circle (1);
\draw (0,0) circle (3);
\end{tikzpicture}
\quad =\quad - \;
\begin{tikzpicture}[anchorbase, scale=.4]
\draw (0,0) circle (1);
\draw (0,0) circle (3);
\draw [double ,directed=.55] (0,0) circle (2);
\draw [white,line width=.15cm] (.9,.4) to (2.8,1.2);
\draw [very thick,->] (.9,.4) to (2.8,1.2);
\node at (0,-1) {$*$};
\node at (0,-3) {$*$};
\draw [dashed] (0,-1) to (0,-3);
\end{tikzpicture} 
\end{equation}

\subsection{Slope subcategories}
In this section, we will study the category $\Tfoam$, one slope $m/n$ at a time. The main result is that the degree zero morphism spaces in these pieces are controlled by the affine web category $\AWeb$. For this, we will assume a weak form of Conjecture~\ref{conj:functoriality}, namely that the $\stwo$-operations from Section~\ref{sec:stwo} give auto-equivalences of the toric foam category. We start with the following basic observation which is analogous to Lemma~\ref{lem:WebtoAfoam}.

\begin{lemma}\label{lem:AWebtoTfoam} For any oriented simple closed curve $c$ on the torus, there exists a functor $\AWeb\to \Tfoam$ which sends affine webs $W$ to the rotation foams $W\times \Ss^1$, where rotation is performed along $c$.
\end{lemma}

\begin{definition}\label{def:slopesubcat} Let $m/n$ be a slope. We define: 
\begin{itemize}
\item  The slope subcategory $\Tfoam_{m/n}$, the full subcategory of $\Tfoam$ with objects given by the webs of slope $m/n$ (including all inessential webs). 
\item  The parallel slope category $\Tfoam^p_{m/n}$, the full subcategory of $\Tfoam_{m/n}$ with objects given by collections of 1- and $2$-labeled parallel copies of the slope $m/n$ with arbitrary orientations. 
\end{itemize}
\end{definition}

The next result follows directly from Lemma~\ref{lem:equivobjectsS}.
\begin{corollary} \label{cor:red2slopes2}
Every web in $\Tfoam_{m/n}$ is isomorphic to an object of $\Tfoam^p_{m/n}$ acted upon by the auto-equivalence of superposing by a suitable $2$-labeled multi-curve and a compensating shift in homological degree. 
\end{corollary}

If superposition with $2$-labeled multi-curves gives auto-equivalences of $\Tfoam_{\mathrm{dg}}$ as in Section~\ref{sec:stwo}, then the $\Hom$-spaces of $\Tfoam_{m/n}$ are controlled by the $\Hom$-spaces of $\Tfoam^p_{m/n}$.

\begin{corollary**} Let $W_1$ and $W_2$ be webs in $\Tfoam_{m/n}$ with $[W_1]=[W_2]$. Then there exists a $2$-labeled multi-curve $Z$ and $k\in \Z$, such that $t^k W_1*Z\cong W_1^p$ and $t^k W_2*Z\cong W_2^p$, where $W_1^p$ and $W_2^p$ are objects in $\Tfoam^p_{m/n}$. This implies $\Tfoam(W_1,W_2)\cong \Tfoam(W_1^p,W_2^p)$. 
\end{corollary**}

It thus remains to understand the morphism spaces in the parallel slope subcategories. By definition, their objects are $\Ss^1$-equivariant along the slope direction. Our next goal is to show that we may also assume that this is the case for the morphisms. 

To this end, we want to write arbitrary foams in a parallel slope subcategory as compositions of ``wrap-around foams'' and foams that are supported in an annular neighborhood $\A$ of the slope, i.e. foams that live in the subcategory $\AFoam \hookrightarrow \Tfoam_{m/n}$.

\begin{lemma}
\label{lem:wraparound}
Let $W_1$ and $W_2$ be webs in $\Tfoam$ supported in an annular neighborhood of the slope $m/n$. Then every morphism $F\in \Tfoam(W_1,W_2)$ factors into a composition of morphisms in $\AFoam\subset \Tfoam_{m/n}$ and $\Ss^1$-equivariant wrap-around foams in the image of $\AWeb\to \Tfoam_{m/n}$.
\end{lemma}
An paradigmatic example of such a factorization in the case of the slope $1/0$ is shown in Figure~\ref{fig:wrap}.
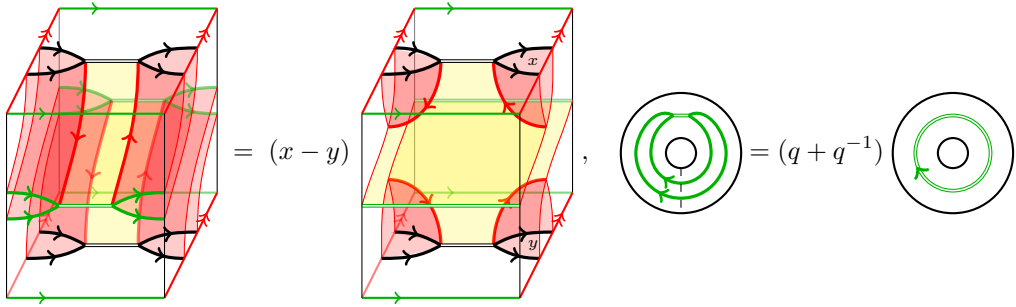
\begin{figure}[ht]
\[\begin{tikzpicture}[fill opacity=.2,anchorbase,xscale=.7, yscale=0.7]
\torusback{3}{2}{3.5}
\fill [fill=yellow] (2,3.75) to [out=255,in=90] (1.5,1) to (2.5,1) to [out=90,in=255](3,3.75) to (2,3.75);
\fill [fill=yellow] (2.5,4.5) to [out=270,in=75] (2,1.75)  to (1,1.75) to [out=75,in=270](1.5,4.5) to (2.5,4.5);
\fill [fill=red] (0,2) to [out=0,in=150] (1,1.75) to [out=75,in=270](1.5,4.5) to [out=210,in=0] (0.375,4.25) to [out=270,in=75] (0,2);
\fill [fill=red] (0,1.5) to [out=0,in=210] (1,1.75) to [out=75,in=270](1.5,4.5) to [out=150,in=0] (0.625,4.75) to [out=270,in=75] (0,1.5);
\fill [fill=red] (3,1.5) to [out=180,in=330] (2,1.75) to [out=75,in=270](2.5,4.5) to [out=30,in=180] (3.625,4.75) to [out=270,in=75] (3,1.5);
\fill [fill=red] (3,2) to [out=180,in=30] (2,1.75) to [out=75,in=270](2.5,4.5) to [out=330,in=180] (3.375,4.25) to [out=270,in=75] (3,2);
\fill [fill=red] (1,3.5) to [out=0,in=210] (2,3.75) to [out=255,in=90] (1.5,1) to [out=150,in=0] (.625,1.25) to [out=90,in=255] (1,3.5);
\fill [fill=red] (1,4) to [out=0,in=150] (2,3.75) to [out=255,in=90] (1.5,1) to [out=210,in=0] (.375,0.75) to [out=90,in=255] (1,4);
\fill [fill=red] (4,4) to [out=180,in=30] (3,3.75) to [out=255,in=90] (2.5,1) to [out=330,in=180] (3.375,0.75) to [out=90,in=255] (4,4);
\fill [fill=red] (4,3.5) to [out=180,in=330] (3,3.75) to [out=255,in=90] (2.5,1) to [out=30,in=180] (3.625,1.25) to [out=90,in=255] (4,3.5);
\coordinate (a) at (0,1.25);
\draw[green,very thick, directed=.5, opacity=.5] (1,2.25)+(a) to [out=0,in=210] ($(2,2.5)+(a)$);
\draw[green,very thick, directed=.5, opacity=.5] (1,2.75)+(a) to [out=0,in=150] ($(2,2.5)+(a)$);
\draw[green,double, opacity=.5] (2,2.5)+(a) to ($(3,2.5)+(a)$);
\draw[green,very thick, directed=.6, opacity=.5] (3,2.5)+(a) to [out=30,in=180] ($(4,2.75)+(a)$);
\draw[green,very thick, directed=.6, opacity=.5] (3,2.5)+(a) to [out=330,in=180] ($(4,2.25)+(a)$);
\draw [red] (0.375,4.25) to [out=270,in=75] (0,2);
\draw [red] (0.625,4.75) to [out=270,in=75] (0,1.5);
\draw [red] (3.625,4.75) to [out=270,in=75] (3,1.5);
\draw [red] (3.375,4.25) to [out=270,in=75] (3,2);
\draw [red] (.625,1.25) to [out=90,in=255] (1,3.5);
\draw [red] (.375,0.75) to [out=90,in=255] (1,4);
\draw [red] (3.375,0.75) to [out=90,in=255] (4,4);
\draw [red] (3.625,1.25) to [out=90,in=255] (4,3.5);
\draw [very thick, red, directed=.55,opacity=.5] (2,3.75) to [out=255,in=90] (1.5,1);
\draw [very thick, red, rdirected=.55,opacity=.5] (3,3.75) to [out=255,in=90] (2.5,1);
\draw [very thick, red, rdirected=.55] (2.5,4.5) to [out=270,in=75] (2,1.75);
\draw [very thick, red, directed=.55] (1.5,4.5) to [out=270,in=75] (1,1.75);
\coordinate (a) at (0.125,-2.75);
\draw[very thick, directed=.5] (.25,3.5)+(a) to [out=0,in=210] ($(1.375,3.75)+(a)$);
\draw[very thick, directed=.5] (.5,4)+(a) to [out=0,in=150] ($(1.375,3.75)+(a)$);
\draw[double] (1.375,3.75)+(a) to ($(2.375,3.75)+(a)$);
\draw[very thick, directed=.6] (2.375,3.75)+(a) to [out=30,in=180] ($(3.5,4)+(a)$);
\draw[very thick, directed=.6] (2.375,3.75)+(a) to [out=330,in=180] ($(3.25,3.5)+(a)$);
\coordinate (a) at (0.125,0.75);
\draw[very thick, directed=.5] (.25,3.5)+(a) to [out=0,in=210] ($(1.375,3.75)+(a)$);
\draw[very thick, directed=.5] (.5,4)+(a) to [out=0,in=150] ($(1.375,3.75)+(a)$);
\draw[double] (1.375,3.75)+(a) to ($(2.375,3.75)+(a)$);
\draw[very thick, directed=.6] (2.375,3.75)+(a) to [out=30,in=180] ($(3.5,4)+(a)$);
\draw[very thick, directed=.6] (2.375,3.75)+(a) to [out=330,in=180] ($(3.25,3.5)+(a)$);
\coordinate (a) at (-1,-.75);
\draw[green,very thick, directed=.5] (1,2.25)+(a) to [out=0,in=210] ($(2,2.5)+(a)$);
\draw[green,very thick, directed=.5] (1,2.75)+(a) to [out=0,in=150] ($(2,2.5)+(a)$);
\draw[green,double] (2,2.5)+(a) to ($(3,2.5)+(a)$);
\draw[green,very thick, directed=.6] (3,2.5)+(a) to [out=30,in=180] ($(4,2.75)+(a)$);
\draw[green,very thick, directed=.6] (3,2.5)+(a) to [out=330,in=180] ($(4,2.25)+(a)$);
\torusfront{3}{2}{3.5}
\end{tikzpicture}
\;=\;
%
(x-y)\;\begin{tikzpicture}[fill opacity=.2,anchorbase,xscale=.7, yscale=0.7]
\torusback{3}{2}{3.5}
\draw [very thick, red, directed=.55] (.5,2.25) to [out=0,in=90] (1.5,1);
\draw [very thick, red, rdirected=.55] (3.5,2.25) to [out=180,in=90] (2.5,1);
\fill [fill=yellow] (1,3.75) to (4,3.75) to [out=255,in=75] (3.5,2.25) to [out=180,in=90] (2.5,1) to (1.5,1) to [out=90,in=0] (.5,2.25) to [out=75,in=255] (1,3.75);
\fill [fill=yellow] (0,1.75) to (3,1.75) to [out=75,in=255] (3.5,3.25) to [out=180,in=270] (2.5,4.5) to (1.5,4.5) to [out=270,in=0] (.5,3.25) to [out=255,in=75] (0,1.75);
\fill [fill=red] (3.5,3.25) to [out=180,in=270] (2.5,4.5) to [out=30,in=180]  (3.625,4.75) to [out=270,in=75] (3.5,3.25);
\fill [fill=red] (3.5,3.25) to [out=180,in=270] (2.5,4.5) to [out=330,in=180]  (3.375,4.25) to [out=270,in=105] (3.5,3.25);
\fill [fill=red] (0.5,3.25) to [out=0,in=270] (1.5,4.5) to [out=150,in=180]  (0.625,4.75) to [out=270,in=75] (.5,3.25);
\fill [fill=red] (0.5,3.25) to [out=0,in=270] (1.5,4.5) to [out=210,in=180]  (0.375,4.25) to [out=270,in=105] (.5,3.25);
\fill [fill=red] (0.5,2.25) to [out=0,in=90] (1.5,1) to [out=150,in=180]  (.625,1.25) to [out=90,in=285] (.5,2.25);
\fill [fill=red] (0.5,2.25) to [out=0,in=90] (1.5,1) to [out=210,in=180]  (.375,0.75) to [out=90,in=255] (.5,2.25);
\fill [fill=red] (3.5,2.25) to [out=180,in=90] (2.5,1) to [out=30,in=180]  (3.625,1.25) to [out=90,in=285] (3.5,2.25);
\fill [fill=red] (3.5,2.25) to [out=180,in=90] (2.5,1) to [out=330,in=180]  (3.375,0.75) to [out=90,in=255] (3.5,2.25);
\coordinate (a) at (0,1.25);
\draw[green,double, opacity=.5] (1,2.5)+(a) to ($(4,2.5)+(a)$);
\draw [red] (0.375,4.25) to [out=270,in=105] (.5,3.25);
\draw [red] (0.625,4.75) to [out=270,in=75] (.5,3.25);
\draw [red] (3.625,4.75) to [out=270,in=75] (3.5,3.25);
\draw [red] (3.375,4.25) to [out=270,in=105] (3.5,3.25);
\draw [red] (3.5,3.25) to [out=255,in=75] (3,1.75);
\draw [red] (0.5,3.25) to [out=255,in=75] (0,1.75);
\draw [red] (.625,1.25) to [out=90,in=285] (.5,2.25);
\draw [red] (.375,0.75) to [out=90,in=255] (.5,2.25);
\draw [red] (3.375,0.75) to [out=90,in=255] (3.5,2.25);
\draw [red] (3.625,1.25) to [out=90,in=285] (3.5,2.25);
\draw [red] (4,3.75) to [out=255,in=75] (3.5,2.25);
\draw [red] (1,3.75) to [out=255,in=75] (0.5,2.25);
\draw [very thick, red, rdirected=.55] (2.5,4.5) to [out=270,in=180] (3.5,3.25);
\draw [very thick, red, directed=.55] (1.5,4.5) to [out=270,in=0] (.5,3.25);
\coordinate (a) at (0.125,-2.75);
\draw[very thick, directed=.5] (.25,3.5)+(a) to [out=0,in=210] ($(1.375,3.75)+(a)$);
\draw[very thick, directed=.5] (.5,4)+(a) to [out=0,in=150] ($(1.375,3.75)+(a)$);
\draw[double] (1.375,3.75)+(a) to ($(2.375,3.75)+(a)$);
\draw[very thick, directed=.6] (2.375,3.75)+(a) to [out=30,in=180] ($(3.5,4)+(a)$);
\draw[very thick, directed=.6] (2.375,3.75)+(a) to [out=330,in=180] ($(3.25,3.5)+(a)$);
\coordinate (a) at (0.125,0.75);
\draw[very thick, directed=.5] (.25,3.5)+(a) to [out=0,in=210] ($(1.375,3.75)+(a)$);
\draw[very thick, directed=.5] (.5,4)+(a) to [out=0,in=150] ($(1.375,3.75)+(a)$);
\draw[double] (1.375,3.75)+(a) to ($(2.375,3.75)+(a)$);
\draw[very thick, directed=.6] (2.375,3.75)+(a) to [out=30,in=180] ($(3.5,4)+(a)$);
\draw[very thick, directed=.6] (2.375,3.75)+(a) to [out=330,in=180] ($(3.25,3.5)+(a)$);
\coordinate (a) at (-1,-.75);
\draw[green,double] (1,2.5)+(a) to ($(4,2.5)+(a)$);
\torusfront{3}{2}{3.5}
\node[black,opacity=1] at (3.25,1) {\tiny $y$};
\node[black,opacity=1] at (3.25,4.5) {\tiny $x$};
\end{tikzpicture}
\;,\quad
\begin{tikzpicture}[fill opacity=.2,anchorbase,scale=.4]
\draw[thick] (0,0) circle (.5);
\draw[thick] (0,0) circle (2);
\draw[green,double] (-.25,1.25) to (.25,1.25);
\draw[green,very thick,directed=.59] (.25,1.25) to [out=30,in=90] (1.5,0) to [out=270,in=0] (0,-1.5) to [out=180,in=270] (-1.5,0) to [out=90,in=150] (-.25,1.25);
\draw[green,very thick,directed=.60] (.25,1.25) to [out=330,in=90] (1,0)to [out=270,in=0] (0,-1) to [out=180,in=270] (-1,0) to [out=90,in=210] (-.25,1.25);
\draw[dashed] (0,-.5) to (0,-2);
\end{tikzpicture}
=
(q+q^{-1})\;
\begin{tikzpicture}[fill opacity=.2,anchorbase,scale=.4]
\draw[thick] (0,0) circle (.5);
\draw[thick] (0,0) circle (2);
\draw[green,double,rdirected=.60] (0,0) circle (1.25);
\end{tikzpicture}
\]
\caption{A foam between toric webs supported in the neighborhood of the slope $1/0$ is expanded as a composition of annular and wrap-around foams. 
}\label{fig:wrap}
\end{figure}
\begin{proof}
Consider the affine web given by generically intersecting $F$ with the identity foam on the antipodal curve of the slope. This is drawn in green in Figure~\ref{fig:wrap}. By Lemma~\ref{lem:equivobjectsS} in the special case of the annulus, we can apply foam relations to $F$ to make this intersection $\Ss^1$-equivariant. Every circle in this web corresponds to a wrap-around foam and the rest of $F$ can be isotoped into the cylinder over the neighborhood of the slope. This decomposes $F$ as desired.
\end{proof}

\begin{lemma} \label{lem:equivariantmorph} Every morphism in $\Tfoam^p_{m/n}$ 
can be expressed as a linear combination of foams that are $\Ss^1$-equivariant along the slope direction, possibly with dots.
\end{lemma}
\begin{proof}
By Lemma~\ref{lem:wraparound}, foams in the parallel slope categories can be factored into wrap-arounds (clearly $\Ss^1$-equivariant) and foams supported in an annular neighborhood $\A$ of the slope between $\Ss^1$-equivariant webs. We shall now show that these annular foams can be made $\Ss^1$-equivariant as well.

  Consider $F\in \Afoam(W_1,W_2)$ with $W_1$ and $W_2$ $\Ss^1$-equivariant, that is, parallel copies of essential $1$- and $2$-labeled circles in the annulus. Note that proving that $F$ can be written as a linear combination of $\Ss^1$-equivariant foams is equivalent to proving that any pre- or post-composition of $F$ with the $\Ss^1$-equivariant foam over an invertible linear combination of webs can be written as a linear combination of $\Ss^1$-equivariant foams. In particular, this allows to braid boundary circles.\\
  Furthermore, the problem of decomposing $F$ is also preserved by the use of duality maps given by \textit{bending} a boundary circle lying on one side to the other side (which reverses its orientation). Using braiding and duality, we may thus assume that all circles come with the same orientation in $W_1$ and $W_2$, that $2$-labeled circles appear only in $W_1$ or in $W_2$, and they do so on the left, relative to the orientation of the other circles. Note that the sums of the labels of the circles in $W_1$ and $W_2$ are equal. The proof proceeds by induction on this sum $n$. 
  
  If $n=0$, then $F$ is closed and can be evaluated as follows: A generic slice is a closed web, along which $F$ can be neck-cut by first using Lemma~\ref{lem:neckcut2} and then $2$-labeled neck-cutting. The result is a closed foam in a ball, which evaluates by Lemma~\ref{lem:closedeval}. 
  
  If $n=1$, then $W_1$ and $W_2$ are both given by a single circle, and a generic section of $F$ is a web with one boundary point at the bottom and one at the top. Such webs are isomorphic to multiples of the identity web, and thus the foam can be cut accordingly by use of Lemma \ref{lem:webisos}. Then one can apply a neck-cutting relation on a push-out of the square formed by the top and bottom segments circles and two copies of the vertical segment given by the section. The result is a sum of unions of a possibly dotted identity foam and closed foams that can be evaluated to scalars.
  
  Next, we consider the case of $n \geq 2$, and we may assume that it is $W_1$ that contains no $2$-labeled circles. If the connected component $F_1$ of $F$ that contains the leftmost circle in $W_1$ does not contain any other circles in $W_1$, then $F_1$ contains only the leftmost circle in $W_2$ (which is thus $1$-labeled), and we can apply the argument for the previous case to make this component of $F$ $S_1$-equivariant. Moreover, $F\setminus F_1$ can be made $\Ss^1$-equivariant by the induction hypothesis.
  
 Now, consider the case where the two leftmost circles in $W_1$ belong to the same connected component of $F$. This implies that there exist two paths starting at points in the circles, which meet on a $2$-labeled facet in $F$. This means we can create an $\Ss^1$-equivariant and possibly dotted standard turnback at the bottom of $F_1$ as in the proof of Lemma~\ref{lem:turnbacks}. Next we distinguish two cases.
 
If $W_2$ contains a $2$-labeled circle at the left, then two applications of the $2$-labeled neck-cutting relation will connect the merging turnback directly with the $2$-labeled circle in $W_2$, and this $S_1$-equivariant merge foam can be stripped off $F$, thereby reducing the sum of boundary labels. The remaining foam can be made $\Ss^1$-equivariant by the induction hypothesis.

If $W_2$ does not have a $2$-labeled circle, then a similar argument will find a split foam terminating in the two leftmost $1$-labeled circles in $W_2$. In this case, two $2$-labeled neck-cutting relations will connect the merge and the split foam directly, thereby creating a component that can be stripped off $F$. The remaining foam can again be made $\Ss^1$-equivariant by the induction hypothesis.
\end{proof}

In other words, the morphisms in the parallel slope subcategories are given by rotation foams generated by affine webs, possibly decorated with dots. Indeed, the (essential) $q$-degree of such a morphism is twice its number of dots.

\begin{corollary}
For every slope $m/n$, the rotation functor $\AWeb\rightarrow \Tfoam^p_{m/n}$ from Lemma~\ref{lem:AWebtoTfoam} is surjective on objects and full onto the essential $q$-degree zero part of $\Tfoam^p_{m/n}$. 
\end{corollary}

\begin{proof} Surjectivity on objects is clear from the definition and the claimed fullness was proved in Lemma~\ref{lem:equivariantmorph}.
\end{proof}
 
\begin{proposition} \label{prop:isodegzeroTfoamAweb}
The rotation functor $\AWeb\rightarrow \Tfoam^p_{m/n}$ is faithful, making the category $\AWeb$ isomorphic to the essential $q$-degree zero part of $\Tfoam^p_{m/n}$.
\end{proposition}

\begin{proof} We prove that the functor takes certain spanning sets for the morphism spaces of $\AWeb$ to linearly independent sets of morphisms in $\Tfoam^p_{m/n}$.

Consider a particular morphism space in $\AWeb$ and its spanning set which is given by affine webs $W$ with underlying $1$-labeled curve $c(W)$ without inessential closed components. It is not hard to see that if two such webs have isotopic underlying curves $c(W)$ and equal winding degree, then they are equal up to a sign. So we keep only one web $W$ per pair of isotopy class of $c(W)$ and winding degree of $W$ in the spanning set. We already know that the foams given as the images of these webs under the functor $-\times \Ss^1$ span the relevant degree zero morphism space in $\Tfoam^p_{m/n}$. 

From Proposition~\ref{prop:linindep} we see that these foams are non-zero and that there can be no non-trivial $\Q$-linear relation between such foams if they have non-isotopic underlying surfaces. Additionally, there can be no non-trivial $\Q$-linear relation between foams that come from webs of distinct winding degree, since they will represent distinct relative second homology classes.
\end{proof}

\begin{remark} Under the isomorphism of the essential $q$-degree zero subcategory of $\Tfoam_{m/n}^p$ with $\AWeb$, the superposition operations with $2$-labeled copies of the slope $m/n$ correspond to the auto-equivalences $\lambda$ and $\lambda^*$ from Section~\ref{sec:projectors}.
\end{remark}

The interplay between different slopes is fairly restricted, as shown by the following lemma. 
\begin{lemma} \label{lem:betweenslopes} Let $W_1$ and $W_2$ be essential webs in $\Tfoam$ of different slopes, and $F$ be an orientable foam between $W_1$ and $W_2$. Then $F$ factors through a web without $1$-labeled edges.
\end{lemma}
\begin{proof}
Consider an orientable foam $F\in \Tfoam(W_1,W_2)$ and denote by $c(F)$ its underlying orientable surface, which bounds the multi-curves $c(W_1)$ and $c(W_2)$ obtained from the webs $W_1$ and $W_2$. By assumption $c(W_1)$ and $c(W_2)$ represent multiples of different primitive homology classes, however, they are homologous via $S$ and thus individually null-homologous. By Corollary~\ref{cor:neckcut} we may assume that they consist of even numbers of parallels of the slope. Our goal is to find a sequence of neck-cutting relations that we can apply to $c(F)$ in order to disconnect the multi-curves $c(W_1)$ and $c(W_2)$. A corresponding sequence of foam relations will then show that $F$ can be written as a linear combination of foams that factor through a web without $1$-labeled edges.

At the expense of performing a small isotopy, suppose that the standard height function on $c(F)\subset \T\times I$ is a separated Morse function. Then we consider the sequence of index $1$ critical points in order of height and their unstable manifolds, i.e. the gradient flow lines flowing downward out of such a critical point. For each such critical point, starting with the lowest, there are three possibilities: 
\begin{enumerate}
\item One flow line hits an index $0$ critical point, in which case we can cancel the two critical points by an isotopy (this works in $c(F)$ as well as in $F$). 
\item Both flow lines hit the same component of $c(W_1)$. Since $c(F)$ is orientable, this corresponds to splitting off a non-essential circle. Indeed, a compression disk can be constructed as shown in Figure~\ref{fig:compdisk}. Neck-cutting along this eliminates the critical point. 
\item The two flow lines hit different components of $c(W_1)$, which are necessarily adjacent in the sense that they bound an annulus in $\T\setminus c(W_1)$. From this, we obtain a compression disk (as shown on the right in Figure~\ref{fig:compdisk}), neck-cutting along which results in an annulus capping off the two components of $c(W_1)$. 
\end{enumerate} 

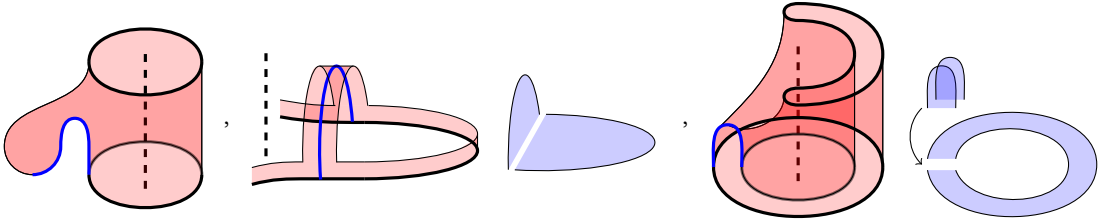
\begin{figure}[ht]
\begin{tikzpicture}[fill opacity=.2,anchorbase,scale=.75]
\draw [fill=red] (0,0) to[out=90,in=90] (2,0) to (2,2) to [out=90,in=90] (0,2) 
to[out=270,in=90] (-1.5,0.5) to [out=270,in=180] (-1,0)to [out=0,in=270](-.5,0.5) to [out=90,in=180] (-.25,1) to [out=0,in=90] (0,.5) to (0,0);
\draw[very thick,dashed] (1,-.25) to (1,2.25);
\draw [fill=red] (0,0) to[out=270,in=270] (2,0) to (2,2) to [out=270,in=270] (0,2) 
to[out=270,in=90] (-1.5,0.5) to [out=270,in=180] (-1,0)to [out=0,in=270](-.5,0.5) to [out=90,in=180] (-.25,1) to [out=0,in=90] (0,.5) to (0,0);
\draw[very thick] (0,0) to[out=270,in=270] (2,0);
\draw[very thick] (0,2) to[out=270,in=270] (2,2);
\draw[very thick] (0,2) to[out=90,in=90] (2,2);
\draw[very thick, opacity=.5] (2,0) to [out=90,in=90] (0,0);
\draw[very thick, blue] (-1,0) to [out=0,in=270](-.5,0.5) to [out=90,in=180] (-.25,1) to [out=0,in=90] (0,.5) to (0,0);
\end{tikzpicture}
\; , \;
\begin{tikzpicture}[fill opacity=.2,anchorbase,scale=.75]
\fill[red](-1.5,1.1) to [out=350,in=0] (-.6,1)to (0,1) to[out=0,in=170] (1,.92) to [out=350,in=140] (1.75,.725) to [out=320,in=90] (2,0.5) to (2,0.8) to [out=90,in=330]  (1.75,1.03) to [out=150,in=335] (1.3,1.18) to [out=155,in=350] (1,1.22) to [out=170,in=0] (0,1.3)to (-.6,1.3) to [out=180,in=350](-1.5,1.4); 
\fill[red] (-1.1,.3) to [out=90,in=260] (-1.05,1.25) to [out=90,in=250] (-1.01,1.6) to [out=65,in=180]  (-.8,2) to (-.2,2) to [out=180,in=65] (-.41,1.6) to [out=255,in=85] (-.45,1.25) to [out=260,in=90] (-.5,.3) to (-1.1,.3)  ;
\fill[red]  (-0.59,1.6) to [out=115,in=0]  (-.8,2) to (-.2,2) to [out=0,in=115] (.01,1.6) to [out=295,in=100] (.05,1.3) to (-.55,1.3) ;
\fill[red] (-2,-.1) to [out=15,in=180] (-1,0) to (0,0) to[out=0,in=190] (1,.08) to [out=10,in=220] (1.75,.275) to [out=40,in=270] (2,0.5) to (2,0.8) to [out=270,in=40]  (1.75,.575)  to [out=220,in=25] (1.3,.42) to [out=205,in=10] (1,.38) to [out=190,in=0] (0,0.3)to(-1.1,0.3) to [out=180,in=15](-2,.2) to (-2,-.1); 
\draw[very thick,dashed] (-1.75,.4) to (-1.75,2.25);
\draw[very thick] (-2,-.1) to [out=15,in=180] (-1,0) to (0,0);
\draw[very thick] (-1.5,1.1) to [out=350,in=180] (0,1);
\draw[very thick] (0,.5) [partial ellipse=270:450:2cm and .5cm];
\draw[very thick,blue] (-.5,.5) [partial ellipse=20:200: .3cm and 1.5cm];
\draw(-.5,.3) to (0,.3) ;
\draw(2,.5) to (2,.8) ;
\draw(-.8,2) to (-.2,2) ;
\draw (0,.8) [partial ellipse=270:448:2cm and .5cm];
\draw (-.2,.5) [partial ellipse=32:188: .3cm and 1.5cm];
\draw (-.8,.5) [partial ellipse=32:188: .3cm and 1.5cm];
\draw (-2,.2) to [out=15,in=180] (-1.1,.3);
\draw (-1.5,1.4) to [out=350,in=180] (-.52,1.3);
\end{tikzpicture}
\;\;
\begin{tikzpicture}[fill opacity=.2,anchorbase,scale=.75]
\draw[fill=blue] (0,1) [partial ellipse=256:448:2cm and .5cm];
\draw[fill=blue] (-.3,.7) [partial ellipse=32:187: .3cm and 1.5cm];
\end{tikzpicture}
\; , \;
\begin{tikzpicture}[fill opacity=.2,anchorbase,scale=.75]
\draw[very thick,dashed] (1,-.25) to (1,2.25);
\draw [fill=red] (-.5,0) to[out=270,in=270] (2.5,0) to (2.5,2) to [out=270,in=0](1,1.1) to [out=180,in=270] (.75,1.25) to [out=270,in=330] (-.25,.75)to [out=180,in=90]  (-.5,0);
\draw [fill=red] (0,0) to[out=270,in=270] (2,0) to (2,2) to [out=270,in=0](1,1.4) to [out=180,in=90] (.75,1.25) to [out=270,in=330] (-.25,.75)to [out=0,in=90]  (0,0);
\draw [fill=red] (-.5,0) to[out=90,in=90] (2.5,0) to (2.5,2) to [out=90,in=0](1,2.9) to [out=180,in=90] (0.75,2.75) to [out=270,in=45]  (-.35,.75) to [out=20,in=180](-.25,.75)to [out=180,in=90]  (-.5,0);
\draw [fill=red] (0,0) to[out=90,in=90] (2,0) to (2,2) to [out=90,in=0](1,2.6) to [out=180,in=270] (0.75,2.75) to [out=270,in=45] (-.35,.75) to [out=20,in=180](-.25,.75) to [out=0,in=90]  (0,0);
\draw[very thick] (2,2) to [out=90,in=0] (1,2.6) to [out=180,in=270] (0.75,2.75) to [out=90,in=180] (1,2.9) to [out=0,in=90] (2.5,2) to [out=270,in=0](1,1.1) to [out=180,in=270] (.75,1.25) to [out=90,in=180](1,1.4) to[out=0,in=270] (2,2);
\draw[very thick, opacity=.5] (0,0) to[out=270,in=270] (2,0);
\draw[very thick, opacity=.5] (2,0) to [out=90,in=90] (0,0);
\draw[very thick] (-.5,0) to[out=270,in=270] (2.5,0);
\draw[very thick] (-.5,0) to [out=90,in=90] (2.5,0);
\draw[very thick, blue] (-.5,0) to [out=90,in=180](-.25,.75)to [out=0,in=90]  (0,0);
\end{tikzpicture}\;\;
\begin{tikzpicture}[fill opacity=.2,anchorbase,scale=.75]
\fill [blue] (-.5,-.1) to[out=280,in=270] (2.5,0) to [out=90,in=80] (-.5,.1) to (0,.1)to [out=80,in=90]  (2,0) to[out=270,in=290] (0,-.1) to (-.5,-.1);
\draw (0,.1)to [out=80,in=90]  (2,0) to[out=270,in=290] (0,-.1);
\draw(-.5,-.1) to [out=280,in=270] (2.5,0) to [out=90,in=80] (-.5,.1);
\fill [blue] (-.5,1) to [out=90,in=180](-.25,1.75)to [out=0,in=90]  (0,1) to (-.5,1);
\draw (-.5,1) to [out=90,in=180](-.25,1.75)to [out=0,in=90]  (0,1);
\fill [blue] (-.35,1.15) to [out=90,in=180](-.1,1.9)to [out=0,in=90]  (0.15,1.15) to (-.35,1.15);
\draw (-.35,1.15) to [out=90,in=180](-.1,1.9)to [out=0,in=90]  (0.15,1.15);
\draw[white] (0,2.9)to(0.5,2.9);
\draw[->] (-.6,1)to[out=225,in=135](-.6,0);
\end{tikzpicture}
\caption{Types of index $1$ critical points, the last two with the desired compression disks.}
\label{fig:compdisk}
\end{figure}

In either case, we call the result again $c(F)$. This procedure can be iterated until $c(W_1)$ and $c(W_2)$ are disconnected along $c(F)$ because sufficiently many index $1$ critical points exist by the homological assumptions.
Note that we have used that for any compression disk $D$ for $c(W_1)$, it is possible to apply a neck-cutting relation to the foam $F$ with the same result after forgetting $2$-labeled facets. This was shown in Lemma~\ref{lem:neckcut2}.  \end{proof}

Unorientable foams can map between essential webs of different slopes via unorientable saddles, but by Lemma~\ref{lem:Eulerchar} such a saddle has a positive essential $q$-degree. Thus, Lemma~\ref{lem:betweenslopes} implies that all slope-changing foams in $\Tfoam_0$ factor through a web without $1$-labeled edges.

\subsection{A quotient of the toric foam category} In this section, we will study a particular quotient of the toric foam category. To streamline the exposition, we will assume that Conjecture~\ref{conj:functoriality} holds. However, most results here can be described and proven as in Section~\ref{sec:JWFoam} without relying on $\stwo$-operations or Conjecture~\ref{conj:functoriality}.

Recall from Proposition~\ref{prop:bifunctor} that superposition induces a bifunctor $*$ from $\Tfoam\times \Tfoam$ to $\HC(\Tfoam)$. For the following, we denote by $T_{\pm 1}$ the endomorphism of the empty web in $\Tfoam$ given by the boundary parallel essential torus, with the standard or the opposite orientation. Note that the contraction $T_{\pm 1}* -$ of the superposition bifunctor with these essential tori gives degree-preserving endofunctors of $\Tfoam$, which thus restrict to endofunctors of $\Tfoam_0$. 

\begin{definition} The foam category $\essTfoam$ is the quotient $\Tfoam_0/Q$, where $Q$ is defined to be the ideal of $\Tfoam_0$ generated by the morphisms of the form $T_{\pm 1}* F$ for foams $F$ in $\Tfoam_0$. 
\end{definition}

Above we have seen that $\Tfoam_0$ can be decomposed into full subcategories $\Tfoam_{m/n,0}$ corresponding to slopes on $\T$. After applying auto-equivalences given by $\stwo$-operations and decomposing webs with inessential $1$-labeled components, we arrive at the subcategories $\Tfoam_{m/n,0}^p$, which can be described by the affine web category $\AWeb$. In these subcategories, $Q$ precisely corresponds to the ideal generated by $1$-labeled essential circles in $\AWeb$. 

\begin{proposition} The graded additive $\Q$-linear category $\essTfoam$ satisfies the following properties 
\begin{enumerate}
\item $\essTfoam$ decomposes into blocks indexed by $H_1(\T)$,
\item $\stwo$-operations provide equivalences between these blocks,
\item the parallel slope subcategories of $\essTfoam$ are isomorphic to $\essbAWeb$,
\item all unorientable foams are zero in $\essTfoam$,
\item  slope changing foams factor through purely $2$-labeled webs.
\end{enumerate}
\end{proposition}
\begin{proof} The first two properties are immediate since the direct sum decompositions and auto-equivalences are inherited by the quotient. The third property was discussed above and the fourth is inherited from $\Tfoam_0$, see Corollary~\ref{cor:oriented}. Finally, in the proof of Lemma~\ref{lem:betweenslopes} we have observed that slope-changing foams factor through a purely $2$-labeled web if they do not contain an unorientable saddle. The claim (5) then follows from (4). 
\end{proof}

\begin{lemma*}
\label{lem:essstar}
Superposition induces a bifunctor $*\colon \essTfoam\times \essTfoam \to \HC(\essTfoam)$.
\end{lemma*}
\begin{proof} It suffices to prove that if $W_1$ and $W_2$ are toric webs and $F=T_{\pm}*\id_{W_1}$, then $F*\id_{W_2}\colon \T\Kh(W_1\star W_2)\to \T\Kh(W_1\star W_2)$ is a chain map whose components are of the form $T_{\pm}*\id_X$, where $X$ denotes a web appearing in the complex $\T\Kh(W_1\star W_2)$.

To this end, we claim that the two chain maps given by $F*\id_{W_2}=\T\Kh(\iota(\T\Kh(T_{\pm}\star\id_{W_1}))\star\id_{W_2})$ and $\T\Kh(T_{\pm} \star \id_{W_1} \star \id_{W_2})$ are homotopic. In the latter, we can omit parentheses because the superposition product on $\Q\Ttanweb$ has associators given by isotopies that are generically vertical and are thus sent to identities by the Khovanov functor.

For the proof of the claim, recall that the computation of the map induced by a foam in $\Ttanweb$ relies on a presentation of the foam as a movie of tangled webs in $\T\times [0,1]$. Each basic movie, which is supported in a small disk, induces a basic chain map between the corresponding Khovanov complexes, whose components are given by foams which differ from identity foams only over the same small disk. By virtue of this locality, we may compute the chain map $\T\Kh(T_{\pm} \star \id_{W_1} \star \id_{W_2})$ in two steps. In the first step we only resolve crossings that involve $W_1$ and the webs given by slices through the foam $T$, but none of the crossings with the web $W_2$. The result of this partial computation is a movie of tangled webs, which represents the foam $\iota(\T\Kh(T_{\pm}\star \id_{W_1}))$ superposed over the identity foam on $W_2$. This intermediate result of the first computation step agrees precisely with the movie of tangled webs from which $\T\Kh(\iota(\T\Kh(T_{\pm}\star\id_{W_1}))\star \id_{W_2})$ is computed. This implies that the two chain maps are homotopic.

Finally, we re-associate once more and compute the chain map $\T\Kh(T_{\pm} \star \id_{W_1} \star \id_{W_2})$ by first resolving all crossings between $W_1$ and $W_2$ before computing the action of the torus $T_{\pm}$ on all webs $X$ that appear in $\T\Kh(W_1\star W_2)$. By definition, this action is given by foams $T_{\pm}*\id_X$, which concludes the proof.
\end{proof}

Since the parallel slope subcategories of $\essTfoam$ are isomorphic to $\essbAWeb$, they contain idempotent endomorphisms corresponding to the extremal weight projectors from Section~\ref{sec:projectors}. These idempotents represent objects in $\Kar(\essTfoam)$ which categorify the basis elements $(m,n)_T\in \T\basis$ of the $\glnn{2}$ skein algebra of the torus, as described in Section~\ref{sec:skeinalgtorus}. 

For the following, consider $m,n\in \Z$, not both equal to zero, and $d=\gcd(m,n)$. We choose (once and for all) a $1$-labeled, oriented multi-curve $(m,n)$ of homology class $m [\lambda]+ n [\mu]$ on $\T$, which consists of $d$ parallel copies of a simple closed curve. Note that this is consistent with Section~\ref{sec:skeinalgtorus}, but now we consider $(m,n)$ as an object of $\essTfoam$ and hence do not allow isotopies.

\begin{definition}
We denote by $(m,n)^F_T$ the object in the Karoubi envelope of $\essTfoam$, given by the web $(m,n)$ together with the idempotent rotation foam $T_d\times \Ss^1$ generated by the extremal weight projector $T_d$.
\end{definition}

\begin{lemma}\label{lem:KarTmor} The morphism spaces in the Karoubi envelope of $\essTfoam$ between objects of the form $(m,n)^F_T*\w{r,s}$ and $\w{r,s}$ for $(m,n), (r,s)\in \Z^2$ with $(m,n)\neq (0,0)$ satisfy the following properties:
\begin{itemize}
\item The endomorphism ring of $(m,n)^F_T*\w{r,s}$ is isomorphic to $\C[(\wrap*\w{r,s})^{\pm 1}]$, with the generator induced by the $1$-labeled wrap endomorphism $\wrap$ of $(m,n)$. 
\item The endomorphism ring of $\w{r,s}$ for $(r,s)\neq (0,0)$ is isomorphic to $\C[\wrap^{\pm 1}_2]$, with the generator $\wrap_2$ given by the $2$-labeled wrap.
\item The endomorphism ring of the empty web $\w{0,0}=\emptyset$ is isomorphic to $\C[c_2^{\pm 1}]$, with the generator $c_2$ given by the boundary-parallel $2$-labeled torus with the standard orientation.
\item There are no non-trivial morphisms between distinct objects of the form $(m,n)^F_T*\w{r,s}$ or $\w{r,s}$ with $m>0$ or $n>m=0$.
\item Otherwise, the only non-trivial morphisms between such objects are realized by $(m,n)^F_T*\w{-m,-n} \cong (-m,-n)^F_T$ and, more generally, $(m,n)^F_T*\w{r-m,s-n} \cong (-m,-n)_T*\w{r,s}$.
\end{itemize}
\end{lemma}
\begin{proof} First of all, there is a homological obstruction for having a morphism between such objects since any foam involving $(m,n)^F_T*\w{r,s}$ needs to preserve the class $[(m,n)_T*\w{r,s}]\in H_1(\T)$. Two objects in the same slope and the same homology class can be simultaneously transported into a parallel slope subcategory, where the claims follow from the corresponding results in $\essbAWeb$, see Proposition*~\ref{prop:webdecomp}. Finally, a morphism between two such objects in different slopes factors through a purely $2$-labeled web, and thus into two slope-preserving morphisms. Since the objects are in different slopes, both of them are not $2$-labeled, and so the two slope-preserving morphisms are again zero by Proposition~\ref{prop:webdecomp}. \end{proof}

Next, we choose a basepoint $p\in \T$ disjoint from all multi-curves $(m,n)$. Then the spaces of morphisms between toric webs $W_1$ and $W_2$, which are disjoint from $p$, admit an additional $\Z$-grading, which can be computed as the algebraic intersection number of foams in $\T\times [0,1]$ with the oriented arc $p\times [0,1]$. This corresponds to the winding grading on affine webs in any slope. We can now add formal winding shifts of objects and restrict to morphisms that respect such shifts, c.f. the discussion after Proposition~\ref{prop:webdecomp}. In the following, we use the notation $\Kar(\essTfoam)$ for this winding-graded version.

In Lemma~\ref{lem:KarTmor} we have disregarded the winding grading, and have thus obtained non-trivial endomorphism rings. If we instead work in $\Kar(\essTfoam)$, we get shifted objects $\sh^a (m,n)^F_T*\w{r,s}$ and $\sh^b \w{r,s}$, which have $1$-dimensional endomorphism rings. As $a\in \Z$ varies, all objects $\sh^a (m,n)^F_T*\w{r,s}$ are isomorphic to each other, via isomorphisms that are unique up to scalars. Similarly, objects of the form $\sh^a \w{r,s}$ are equivalent if and only if their shifts have the same parity. 

\begin{proposition}\label{prop:Tss} $\Kar(\essTfoam)$ is semisimple, with non-isomorphic simple objects $t^{ms-nr}(m,n)^F_T*\w{r,s}$ and $\w{r,s}$ as well as $\sh \w{r,s}$ for $m,n,r,s\in \Z$, $m>0$ or $n>m=0$, and their $q$-grading shifts. 
\end{proposition}
\begin{proof} We have seen that every web in $\essTfoam$ is isomorphic to a direct sum of parallels of a slope, acted upon by $\stwo$-operations. Via Proposition~\ref{prop:isodegzeroTfoamAweb}, Proposition~\ref{prop:webdecomp} implies that every such object can be further decomposed into a direct sum of simple objects as listed above. Any idempotent in $\essTfoam$ thus gives rise to an idempotent matrix of morphisms between simples. Such matrices can be diagonalized, so all idempotents split into simples. 
\end{proof}

Now we can complete the proof of Theorem~\ref{thm:decat} in the case of the torus. For this, we wanted to see that the map $\gamma \colon \TWebq \to K_0(\Tfoam)$ sends the standard basis $\T\basisstd$ to a linearly independent set. This follows by considering the composition of $\gamma$ with the natural maps $K_0(\Tfoam)\to K_0(\essTfoam) \to K_0(\Kar(\essTfoam))$, which sends the basis $\T\basis$ to the set of classes of the non-isomorphic simple objects from Proposition~\ref{prop:Tss}, which is thus linearly independent.

\subsection{The superposition products of simples}
\label{sec:examples}
In this section, we sketch how the superposition bifunctors from Proposition~\ref{prop:bifunctor} and Lemma~\ref{lem:essstar} extend to bifunctors on the semisimple categories $\Kar(\Sfoam_0)_{\mathrm{dg}}$ and $\Kar(\essTfoam)_{\mathrm{dg}}$ respectively. We then compute the examples of the superposition of simples in the case of the torus, which suggest that the superposition bifunctor satisfies a categorified analogue of the Frohman--Gelca formula for the multiplication in the toric skein algebra. Throughout we will assume that Conjecture~\ref{conj:functoriality} holds.

In order to define the bifunctor $*$ on $\Kar(\Sfoam_0)_{\mathrm{dg}}$ or $\Kar(\essTfoam)_{\mathrm{dg}}$ it suffices to define it on simple objects, since all non-trivial morphisms are scalar multiples of identity morphisms. Let $F_1$ and $F_2$ be idempotent endomorphisms of webs $W_1$ and $W_2$ respectively, which represent such simple objects. Then $W_1*W_2$ is a chain complex which deformation retracts via maximal Gaussian elimination onto an essentially unique minimal complex, which is again given by a direct sum of simple objects. Next, by the monoidality assumption on $*$, the chain endomorphism $(F_1*\id_{W_2})\circ(\id_{W_1}*F_2)$ of $W_1*W_2$ (defined via Proposition~\ref{prop:bifunctor}) is idempotent up to homotopy. The induced endomorphism on the minimal complex is honestly idempotent, and we define $(W_1,F_1)*(W_2,F_2)$ to be its image.

\begin{example} \label{exa:1001-cat} In $\Kar(\essTfoam)_{\mathrm{dg}}$ we have
\[(1,0)^F_T*(0,1)^F_T = (1,1)^F_T\oplus (1,-1)^F_T*\w{0,1} \]
\end{example}
This is because $(a,b)_T=(a,b)$ for $\gcd(a,b)=1$, and the chain complex for $(1,0)*(0,1)$ in Example~\ref{exa:1001} splits thanks to the degree zero truncation in $\essTfoam$.

\begin{example} \label{exa:2101-cat} In $\Kar(\essTfoam)_{\mathrm{dg}}$ we have
\[(2,1)^F_T*(0,1)^F_T\cong (2,2)^F_T\oplus(2,0)^F_T*\w{0,1}.\]
\end{example}

Indeed, because of the degree zero truncation, the chain complex for $(2,1)*(0,1)$ splits into two halves, as already suggested in the illustration in Example~\ref{exa:2101-2}. Consider the first half, which is a complex over the parallel slope category for $1/1$. After transporting to $\essbAWeb$, we see the following complex:
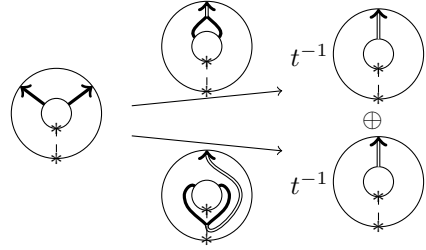
\[
 \begin{tikzpicture}[anchorbase]
 \node at (0,0) {$
 \begin{tikzpicture}[anchorbase, scale=.2]
\draw (0,0) circle (1);
\draw (0,0) circle (3);
\draw [very thick,->] (.8,.6) to (2.4,1.8);
\draw [very thick,->] (-.8,.6) to (-2.4,1.8);
\node at (0,-1) {$*$};
\node at (0,-3) {$*$};
\draw [dashed] (0,-1) to (0,-3);
\end{tikzpicture}
$};
 \node at (4,.8) {$t^{-1}\,
 \begin{tikzpicture}[anchorbase, scale=.2]
\draw (0,0) circle (1);
\draw (0,0) circle (3);
\draw [double,->] (0,1) to (0,3);
\node at (0,-1) {$*$};
\node at (0,-3) {$*$};
\draw [dashed] (0,-1) to (0,-3);
\end{tikzpicture}
$};
 \node at (4.2,0) {$\oplus$};
 \node at (4,-.9) {$t^{-1}\,
 \begin{tikzpicture}[anchorbase, scale=.2]
\draw (0,0) circle (1);
\draw (0,0) circle (3);
\draw [double,->] (0,1) to (0,3);
\node at (0,-1) {$*$};
\node at (0,-3) {$*$};
\draw [dashed] (0,-1) to (0,-3);
\end{tikzpicture}
$};
\draw [->] (1,.2) to (3,.4);
\draw [->] (1,-.2) to (3,-.4);
\node at (2,.9) {$
 \begin{tikzpicture}[anchorbase, scale=.2]
\draw (0,0) circle (1);
\draw (0,0) circle (3);
\draw [very thick] (.8,.6) to [out=45, in=315] (0,2);
\draw [very thick] (-.8,.6) to [out=135,in=225] (0,2);
\draw [double,->] (0,2) to (0,3);
\node at (0,-1) {$*$};
\node at (0,-3) {$*$};
\draw [dashed] (0,-1) to (0,-3);
\end{tikzpicture}
$};
\node at (2,-1.08) {$ 
 \begin{tikzpicture}[anchorbase, scale=.2]
\draw (0,0) circle (1);
\draw (0,0) circle (3);
\draw [very thick] (.8,.6) to [out=45, in=90] (1.5,0) to [out=270,in=45] (0,-2);
\draw [very thick] (-.8,.6) to [out=135,, in=90] (-1.5,0) to [out=270,in=135] (0,-2);
\draw[double,->] (0,-2) to [out=270,in=270] (2.25,0) to [out=90,in=270](0,3);
\node at (0,-1) {$*$};
\node at (0,-3) {$*$};
\draw [dashed] (0,-1) to (0,-3);
\end{tikzpicture}
$};
\end{tikzpicture}
\]
The identity web in homological degree zero splits into orthogonal idempotents as follows:
\[\begin{tikzpicture}[anchorbase, scale=.2]
\draw (0,0) circle (1);
\draw (0,0) circle (3);
\draw [very thick,->] (.8,.6) to (2.4,1.8);
\draw [very thick,->] (-.8,.6) to (-2.4,1.8);
\node at (0,-1) {$*$};
\node at (0,-3) {$*$};
\draw [dashed] (0,-1) to (0,-3);
\end{tikzpicture} = T_2 
+\frac{1}{2} \;
 \begin{tikzpicture}[anchorbase, scale=.2]
\draw (0,0) circle (1);
\draw (0,0) circle (3);
\draw [very thick] (.8,.6) to [out=45, in=315] (0,1.5);
\draw [very thick] (-.8,.6) to [out=135,in=225] (0,1.5);
\draw [double] (0,1.5) to (0,2.25);
\draw [very thick,->] (0,2.25) to [out=45,in=225](2.16,1.62) to(2.4,1.8);
\draw [very thick,->] (0,2.25) to [out=135,in=315](-2.16,1.62) to(-2.4,1.8);
\node at (0,-1) {$*$};
\node at (0,-3) {$*$};
\draw [dashed] (0,-1) to (0,-3);
\end{tikzpicture}
+\frac{1}{2} \;
 \begin{tikzpicture}[anchorbase, scale=.2]
\draw (0,0) circle (1);
\draw (0,0) circle (3);
\draw [very thick] (.8,.6) to [out=45, in=90] (1.5,0) to [out=270,in=0] (0,-1.5);
\draw [very thick] (-.8,.6) to [out=135,, in=90] (-1.5,0) to [out=270,in=180] (0,-1.5);
\draw[double] (.2,-1.5) to (.2,-2.25);
\draw [very thick,->] (0,-2.25) to [out=0,in=270] (2.25,0) to [out=90,in=225](2.16,1.62) to(2.4,1.8);
\draw [very thick,->] (0,-2.25) to [out=180,in=270] (-2.25,0) to [out=90,in=315](-2.16,1.62) to(-2.4,1.8);
\node at (0,-1) {$*$};
\node at (0,-3) {$*$};
\draw [dashed] (0,-1) to (0,-3);
\end{tikzpicture}
\]
The two differentials restrict to zero on $T_2$ and to isomorphisms on the other two objects respectively. After Gaussian elimination, only $T_2$ remains, which corresponds to $(2,2)^F_T$ in $\Kar(\essTfoam)$. The second half of the chain complex analogously retracts onto $(2,0)^F_T*\w{0,1}$.

\begin{example}\label{exa:slope-cat} In $\Kar(\essTfoam)_{\mathrm{dg}}$ we have
\begin{equation}\label{eqn:catFG}
(m,n)^F_T*(r,s)^F_T\cong (m+r,n+s)^F_T \oplus (m-r,n-s)^F_T*\w{r,s}
\end{equation} whenever these simples lie in the same slope, i.e. if $\frac{m}{n}=\frac{r}{s}$. This follows from the isomorphisms between $\essbAWeb$ and the degree zero parallel slope subcategories of $\essTfoam$, see Proposition~\ref{prop:isodegzeroTfoamAweb}.
\end{example}

\begin{remark} If $*$ extends to a monoidal structure on $\Kar(\essTfoam)_{\mathrm{dg}}$, then the preceding examples are sufficient to prove the categorified Frohman--Gelca formula \eqref{eqn:catFG} in full generality, see Remark~\ref{rem:catproof}.
\end{remark}

The following computation verifies a non-trivial case of the categorified Frohman--Gelca formula \eqref{eqn:catFG}, and thus provides additional evidence for Conjecture~\ref{conj_catFG}.

\begin{example} In $\Kar(\essTfoam)_{\mathrm{dg}}$ we have
\[(2,0)^F_T*(0,1)^F_T \cong (2,1)^F_T \oplus (2,-1)^F_T*\w{0,1}\]
\end{example}

We will prove this by computing the image of the idempotent $T_2* \id$ on the cube of resolutions of $(2,0)*(0,1)$:\footnote{Here we suppress the choice of an ordering for the crossings, as it is not relevant in the computation.}
\begin{equation*}
\xy
(-10,0)*{
\begin{tikzpicture}[anchorbase]
\draw[very thick, directed=.9] (1.2,0) to (1.2,1.6);
\draw[white, line width=.15cm] (0,.5) to (2.4,.5);
\draw[white, line width=.15cm] (0,1.1)  to (2.4,1.1);
\draw[very thick, directed=.75] (0,0.5) to (2.4,.5);
\draw[very thick, directed=.75] (0,1.1)  to (2.4,1.1);
\torus{2.4}{1.6}
\end{tikzpicture}
};
(10,0)*{\cong
};
(47,0)*{(A)};
(74,4)*{(B)};
(74,-4)*{(C)};
(88,0)*{(D)};
(110,0)*{
 q^2 t^{-2} \;\;
\begin{tikzpicture}[anchorbase]
\draw[very thick, directed=.55] (0,0.5) [out=0,in=180] to (1.05,.35);
\draw[very thick, directed=.55] (0,1.1)to [out=0,in=180] (1.05,.95);
\draw[very thick] (1.35,1.25)  to [out=90,in=270] (1.2,1.6);
\draw[very thick] (1.2,0) [out=90,in=270] to (1.05,.35);
\draw[very thick, directed=.55] (1.35,.65) [out=0,in=180] to (2.4,.5);
\draw[double] (1.05,.95) to (1.35,1.25) ;
\draw[double]  (1.05,.35) to (1.35,.65) ;
\draw[very thick, directed=.55] (1.35,.65) to (1.05,.95);
\draw[very thick, directed=.55] (1.35,1.25) [out=0,in=180] to (2.4,1.1);
\torus{2.4}{1.6}
\end{tikzpicture}
};
(70,15)*{
 q t^{-1} \;\;
\begin{tikzpicture}[anchorbase]
\draw[very thick, directed=.55] (0,0.5) [out=0,in=180] to (1.05,.35);
\draw[very thick, directed=.55] (0,1.1)to [out=0,in=270] (1.2,1.6);
\draw[very thick] (1.2,0) [out=90,in=270] to (1.05,.35);
\draw[very thick, directed=.55] (1.35,.65) [out=0,in=180] to (2.4,.5);
\draw[double]  (1.05,.35) to (1.35,.65);
\draw[very thick, directed=.55] (1.35,.65) to [out=90,in=180] (2.4,1.1);
\torus{2.4}{1.6}
\end{tikzpicture}
};
(70,-15)*{
 q t^{-1} \;\;
\begin{tikzpicture}[anchorbase]
\draw[very thick, directed=.55] (0,0.5) [out=0,in=270] to (1.05,.95);
\draw[very thick, directed=.55] (0,1.1)to [out=0,in=180] (1.05,.95);
\draw[very thick] (1.35,1.25)  to [out=90,in=270] (1.2,1.6);
\draw[very thick,directed=.55] (1.2,0) [out=90,in=180] to (2.4,.5);
\draw[double] (1.05,.95) to (1.35,1.25) ;
\draw[very thick, directed=.55] (1.35,1.25) [out=0,in=180] to (2.4,1.1);
\torus{2.4}{1.6}
\end{tikzpicture}
};
(30,0)*{
\begin{tikzpicture}[anchorbase]
\draw[very thick, directed=.55] (0,0.5) [out=0,in=180] to (2.4,1.1);
\draw[very thick, directed=.55] (0,1.1) [out=0,in=270] to (1.2,1.6);
\draw[very thick, directed=.55] (1.2,0) [out=90,in=180] to (2.4,.5);
\torus{2.4}{1.6}
\end{tikzpicture}
};
(50,9)*{
\begin{tikzpicture}[anchorbase]
\draw[->] (0,0) to (1,.5);
\end{tikzpicture}
};
(50,-9)*{
\begin{tikzpicture}[anchorbase]
\draw[->] (0,0) to (1,-.5);
\end{tikzpicture}
};
(95,-9)*{
\begin{tikzpicture}[anchorbase]
\draw[->] (0,0) to (1,.5);
\end{tikzpicture}
};
(95,9)*{
\begin{tikzpicture}[anchorbase]
\draw[->] (0,0) to (1,-.5);
\end{tikzpicture}
};
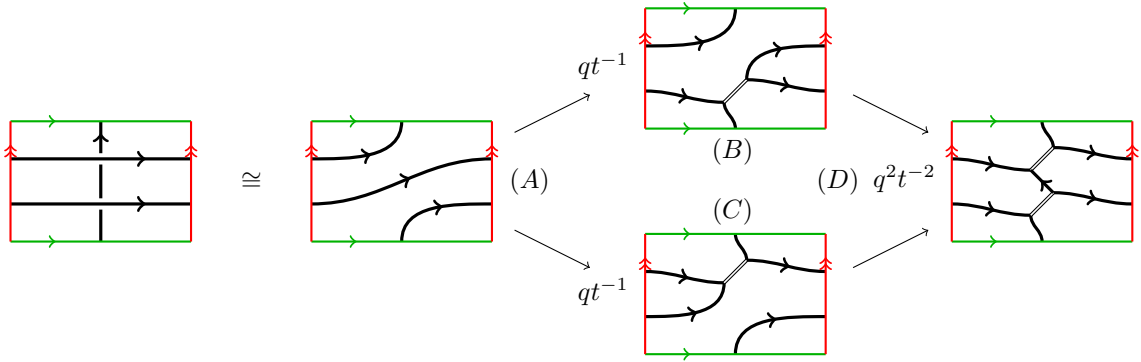
\endxy
\end{equation*}
This chain complex already contains the desired summands $(2,1)^F_T\cong (A)$ and $(2,-1)^F_T*\w{0,1}\cong (D)$. We thus need to prove that $T_2*\id$ acts as the identity on these chain groups, and by zero on $(B)\oplus (C)$.

Recall from Definition~\ref{def:extrwp} that the idempotent decomposes as 
\[T_2*\id=\id*\id + \frac{-u_1}{2}*\id + \frac{-u_2}{2}*\id\]
and we will compute the individual actions of these summands on the chain groups. The summand $\id*\id$ acts as the identity on all chain groups, and we claim that the other two summands act as follows: 

\begin{center}
\begin{tabular}{c | c  c c c}
$-u_1/2 * \id$ & $(A)$ & $(B)$ & $(C)$ & $(D)$\\
\hline 
 $(A)$ & $0$ &  &  &  \\
 $(B)$ & & $-\id/2$ & $\phi$ & \\
 $(C)$ & & $\psi$ & $-\id/2$ & \\ 
 $(D)$ & & & & $0$
\end{tabular}
\begin{tabular}{c | c  c c c}
$-u_2/2 * \id$ & $(A)$ & $(B)$ & $(C)$ & $(D)$\\
\hline 
 $(A)$ & $0$ & & & \\
 $(B)$ & & $-\id/2$ & $-\phi$ & \\
 $(C)$ & & $-\psi$ & $-\id/2$ & \\ 
 $(D)$ & & & & $0$
\end{tabular}
\end{center}

Here $\phi$ and $\psi$ are certain foams that are described in the proof of the claim. Note that the claim implies that the categorified Frohman--Gelca formula holds for $(2,0)^F_T*(0,1)^F_T$.

In order to prove the claim, we first explain why $-u_1/2*\id$ and $-u_2/2*\id$ act by zero on the chain groups $(A)$ and $(D)$. For this, note that these morphisms factor through the following superposition:
\begin{equation*}
(E):=\xy
(-10,0)*{
\begin{tikzpicture}[anchorbase]
\draw[very thick, directed=.9] (1.2,0) to (1.2,1.6);
\draw[white, line width=.15cm] (0,.8) to (2.4,.8);
\draw[double, directed=.75] (0,.8) to (2.4,.8);
\torus{2.4}{1.6}
\end{tikzpicture}
};
(10,0)*{=
};
(30,0)*{
 q t^{-1} \;\;
\begin{tikzpicture}[anchorbase]
\draw[very thick, directed=.55] (1,.8) to [out=90,in=270] (1.2,1.6);
\draw[very thick] (1,.8) to (1.4,.8);
\draw[very thick, directed=.55] (1.2,0) to [out=90,in=270] (1.4,.8);
\draw[double,directed=.55] (1.4,.8) to [out=0,in=180] (2.4,.8);
\draw[double,directed=.55] (0,.8) to [out=0,in=180] (1,.8);
\torus{2.4}{1.6}
\end{tikzpicture}
};
\endxy
\end{equation*}
which is supported in homological degree $-1$, but the chain groups $(A)$ and $(D)$ are of homological degree $0$ and $-2$.

Next we compute the maps between the chain groups $(B)$ and $(C)$ induced by $-u_1/2*\id$ and $-u_2/2*\id$. First we consider the chain maps represented by the following movie, which can be read right or left:

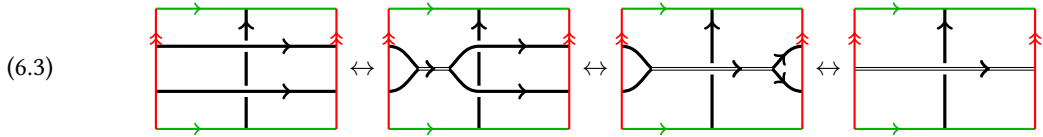
\begin{equation}
\label{eqn:u1} \begin{tikzpicture}[anchorbase]
\draw[very thick, directed=.9] (1.2,0) to (1.2,1.6);
\draw[white, line width=.15cm] (0,.5) to (2.4,.5);
\draw[white, line width=.15cm] (0,1.1)  to (2.4,1.1);
\draw[very thick, directed=.75] (0,0.5) to (2.4,.5);
\draw[very thick, directed=.75] (0,1.1)  to (2.4,1.1);
\torus{2.4}{1.6}
\end{tikzpicture} 
\leftrightarrow
\begin{tikzpicture}[anchorbase]
\draw[very thick, directed=.9] (1.2,0) to (1.2,1.6);
\draw[white, line width=.15cm] (0,.5) to (2.4,.5);
\draw[white, line width=.15cm] (0,1.1)  to (2.4,1.1);
\draw[very thick, directed=.75] (0,0.5) to [out=0,in=225](.4,.8) (.8,.8) to [out=315,in=180] (1.2,.5) to  (2.4,.5);
\draw[very thick, directed=.75] (0,1.1) to [out=0,in=135] (.4,.8) (.8,.8) to [out=45,in=180] (1.2,1.1) to  (2.4,1.1);
\draw[double,directed=.55] (.4,.8) to (.8,.8);
\torus{2.4}{1.6}
\end{tikzpicture}
\leftrightarrow
\begin{tikzpicture}[anchorbase]
\draw[very thick, directed=.9] (1.2,0) to (1.2,1.6);
\draw[white, line width=.15cm] (0,.8) to (2.4,.8);
\draw[very thick, directed=.75] (0,0.5) to [out=0,in=225](.4,.8) (2,.8) to [out=315,in=180] (2.4,.5);
\draw[very thick, directed=.75] (0,1.1) to [out=0,in=135] (.4,.8) (2,.8) to [out=45,in=180] (2.4,1.1);
\draw[double,directed=.75] (.4,.8) to (2,.8);
\torus{2.4}{1.6}
\end{tikzpicture}
\leftrightarrow
\begin{tikzpicture}[anchorbase]
\draw[very thick, directed=.9] (1.2,0) to (1.2,1.6);
\draw[white, line width=.15cm] (0,.8) to (2.4,.8);
\draw[double, directed=.75] (0,.8) to (2.4,.8);
\torus{2.4}{1.6}
\end{tikzpicture}
 \end{equation}
 Reading from left to right we first see a zip foam, then a fork slide, see Section~\ref{sec:forkslide}, and finally a digon collapse.
 
In terms of annular webs, we see $ \begin{tikzpicture}[anchorbase, scale=.2]
\draw (0,0) circle (1);
\draw (0,0) circle (3);
\draw [very thick] (.8,.6) to [out=45, in=315] (0,2);
\draw [very thick] (-.8,.6) to [out=135,in=225] (0,2);
\draw [double,->] (0,2) to (0,3);
\node at (0,-1) {$*$};
\node at (0,-3) {$*$};
\draw [dashed] (0,-1) to (0,-3);
\end{tikzpicture}* \id$
 and 
 $ \begin{tikzpicture}[anchorbase, scale=.2]
\draw (0,0) circle (1);
\draw (0,0) circle (3);
\draw [very thick,->] (0,1.5) to [out=45, in=225] (1.8,2.4);
\draw [very thick,->] (0,1.5) to [out=135,in=315] (-1.8,2.4);
\draw [double] (0,1) to (0,1.5);
\node at (0,-1) {$*$};
\node at (0,-3) {$*$};
\draw [dashed] (0,-1) to (0,-3);
\end{tikzpicture}* \id$ in the rightward and leftward compositions respectively. Going all the way to the right and then back to the left, we obtain the chain map induced by $u_1 * \id$. \\

We can now compute the action of the chain maps from \eqref{eqn:u1} on the object $(B)$.

\begin{equation}
\label{eqn:u1B}
(B)=\begin{tikzpicture}[anchorbase]
\draw[very thick, directed=.55] (0,0.5) [out=0,in=180] to (1.05,.35);
\draw[very thick, directed=.55] (0,1.1)to [out=0,in=270] (1.2,1.6);
\draw[very thick] (1.2,0) [out=90,in=270] to (1.05,.35);
\draw[very thick, directed=.55] (1.35,.65) [out=0,in=180] to (2.4,.5);
\draw[double]  (1.05,.35) to (1.35,.65);
\draw[very thick, directed=.55] (1.35,.65) to [out=90,in=180] (2.4,1.1);
\torus{2.4}{1.6}
\end{tikzpicture}
\leftrightarrow
\begin{tikzpicture}[anchorbase]
\draw[very thick, directed=.75] (0,0.5) to [out=0,in=225](.4,.8) (.8,.8) to [out=315,in=180] (1.05,.35);
\draw[very thick, directed=.75] (0,1.1) to [out=0,in=135] (.4,.8) (.8,.8) to [out=45,in=270] (1.2,1.6);
\draw[double,directed=.55] (.4,.8) to (.8,.8);
\draw[very thick] (1.2,0) [out=90,in=270] to (1.05,.35);
\draw[very thick, directed=.55] (1.35,.65) [out=0,in=180] to (2.4,.5);
\draw[double]  (1.05,.35) to (1.35,.65);
\draw[very thick, directed=.55] (1.35,.65) to [out=90,in=180] (2.4,1.1);
\torus{2.4}{1.6}
\end{tikzpicture} 
\leftrightarrow
\begin{tikzpicture}[anchorbase]
\draw[very thick, directed=.55] (1,.8) to [out=90,in=270] (1.2,1.6);
\draw[very thick] (1,.8) to (1.4,.8);
\draw[very thick, directed=.55] (1.2,0) to [out=90,in=270] (1.4,.8);
\draw[double,directed=.55] (1.4,.8) to [out=0,in=180] (2,.8);
\draw[double,directed=.55] (0.4,.8) to [out=0,in=180] (1,.8);
\draw[very thick, directed=.75] (0,0.5) to [out=0,in=225](.4,.8) (2,.8) to [out=315,in=180] (2.4,.5);
\draw[very thick, directed=.75] (0,1.1) to [out=0,in=135] (.4,.8) (2,.8) to [out=45,in=180] (2.4,1.1);
\torus{2.4}{1.6}
\end{tikzpicture} 
\leftrightarrow
\begin{tikzpicture}[anchorbase]
\draw[very thick, directed=.55] (1,.8) to [out=90,in=270] (1.2,1.6);
\draw[very thick] (1,.8) to (1.4,.8);
\draw[very thick, directed=.55] (1.2,0) to [out=90,in=270] (1.4,.8);
\draw[double,directed=.55] (1.4,.8) to [out=0,in=180] (2.4,.8);
\draw[double,directed=.55] (0,.8) to [out=0,in=180] (1,.8);
\torus{2.4}{1.6}
\end{tikzpicture}
=(E)
\end{equation}
Reading left-to-right, we first see a zip foam, then the morphism induced by the fork slide, which is just a local isotopy, and finally a digon closure. It is easy to see that the composite is an isotopy that slides the split-vertex rightward through the red gluing edge of the torus
. In the opposite direction, we see the inverse isotopy. This implies that the map $\frac{-1}{2}u_1*\id$ acts by the scalar $-1/2$ on the resolution $(B)$. Moreover, since, $u_2$ is obtained by conjugating $u_1$ by a half-rotation of the torus, $-u_2/2*\id$ also acts by the scalar $-1/2$ on $(B)$.\\

Next, we look at the action of the chain maps from \eqref{eqn:u1} on the object $(C)$.
\begin{equation}
\label{eqn:u1C}
(C)=
\begin{tikzpicture}[anchorbase]
\draw[very thick, directed=.55] (0,0.5) [out=0,in=270] to (1.05,.95);
\draw[very thick, directed=.55] (0,1.1)to [out=0,in=180] (1.05,.95);
\draw[very thick] (1.35,1.25)  to [out=90,in=270] (1.2,1.6);
\draw[very thick,directed=.55] (1.2,0) [out=90,in=180] to (2.4,.5);
\draw[double] (1.05,.95) to (1.35,1.25) ;
\draw[very thick, directed=.55] (1.35,1.25) [out=0,in=180] to (2.4,1.1);
\torus{2.4}{1.6}
\end{tikzpicture}
\leftrightarrow
\begin{tikzpicture}[anchorbase]
\draw[very thick, directed=.75] (0,0.5) to [out=0,in=225](.3,.8) (.6,.8) to [out=315,in=270] (1.05,.95);
\draw[very thick, directed=.75] (0,1.1) to [out=0,in=135] (.3,.8) (.6,.8) to [out=45,in=180] (1.05,.95);
\draw[double] (.3,.8) to (.6,.8);
\draw[very thick] (1.35,1.25)  to [out=90,in=270] (1.2,1.6);
\draw[very thick,directed=.55] (1.2,0) [out=90,in=180] to (2.4,.5);
\draw[double] (1.05,.95) to (1.35,1.25) ;
\draw[very thick, directed=.55] (1.35,1.25) [out=0,in=180] to (2.4,1.1);
\torus{2.4}{1.6}
\end{tikzpicture}
\leftrightarrow
\begin{tikzpicture}[anchorbase]
\draw[very thick, directed=.55] (1,.8) to [out=90,in=270] (1.2,1.6);
\draw[very thick] (1,.8) to (1.4,.8);
\draw[very thick, directed=.55] (1.2,0) to [out=90,in=270] (1.4,.8);
\draw[double,directed=.55] (1.4,.8) to [out=0,in=180] (2,.8);
\draw[double,directed=.55] (0.4,.8) to [out=0,in=180] (1,.8);
\draw[very thick, directed=.75] (0,0.5) to [out=0,in=225](.4,.8) (2,.8) to [out=315,in=180] (2.4,.5);
\draw[very thick, directed=.75] (0,1.1) to [out=0,in=135] (.4,.8) (2,.8) to [out=45,in=180] (2.4,1.1);
\torus{2.4}{1.6}
\end{tikzpicture} 
\leftrightarrow
\begin{tikzpicture}[anchorbase]
\draw[very thick, directed=.55] (1,.8) to [out=90,in=270] (1.2,1.6);
\draw[very thick] (1,.8) to (1.4,.8);
\draw[very thick, directed=.55] (1.2,0) to [out=90,in=270] (1.4,.8);
\draw[double,directed=.55] (1.4,.8) to [out=0,in=180] (2.4,.8);
\draw[double,directed=.55] (0,.8) to [out=0,in=180] (1,.8);
\torus{2.4}{1.6}
\end{tikzpicture}
=(E)
\end{equation}
Reading left-to-right, we first see a zip, then an interesting negated foam appearing in the fork slide \eqref{eqn:forkslidefoam}, followed by the digon collapse. However, the fork slide foam is itself built from a digon opening and a zip morphisms, which cancel with the other two present foams. The total map is the negative of an isotopy, which slides the merge vertex left through the red gluing edge of the torus. Reading right-to-left, we obtain the inverse map. As before, this shows that $-u_1/2*\id$ and $-u_2/2*\id$ each act by the scalar $-1/2$ on the resolution $(C)$.

Next, we have to compute the components of $u_1$ mapping $(B)\to(C)$ and $(C)\to(B)$. For the first one, we compose the left-to-right map from \eqref{eqn:u1B} with the right-to-left map from \eqref{eqn:u1C}. The result is the negative of the isotopy that moves the entire double edge in $(B)$ rightwards through the red gluing edge of the torus. 
 Similarly, the component $(C)\to (B)$ is given by the negative of the leftward isotopy.

Finally, the components of $u_2*\id$ are obtained by conjugating the action of $u_1*\id$ by a half-rotation of the torus. Note however, that this interchanges $(B)$ and $(C)$. The $u_2*\id$ component $(B)\to (C)$ is the negative of the isotopy which moves the double edge in $(B)$ first down and through the green gluing edge of the torus into position $(C)$, then leftward to position $(B)$ and then again upward (but not passing the gluing edge) to position $(C)$. 
Similarly, the $u_2*\id$ component $(C)\to (B)$ is given by the negative of the isotopy sliding the double edge first down (not passing gluing edges), then rightward, and finally up through the green gluing edge.
\begin{center}
\begin{tabular}{c | c | c}
~ & $(B)\to (C)$ & $(C) \to (B)$\\
\hline
 $u_1$ & right & left\\
\hline
 $u_2$& down and left & up and right 
\end{tabular}
\end{center}
In order to see that the off-diagonal contributions of $u_1*\id$ and $u_2*\id$ cancel, we argue that the isotopy from $(B)$ to itself, which moves the double edge down once and left twice, is the negative of the identity; and similarly for the isotopy from $(C)$ to itself, which moves the double edge up once and right twice.

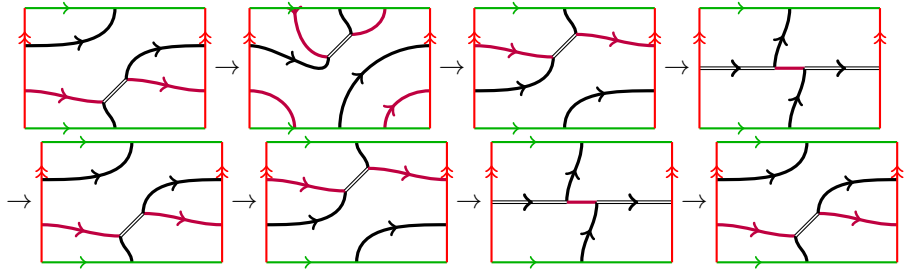
\begin{gather*}
\begin{tikzpicture}[anchorbase]
\draw[very thick, purple, directed=.55] (0,0.5) [out=0,in=180] to (1.05,.35);
\draw[very thick, directed=.55] (0,1.1)to [out=0,in=270] (1.2,1.6);
\draw[very thick] (1.2,0) [out=90,in=270] to (1.05,.35);
\draw[very thick,purple, directed=.55] (1.35,.65) [out=0,in=180] to (2.4,.5);
\draw[double]  (1.05,.35) to (1.35,.65);
\draw[very thick, directed=.55] (1.35,.65) to [out=90,in=180] (2.4,1.1);
\torus{2.4}{1.6}
\end{tikzpicture}
\to
\begin{tikzpicture}[anchorbase]
\draw[very thick,purple, directed=.55] (0,0.5) [out=0,in=90] to (.6,0) (.6,1.6) to [out=270,in=180] (1.05,.95);
\draw[very thick, directed=.55] (0,1.1)to [out=0,in=270] (1.05,.95) ;
\draw[very thick,purple, directed=.65] (1.35,1.25) to [out=0,in=270](1.8,1.6)(1.8,0) [out=90,in=180] to (2.4,.5);
\draw[double] (1.05,.95) to (1.35,1.25) ;
\draw[very thick, directed=.55] (1.35,1.25) to [out=90,in=270](1.2,1.6) (1.2,0) to [out=90,in=180] (2.4,1.1);
\torus{2.4}{1.6}
\end{tikzpicture}
\to
\begin{tikzpicture}[anchorbase]
\draw[very thick, directed=.55] (0,0.5) [out=0,in=270] to (1.05,.95);
\draw[very thick,purple, directed=.55] (0,1.1)to [out=0,in=180] (1.05,.95);
\draw[very thick] (1.35,1.25)  to [out=90,in=270] (1.2,1.6);
\draw[very thick,directed=.55] (1.2,0) [out=90,in=180] to (2.4,.5);
\draw[double] (1.05,.95) to (1.35,1.25) ;
\draw[very thick,purple, directed=.55] (1.35,1.25) [out=0,in=180] to (2.4,1.1);
\torus{2.4}{1.6}
\end{tikzpicture}
\to
\begin{tikzpicture}[anchorbase]
\draw[very thick, directed=.55] (1,.8) to [out=90,in=270] (1.2,1.6);
\draw[very thick,purple] (1,.8) to (1.4,.8);
\draw[very thick, directed=.55] (1.2,0) to [out=90,in=270] (1.4,.8);
\draw[double,directed=.55] (1.4,.8) to [out=0,in=180] (2.4,.8);
\draw[double,directed=.55] (0,.8) to [out=0,in=180] (1,.8);
\torus{2.4}{1.6}
\end{tikzpicture}\\
\to
\begin{tikzpicture}[anchorbase]
\draw[very thick,purple, directed=.55] (0,0.5) [out=0,in=180] to (1.05,.35);
\draw[very thick, directed=.55] (0,1.1)to [out=0,in=270] (1.2,1.6);
\draw[very thick] (1.2,0) [out=90,in=270] to (1.05,.35);
\draw[very thick,purple, directed=.55] (1.35,.65) [out=0,in=180] to (2.4,.5);
\draw[double]  (1.05,.35) to (1.35,.65);
\draw[very thick, directed=.55] (1.35,.65) to [out=90,in=180] (2.4,1.1);
\torus{2.4}{1.6}
\end{tikzpicture}
\to
\begin{tikzpicture}[anchorbase]
\draw[very thick, directed=.55] (0,0.5) [out=0,in=270] to (1.05,.95);
\draw[very thick,purple, directed=.55] (0,1.1)to [out=0,in=180] (1.05,.95);
\draw[very thick] (1.35,1.25)  to [out=90,in=270] (1.2,1.6);
\draw[very thick,directed=.55] (1.2,0) [out=90,in=180] to (2.4,.5);
\draw[double] (1.05,.95) to (1.35,1.25) ;
\draw[very thick,purple, directed=.55] (1.35,1.25) [out=0,in=180] to (2.4,1.1);
\torus{2.4}{1.6}
\end{tikzpicture}
\to
\begin{tikzpicture}[anchorbase]
\draw[very thick, directed=.55] (1,.8) to [out=90,in=270] (1.2,1.6);
\draw[very thick,purple] (1,.8) to (1.4,.8);
\draw[very thick, directed=.55] (1.2,0) to [out=90,in=270] (1.4,.8);
\draw[double,directed=.55] (1.4,.8) to [out=0,in=180] (2.4,.8);
\draw[double,directed=.55] (0,.8) to [out=0,in=180] (1,.8);
\torus{2.4}{1.6}
\end{tikzpicture}
\to
\begin{tikzpicture}[anchorbase]
\draw[very thick,purple, directed=.55] (0,0.5) [out=0,in=180] to (1.05,.35);
\draw[very thick, directed=.55] (0,1.1)to [out=0,in=270] (1.2,1.6);
\draw[very thick] (1.2,0) [out=90,in=270] to (1.05,.35);
\draw[very thick,purple, directed=.55] (1.35,.65) [out=0,in=180] to (2.4,.5);
\draw[double]  (1.05,.35) to (1.35,.65);
\draw[very thick, directed=.55] (1.35,.65) to [out=90,in=180] (2.4,1.1);
\torus{2.4}{1.6}
\end{tikzpicture}
\end{gather*}
Here we have color-coded the $1$-labeled edges for increased clarity. Additionally, we can interpret this foam as a superposition as follows:
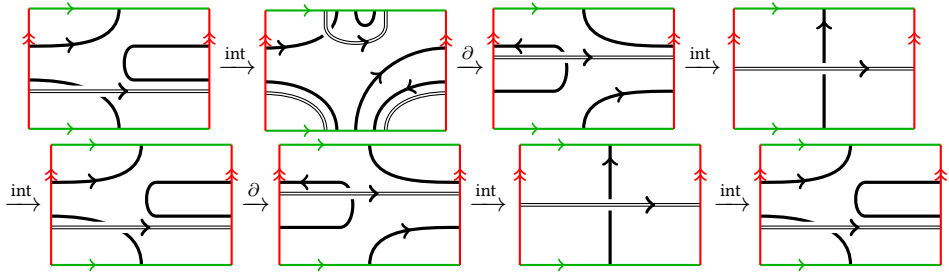
\begin{gather*}
\begin{tikzpicture}[anchorbase]
\draw[very thick,directed=.85] (1.2,0) [out=90,in=0] to (0,.65) (2.4,.65) to [out=180,in=0] (1.4,.65) to [out=180,in=180] (1.4,1.1) to [out=0,in=180] (2.4,1.1) (0,1.1) to [out=0,in=270] (1.2,1.6);
\draw[white, line width=.15cm]  (0.4,.5) to (2,.5);
\draw[double, directed=.55]  (0,.5) to (2.4,.5);
\torus{2.4}{1.6}
\end{tikzpicture}
\xrightarrow{\text{int}}
\begin{tikzpicture}[anchorbase]
\draw[very thick,directed=.30, directed=.65, directed=.85] (.95,0) to [out=90,in=0] (0,.65) (2.4,.65) to [out=180,in=90] (1.45,0) (1.45,1.6) to [out=270,in=0] (1.325,1.4) to [out=180,in=270] (1.2,1.6) (1.2,0) to [out=90,in=180]  (2.4,1.1) (0,1.1) to [out=0,in=270] (.95,1.6);
\draw[white, line width=.15cm]  (.8,1.6) to [out=270,in=90] (.8,1.4) to [out=270,in=270] (1.6,1.4);
\draw[double, directed=.55]  (0,.5) to [out=0,in=90] (.8,0) (.8,1.6) to [out=270,in=90] (.8,1.4) to [out=270,in=270] (1.6,1.4) to [out=90,in=270]  (1.6,1.6) (1.6,0) to [out=90,in=180] (2.4,.5);
\torus{2.4}{1.6}
\end{tikzpicture}
\xrightarrow{\partial}
\begin{tikzpicture}[anchorbase]
\draw[very thick, directed=.55] (0,0.5) to (.8,.5) to [out=0,in=0] (.8,1.1) to (0,1.1) (2.4,1.1) to [out=180,in=270] (1.2,1.6);
\draw[very thick,directed=.55] (1.2,0) [out=90,in=180] to (2.4,.5);
\draw[white, line width=.15cm]  (0.4,.95) to (2,.95);
\draw[double, directed=.55]  (0,0.95) to (2.4,0.95);
\torus{2.4}{1.6}
\end{tikzpicture}
\xrightarrow{\text{int}}
\begin{tikzpicture}[anchorbase]
\draw[very thick, directed=.9] (1.2,0) to (1.2,1.6);
\draw[white, line width=.15cm] (0,.8) to (2.4,.8);
\draw[double, directed=.75] (0,.8) to (2.4,.8);
\torus{2.4}{1.6}
\end{tikzpicture}
\\
\xrightarrow{\text{int}}
\begin{tikzpicture}[anchorbase]
\draw[very thick,directed=.85] (1.2,0) [out=90,in=0] to (0,.65) (2.4,.65) to [out=180,in=0] (1.4,.65) to [out=180,in=180] (1.4,1.1) to [out=0,in=180] (2.4,1.1) (0,1.1) to [out=0,in=270] (1.2,1.6);
\draw[white, line width=.15cm]  (0.4,.5) to (2,.5);
\draw[double, directed=.55]  (0,.5) to (2.4,.5);
\torus{2.4}{1.6}
\end{tikzpicture}
\xrightarrow{\partial}
\begin{tikzpicture}[anchorbase]
\draw[very thick, directed=.55] (0,0.5) to (.8,.5) to [out=0,in=0] (.8,1.1) to (0,1.1) (2.4,1.1) to [out=180,in=270] (1.2,1.6);
\draw[very thick,directed=.55] (1.2,0) [out=90,in=180] to (2.4,.5);
\draw[white, line width=.15cm]  (0.4,.95) to (2,.95);
\draw[double, directed=.55]  (0,0.95) to (2.4,0.95);
\torus{2.4}{1.6}
\end{tikzpicture}
\xrightarrow{\text{int}}
\begin{tikzpicture}[anchorbase]
\draw[very thick, directed=.9] (1.2,0) to (1.2,1.6);
\draw[white, line width=.15cm] (0,.8) to (2.4,.8);
\draw[double, directed=.75] (0,.8) to (2.4,.8);
\torus{2.4}{1.6}
\end{tikzpicture}
\xrightarrow{\text{int}}
\begin{tikzpicture}[anchorbase]
\draw[very thick,directed=.85] (1.2,0) [out=90,in=0] to (0,.65) (2.4,.65) to [out=180,in=0] (1.4,.65) to [out=180,in=180] (1.4,1.1) to [out=0,in=180] (2.4,1.1) (0,1.1) to [out=0,in=270] (1.2,1.6);
\draw[white, line width=.15cm]  (0.4,.5) to (2,.5);
\draw[double, directed=.55]  (0,.5) to (2.4,.5);
\torus{2.4}{1.6}
\end{tikzpicture}
\end{gather*}
Here we see the superposition of a rightward wrap on $\w{1,0}$ and two leftward wraps of $(1,0)$. Both of these wrap morphisms can be written as the superposition of identity morphisms with $2$-labeled essential tori, which cancel each other, since they have opposite orientations. However, in expanding the double wrap on $(1,0)$, we obtain an extra factor of $-1$, see \eqref{eqn:doublewrap}.

This shows that the ``down once and left twice'' isotopy foam agrees with the negated identity foam. This implies that $-u_1/2*\id$ and $-u_2/2*\id$ give cancelling contributions to the $(B)\to (C)$ component of the chain map. The case of the $(C)\to (B)$ component is completely analogous. This proves the claim.

\settocdepth{section}
\section{Appendix A: Proof of the Frohman--Gelca formula}
\label{sec:FGfla}
The purpose of this section is to prove the $\glnn{2}$ version of the Frohman--Gelca formula in $\TWebq$:
\[(m,n)_T*(r,s)_T = (m+r, n+s)_T + (m-r,n-s)_T*\w{r,s}\]

The proof is analogous to the one for the $\slnn{2}$-case given in \cite{FG}.

\begin{definition} The geometric intersection number of $(m,n)_T$ and $(r,s)_T$ is defined to be $|m s-n r|$.
\end{definition}

The proof of the Frohman-Gelca formula proceeds by induction on the geometric intersection number of the two basis elements to be multiplied. The base of the induction is given by the cases of intersection number zero and one. The induction step is split into two cases, depending on whether $(m,n)_T$ and $(r,s)_T$ are both just slopes (when $(m,n)$ and $(r,s)$ are relatively prime) or not.

\subsection{Induction base: intersection number zero and one}
In this case of intersection number zero, $(m,n)_T$ and $(r,s)_T$ live in the same slope, i.e. are of the form $(m,n)_T=(ka,kb)_T$ and $(r,s)_T=(la,lb)_T$ for $\gcd(a,b)=1$ and $k,l\in \Z$

First we assume that they are coherently oriented, i.e. $k$ and $l$ have the same sign, which may assume to be positive.

\begin{lemma} The Frohman-Gelca formula holds for intersection number zero and parallel orientations.
\end{lemma}
\begin{proof}
For $\gcd(a,b)=1$, the elements $(ka,kb)*\w{la,lb}$ for $k,l\in \Z_{\geq 0}$ generate a $\Z[q^{\pm 1}]$-subalgebra of $\TWebq$ isomorphic to the polynomial ring $\Z[q^{\pm 1}][(a,b), \w{a,b}]$. This can be further identified with the symmetric polynomial ring $\Z[q^{\pm 1}][x_1,x_2]^{\mathcal{S}_2}$ via the isomorphism that sends: 
\[(a,b)\mapsto x_1+x_2, \quad \w{a,b}\mapsto x_1 x_2, \quad (ka,kb)_T\mapsto x_1^k + x_2^k.\] In particular, for $k\geq l$ the equality $(x_1^l + x_2^l)
= (x_1^{k+l} + x_2^{k+l})+ x_1^lx_2^l(x_1^{k-l} + x_2^{k-l})$ implies:
\[
(ka,kb)_T*(la,lb)_T= ((k+l)a,(k+l)b)_T+ ((k-l)a,(k-l)b)_T*\w{la,lb} 
\]
For $k<l$ we use commutativity and the previous expansion:
\[(ka,kb)_T*(la,lb)_T = (la,lb)_T*(ka,kb)_T = ((k+l)a,(k+l)b)_T+ ((l-k)a,(l-k)b)_T*\w{ka,kb} \]
Now note that $((l-k)a,(l-k)b)_T*\w{ka,kb} = ((k-l)a,(k-l)b)_T*\w{la,lb}$
\end{proof}

Now suppose that $(m,n)_T$ and $(r,s)_T$ have intersection number zero, but are oppositely oriented. Then we write $(r,s)_T=(-r,-s)_T*\w{r,s}$. With the same notation as above, we multiply as follows:
\begin{align*}(m,n)_T*(r,s)_T &= (m,n)_T*(-r,-s)_T*\w{r,s} 
\\
&= (m-r,n-s)_T*\w{r,s} + (m+r,n+s)_T*\w{-r,-s}*\w{r,s} 
\\
&= (m+r,n+s)_T + (m-r,n-s)_T*\w{r,s}	
\end{align*}

 If the intersection number is one, this implies $\gcd(m,n)=\gcd(r,s)=1$ and the product formula holds by a trivial computation similar to Example~\ref{exa:1001}. 

This settles the base of the induction.

\subsection{Induction step, first case: multiplication of slopes}
\label{sec:case1}
Assume that $\gcd(m,n)=\gcd(r,s)=1$. In this case we drop the subscript and write $(m,n)$ and $(r,s)$ for the objects, which are just single copies of the slope. We need to show
\[
 (m,n)*(r,s) =(m+r,n+s)_T + (m-r,n-s)_T*\w{r,s} 
\]
but after applying a homeomorphism of the torus, it is sufficient to prove this in the special cases
\begin{equation}\label{eq:caseone}
 (m,n)*(0,1) = (m,n+1)_T + (m,n-1)_T*\w{0,1} 
\end{equation}
where $0\leq n < m$ and we still have $\gcd(m,n)=1$.\footnote{Here we need to allow orientation reversing homeomorphisms, or also consider the cases $m<n\leq 0$, which are analogous.} For intersection number $1$ we have $(m,n)=(1,0)$ and this case has already been settled. For intersection number $2$, there is also only one case $(m,n)=(2,1)$ which was verified in Example~\ref{exa:2101}. For the induction step we assume that the Frohman-Gelca formula already holds for intersection numbers less than $m\geq 3$.

\begin{lemma}For $0 <n<m\geq 3$ with $\gcd(m,n)=1$, there exist positive integers $u$, $w$ and $z$, and $v\geq 0$, such that $u+w=m$, $v+z=n$, $|u z - v w|=1$ and $w<m$ and $u<m-1$.
\end{lemma}
\begin{proof}  This is \cite[Lemma 4.2]{FG}, but mind the typo there: $v$ and $z$ cannot both be positive for $n=1$.
\end{proof} 
Note that the condition on the determinant implies $\gcd(u,v)=\gcd(w,z)=\gcd(u-w,v-z)=1$. In order to verify \eqref{eq:caseone}, we expand the product $(u,v)*(w,z)*(0,1)$ in two ways. Multiplying the first two factors first, we get the following expansion for $(u,v)*(w,z)*(0,1)$:

\begin{align*}
=& (m,n)*(0,1) + (u-w,v-z)*\w{w,z}  *(0,1)
\\
=& (m,n)*(0,1) + q^{2w} (u-w,v-z)*(0,1)*\w{w,z}\\
=& (m,n)*(0,1) + q^{2w} (u-w,v-z+1)_T*\w{w,z} + q^{2w} (u-w,v-z-1)_T*\w{0,1}*\w{w,z}\\
=& (m,n)*(0,1) + q^{2w} (u-w,v-z+1)_T*\w{w,z} + (u-w,v-z-1)_T*\w{w,z+1}
\end{align*}
Here we have used the Frohman-Gelca formula twice, in cases of intersection number $1$ and $|u-w|<m$. We have also commuted $(0,1)$ and $\w{w,z}$ past each other at the expense of a $q$-shift and collapsed $\w{0,1}*\w{w,z}=q^{-2w}\w{w,z+1}$. Alternatively we can first multiply the second two factors in $(u,v)*(w,z)*(0,1)$ to get the expansion:
 
\begin{align*}
=&(u,v)*\left((w,z+1)_T +(w,z-1)_T*\w{0,1}\right )  \\
=&(u+w,v+z+1)_T+(u-w,v-z-1)_T*\w{w,z+1} 
\\~&+(u+w,v+z-1)_T*\w{0,1} + (u-w,v-z+1)_T*\w{w,z-1} *\w{0,1} 
\\=&
(m,n+1)_T+(m,n-1)_T*\w{0,1}+ q^{2w} (u-w,v-z+1)_T*\w{w,z} + (u-w,v-z-1)_T*\w{w,z+1}
\end{align*}

Here we have applied the Frohman-Gelca formula to $(w,z)*(0,1)$, which has intersection number $w<m$ and to $(u,v)*(w,z\pm 1)_T$ which have intersection number  $|u z \pm u -v w| = u\pm 1<m$. Comparing the two expansions confirms that \eqref{eq:caseone} holds.

\subsection{Induction step, second case: multiplication of general basis elements}
\label{sec:case2}
Now we deal with the case of $(m,n)_T$ and $(r,s)_T$, at least one of which is not relatively prime. Assume that $\gcd(m,n)=d\geq 2$ and write $(a,b)=(m/d,n/d)$ for the corresponding slope. The proof of \eqref{eq:FG} proceeds by computing $(a,b)*(m-a,n-b)_T*(r,s)_T$ in two ways. Multiplying the first two factors first, we get the expansion:

\begin{align*}
 =&  (m,n)_T*(r,s)_T + (2a-m,2b-n)_T*\w{m-a,n-b}*(r,s)_T
\\
=&  (m,n)_T*(r,s)_T + q^{2(d-1)(a s - b r)} (2a-m,2b-n)_T*(r,s)_T*\w{m-a,n-b}
\\
=& (m,n)_T*(r,s)_T + (2a-m-r,2b-n-s)_T*\w{m-a+r,n-b+s}  \\
~&+ q^{2(d-1)(a s - b r)}(2a-m+r,2b-n+s)_T*\w{m-a,n-b}
\end{align*}
Here we have used that  $|(2-d)a s - (2-d)b  r|  = |(2-d)(a s- b r)|<|ms-rn|$, so we can  multiply $(2a-m,2b-n)_T*(r,s)_T$ via the Frohman-Gelca formula by the induction hypothesis. For the same reason we can multiply the other two factors first and then expand further to get:
\begin{align*}
=& (a,b)_T*\left( (m-a+r,n-b+s)_T + (m-a-r,n-b-s)_T*\w{r,s} \right)  
\\
=& (m+r, n+s)_T + (m - r, n-s)_T*\w{r,s} + (2a-m-r,2b-n-s)_T*\w{m-a+r,n-b+s}  \\
~&+ q^{2(d-1)(a s - b r)}(2a-m+r,2b-n+s)_T*\w{m-a,n-b}
\end{align*}
Here we have used that the intersection numbers of the products in the first line are $|a((d-1)b+s) - b((d-1)a+r)|=|a s- b r|$, permitting the use of the induction hypothesis.

Comparing the two expansions of $(a,b)*(m-a,n-b)_T*(r,s)_T$ confirms \eqref{eq:FG} in the present case.

\begin{remark} \label{rem:catproof} 
Note that the induction scheme employed in this proof immediately extends to a proof of the categorified Frohman--Gelca formula provided that $*$ actually gives a monoidal structure on $\Kar(\essTfoam)_{\mathrm{dg}}$. 
\end{remark}
The basic ingredients for such a proof are:
\begin{enumerate}
\item The induction base cases of intersection number $0$ and $1$, which were checked in Examples~\ref{exa:1001-cat}, \ref{exa:2101-cat} and \ref{exa:slope-cat}.
\item Associativity isomorphisms for triple tensor products and distributivity isomorphisms for tensor products of direct sums.
\item The Krull-Schmidt property that allows us to conclude the existence of an isomorphism $A\cong B$ whenever $A\oplus C \cong B\oplus C$. 
\end{enumerate}

\section{Appendix B: Reidemeister and fork slide chain maps}
Here we explicitly describe the chain maps associated to Reidemeister II moves and certain fork slide moves. In doing so, we encode signs appearing in cube of resolutions chain complexes via the following formalism without choosing a particular order of crossings. Let $\mathrm{cr}=\{c,c^\prime, \dots\}$ be the set of crossings in the tangle diagram that involve a $1$-labeled edge and $\mathrm{Cl}(\mathrm{cr})$ the Clifford ring with generators $c\in \mathrm{cr}$ and relations $c^2=1$ and $c c^\prime+c ^\prime c=0$ if $c \neq c^\prime$.

To a web, which has been resolved with a double edge at the $1-1$-crossings $c_1,\dots, c_v$, and with the unique web at each $1-2$- or $2-1$-labeled crossing $C_1,\dots, C_w$, we associate the direct summand $\Z\la c_1 \cdots c_v\cdot C_1 \cdots C_w\ra$ of $\mathrm{Cl}(\mathrm{cr})$, which is (non-canonically) isomorphic to $\Z$. A component of the differential coming from the $1-1$-crossing $c$ is now additionally assigned the action of right-multiplication $r_c$ by $c$ in $\mathrm{Cl}(\mathrm{cr})$. In composing such differentials one now also multiplies the corresponding generators in the Clifford ring. Its defining relations now guarantee that squares of differentials in the tensor product multi-complex anti-commute. 

\subsection{Reidemeister II}
\label{sec:Reidemeister}

Consider the chain complex associated to the complicated side of a Reidemeister II move between strands of parallel orientation:

\begin{gather*}
\xy
(15,20)*{
=
};
(6.5,23)*{
c^\prime
};
(-1,22.5)*{
c
};
(0,20)*{
\Kh\left(
\begin{tikzpicture}[anchorbase, scale=.5]
\draw [very thick, ->] (3,0) to [out=180,in=0] (1.5,1) to [out=180,in=1]
(0,0);
 \draw [white,line width=.15cm] (3,1) to [out=180,in=0] (1.5,0) to [out=180,in=0] 
(0,1);
\draw [very thick, ->] (3,1) to [out=180,in=0] (1.5,0) to [out=180,in=0]  (0,1);
\end{tikzpicture}
\right)
};
(35,20)*{
\begin{tikzpicture}[anchorbase, scale=.5]
\draw [very thick] (3,0) to (2.9,0) to [out=180,in=315] (2.4,0.5);
\draw [very thick] (3,1) to (2.9,1) to[out=180,in=45] (2.4,0.5);
\draw [double]  (2.4,0.5) -- (2,0.5);
\draw [very thick, ->] (2,0.5) to [out=135,in=0] (1.5,1) to (0,1);
\draw [very thick, ->] (2,0.5) to [out=225,in=0]  (1.5,0) to (0,0);
\end{tikzpicture}
\otimes \Z\la c^\prime\ra
};
(80,15)*{
\begin{tikzpicture}[anchorbase, scale=.5]
\draw [very thick, ->] (3,0) -- (0,0);
\draw [very thick, ->] (3,1) -- (0,1);
\end{tikzpicture}
\otimes \Z\la 1 \ra
};
(80,25)*{
\begin{tikzpicture}[anchorbase, scale=.5]
\draw [very thick] (3,0) to (2.9,0) to [out=180,in=315] (2.4,0.5);
\draw [very thick] (3,1) to (2.9,1) to [out=180,in=45] (2.4,0.5);
\draw [double]  (2.4,0.5) -- (2,0.5);
\draw [very thick] (2,0.5) to [out=135,in=45] (1,0.5);
\draw [very thick] (2,0.5) to [out=225,in=315] (1,0.5);
\draw [double]  (1,0.5) -- (.6,0.5);
\draw [very thick, ->] (.6,.5) to [out=135,in=0] (.1,1) to (0,1);
\draw [very thick, ->] (.6,.5) to [out=225,in=0](.1,0) to(0,0);
\end{tikzpicture}
\otimes \Z\la c\cdot c^\prime\ra
};
(130,20)*{
\begin{tikzpicture}[anchorbase, scale=.5]
\draw [very thick] (3,0) to (1.5,0)to [out=180,in=315] (1,0.5);
\draw [very thick] (3,1) to (1.5,1) to [out=180,in=45] (1,0.5);
\draw [double]  (1,0.5) -- (.6,0.5);
\draw [very thick, ->] (.6,.5) to [out=135,in=0] (.1,1) to (0,1);
\draw [very thick, ->] (.6,.5) to [out=225,in=0](.1,0) to(0,0);
\end{tikzpicture}
\otimes \Z\la c\ra
};
(55,24)*{
\begin{tikzpicture}[scale=.5]
\draw [->] (0,0) to (3,.5);
\end{tikzpicture}
};
(55,27)*{
\textrm{zip} \otimes r_c
};
(55,16)*{
\begin{tikzpicture}[scale=.5]
\draw [->] (0,0) to (3,-.5);
\end{tikzpicture}
};
(55,13)*{
\textrm{unzip} \otimes r_{c^\prime}
};
(105,24)*{
\begin{tikzpicture}[scale=.5]
\draw [->] (0,0) to (3,-.5);
\end{tikzpicture}
};
(105,27)*{
\textrm{unzip} \otimes r_{c^\prime}
};
(105,16)*{
\begin{tikzpicture}[scale=.5]
\draw [->] (0,0) to (3,.5);
\end{tikzpicture}
};
(105,13)*{
\textrm{zip} \otimes r_c
};
\endxy
\end{gather*}
This complex is chain homotopy equivalent to the complex consisting of a single web consisting of two parallel strands, concentrated in homological and q-degree zero. The homotopy equivalences are given by identity foams connecting the parallel strand webs in both complexes, as well as the following more complicated composite foam and its reflection:

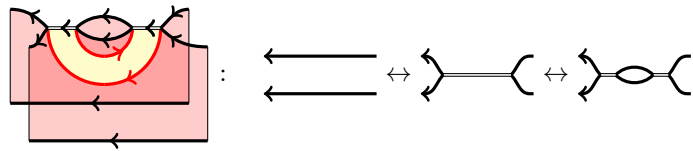
\begin{gather*}
\begin{tikzpicture} [anchorbase,scale=.5,fill opacity=0.2]
	\path [fill=red]  (.75,2.5) to [out=270,in=180] (1.5,1.75) to [out=0,in=270] 	(2.25,2.5) to [out=135,in=45](.75,2.5);
	\path [fill=red]  (.75,2.5) to [out=270,in=180] (1.5,1.75) to [out=0,in=270] 	(2.25,2.5) to [out=225,in=315](.75,2.5);
	\path [fill=red] (4.25,2) to (4.25,-.5) to (-.5,-.5) to (-.5,2) to
		[out=0,in=225] (0,2.5) to [out=270,in=180] (1.5,1) to [out=0,in=270] 
			(3,2.5) to [out=315,in=180] (4.25,2);
	\path [fill=red] (3.75,3) to (3.75,.5) to(-1,.5) to (-1,3) to [out=0,in=135]
		(0,2.5) to [out=270,in=180] (1.5,1) to [out=0,in=270] 
			(3,2.5) to [out=45,in=180] (3.75,3);
	\path[fill=yellow] (2.25,2.5) to [out=270,in=0] (1.5,1.75) to [out=180,in=270] (.75,2.5) to (0,2.5) to [out=270,in=180] (1.5,1) to [out=0,in=270] (3,2.5) to  (2.25,2.5) ;
	\draw [very thick,directed=.55] (4.25,-.5) to  (-.5,-.5);
	\draw [very thick, directed=.55] (3.75,.5) to  (-1,.5);
	\draw [very thick, red, directed=.75] (3,2.5) to [out=270,in=0] (1.5,1);
	\draw [very thick, red] (1.5,1) to [out=180,in=270] (0,2.5);
	\draw [very thick, red, rdirected=.75] (2.25,2.5) to [out=270,in=0] (1.5,1.75);
	\draw [very thick, red] (1.5,1.75) to [out=180,in=270] (.75,2.5);
	\draw  (3.75,3) to (3.75,.5);
	\draw (4.25,2) to (4.25,-.5);
	\draw (-1,3) to (-1,.5);
	\draw  (-.5,2) to (-.5,-.5);
	\draw [double,directed=.55] (3,2.5) to (2.25,2.5);
	\draw [double,directed=.55] (.75,2.5) to (0,2.5);
	\draw [very thick,directed=.55] (0,2.5) to [out=135,in=0] (-1,3);
	\draw [very thick,directed=.75] (0,2.5) to [out=225,in=0] (-.5,2);
	\draw [very thick,directed=.55] (3.75,3) to [out=180,in=45] (3,2.5);
	\draw [very thick,directed=.75] (4.25,2) to [out=180,in=315] (3,2.5);
	\draw [very thick,directed=.55] (2.25,2.5) to [out=135,in=45] (.75,2.5);
	\draw [very thick,directed=.55] (2.25,2.5) to [out=225,in=315] (.75,2.5);
\end{tikzpicture}\;\colon \quad
\begin{tikzpicture}[anchorbase, scale=.5]
\draw [very thick, ->] (3,0) -- (0,0);
\draw [very thick, ->] (3,1) -- (0,1);
\end{tikzpicture}
\leftrightarrow
\begin{tikzpicture}[anchorbase, scale=.5]
\draw [very thick] (3,0) to (2.9,0) to [out=180,in=315] (2.4,0.5);
\draw [very thick] (3,1) to (2.9,1) to [out=180,in=45] (2.4,0.5);
\draw [double]  (2.4,0.5) -- (.6,0.5);
\draw [very thick, ->] (.6,.5) to [out=135,in=0] (.1,1) to (0,1);
\draw [very thick, ->] (.6,.5) to [out=225,in=0](.1,0) to(0,0);
\end{tikzpicture}
\leftrightarrow
\begin{tikzpicture}[anchorbase, scale=.5]
\draw [very thick] (3,0) to (2.9,0) to [out=180,in=315] (2.4,0.5);
\draw [very thick] (3,1) to (2.9,1) to [out=180,in=45] (2.4,0.5);
\draw [double]  (2.4,0.5) -- (2,0.5);
\draw [very thick] (2,0.5) to [out=135,in=45] (1,0.5);
\draw [very thick] (2,0.5) to [out=225,in=315] (1,0.5);
\draw [double]  (1,0.5) -- (.6,0.5);
\draw [very thick, ->] (.6,.5) to [out=135,in=0] (.1,1) to (0,1);
\draw [very thick, ->] (.6,.5) to [out=225,in=0](.1,0) to(0,0);
\end{tikzpicture}
\end{gather*}
The action on the Clifford ring labels is given by $r_{c^\prime \cdot c}$ (rightward) and $r_{c^\prime \cdot c}$ (leftward) respectively.

In the case of the Reidemeister II move with the signs of the two crossings switched, the non-trivial foams in the chain map acquire a minus sign.

\subsection{Fork slides}
\label{sec:forkslide}
We describe the chain homotopy equivalences corresponding to the following fork slide move:
\[\Kh\left(
\begin{tikzpicture}[anchorbase, scale=.4]
\draw [very thick,<-] (0,0) to [out=0,in=180] (4,2);
\draw [white,line width=.15cm](1,1.5) to [out=0,in=180] (4,.5);
\draw [very thick,<-] (0,2) to [out=0,in=135] (1,1.5);
\draw [very thick,<-] (0,1) to [out=0,in=225] (1,1.5);
\draw [double] (1,1.5) to [out=0,in=180] (4,.5);
\end{tikzpicture}
\right)
\quad \cong \quad 
\Kh\left(
\begin{tikzpicture}[anchorbase, scale=.4]
\draw [very thick,<-] (0,0) to [out=0,in=180] (4,2);
\draw [white,line width=.15cm] (0,2) to [out=0,in=135] (3,0.5);
\draw [white,line width=.15cm] (0,1) to [out=0,in=225] (3,0.5);
\draw [very thick,<-] (0,2) to [out=0,in=135] (3,0.5);
\draw [very thick,<-] (0,1) to [out=0,in=225] (3,0.5);
\draw [double] (3,0.5) to [out=0,in=180] (4,.5);
\end{tikzpicture}
\right)
\]

The cube of resolutions chain complexes for both sides are the following:
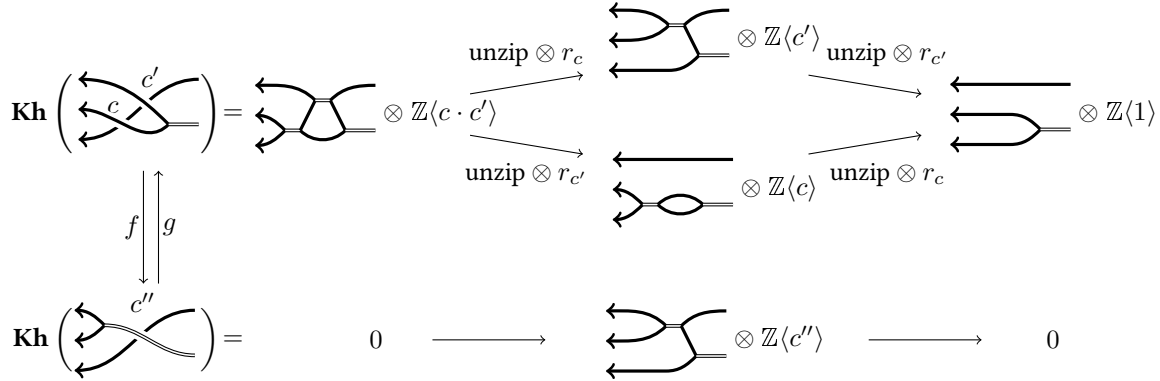
\begin{gather*}
\xy
(11,20)*{
=
};
(0,25)*{
c^\prime
};
(-5,21)*{
c
};
(-4.5,20)*{
\Kh\left(
\begin{tikzpicture}[anchorbase, scale=.4]
\draw [very thick,<-] (0,0) to [out=0,in=180] (4,2);
\draw [white,line width=.15cm] (0,2) to [out=0,in=135] (3,0.5);
\draw [white,line width=.15cm] (0,1) to [out=0,in=225] (3,0.5);
\draw [very thick,<-] (0,2) to [out=0,in=135] (3,0.5);
\draw [very thick,<-] (0,1) to [out=0,in=225] (3,0.5);
\draw [double] (3,0.5) to [out=0,in=180] (4,.5);
\end{tikzpicture}
\right)
};
(30,20)*{
\begin{tikzpicture}[anchorbase, scale=.4]
\draw [very thick,<-] (0,2) to [out=0,in=135] (2,1.5);
\draw [very thick,<-] (0,1) to [out=0,in=135] (1,.5);
\draw [very thick,<-] (0,0) to [out=0,in=225] (1,.5);
\draw [double] (1,.5) to [out=0,in=180] (1.5,.5);
\draw [very thick] (1.5,.5) to [out=315,in=225] (3,.5);
\draw [very thick] (1.5,.5) to (2,1.5);
\draw [double] (2,1.5) to [out=0,in=180] (2.5,1.5);
\draw [very thick] (2.5,1.5) to [out=45,in=180] (4,2);
\draw [very thick] (2.5,1.5) to (3,.5);
\draw [double] (3,.5) to [out=0,in=180] (4,0.5);
\end{tikzpicture}\otimes \Z\la c\cdot c^\prime\ra
};
(75,30)*{
\begin{tikzpicture}[anchorbase, scale=.4]
\draw [very thick,<-] (0,2) to [out=0,in=135] (2,1.5);
\draw [very thick,<-] (0,1) to [out=0,in=225] (2,1.5);
\draw [very thick,<-] (0,0) to (2,0) to [out=0,in=225] (3,.5);
\draw [double] (2,1.5) to [out=0,in=180] (2.5,1.5);
\draw [very thick] (2.5,1.5) to [out=45,in=180] (4,2);
\draw [very thick] (2.5,1.5) to(3,.5);
\draw [double] (3,.5) to [out=0,in=180] (4,0.5);
\end{tikzpicture}
\otimes \Z\la c^\prime\ra
};
(75,10)*{
\begin{tikzpicture}[anchorbase, scale=.4]
\draw [very thick,<-] (0,2) to (4,2);
\draw [very thick,<-] (0,1) to [out=0,in=135] (1,.5);
\draw [very thick,<-] (0,0) to [out=0,in=225] (1,.5);
\draw [double] (1,.5) to [out=0,in=180] (1.5,.5);
\draw [very thick] (1.5,.5) to [out=315,in=225] (3,.5);
\draw [very thick] (1.5,.5) to [out=45,in=135] (3,.5);
\draw [double] (3,.5) to [out=0,in=180] (4,0.5);
\end{tikzpicture}
\otimes \Z\la c\ra
};
(120,20)*{
\begin{tikzpicture}[anchorbase, scale=.4]
\draw [very thick,<-] (0,2) to (4,2);
\draw [very thick,<-] (0,1) to (2,1) to [out=0,in=135] (3,.5);
\draw [very thick,<-] (0,0) to (2,0) to [out=0,in=225] (3,.5);
\draw [double] (3,.5) to [out=0,in=180] (4,0.5);
\end{tikzpicture}
\otimes \Z\la 1 \ra
};
(50,24)*{
\begin{tikzpicture}[scale=.5]
\draw [->] (0,0) to (3,.5);
\end{tikzpicture}
};
(50,28)*{
\textrm{unzip} \otimes r_c
};
(50,16)*{
\begin{tikzpicture}[scale=.5]
\draw [->] (0,0) to (3,-.5);
\end{tikzpicture}
};
(50,12)*{
\textrm{unzip} \otimes r_{c^\prime}
};
(95,24)*{
\begin{tikzpicture}[scale=.5]
\draw [->] (0,0) to (3,-.5);
\end{tikzpicture}
};
(98,28)*{
\textrm{unzip} \otimes r_{c^\prime}
};
(95,16)*{
\begin{tikzpicture}[scale=.5]
\draw [->] (0,0) to (3,.5);
\end{tikzpicture}
};
(98,12)*{
\textrm{unzip} \otimes r_c
};
(-4.5,-10)*{
\Kh\left(
\begin{tikzpicture}[anchorbase, scale=.4]
\draw [very thick,<-] (0,0) to [out=0,in=180] (4,2);
\draw [white,line width=.15cm](1,1.5) to [out=0,in=180] (4,.5);
\draw [very thick,<-] (0,2) to [out=0,in=135] (1,1.5);
\draw [very thick,<-] (0,1) to [out=0,in=225] (1,1.5);
\draw [double] (1,1.5) to [out=0,in=180] (4,.5);
\end{tikzpicture}
\right)
};
(-1,-5)*{
c^{\prime\prime}
};
(11,-10)*{
=
};
(45,-10)*{
\begin{tikzpicture}[scale=.5]
\draw [->] (0,0) to (3,0);
\end{tikzpicture}
};
(30,-10)*{
0
};
(75,-10)*{
\begin{tikzpicture}[anchorbase, scale=.4]
\draw [very thick,<-] (0,2) to [out=0,in=135] (2,1.5);
\draw [very thick,<-] (0,1) to [out=0,in=225] (2,1.5);
\draw [very thick,<-] (0,0) to (2,0) to [out=0,in=225] (3,.5);
\draw [double] (2,1.5) to [out=0,in=180] (2.5,1.5);
\draw [very thick] (2.5,1.5) to [out=45,in=180] (4,2);
\draw [very thick] (2.5,1.5) to(3,.5);
\draw [double] (3,.5) to [out=0,in=180] (4,0.5);
\end{tikzpicture}
\otimes \Z\la c^{\prime\prime}\ra
};
(103,-10)*{
\begin{tikzpicture}[scale=.5]
\draw [->] (0,0) to (3,0);
\end{tikzpicture}
};
(120,-10)*{
0
};
(0,5)*{
\begin{tikzpicture}[scale=.5]
\draw [->] (.2,0) to (.2,3);
\node at (-.5,1.5) {$f$};
\node at (.5,1.5) {$g$};
\draw [<-] (-.2,0) to (-.2,3);
\end{tikzpicture}
};
\endxy
\end{gather*}
Restricted to the web on the top of the cube of resolutions, the chain homotopy equivalence $f$ acts as $\id \otimes r_{c^\prime c^{\prime\prime}}$ and the corresponding component of $g$ is $\id \otimes r_{c^{\prime\prime}\cdot c^\prime}$.
The remaining non-trivial components of $f$ and $g$ are formed by the signed non-trivial foams obtained as the following composition of elementary foams, read rightward (as illustrated) or leftward respectively (reflection of the illustrated foam). 
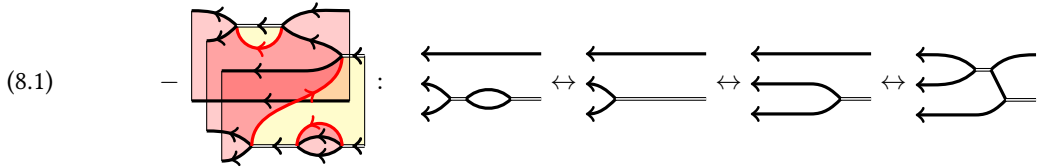
\begin{equation}
\label{eqn:forkslidefoam} - \,
\begin{tikzpicture} [anchorbase,scale=.4,fill opacity=0.2]
	\path [fill=red]  (2.25,.5) to [out=90,in=0] (1.5,1.25) to [out=180,in=90] (0.75,.5) to [out=315,in=225] (2.25,.5);
	\path [fill=red]  (2.25,.5) to [out=90,in=0] (1.5,1.25) to [out=180,in=90] (0.75,.5) to [out=45,in=135] (2.25,.5);
	\path [fill=red] (-1.25,4.5) to [out=135,in=0] (-2.75,5) to (-2.75,2) to (2.5,2) to (2.5,5) to [out=180,in=45] (.25,4.5) to [out=270,in=0] (-.5,3.75) to [out=180,in=270] (-1.25,4.5);
		\path [fill=red] (-1.25,4.5) to [out=225,in=0] (-2.25,4) to (-2.25,1) to [out=0,in=135]  (-.75,.5) to [out=90,in=270]  (2.25,3.5) to [out=135,in=315] (.25,4.5) to [out=270,in=0] (-.5,3.75) to [out=180,in=270] (-1.25,4.5);
		\path [fill=red]  (-.75,.5) to [out=90,in=270]  (2.25,3.5) to [out=225,in=0] (1,3) to (-1.75,3) to (-1.75,0) to [out=0,in=225] (-.75,.5);
		\path [fill=yellow]  (-.75,.5) to [out=90,in=270]  (2.25,3.5) to (3,3.5) to (3,.5)to (2.25,.5) to [out=90,in=0] (1.5,1.25) to [out=180,in=90] (0.75,.5) to (-.75,.5) ;
		\path[fill=yellow] (.25,4.5) to [out=270,in=0] (-.5,3.75) to [out=180,in=270] (-1.25,4.5) to (.25,4.5);
\draw [double,directed=.55] (3,.5) to (2.25,.5);
\draw [very thick,directed=.55] (2.25,.5) to [out=135,in=45] (.75,.5);
\draw [very thick,directed=.55] (2.25,.5) to [out=225,in=315] (.75,.5);
\draw [double,directed=.55] (.75,.5) to (-.75,.5);
\draw [very thick,directed=.55] (-.75,.5) to [out=135,in=0] (-2.25,1);
\draw [very thick,directed=.75] (-.75,.5) to [out=225,in=0] (-1.75,0);
\draw [very thick,directed=.55] (2.5,2) to (-2.75,2);
\draw [red, very thick,directed=.65] (-.75,.5) to [out=90,in=270] (2.25,3.5);
\draw [red, very thick,directed=.65] (2.25,.5) to [out=90,in=0] (1.5,1.25) to [out=180,in=90] (0.75,.5);
\draw [red, very thick, directed=.55] (.25,4.5) to [out=270,in=0] (-.5,3.75) to [out=180,in=270] (-1.25,4.5);
\draw [double,directed=.55] (3,3.5) to (2.25,3.5);
\draw [very thick,directed=.75] (2.25,3.5) to [out=225,in=0] (1,3) to (-1.75,3);
\draw [very thick,directed=.55] (-1.25,4.5) to [out=135,in=0] (-2.75,5);
\draw [double,directed=.55] (.25,4.5) to (-1.25,4.5);
\draw [very thick,directed=.75] (-1.25,4.5) to [out=225,in=0] (-2.25,4);
\draw [very thick,directed=.55] (2.5,5) to [out=180,in=45] (.25,4.5); 
\draw [very thick,directed=.55] (2.25,3.5) to [out=135,in=315] (.25,4.5); 
\draw (-1.75,3) to (-1.75,0);
\draw (-2.25,4) to (-2.25,1);
\draw (-2.75,5) to (-2.75,2);
\draw (3,3.5) to (3,.5);
\draw (2.5,5) to (2.5,2);
\end{tikzpicture}
\;\colon \quad
\begin{tikzpicture}[anchorbase, scale=.4]
\draw [very thick,<-] (0,2) to (4,2);
\draw [very thick,<-] (0,1) to [out=0,in=135] (1,.5);
\draw [very thick,<-] (0,0) to [out=0,in=225] (1,.5);
\draw [double] (1,.5) to [out=0,in=180] (1.5,.5);
\draw [very thick] (1.5,.5) to [out=315,in=225] (3,.5);
\draw [very thick] (1.5,.5) to [out=45,in=135] (3,.5);
\draw [double] (3,.5) to [out=0,in=180] (4,0.5);
\end{tikzpicture}
\leftrightarrow
\begin{tikzpicture}[anchorbase, scale=.4]
\draw [very thick,<-] (0,2) to (4,2);
\draw [very thick,<-] (0,1) to [out=0,in=135] (1,.5);
\draw [very thick,<-] (0,0) to [out=0,in=225] (1,.5);
\draw [double] (1,.5) to [out=0,in=180] (4,.5);
\end{tikzpicture}
\leftrightarrow
\begin{tikzpicture}[anchorbase, scale=.4]
\draw [very thick,<-] (0,2) to (4,2);
\draw [very thick,<-] (0,1) to (2,1) to [out=0,in=135] (3,.5);
\draw [very thick,<-] (0,0) to (2,0) to [out=0,in=225] (3,.5);
\draw [double] (3,.5) to [out=0,in=180] (4,0.5);
\end{tikzpicture}
\leftrightarrow
\begin{tikzpicture}[anchorbase, scale=.4]
\draw [very thick,<-] (0,2) to [out=0,in=135] (2,1.5);
\draw [very thick,<-] (0,1) to [out=0,in=225] (2,1.5);
\draw [very thick,<-] (0,0) to (2,0) to [out=0,in=225] (3,.5);
\draw [double] (2,1.5) to [out=0,in=180] (2.5,1.5);
\draw [very thick] (2.5,1.5) to [out=45,in=180] (4,2);
\draw [very thick] (2.5,1.5) to (3,.5);
\draw [double] (3,.5) to [out=0,in=180] (4,0.5);
\end{tikzpicture}
\end{equation}
The action on the Clifford generators is given by $r_{c\cdot c^{\prime\prime}}$ and $r_{c^{\prime\prime}\cdot c}$ respectively.

\bibliographystyle{plain}

\end{document}